\documentclass{memoir}
\usepackage{mathptmx}
\usepackage{helvet}
\usepackage{courier}
\usepackage{type1cm}

\usepackage{makeidx}         % allows index generation
\usepackage{graphicx}        % standard LaTeX graphics tool
\usepackage{multicol}        % used for the two-column index
\usepackage[bottom]{footmisc}% places footnotes at page bottom

\usepackage{amssymb,amsmath,amsfonts}
\usepackage{enumerate}

\usepackage{bm}

\usepackage[frenchb,english]{babel}
\usepackage{dsfont}
\usepackage{changepage}
\usepackage[T1]{fontenc}
\usepackage[utf8]{inputenc}

\usepackage{amsmath}

\usepackage{amsthm}

\usepackage{bookmark}

\usepackage{numprint}

\usepackage[left]{lineno}

\usepackage{rotating}

\newcommand{\B} {{ \mathcal B }}

\newcommand{\N} {{{\mathbb N}}}
\newcommand{\M} {{ \mathcal M }}

\newcommand{\bpf}{\begin{proof}}
\newcommand{\epf}{\end{proof}}

\newcommand{\ee}{{\rm e}}

\newcommand{\E}{\mathbb{E}}
\newcommand{\ind}{\mathbf{1}}

\newcommand{\D}{\mathbb{D}}
\newcommand{\Z}{\mathbb{Z}}
\newcommand{\R}{\mathbb{R}}

\renewcommand{\S}{\textsc{s}}

\newcommand{\Card}{\mathrm{Card}}

\newcommand{\Var}{\mathrm{Var}}
\newcommand{\Vect}{\mathrm{Vect}}
\renewcommand{\P}{\mathbb{P}}
\newcommand{\pen}{\mathrm{pen}}

\newcommand{\x}{\mathbf{x}}

\newcommand{\norm}[1]{\left\Vert#1\right\Vert}

\def\B{\mathcal{B}}

\def\interior#1{\smash{\mathop{#1}\limits^{\lower1pt\hbox{$\scriptscriptstyle\circ$}}}}
\theoremstyle{plain}
\newtheorem{theorem}{Theorem}[section]
\newtheorem{lemma}[theorem]{Lemma}
\newtheorem{corollary}[theorem]{Corollary}

\newtheorem{definition}[theorem]{Definition}
\newtheorem{proposition}[theorem]{Proposition}
\newtheorem{remark}[theorem]{Remark}

\allowdisplaybreaks

\numberwithin{equation}{section}
\numberwithin{theorem}{section}
\numberwithin{figure}{section}
\numberwithin{section}{chapter}
\numberwithin{table}{section}

\def\be{\begin{eqnarray}}
\def\ee{\end{eqnarray}}
\def\ben{\begin{eqnarray*}}
\def\een{\end{eqnarray*}}

\title{Statistical inference for epidemic processes in a homogeneous community}

\author{Catherine Lar\'edo (with Viet Chi Tran for Chapter 4)}

\begin{document}

\maketitle

This document is the Part IV of the book \textit{Stochastic Epidemic Models with Inference} edited by Tom Britton and Etienne Pardoux \cite{brittonpardoux}. It is written by Catherine Lar\'edo, with the contribution of Viet Chi Tran for the Chapter 4.

 \tableofcontents

%\emph{Acknowledgements:} This research has been supported by the ``Chaire Mod\'elisation Math\'ematique et Biodiversit\'e" of Veolia Environnement-Ecole Polytechnique-Museum National d'Histoire Naturelle-Fondation X. V.C.T. also acknowledges support from Labex CEMPI (ANR-11-LABX-0007-01), GdR GeoSto 3477, ANR Project Cadence (ANR-16-CE32-0007) and ANR Project Econet (ANR-18-CE02-0010). \\

\setcounter{chapter}{0}

\chapter*{Introduction}
\addcontentsline{toc}{chapter}{Introduction}

Mathematical modeling of epidemic spread and estimation of key parameters from data provided much insight in the understanding of public health problems related to infectious diseases. These models are naturally  parametric models, where the present parameters  rule the evolution of the epidemics under study.

Multidimensional continuous-time Markov jump processes {$(\mathcal{Z}(t))$} on $\Z^p$ form a usual set-up for modeling epidemics on the basis of   compartmental approaches as for instance   the  $SIR$-like  (Susceptible-Infectious-Removed)  epidemics (see Part I of these notes and also   \cite{and00IV}, \cite{die13IV}, \cite{kee11IV}).
However, when facing incomplete epidemic data, inference based on {$(\mathcal{Z}(t))$} is not easy to be achieved.

There are different situations where missing data are present.
One situation concerns  Hidden Markov Models, which are in most cases  Markov processes observed with noise. It corresponds for epidemics to the fact that the exact status of all the individuals  within a population are not observed, or that detecting the status has some noise (see \cite{cap05IV} for instance).
Another situation comes from the fact that observations are performed at discrete times. They can also be aggregated  (e.g.\ number of infected per day).
A third case, for multidimensional processes, is that some coordinates cannot be observed in practice.
 While the statistical inference has a longstanding theory for complete data, this is no longer true for many cases that occur in practice.
 Many methods have been proposed to fill this gap  starting from  the Expectation-Maximization algorithm (\cite{dempsterlairdrubinIV}, \cite{kuh04IV}) up to various Bayesian methods (\cite{cau04IV}, \cite{oneillrobertsIV}), Monte Carlo methods (\cite{gil96IV}, \cite{one02IV}), based on particle filtering (\cite{fea08IV}, \cite{fea12IV}), Approximate Bayesian Computation methods (\cite{bea09IV}, \cite{blumIV}, \cite{sis18IV}, \cite{tonietalIV}), maximum iterating filtering (\cite{ion11IV}),  Sequential Monte Carlo or  Particle MCMC (\cite{and10IV}, \cite{dou01IV}), see also the R package POMP (\cite{kin16IV}).
  Nevertheless, these methods do not completely circumvent the issues related to incomplete data. Indeed, as summarized in  \cite{bri16IV},
there are some limitations in practice due to the size of missing data  and to the various tuning parameters to be adjusted. \\

The aim of this part  is to provide some tools to estimate the parameters ruling the epidemic dynamics on the basis of available data.
We begin with a chapter  about inferential methodology  for stochastic processes which is not specific to applications to epidemics but  is the  backbone
 of the various inference methods  detailed in the next chapters of this part.

The methods used to build estimators are  linked with the precise nature of the observations, each kind of observations generating a different statistical problem.
We detail these facts in the first chapter.  We have intentionally omitted in this chapter the additional  problem of noisy observations, which often occurs in practice.
This is another layer which comes on top. It  entails Hidden Markov Models  and State space Models (see  \cite{cap05IV} or  \cite{vanh08IV}) and  also the R-package
Pomp (\cite{kin16IV}).
%  keywords = {Markov processes; hidden Markov model; state space model; stochastic dynamical system; maximum % likelihood; plug-and-play; time series; mechanistic model; sequential Monte Carlo; R},
 % abstract = {Partially observed Markov process (POMP) models, also known as hidden Markov models or state space models, are ubiquitous tools for time series analysis. The R package pomp provides a very flexible framework for Monte Carlo statistical investigations using nonlinear, non-Gaussian POMP models. A range of modern statistical methods for POMP models have been implemented in this framework including sequential Monte Carlo, iterated filtering, particle Markov chain Monte Carlo, approximate Bayesian computation, maximum synthetic likelihood estimation, nonlinear forecasting, and trajectory matching. In this paper, we demonstrate the application of these methodologies using some simple toy problems. We also illustrate the specification of more complex POMP models, using a nonlinear epidemiological model with a discrete population, seasonality, and extra-demographic stochasticity. We discuss the specification of user-defined models and the development of additional methods within the programming environment provided by pomp.},

Chapter \ref{StatMC} is devoted to
%In order to present an overview of the statistical problems, we detail in the next chapter
the statistical inference  for Markov chains. Indeed, discrete time Markov chains models are interesting here because many questions that arise for more complex epidemic models can be illustrated in this set-up.

We had rather focus here on parametric inference since epidemic models always include in their dynamics  parameters that need to be estimated in order to derive predictions.
At the early stage of an outbreak, a good approximation for the epidemic dynamics is to consider that the population of Susceptible is infinite and that
Infected individuals evolve according to a branching process (see   Part I,  Section \ref{sec-early-stage}  of these notes).
We also present in this chapter some classical statistical results in this domain.

As detailed in Part I, Chapter \ref{TB-EP_chap_StochMood},  epidemics in a close population of size $N$ are naturally modeled  by pure jump processes $(\mathcal{Z}^N(t))$.
However, inference for such models requires that all the jumps (i.e.\  times of infection and recovery for the $SIR$ model) are observed. Since these data are rarely available in practice, statistical methods often rely on data augmentation, which allows us to complete the data and add in the analysis all the missing jumps. For moderate to large populations, the complexity increases rapidly, becoming the source of additional problems.
Various approaches were developed during the last years to deal with partially observed epidemics. Data augmentation and likelihood-free methods such as the Approximate Bayesian Computation
(ABC) opened some of the most promising pathways for improvement  (see e.g. \cite{bre09IV}, \cite{mck09IV}). Nevertheless, these methods do not completely circumvent the issues related to incomplete data. As stated also in \cite{bri16IV}, \cite{cau12IV},  there are some limitations in practice, due to the size of missing data and to the various tuning parameters to be adjusted (see also \cite{and00IV}, \cite{one10IV}).

In this context,  it appears that diffusion processes  satisfactorily approximating epidemic dynamics  can be profitably used for inference of model parameters for epidemic data, due  to their analytical power  (see e.g.\ \cite{fuc13IV}, \cite{ros09IV}).
More precisely,   when normalized by $N$, $(Z^N(t)= N^{-1} \mathcal{Z}^N(t))$ satisfies an ODE  as the population size $N$ goes to infinity  and moreover, in the first part of these notes, it is proved that the  Wasserstein $L_1$-distance between  $(Z^N(t))$  and a multidimensional diffusion process with   diffusion coefficient proportional to $1/\sqrt{N}$ is of order $o(N^{-1/2})$ on a finite interval $[0,T]$ (see Part I, Sections \ref{TB-EP_sec_CLT} and
\ref{TB-EP_sec_DiffusApprox}).  Hence, in the case of a major outbreak in a large community, epidemic dynamics can be described  using  multidimensional diffusion processes
$(X^N(t))_{t\geq 0}$ with a small diffusion coefficient proportional to $1/\sqrt{N}$.
%Since epidemics are usually observed over limited time periods, we consider in what follows the parametric inference based on observations of the epidemic dynamics on a fixed interval $[0,T]$.
We detail in Chapter \ref{Diffusions} the parametric inference for epidemic dynamics described  using  multidimensional diffusion processes
$(X^N(t))_{t\geq 0}$ with a small diffusion coefficient proportional to $1/\sqrt{N}$ based on discrete observations. Since epidemics are usually observed over limited time periods, we consider the parametric inference based on observations of the epidemic dynamics on a fixed interval $[0,T]$.

The last  chapter is devoted to the inference for the continuous time $SIR$ model.
We present several algorithms which address the problem of incomplete data in this set-up: Expectation-Maximization algorithm,  Monte Carlo methods and Approximate Bayesian Computation methods.
Finally, all the classical statistical results  required for this part are detailed  in the Appendix.

\chapter{Observations and Asymptotic Frameworks} \label{chap1:intro}
%\chapter{Introduction}
%\label{intro-ch1}

Multidimensional continuous-time Markov jump processes {$(\mathcal{Z}(t))$} on $\Z^p$ form a usual set-up for modeling epidemics on the basis of   compartmental approaches as for instance   the  $SIR$-like  (Susceptible-Infectious-Removed)  epidemics (see Part I of these notes and also   \cite{and00IV}, \cite{die13IV}, \cite{kee11IV}).
However, when facing incomplete epidemic data, inference based on {$(\mathcal{Z}(t))$} is not easy to be achieved.

Assume that a stochastic process  $(\mathcal{Z}(t), t \in [0,T])$ models the epidemic dynamics with parameters associated with this process (transition kernels depending on a parameter $\theta$ for Markov chains, drift and diffusion coefficients for a diffusion process, infinitesimal generator for a Markov pure jump process). The observed process corresponds to the value $\theta_0$ of this parameter. This value $\theta_0$ is called  the true (unknown) parameter value. Our concern
here is the estimation of $\theta_0$ from the observations  that are available
%. Our aim is to build estimators $\hat{\theta}$ of $\theta_0$ depending on the observations
and the study of their properties.
The methods used to build estimators are  linked with the precise nature of the observations, each kind of observations generating a different statistical problem.
We detail these facts in the next sections.  We have intentionally omitted in this chapter the additional  problem of noisy observations, which often occurs in practice.
This is another layer which comes on top. It  entails Hidden Markov Models  and State space Models (see  \cite{cap05IV} or  \cite{vanh08IV}) and  also the R-package
Pomp (\cite{kin16IV}).

\section{Various kinds of observations and  asymptotic frameworks} \label{various}
As developed in  Part I of these notes, the epidemic dynamics is  modeled by a stochastic process $(\mathcal{Z}(t))$ defined on $[0,T]$ with values in $\R^p$, which
 describes at each time $t$ the number of individuals in each of the $p$ health states (e.g.\ $p=3$ for the $SIR$ model).
Inference for epidemic models is complicated by the fact  that collected observations usually do not contain all the information on the whole path of
$(\mathcal{Z}(t),0\leq t\leq T)$. Moreover, the inference method  relies on an asymptotic framework which allows us to control the properties of estimators.
We  detail here  in a general set-up  these facts, which are not specific to the inference for epidemic dynamics, but rely on general properties  of inference for stochastic processes,  this  knowledge being useful for applications to epidemics.

\subsection{Observations} \label{varobs}
Historically, continuous  observation  of $(\mathcal{Z}(t),0\leq t\leq T)$ was systematically assumed in the literature concerning the statistics of continuous time stochastic processes (see  \cite{ibr81IV}, \cite{lip01IV}, \cite{lip201IV}). It is justified by the property that theoretical results can be obtained. However, many various cases can occur in practice. Among them, including the complete case,  the more frequent are \\
\textbf{ Case (a)}. Continuous observation  of $(\mathcal{Z}(t))$ on $[0,T]$.\\
%Moreover, the current trend is nowadays to get more observations corresponding to data  sampled at very close times.
 \textbf{ Case (b)}.  Discrete observations:  $(\mathcal{Z}(t_1),\dots,\mathcal{Z}(t_n))$ with $0\leq t_1<t_2<\dots < t_n \leq T$.\\
\textbf{ Case (c)}. Aggregated observations $(J_0,\dots, J_{n-1}) $ with $J_i=\int_{t_{i}}^{t_{i+1} }\mathcal{Z}(s) ds$.\\
\textbf{ Case (d)}.  Model with latent variables: Some coordinates of $(\mathcal{Z}(t), t \in [0,T])$ are unobserved.\\

Case (a)  corresponds to complete data. For the $SIR$ epidemics, it means that the times of infection and recovery are observed for each individual in the population.
Case (b)  corresponds to the fact that observations are made at successive known times (one observation per day or per week during the epidemic outburst
(see \cite{bjo02IV}, \cite{cau08IV}, \cite{bre09IV}, \cite{cau12IV}).  Case (c) occurs in epidemics when the available observations are the number of Infected individuals and Removed per week  for instance.  Case (d) deals with the fact that, in routinely collected observations of epidemic models, one or several model variables are unobserved (or latent) (see e.g.\ \cite{cap05IV}, \cite{fea08IV} for general references and  \cite{bre09IV}, \cite{bri16IV}, \cite{ion11IV}, \cite{ion06IV},  \cite{oneillrobertsIV}, \cite{tonietalIV}
  for applications to epidemics).

\subsection{Various asymptotic frameworks}\label{varframe}
%From now on, we consider that the $p$-dimensional process $\mathcal{Z} (t)= (\mathcal{Z}_1(t),\dots \mathcal{Z}_p(t))'  $ is observed on the time interval $[0,T]$.\\
Taking into account an asymptotic framework is necessary to study and compare the
properties of different estimators. It is also a preliminary step for the study of non-asymptotic properties.
While for  i.i.d.\ observations, the natural asymptotic framework is that the number $n$ of observations goes to infinity, for stochastic processes various approaches are used according to the model properties or to the available observations.
Two different situations need to be considered according to the time interval of observation $ [0,T]$, where $T$ either goes to infinity or is fixed.
%\pagebreak

\subsubsection{Increasing time of observation $ [0,T]$ with $T\rightarrow \infty$}\label{IncreasT}
 %\begin{enumerate}
%\item  For a continuously observed process: $T \rightarrow \infty$.
If  $(\mathcal{Z} (t))$  on $[0,T]$ is continuously  observed, a general theory is available
 for ergodic processes and for stationary mixing processes.
 Inference can also be performed for
some special models but  does no longer rely on a general theory. This occurs for  supercritical branching processes  and for the explosive $AR(1)$ process.

Let us consider the case of discrete observations of a continuous time process with regular sampling $\Delta$.
The observations are:  $(\mathcal{Z} (t_1),\mathcal{Z} (t_2),\dots,\mathcal{Z} (t_n))$  with $t_i=i\Delta$ and  $T=n\Delta$.\\
%$0\leq t_1 <t_2 ...<t_n\leq T$.
% For sake of simplicity, we consider the regular sampling : $t_i=i\Delta, i=1,\dots,n$ with $T=n\Delta$.\\
%%Assumption: The number of observations $n$ goes to infinity.\\
Two distinct cases arise from the study of parametric inference for diffusion processes \\
\textbf{ (1)}  The sampling interval $\Delta$ is fixed (  $T=n\Delta$ and $n\rightarrow \infty $).\\
\textbf{ (2)}  The sampling interval $ \Delta=\Delta_n \rightarrow 0$ with $T=n\Delta_n\rightarrow \infty$ as $n\rightarrow \infty$.\\
Since the likelihood is not explicit and difficult to compute, it raises many theoretical problems. References for the inference in these  cases are  Kessler \cite{kes97IV},  \cite{kes00IV}  followed by many others \cite{kes12IV}.\\

\noindent
In practice, when a sampling interval $\Delta$ is present in the data collecting, it might be important to take it explicitly into account.
Deciding whether $\Delta$ is small or not depends more on the time scale than on its precise value.
However this parameter $\Delta$ explicitly enters in the estimators,  and some estimators with apparently good properties  for $\Delta$ fixed might explode for small $\Delta$.
It corresponds in theory to different rates of convergence for the various coordinates of the unknown parameter $\theta$
as $n\rightarrow \infty$.
This typically occurs for discrete observations of a diffusion process (see Section \ref{ch1_Illus}).

\subsubsection{Fixed observation time $[0,T]$}\label{FixedT}
Several asymptotic frameworks are used.\\

\noindent(1) \emph{Discrete observations on [0,T] with $T=n\Delta_n$  fixed} \\
The sampling interval $\Delta_n \rightarrow 0$ while the number of observations $n$ tends to infinity.\\
%(more observations in the  same time interval).\\
For diffusion processes, only parameters in the diffusion coefficient can be estimated (see \cite{gen93IV}, \cite{jac12IV}).\\

\noindent(2) \emph{Observation of   $k$ i.i.d.\ sample paths of  $(\mathcal{Z}^i(t),0 \leq t \leq T)$, $i=1,\dots k$ with $k\rightarrow \infty$.}\\
Observations of $(\mathcal{Z}^i(t))$ can be continuous or discrete.This framework is relevant for panel data  which describe for instance the dynamics of several epidemics
in different locations. It allows us to include covariates or additional random effects in the model.
The assumption is that the number of paths  $k$ goes to infinity (see e.g \cite{gut91IV}).\\

\noindent(3) \emph{Presence of a ``Small parameter'' $\epsilon >0: (\mathcal{Z}^{\epsilon}(t),0\leq t\leq T) $, and  $\epsilon \rightarrow 0$.}\\
 Inference is studied in the set-up of  a family of stochastic models $(\mathcal{Z}^{\epsilon}(t), 0\leq t\leq T)$ depending on a parameter    $\epsilon >0$.
 Such a family of processes  naturally appears in the theory of  "Small perturbations of dynamical systems", where $(X^{\epsilon}(t))$ denotes a diffusion process with small diffusion coefficient $\epsilon \sigma(\cdot) $ (see e.g.\ \cite{fre84IV}).
The presence of a small parameter  occurs in the study of epidemics in large closed populations of size $N$, when they are density dependent. The small parameter $\epsilon$ is associated to the population size $N$ by the relation $\epsilon= 1/\sqrt{N}$ leading to the family of processes $\mathcal{Z}^{\epsilon} (t)= \epsilon^2 \mathcal{Z} (t)$ (normalization by the population size of the process). From a probability perspective, we refer to  Part I, Sections
 \ref{TB-EP_sec_CLT} and \ref{TB-EP_sec_DiffusApprox} (see also \cite[Chapter 8]{eth05IV}). For statistical purposes, we investigate in Chapter \ref{Diffusions} of this part
the asymptotic framework  "$\epsilon\rightarrow 0$"  and, for discrete observations,  the cases where the sampling interval $\Delta$ can be fixed or $\Delta=\Delta_n \rightarrow  0$.\\

\noindent(4) \emph{Asymptotics on the initial population number.}\\
 It consists in assuming that one coordinate of $(\mathcal{Z}(t))$ at time $0$ satisfies that  $\mathcal{Z}^i(0)= M\rightarrow \infty$.
 The parametric  inference for the continuous time $SIR$ model is performed in this framework (see  the results recalled in Section  \ref{sec:likelihood} or \cite{and00IV}).
 %as the total population size $N\rightarrow \infty$.
This is also used for  subcritical branching processes where the initial number of ancestors goes to $\infty$ (see e.g.\cite{gut91IV}).

\subsection{Various estimation methods} \label{varest}
As pointed out in the introduction of this part, we are mainly concerned by  the problem of parametric inference. There exist several estimation methods.\\
%\textcolor{red}{RAJOUTER Preliminaires article Cauchemez et al JRSSB interface}

\noindent
\textbf{ Maximum Likelihood Estimation}\\
This entails that one can compute the likelihood of the observation. For a continuously observed process, this is generally possible, but for a discrete time observation of a continuous-time process or for other kinds of incomplete observations, it is often intractable. This opens the whole domain of stochastic algorithms which aim at completing the data in order to estimate parameters with Maximum Likelihood methods. In particular, the well-known Expectation-Maximisation algorithm (\cite{dempsterlairdrubinIV}) and other related algorithms  (see e.g.\ \cite{and10IV}, \cite{kuh04IV}, \cite{oneillrobertsIV}) are based on the likelihood. For regular statistical models, Maximum Likelihood Estimators (MLE) are consistent and efficient (best theoretical variance).\\

%\textcolor{red}{RAJOUTER  Algo types EM ... bases sur la likelihood}

\noindent
\textbf{ Minimum Contrast Estimation or Estimating Functions}\\
 When it is difficult to use the accurate (exact) likelihood, pseudo-likelihoods (contrast functions; approximate likelihoods,..), or pseudo -score functions (approximations of the score function, estimating functions) are often used. When they are well designed, these methods lead to consistent estimators converging at the right rate. They might loose the efficiency property of MLE in regular statistical models (see e.g. \cite{vaa00IV} for the general theory  and \cite{dac93IV}, \cite{hop14IV} for stochastic processes).\\

\noindent
\textbf{ Empirical  and non-parametric Methods}\\
This comprises  all the methods that rely on limit theorems (such as the ergodic theorem) associated with various functionals of the observations. Among these methods, we can refer to Moments methods and  Generalized Moment Methods (see e.g. \cite{vaa00IV} for the general theory and
% \cite{hans95} for moment estimators for discretely and possibly randomly sampled diffusions.
\cite{hans95IV} for discrete observation of continuous-time Markov processes).\\
%Here, inference can be parametric or non-parametric.

\noindent
\textbf{ Algorithmic Methods}\\
Many methods have been developed to perform estimation for incomplete data. It is difficult to be exhaustive. Let us quote  \cite{and10IV}, \cite{dou01IV}  for Particle Markov Monte Carlo methods; \cite{bea02IV}, \cite{blumIV}, \cite{blumtranIV},  \cite{sis18IV}, \cite{tonietalIV} for Approximate Bayesian Computation; \cite{cau04IV}, \cite{mck09IV}  for Bayesian MCMC; \cite{ion11IV}, \cite{kin16IV} for iterated filtering  and the R-package POMP.  In the last chapter of this  part, MCMC and ABC methods are detailed for the $SIR$ model.

\section{An example illustrating the  inference in these various situations}\label{ch1_Illus}
Let us investigate here the consequences of these various situations for the statistical inference
on  a simple stochastic model for describing  a population dynamics: the AR(1)  model which is a simple model for describing  dynamics in discrete time, its continuous time description corresponding  to the Ornstein--Uhlenbeck diffusion process.
 Besides studying a simplified population  model, the main interest of this example lies in the property that  computations are explicit for the various inference approaches
listed in the previous section.

 \subsection{A simple model for population dynamics: AR(1)}\label{subAR1}
The AR(1) model is a classical model for describing population dynamics in discrete time.
%\textbf{ Example 2: A simple model for population dynamics.}\\
On $(\Omega,\mathcal{ F},\P) $  a probability space, let   $(\epsilon_{i})$ be a sequence of  i.i.d.\ random variables on $\R$ with distribution $\mathcal{ N}(0,1)$.
% and $\eta$ a random variable independent of $(\epsilon_i)$.
Consider the autoregressive process on $ \R$ defined, for $i\geq 0$,
\begin{equation}\label{AR1}
X_{i+1} = a X_{i} + \gamma \epsilon_{i+1}, \quad X_0=x_0.
\end{equation}
In order to compare  this model with its continuous  time version, the Ornstein--Uhlenbeck diffusion process,
we assume that $a>0$ and that $ x_0$ is deterministic  and known.
%either a random variable independent of the $(\epsilon_i)$ or
The observations are $(X_i, i=1,\dots,n)$ and  the unknown parameters $(a, \gamma ) \in \ (0,+\infty)^2 $.
The distribution  $\P^n_{a,\gamma}$ of the $n$-tuple $(X_1,\dots,X_n)$ is easy to compute, since the random
variables $(X_i-a X_{i-1}, i=1,\dots,n)$ are independent and identically distributed $\mathcal{ N}(0,\gamma ^2)$. If $\lambda_n$
%Hence the density of $P_{a,\gamma}^n$ with respect to the
denotes the Lebesgue measure on $\R^n$, then
% is, assuming that $X_0=x_0$,
$$\frac{d\P_{a,\gamma}^n}{d\lambda_n}(x_i, i=1,\dots,n)= \frac{1}{(\gamma \sqrt {2 \pi})^n}
\exp(-\frac{1}{ 2\gamma ^2}\sum_{i=1}^n (x_i-ax_{i-1})^2).$$
%Choosing for dominating measure $ \frac{1}{(\sqrt {2 \pi})^n} \lambda_n$ yields that a
Hence, the loglikelihood function is
\begin{equation}\label{loglikAR}
\log L_n(a,\gamma) ={\ell}_n (a,\gamma)= -\frac{n}{2} \log (2\pi)
-\frac{n}{2} \log \gamma^2 - \frac{1}{2\gamma^2} \sum_{i=1}^n (X_i-aX_{i-1})^2.
\end{equation}
The maximum likelihood estimators are
\begin{equation} \label{estimagamma}
	{\hat a}_n = \frac{\sum_{i=1}^n X_{i-1}X_i}{\sum_ {i=1}^n X_{i-1}^2};\quad
	{\hat \gamma^2}_n = \frac{1}{n}\sum_{i=1}^n (X_i-{\hat a}_n X_{i-1})^2.
\end{equation}
The properties of  $({\hat a}_n,{\hat \gamma^2}_n)$ can be studied as  $n \rightarrow \infty$:
$({\hat a}_n,  {\hat \gamma^2}_n )$ is strongly consistent: $$({\hat a}_n,  {\hat \gamma^2}_n ) \rightarrow (a,\gamma^2)\; \;  \mbox{a. s. under } \P_{a,\gamma}  \mbox { as } n\rightarrow \infty.$$
The rates of convergence differ according to the probabilistic properties of $(X_i)$.\\
\textbf{(1)} If $0<a<1$, $(X_i)$ is a Harris recurrent  Markov chain with stationary distribution\\ $\mu_{a,\gamma}(dx) = \mathcal{ N}(0, \frac{\gamma^2}{1-a^2 })$. The estimators  ${\hat a_n},{\hat \gamma}^2_n$ are asymptotically independent and satisfy
\begin{equation}\label{limARagamma}
\begin{pmatrix}\sqrt{n}({\hat a_n} - a)\\ \sqrt{n}( {\hat \gamma}^2_n - \gamma^2) \end{pmatrix} \rightarrow \mathcal{ N}_2 \left(0, \begin{pmatrix} 1-a^2  & 0  \\ 0 & 2 \gamma^4  \end{pmatrix}\right).
\end{equation}
\textbf{(2)} If $a=1$, $(X_i)$ is a null recurrent random walk and  $n ({\hat a_n} - 1)$ converges to a non-Gaussian distribution, while ${\hat \gamma^2}_n$ has the properties of
Case (1).\\
\textbf{(3)} If $a>1$ and $x_0=0$, $(X_i)$ is explosive. One can prove that
$a^n({\hat a_n} - a)$ converges to a random variable $Y= \eta Z $, where  $\eta,Z$ are two independent  random variables, $Z \sim \mathcal{ N}(0,1)$ and $\eta$  is an explicit positive random variable.
 The estimator  ${\hat \gamma^2}_n$ keeps the properties of Case (1).\\
%Details of the proofs are  given in  subsection \ref{AR1details}
%Along this presentation, we recall some classical definitions in statistics.They are well known for i.i.d. observations but we have rather have them recalled in the frame work of dependent data or stochastic processes.
%Enlever?:Modeling the stochastic dynamics of a population can be performed in various ways. We investigate here the consequences of these various models for the statistical inference.
\subsection{Ornstein--Uhlenbeck diffusion process with increasing observation time }\label{subOUlargeT}
This section is based on  Chapter 1 of \cite{gen18IV} where all the statistical inference is detailed.
It is presented here as a starting point for problems that arise when dealing with epidemic data.
In order to investigate  the various situations  detailed in  Section \ref{various}, let us now consider the continuous time version of the $AR(1)$  population model, the  Ornstein--Uhlenbeck diffusion process defined by the stochastic differential equation
%Let $(\xi_t, t\geq 0)$ denote the solution of the stochastic differential equation
\begin{equation}\label{OUsec1}
	d\xi_t= \theta \xi_t dt +\sigma d W_t; \; \xi_0=x_0.
\end{equation}
where
%ote by
$(W_t, t\geq 0) $ denotes a standard Brownian motion on $(\Omega,\mathcal{ F},P)$, and $x_0$  is either deterministic or is a random variable independent of  $(W_t)$.
Then, $(\xi_t, t\geq 0)$ is a diffusion process on $\R$ with continuous sample paths.
This equation can be solved, setting $Y_t=e^{-\theta t} \xi_t$, so that
\begin{equation}\label{OUexpl}
	\xi_t= x_0 e^{\theta t}+ e^{\theta t} \int_0^t  e^{-\theta s} dW_s.
\end{equation}

Let us first consider the case where  $(\xi_t)$ is observed with regular sampling intervals $\Delta$. The observations $(\xi_{t_i};  i=1, \dots,n)$ with $t_i=i\Delta$  satisfy
\begin{equation}\label{OUdisc}
\xi_{t_{i+1}}=  e^{\theta \Delta}\xi_{t_i}+ \sigma e^{\theta (i+1)\Delta} \int_{i\Delta}^{(i+1)\Delta}
e^{-\theta s} dW_s.
\end{equation}
Hence,  $(\xi_{t_{i+1}}- e^{\theta \Delta} \xi_{t_i})$ is independent of $\mathcal{ F}_{t_i}$, where $\mathcal{ F}_{t }= \sigma( \xi_0, W_s, s\leq t)$  and the sequence
$(\xi_{t_i}, i\geq 0)$ is  the autoregressive model $AR(1)$   defined in (\ref{AR1})  setting
\begin{equation}\label{agamma}
X_i= \xi_{t_i}, \quad a=  e^{\theta \Delta}, \quad \gamma^2= \frac{\sigma^2}{2\theta} (e^{2\theta \Delta}-1),
\end{equation}
since the random variables   $((\sigma e^{\theta (i+1)\Delta} \int_{i\Delta}^{(i+1)\Delta}
e^{-\theta s} dW_s), 1\leq i \leq n)$  are independent  Gaussian $\mathcal{ N}(0, \gamma^2)$.\\

\noindent
Cases (1), (2), (3) of the $AR(1)$ are respectively $\{\theta<0\}$, $\{\theta=0\}$ and $\{\theta>0\}$.\\

\noindent\underline{\textbf{ Case (a)} Continuous observation on $[0,T]$.}\\
Let us first start with the parametric inference associated with the complete observation of $(\xi_t)$ on $[0,T]$ \label{subcontOU}.
The space of observations is $(C_T, \mathcal{ C}_T)$,
%$ (C(0,T), \mathcal{ C}_t, 0\leq t \leq T))$,
the space of continuous functions   from $[0,T] $ into $\R$ and $\mathcal{ C}_T$ is the Borel $\sigma$-algebra. associated with the topology of uniform convergence on
$[0,T]$. Let   $\P_{\theta,\sigma^2}$ denote
the probability distribution on $(C_T,\mathcal{ C}_T)$ of the observation $(\xi_t,0\leq t\leq T) $ satisfying (\ref{OUsec1}) . %$P_{\theta_0,\sigma_0}$.
It is well known that if $\sigma^2 \neq {\tau}^2$,
the distributions $\P_{\theta,\sigma^2}$ and  $\P_{\theta,\tau^2}$ are singular on $(C_T, \mathcal{ C}_T)$ (see e.g.\ \cite{lip01IV}).
Indeed, the quadratic variations of $(\xi_t)$
satisfy, as $\Delta_n=t_i-t_{i-1} \rightarrow 0$,
\[\sum_{i=1}^n (\xi_{t_i}-\xi_{t_{i-1}})^2 \rightarrow \sigma ^2 T\mbox{ in }\P_{\theta,\sigma^2}\mbox{-probability}.\]
Therefore,  the set $A = \{\omega, \sum_{i=1}^n (\xi_{t_i}-\xi_{t_{i-1}})^2 \rightarrow \sigma^2 T\}$ satisfies
$\P_{\theta,\sigma^2}(A) =1$ and
 $\P_{\theta,\tau^2}(A) =0$ for  $ \tau^2 \neq.\sigma^2$.\\
 A statistical consequence is that the diffusion coefficient is identified when $(\xi_t)$ is continuously observed.

We assume that  $\sigma $ is fixed and known and omit it in this section.
 Let $\P_{0,\sigma^2}= \P_0$ the distribution
associated with  $\theta=0$ (i.e $ d\xi_t= \sigma dW_t$).
% and let $P_{\theta,\sigma}$ the distribution on $(C_T, \mathcal{ C}_T)$
%of   (\ref{OU})
The Girsanov formula gives an expression of the likelihood function on $[0,T]$,
\begin{equation}\label{LikOU}
	L_T(\theta)= \frac{d\P_{\theta}}{d\P_0}(\xi_t, 0\leq t \leq T)= \exp \left(\frac{\theta }{\sigma^2}\int_0^T \xi_t\; d \xi_t -\frac{\theta ^2}{2\sigma^2}
\int_0^T \xi_t ^2 dt\;\right).
\end{equation}
Substituting  $(\xi_t)$ by its expression in (\ref{OUsec1}), the MLE is
\begin{equation}\label{MLEOU}
	{\hat \theta}_T= \frac{\int_0^T \xi_t d\xi_t}{\int_0^T\xi_t^2 dt}= \theta +{\sigma}  \frac{\int_0^T \xi_t dW_t}
{\int_0^T\xi_t^2 dt}.
\end{equation}
%\underline{Case $T\rightarrow \infty$.}:   T
 The estimator ${\hat \theta}_T$ defined in \eqref{MLEOU}  reads as
% Let us study this estimator as $T \rightarrow \infty$.
%For this, substitute $d\xi_t$ by its expression in (\ref{OU}).
%Set $M_t= \frac{1}{\sigma}\int_0^t  \xi_s dW_s$ . It is a $(\mathcal{ F}_t)$-martingale in $L^2$  with  $\langle M\rangle_t$ is the angle bracket of $(M_t)$ ( i.e.\  the process such that $(M_t^2-\langle M\rangle_t)$ is a martingale).
\begin{equation}\label{MLEbOU}
{\hat \theta}_T= \theta+ \frac{M_T}{\langle M\rangle_T} \mbox{ with } \;  M_t= \frac{1}{\sigma}\int_0^t  \xi_s dW_s.
\end{equation}
where $(M_t)$ is a $(\mathcal{ F}_t)$-martingale in $L^2$  with  angle bracket $\langle M\rangle_t$ (i.e.\  the process such that $(M_t^2-\langle M\rangle_t)$ is a martingale). Noting that $\langle M\rangle_T \rightarrow \infty $  as $T\rightarrow \infty$, the law of large numbers yields that
$\displaystyle{\frac{M_T}{\langle M\rangle_T} \rightarrow 0}$.
 Hence the MLE defined by (\ref{MLEOU}) is consistent.
 %(i.e.\  ${\hat \theta}_T \rightarrow \theta$ in $\P_\theta$-probability).
As for the $AR(1)$- model, the  rate of convergence of ${\hat \theta}_T$ to $ \theta$  depends on the properties of $(M_t)$.
% (which depend on $(\xi_t)$).
Three different cases can be listed as $T\rightarrow \infty$:\\
\textbf{(1)} $\{\theta <0\}$: $(\xi_t)$ is a positive recurrent process with stationary distribution $\mathcal{ N} (0,\frac{\sigma^2}{2|\theta|})$ and $\sqrt{T}({\hat \theta}_T-\theta ) \rightarrow _{\mathcal{ L}} \mathcal{ N}(0,2 |\theta|)$.\\
\textbf{(2)} $\{\theta =0\}$: $(\xi_t)$ is a null recurrent diffusion;  $T {\hat \theta}_T$ converges to a fixed distribution.\\
\textbf{(3)} $\{\theta > 0\}$:  $(\xi_t)$ is a transient diffusion;
  $e^{\theta T} ({\hat \theta}_T-\theta )$ converges in distribution to $Y= \eta\;  Z$, where $\eta,Z$
are two independent random variables, $Z \sim \mathcal{ N}(0,1)$ and $\eta$ is an explicit a positive random variable.\\

\noindent
\underline{\textbf{ Case (b)-1} Discrete observations with sampling interval  $\Delta$ fixed.}\\
Let $t_i=i\Delta, T= n\Delta$ and assume that  the number of observations $n \rightarrow \infty$.\\
Using (\ref{agamma}), $(X_i= \xi_{t_i})$ is an AR(1) with
$a=  e^{\theta \Delta}\;, \gamma^2= \sigma^2  v(\theta) \mbox{ with } v(\theta)= \frac{1}{2\theta} (e^{2\theta \Delta}-1).$

%Set $Y_i= \xi_{t_i}$, then, using (\ref{OUexp}) and the property that $Y_{i+1}- \exp{(\theta \Delta)} Y_i $ is independent of
%$\mathcal{ F}_i= \sigma (\xi_t, t\leq t_i)$ and (\ref{vtheta}),
%\begin{equation}\label{OUdis}
%	Y_{i+1}= \exp{(\theta \Delta)} Y_i + \sigma  \sqrt{v(\theta)} \; \epsilon_{i+1}.
%\end{equation}
%Hence $(Y_i, i=1,\dots,n)$ is an $AR(1)$ process with parameters $(a,\rho)$ with $a= \exp{(\theta \Delta)}$ and $\rho= \sigma \sqrt{v(\theta)}$.
%Hence, a   likelihood function is
 %\begin{equation}\label{Likdis}
%	 L_n(\theta,\sigma)= -\frac{n}{2}\log(\sigma^2 v(\theta))- \frac{1}{2 \sigma^2 v(\theta)}\sum_{i=1}^n (Y_i-\exp(\theta \Delta) \;Y_{i-1})^2.
%	\end{equation}
 %Set $m=(a,\gamma^2)^*$ and
 Let $\phi_{\Delta}:  (0 ,+\infty)^2 \rightarrow  \R \times  (0 ,+\infty) $
 $$\phi_{\Delta} \;:\; m= \begin{pmatrix} a \\ \gamma^2 \end{pmatrix} \rightarrow  \begin{pmatrix} \theta= \frac{\log a }{\Delta}  \\ \sigma^2= \frac{a^2 - 1}{2 \log a} \Delta \gamma^2
 \end{pmatrix}. $$
 %mbox{  with  } \phi_1(a,\gamma^2)= \frac{1}{\Delta} \log a ;  \phi_2(a,\gamma^2)=frac{a^2 - 1}{2 \log a}  \Delta \gamma^2 \
%$$Let us define $\phi_{\Delta}:
%(a, \gamma^2)\rightarrow  (\theta= \frac{1}{\Delta}\log a,\sigma ^2= \frac{\gamma ^2}{v(\theta)}).$
%\rightarrow (a= \exp(\theta \Delta),\gamma^2= \sigma^2 v(\theta).$$
 This is a $C^1 $-diffeomorphism and
% from  $((0,1)\times (0,\infty) $ to  $\R \times (0,\infty)$.
the MLE for $\theta$ and $\sigma^2$ can be deduced from $({\hat a}_n, {\hat \gamma}_n^2)$ obtained in Section
\ref{AR1}. This yields
$${\hat \theta}_n= \frac{1}{\Delta}\log \left(\frac{\sum _{i=1}^n X_{i-1}X_i}{\sum _{i=1}^n X_{i-1}^2}\right); \quad {\hat \sigma}_n^2=
\frac{1}{n}\sum_{i=1}^n (X_i- \exp({\hat \theta}_n \Delta) \;X_{i-1})^2.$$
These two estimators inherit the asymptotic properties  of the maximum likelihood estimators $({\hat a}_n, {\hat \gamma}_n^2)$ obtained in Subsection \ref{subAR1}, their asymptotic variance is obtained using  Theorem \ref{phimle} stated  in the Appendix, Section \ref{Miscstat} (see also \cite{vaa00IV}, Theorem 3.1).
Therefore,  $ ({\hat \theta}_n,{\hat \sigma}^2_n)$ is consistent and, using that
$\displaystyle{a_n( \hat{m}_n-m)}$ converges to a random variable $Y$ yields
\begin{equation}\label{phiestim}
a_n \begin{pmatrix}\hat{\theta}_n- \theta\\ \hat{\sigma}^2_n - \sigma^2 \end{pmatrix} \rightarrow_{\mathcal{ L}} \nabla_x\phi_{\Delta}(m) Y,
\end{equation}
where  $a_n$ is respectively for Cases (1), (2), (3) the matrix
$$ \begin{pmatrix}\sqrt{n}&0\\0&\sqrt{n} \end{pmatrix}, \quad \begin{pmatrix} n&0\\0&\sqrt{n}\end{pmatrix},\quad \begin{pmatrix} e^{n \Delta \theta}&0\\0&
  \sqrt{n} \end{pmatrix}.$$

In particular, for Case (1) where $Y\sim \mathcal{ N}_2(0,\Sigma)$,  the limit distribution   $\mathcal{ N}_2(0, \nabla_x \phi_{\Delta} (m) \Sigma (\nabla_x \phi (m))^*)$ where $\Sigma$  is the matrix obtained in \eqref{limARagamma}.

Looking precisely at the theoretical asymptotic variance of ${\hat \theta}_n$ obtained in \eqref{phiestim},
 %$\nabla  \phi_{\Delta}( a,\gamma^2)$,
 we can observe that, for small  $\Delta$, this  variance  is  $ \frac{2|\theta|}{\Delta} $   and therefore explodes.
It corresponds to the property that $\sqrt{n}$ is not the right rate of convergence of $\theta$ for small  $\Delta$. \\

\noindent\underline{\textbf{ Case (b)-2}  Discrete observations with sampling interval  $\Delta=\Delta_n \rightarrow 0$}\\
We just detail  Case (1), which corresponds to  the ergodic  Ornstein--Uhlenbeck process,  first  studied in \cite{kes97IV}. Under the condition $n\Delta_n^2 \rightarrow 0$, the estimators
${\hat \theta}_n,{\hat\sigma}^2_n $ are consistent and converge at different rates under $\P_{\theta}$,\\
\begin{equation}\label{diffrate}
\begin{pmatrix}\sqrt{n \Delta_n}(\hat{\theta}_n- \theta) \\ \sqrt{n}({\hat \sigma}^2_n -\sigma^2) \end{pmatrix} \overset {\mathcal{ L}}  \rightarrow \mathcal{ N}_2 \left( 0,  \begin{pmatrix} 2 |\theta|& 0 \\
0&2 \sigma^4\end{pmatrix} \right) .
\end{equation}\\

\noindent\underline{\textbf{ Case (c)-1}  Aggregated observations on intervals $[i\Delta,(i+1) \Delta] $ with $\Delta$ fixed.}\\
Assume  now that the available observations are  aggregated data   on successive intervals, $(J_i)$ %i=0, \dots,n-1)$
with
\begin{equation}\label{Ji}
J_i= \int_{t_{i}}^{t_{i+1}} \xi_s \; ds.
\end{equation}
The inference  problem has first  been studied by  \cite{glo00IV},  \cite{glo01IV} for an ergodic stationary  diffusion process. It entails that $\theta <0$ and that  $X_0$ is random, independent of $(W_t, t\geq 0)$, distributed according to the stationary distribution of $(\xi_t)$, $\mathcal{ N}(0,\frac{\sigma^2}{2|\theta|})$.

The process $(J_i)_{i\geq 0}$ is a  non-Markovian strictly stationary centered Gaussian process. Using (\ref{OUexpl}) and (\ref{Ji}), $J_i$ and $J_{i+1}$ are linked by the relation
\begin{align}
J_{i+1}- e^{\theta \Delta} J_i= \frac{\sigma}{\theta }\int_{i\Delta}^{(i+1)\Delta}( e^{\theta\Delta}
&- e^{\theta((i+1)\Delta -s)} )d W_s \label{ARMA}\\
&+ \; \frac{\sigma}{\theta } \int_{(i+1)\Delta}^{(i+2) \Delta}\;(e^{\theta ((i+2)\Delta -s)}-1) d W_s.\nonumber
\end{align}
Hence, for all $ i\geq 1$, $(J_{i+1}- e^{\theta \Delta} J_i )$ is independent of $(J_0,\dots,J_{i-1})$ and  $(J_i)$
possesses the structure of an ARMA(1,1) process, for which the statistical inference
is derived with other tools. Indeed,
%  ItMoreover
% $(J_i)$ is a strictly stationary centered Gaussian process such that
 \begin{eqnarray*}
 \mathrm{Var}(J_i) = \sigma ^2 r_0(\theta) \;& ; &\;  \mathrm{Cov}(J_i,J_j)= \sigma ^2 r_{i-j}(\theta)  \quad \mbox{  with}\\
r_0(\theta)= \frac{1}{\theta^2}\left(\Delta+\frac{1-e^{\theta \Delta}}{\theta} \right)\; &;&\;
r_k(\theta)= -\frac{1}{2\theta^3} e^{-\theta \Delta}(e^{\theta \Delta}-1)^2 \; e^{\theta \Delta |k|} \mbox{ if } k \neq 0.
\end{eqnarray*}
Its spectral density has also an explicit expression, $f_{\theta,\sigma^2}(\lambda)= \sigma ^2 f_{\theta}(\lambda)$.

The likelihood function is known theoretically but its exact expression is intractable.
Instead of the exact likelihood, a well-known method to derive estimators  is to use the Whittle  contrast $U_n(\theta,\sigma^2)$ which provides efficient estimators. It is based  on the periodogram: if ${j}$ denotes now the complex number ${j}^2=-1$,
$$ U_n(\theta,\sigma^2)= \frac{1}{2\pi}\int_{-\pi}^{\pi}\left( \log f_{\theta,\sigma^2}(\lambda)+
\frac{I_n(\lambda)}{f_{\theta,\sigma^2}(\lambda)} \right) d\lambda , \; \mbox { with }\;
I_n(\lambda) = \frac{1}{n}|\sum_{k=0}^{n-1}J_k e^{- {j}k\lambda}|^2. $$
The estimators are then defined as any solution of
$ U_n({ \tilde \theta}_n,{\tilde \sigma}^2_n)= \inf_{\theta,\sigma^2} U_n(\theta,\sigma^2) $.
This yields consistent and asymptotically Gaussian estimators at rate $\sqrt{n}$.  \\

\noindent\underline{\textbf{ Case (c)-2}  Aggregated observations on intervals $[i\Delta,(i+1) \Delta] $ with $\Delta=\Delta_n \rightarrow 0$.}\\
Let us now consider the case of $\Delta=\Delta_n \rightarrow 0, T= n\Delta_n  \rightarrow \infty$ as $ n\rightarrow \infty$.
 Let $J_{i,n} = \int_{i\Delta_n}^{(i+1)\Delta_n} \xi_s ds$.
 Assume that $\theta<0$. The diffusion is positive recurrent with stationary measure $\mu_{\theta,\sigma^2}(dx)\sim \mathcal{ N}(0,\frac{\sigma^2}{2|\theta|})$ . The following two convergences hold in probability (see \cite{glo00IV}).
 \begin{eqnarray*}
 \frac{1}{n} \sum_{i=0}^{n-1}(\Delta_n^{-1} J_{i+1,n}-\Delta_n^{-1} J_{i,n})^2  &\rightarrow & \frac{2}{3}\sigma^2 ,\mbox{  while }\\
\frac{1}{n} \sum_{i=0}^{n-1}( \xi_{(i+1)\Delta_n}- \xi_{i \Delta_n})^2 &\rightarrow &  \sigma^2.
\end{eqnarray*}
Hence,  for small $\Delta_n$, the
heuristics $\frac{1}{\Delta_n} J_{i,n} \sim \xi_{i\Delta_n}$ is too rough and
does not yield good statistical results. The two processes corresponding to these two kinds of observations are structurally distinct:  $(\xi_{i\Delta_n})$ is an  AR(1) process while
$(\frac{1}{\Delta_n} J_{i,n} )$ is ARMA(1,1). Ignoring this can lead to  biased estimators.

\subsection{ Ornstein--Uhlenbeck diffusion with fixed observation time}\label{OUfixed}
\underline{\textbf{ Case (a)} Continuous observation on $[0,T]$}\\
As in Section \ref{subOUlargeT} Case \textbf{ (a)},  the parameter $\sigma^2$ is identified from the  continuous observation of $(\xi_t)$.
Therefore we assume that $\sigma^2$ is known.  The expression for the likelihood \eqref{LikOU} holds.  We get that, without additional assumptions, as for instance the presence of a small parameter $ \epsilon$, the MLE given in \eqref{MLEOU}
 $\hat{\theta}_T$  has a fixed distribution.
 On a fixed time interval, parameters in the drift term of a diffusion cannot be consistently estimated.\\
%\pagebreak

\noindent\underline{\textbf{ Case (b)-1} Discrete observations with fixed sampling $\Delta$}\\
The number of observations $n$ is fixed. Without additional assumptions, neither $\theta$ nor $\sigma^2$ can be consistently estimated.\\

\noindent\underline{\textbf{ Case (b)-2}  Discrete observations with sampling $\Delta_n \rightarrow 0$}\\
Let $\Delta=\Delta_n=T/n \rightarrow 0$ as  $n\rightarrow \infty$.
%\underline{\textbf{ Case (1)} Discrete observations}\\
%Discrete observations  $ (\xi_{t_{i}}, i=0,\dots n)$ with  $t_i=i \Delta_n, Delta=\Delta_n=T/n \rightarrow 0$ as  $n\rightarrow \infty$.}\\
Equation (\ref{OUdisc}) holds and  (\ref{loglikAR}) is the likelihood.  The maximum likelihood estimator  ${\hat\theta}_n$ satisfies
\begin{equation}
	{\hat \theta}_n= \frac{1}{\Delta_n} \log \left( 1+ \Delta_n \frac{\sum_{i=1}^n \xi_{t_{i-1}}( \xi_{t_i}-\xi_{t_{i-1}})}
	{\Delta_n \sum_{i=1}^n \xi_{t_{i-1}}^2}\right).
\end{equation}
Since  $t_i=i \frac{T}{n}$, using the property of stochastic integrals and the Lebesgue integral  yields that, under $\P_{\theta}$,
$$ \sum_{i=1}^n \xi_{t_{i-1}}( \xi_{t_i}-\xi_{t_{i-1}}) \rightarrow \int_0^T \xi_s d\xi_s \mbox{ in probability};\quad
\sum_{i=1}^n \Delta_n \;\xi_{t_{i-1}}^2 \rightarrow \int_0^T\xi_s^2ds \mbox{ a.s.}
$$
Therefore, as $n\rightarrow \infty$, ${\hat \theta}_n$ converges to the  random variable
$\theta_T= \frac{\int_0^T \xi_s d\xi_s}
{\int_0^T\xi_s^2ds}.$
Hence ${\hat \theta}_n$ is not consistent.
Note that $\theta_T$ is precisely the MLE for $\theta$  obtained for continuous observation, which possesses good properties only if $T \rightarrow \infty$.\\
 The story is different for the estimation of $\sigma^2$.
 The normalized  quadratic variations of $(\xi_t)$ is a consistent estimator of $\sigma^2$ and
  $\sum (\xi_{t_i}-\xi_{t_{i-1}})^2  \rightarrow \sigma ^2 T$ in probability. Moreover,
\begin{equation}\label{quadsig}
	\tilde{\sigma}^2= \frac{1}{T}  \sum_{i=1}^n (\xi_{t_i}-\xi_{t_{i-1}})^2  \mbox{ satisfies that } \sqrt{n}({\tilde \sigma}^2- \sigma^2) \rightarrow  \mathcal{ N}(0, 2 \sigma^4).
\end{equation}
%It satisfies that , under $\sqrt{n}({\tilde \sigma}^2- \sigma^2) \rightarrow  \mathcal{ N}(0, 2 \sigma^4)$.\\
Note that this result holds whatever the value of $\theta$.\\

\noindent\underline{\textbf{ Case (c)-1}  Aggregated observations on intervals $[i\Delta,(i+1) \Delta] $ with $\Delta$ fixed}\\
As in Case (b)-1, $\theta$ and $\sigma^2$ cannot be  consistently estimated.\\

\noindent\underline{\textbf{ Case (c)-2}  Aggregated observations on intervals $[i\Delta,(i+1) \Delta] $ with $\Delta=\Delta_n \rightarrow 0$}\\
This has been studied in \cite{glo00IV}.
 Then,  as  $\Delta_n\rightarrow 0$, in probability,
 \begin{equation*}
  \sum_{i=0}^{n-1}(\Delta_n^{-1} J_{i+1,n}-\Delta_n^{-1} J_{i,n})^2  \rightarrow  \frac{2}{3}\sigma ^2 T \;\mbox{ while } \;\sum_{i=0}^{n-1}( \xi_{(i+1)\Delta_n}- \xi_{i \Delta_n})^2 \rightarrow  \sigma^2 T .
 \end{equation*}
% \begin{eqnarray*}
% \sum_{i=0}^{n-1}(\Delta_n^{-1} J_{i+1,n}-\Delta_n^{-1} J_{i,n})^2  &\rightarrow & \frac{2}{3}\sigma ^2 T  \mbox{ in  probability,  while }\\
%\sum_{i=0}^{n-1}( \xi_{(i+1)\Delta_n}- \xi_{i \Delta_n})^2 &\rightarrow & \sigma^2 T  \mbox{ in  probability }.
%\end{eqnarray*}\\
 Here again, the heuristics  $\frac{1}{\Delta_n} J_{i,n} \sim \xi_{i\Delta_n}$ is too rough and
does not yield good statistical results.
%\pagebreak

\subsection{Ornstein--Uhlenbeck diffusion with small diffusion coefficient}\label{OUsmalldiff}
This  asymptotic framework is "$\epsilon \rightarrow 0$".  It naturally occurs for diffusion approximations of epidemic processes.
The equation under study is now
\begin{equation} \label{OUeps}
d\xi_t= \theta \xi_t dt + \epsilon \sigma dW_t \; \; \xi_0=x_0.
\end{equation}
We detail the results for fixed observation time $[0,T]$. \\

\noindent\underline{\textbf{ Case (a)}  Continuous observation on $[0,T]$} \\
As before, we assume that $\sigma^2$ is known and omit it.
Let $\P_{\theta} ^{\epsilon}$  the distribution on $(C_T, \mathcal{ C}_T)$ of $(\xi_t)$  satisfying \eqref{OUeps}. The likelihood is now
\begin{equation}\label{LikOUeps}
	L_{T,\epsilon} (\theta)= \frac{d\P_{\theta} ^{\epsilon}}{dP_{0} ^{\epsilon}}(\xi_s, 0\leq s\leq T)= \exp (\frac{\theta }{\epsilon^2 \sigma^2}\int_0^T \xi_s\; d \xi_s -\frac{\theta ^2}{2\epsilon^2 \sigma^2}
\int_0^T \xi_s ^2 ds\;).
\end{equation}
\begin{equation}\label{MLEOUeps}
	{\hat \theta}_{T,\epsilon}= \theta +\epsilon {\sigma}  \frac{\int_0^T \xi_t dW_t}
{\int_0^T\xi_t^2 dt}.
\end{equation}

Therefore $ {\hat \theta}_{T,\epsilon}  \rightarrow  \theta$ in probability under $P_{\theta} ^{\epsilon}$ as $\epsilon\rightarrow 0$.
Moreover, using results of \cite{kut84IV}),
$$\epsilon^{-1} ({\hat \theta}_{T,\epsilon}- \theta) \rightarrow _{\mathcal{ L} } \mathcal{ N}(0, \tau ^2),\quad \mbox {with }  \tau^{2}= \frac{2 \theta \sigma^2}{ x_0 ^2 (e^{2\theta T}-1)}. $$

\noindent\underline{\textbf{ Case (b)-1} Discrete observations with fixed sampling interval $\Delta$} \\
% with $\Delta=\Delta_n=T/n \rightarrow 0$ as  $n\rightarrow \infty$.}\\
 If $\Delta$ is fixed, only $ \theta$ can be consistently  estimated (see \cite{guy14IV}).
 This is detailed in Chapter \ref{Diffusions}, Section \ref{LFO}.
 Setting $a= e^{\theta \Delta}$, and $X_i=\xi_{i\Delta}$, then, using \eqref{AR1},
 $$X_i=a X_{i-1}+ \epsilon \gamma \eta_{i}, \quad \mbox{where} \quad \gamma^2= \frac{e^{2 \theta \Delta} -1}{2 \theta} \sigma^2  ,$$
and  $(\eta_i) $ i.i.d.\ $\mathcal{ N}(0,1)$
 random variables.  Using \eqref{loglikAR} and \eqref{estimagamma} yields
 $$\hat{a}_{ \epsilon,\Delta}=  a + \epsilon \gamma \frac{\sum_{i=1} ^n  X_{i-1} \eta_i} {\sum_{i=1}^n X_{i-1}^2}.$$
 Therefore, as $\epsilon\rightarrow 0$, $\hat{a}_{ \epsilon,\Delta}$ is consistent and
 $$\epsilon^{-1}( \hat{a}_{ \epsilon,\Delta}- a) \rightarrow \mathcal{ N} (0,V_\Delta), \quad \mbox{with }
 V_{\Delta}= \gamma^2 \frac{ e ^{2\theta \Delta}-1}{x_0^2 (e^{2\theta T}-1)}= \sigma^2 \frac{ (e ^{2\theta \Delta}-1)^2 }{2 x_0^2 \theta (e^{2\theta T}-1)} .$$
 Note that for small $\Delta$, $V_{\Delta} \sim \frac{\Delta}{x_0^2 (e^{2\theta T}-1)}\sigma^2$. \\
%\pagebreak

\noindent\underline{\textbf{ Case (b)-2} Discrete observations with sampling  $\Delta=\Delta_n\rightarrow 0$} \\
 This was first studied in \cite{glo09IV}, \cite{sor03IV}  and is detailed in Chapter \ref{Diffusions}.
 Let $T=n\Delta_n$ (the number of observations $n \rightarrow \infty $ as $\Delta_n\rightarrow 0$).
 Both $\theta$ and $\sigma$ can be estimated from discrete observations. One can prove that they converge at different rates: under $\P_{\theta}$
 as $\epsilon\rightarrow 0,n\rightarrow \infty$,
\begin{equation} \begin{pmatrix}\epsilon^{-1} (\hat{\theta}_{\epsilon,n}-\theta\\
\sqrt{n}( \hat{\sigma}^2_n-\sigma^2 \end{pmatrix}  \rightarrow  \mathcal{ N}_2\left( 0, \begin{pmatrix}\frac{2 \theta \sigma^2}{ x_0 ^2 (e^{2\theta T}-1)}& 0\\0& 2 \sigma^4
\end{pmatrix} \right).
\end{equation}

\subsection{Conclusions}
This detailed example based on the Ornstein--Uhlenbeck  diffusion studied under various asymptotic frameworks and various kinds of observations shows that, before estimating parameters ruling the process under study, one has to carefully consider how the available observations are obtained from the process and to study
their properties. Some approximations are  relevant and keep good statistical properties, while other ones lead to estimators which are not even consistent.

\chapter{Inference for Markov Chain Epidemic Models}\label{StatMC}
%\chapter{Introduction}
In order to present an overview of the statistical problems, we first detail
the statistical inference  for Markov chains. Indeed, discrete time Markov chains models are interesting here because many questions that can arise for more complex models can be illustrated in this set-up. Moreover, continuous-time stochastic models are often observed in practice at discrete times, which might sum up to
a Markov chain model. Therefore, this point of view allows us to illustrate some classical statistical methods for stochastic models used in epidemics.
We have rather focus here on parametric inference since epidemic models always include in their dynamics  parameters that need to be estimated in order to derive predictions.
A recap on parametric inference for Markov chains is given in the Appendix, Section \ref{recapMC}, together with some notations and basic definitions.
We apply  in this chapter these results on
%on some models used for modeling epidemics.   We detail here  some usual inference methods for stochastic models for
some classical stochastic models used in epidemics (see  Part I, Chapter \ref{TB-EP_chap_StochMood} and also \ \cite{and00IV}, \cite{die13IV}).

\section{Markov chains with countable state space}\label{MCcount}
Markov chain models occur  when assuming  that a latent period of fixed length follows the receipt of infection by any susceptible.
According to the epidemic model, the state space of the Markov chain can be finite if the epidemics takes place in a fixed finite population, countable
(birth and death processes,  branching processes, open Markov Models detailed in Part I, Chapter \ref{TB-EP_chap_OpenMarkovMod} of these notes), or continuous
 (see e.g.\ the simple AR(1) dynamic model).\\

 Let us first consider a  Markov chain $(X_n)$ with finite state space  $E=\{0,\dots, N\}$ and  transition matrix
 $(Q(i,j), i,j \in E)$. Assume that $X_0=x_0$ is deterministic and known.
 Our aim is to  estimate  the transition matrix $Q$, which corresponds to $ q= N(N+1)$ parameters since, for all $i \in E$,
$\sum_{j=0}^N Q(i,j)=1 $.\\
Following the definitions recalled in  Section \ref{ApStatMC} in the Appendix, denote by  $\P_{Q}$ the distribution on $(E^{\N},\mathcal{ E}^{\N})$ of $(X_n)$ and
 $\mathcal{ F}_n= \sigma(X_0,\dots,X_n)$.
Let  $\mu_n=  \otimes_{k=1}^n \nu_ k$ with $\nu_k (\cdot)$  the measure on $E$ such that  $\nu_k (i)=1$ for $ i \in E$.

For $A$ a subset of $E$, let $\delta_{A}(\cdot)$ denote  the  Dirac function: $\delta_{A}(x) =1$ if $ x \in A$,  $\delta_{A}(x) =0$ if $x \notin A$.
Define
\begin{equation}\label{Nkl}
N^{ij}_n= \sum_{k=1}^n \delta_{\{i,j\}}(X_{k-1},X_k); \quad N_n^{i.}= \sum_{k=1}^n \delta_{\{i\}}(X_{k-1}).
\end{equation}
%The  random variables $ (N^{ij}_n, k \neq l) $ are  equal to the number of transitions from $k$ to $l$ up to time $n$ and $N_n^{k.}$ is the time spent
%in state $k$ up to time $n$.
Using \eqref{Nkl}, the likelihood   and the loglikelihood read as
\begin{equation}\label{LnQ}
L_n(Q)=  \frac{d \P_{Q}}{d\mu_n}(X_k,k=1,\dots,n)= \prod_{k=1}^n \; Q(X_{k-1},X_k)=  \prod_{i,j \in E } \; Q(i,j)^{N_{ij}^n}\; ,
\end{equation}
\begin{equation}\label{lnQ}
 \ell_n(Q)=  \sum_{i,j \in E} N_n^{ij} \; \log Q(i,j).
 \end{equation}
The computation of the Maximum Likelihood Estimator, $(\hat{Q}_n(i,j,), i,j \in E)$, corresponds to the maximization of $\ell_n(Q)$ under the $(N+1)$ constraints
$ \{\sum_{j=0}^N Q(i,j)\; -1=0 \}$ .
This yields that
\begin{equation}\label{MLEQ}
\hat{Q}_n(i,j)= \frac{N_n^{ij}}{N_n^{i.}}.
\end{equation}
Since the  random variables $ (N^{ij}_n, i \neq j) $ are  equal to the number of transitions from $i$ to $j$ up to time $n$ and $N_n^{i.}$ is the time spent
in state $i$ up to time $n$,  the estimators  $\hat{Q}_n(i,j)$ are equal to the empirical estimates of the transitions.

To study the properties of the MLE, we assume
\begin{enumerate}[(H1)]
\item[\textbf{ (H1)}]  The Markov chain $(X_n)$ with transition matrix  $Q$  is positive recurrent aperiodic on $E$.
\end{enumerate}

Denote by   $\lambda_Q(\cdot)$  the stationary distribution of $(X_n)$. Then, the following holds.
 \begin{proposition}\label{TCLMCF}
 Under (H1), the MLE $(\hat{Q}_n(i,j), i,j \in E)$ is strongly consistent and, under $\P_Q$,
	 $$  \quad \sqrt{n}\left(\hat{Q}_n(i,j)- Q(i,j)\right) _{0\leq i \leq N,0\leq j \leq N-1} \rightarrow _{\mathcal{ L}} \mathcal{ N}_q(0, \Sigma) \mbox{  with } q=N(N+1),$$
$$\Sigma_{ij,ij}=  \frac{Q(i,j)(1-Q(i,j))}{\lambda_Q(i)};\quad  \Sigma_{ij, ij^{\prime}}= - \frac{Q(i,j) Q(i,j^{\prime})}{\lambda_Q(i)};\quad\Sigma_{ij,i^{\prime}j^{\prime}}=0  \mbox{ if } i^{\prime} \neq i.$$
\end{proposition}

\begin{proof}
Under (H1), successive applications of the ergodic theorem yield that, almost surely under $\P_Q$,\\
$\frac{1}{n} N_n^{ij} \rightarrow \lambda_Q(i)Q(i,j) $, $ \frac{1}{n}N_n^{i.} \rightarrow \lambda_Q(i)$ so that  $\hat{Q}_n(i,j) \rightarrow  Q(i,j)$. \\
Let us study  $(\hat{Q}_n(i,j)- Q(i,j))$.  For $0\leq i\leq N,0 \leq j \leq N-1$, define
\begin{equation}\label{Ykl}
Y^{ij}_k=  \left(\delta_{\{j\}}(X_{k})- Q(i,j)\right)  \delta_{\{i\}}(X_{k-1}), \quad  M_n^{ij}= \sum_{k=1}^n Y^{ij}_k .
\end{equation}
%\pagebreak

Then
\begin{equation}\label{QnQ}
\hat{Q}_n(i,j)- Q(i,j)= \frac{N_n^{ij}-Q(i,j)N_n^{i.}}{N_n^{i.}} = \frac{M_n^{ij}}{N_n^{i.}}= \frac{\sum_{k=1}^n Y_k^{ij}}{N_n^{i.}}.
\end{equation}
 %with  $M_n^{kl}= N_n^{kl}-Q(k,l)N_n^{k.} = \sum_{i=1}^{n} {\mathbf 1}_{X_{i-1}=k} ({\mathbf 1}_{X_{i}=l}- Q(k,l))$,
Clearly, $E_Q(Y_k^{ij}|\mathcal{ F}_{k-1})= 0$  and $(M_n^{ij})$ is a centered $ \mathcal{ F}_n$-martingale with values in $\R^q$.
%($q=N(N+1)$).
% Define for $k=1,\dots K, l=1\dots K-1$, the variables  with\\
%$Y_p = \1_{\{X_{p-1}=i,X_p=j\}}-Q(i,j) \1_{\{X_{p-1}=i\}} \Rightarrow
%(M_n)$ centered $L^2$ martingale\\
Its angle bracket  is  the random matrix $\langle M\rangle_n$  with  indices $(ij),(i'j')$ \\
$$ \langle M\rangle_n^{ij,i'j'}=\sum_{k=1}^n  E_{Q}(Y_k^{ij} Y_i^{i'j'} | \mathcal{ F}_{k-1}).$$
Straightforward computations yield that
\begin{align*}
E_{Q}(Y_k^{ij} Y_k^{ij} |\mathcal{ F}_{k-1})) &=  Q(i,j)(1-Q(i,j))  \delta_{\{i\}} (X_{k-1}), \\
E_{Q}(Y_k^{ij} Y_k^{ij'} |\mathcal{ F}_{k-1})) &= - Q(i,j)Q(i,j') \delta_{\{i\}} (X_{k-1})\mbox{ if } j' \neq j\mbox{ and} \\
E_{Q}(Y_k^{ij} Y_i^{i'j'} | \mathcal{ F}_{k-1}) &= 0\mbox{ if }i'\neq i.
\end{align*}
%if $ j' \neq j ,\; E_{Q}(Y_k^{ij} Y_k^{ij'} |\mathcal{ F}_{k-1})) = - Q(i,j)Q(i,j') \delta_{\{i\}} (X_{k-1}) $ and \\
% if $i'=i ,j'=j$, $    E_{Q}(Y_k^{ij} Y_k^{ij} |\mathcal{ F}_{k-1}))=  Q(i,j)(1-Q(i,j))  \delta_{\{i\}} (X_{k-1}) $.
%
Define the $q$-dimensional matrix  $ J_Q$ by
\begin{align*}
J_Q(ij,ij) &= Q(i,j) (1-Q(i,j)) \lambda_Q(i),\\
J_Q(ij,ij') &=- Q(i,j)Q(i,j') \lambda_Q(i)\mbox{ for }j' \neq j\mbox{ and}\\
J_Q (ij,i'j') &=0\mbox{ if }i'\neq i.
\end{align*}
Then, the ergodic theorem yields that $\frac{1}{n}\langle M\rangle_n ^{ij,i'j'} \rightarrow J_Q(ij,i'j')$  a.s.\ under  $\P_Q$.\\
Applying the Central Limit Theorem  for multidimensional martingales (see Appendix, Section \ref{TCLMartmult} )
%Section  \ref{LimitTheo})
yields that, under  $\P_Q$,
 $\frac{1}{\sqrt n} M_n \rightarrow \mathcal{ N}(0,  J_Q)$ in distribution.
Finally,  using that $\frac{1}{n} N_n^{i.} \rightarrow \lambda_Q(i)$ a.s.,  an application of Slutsky's lemma to \eqref{QnQ} achieves the proof  of Proposition \ref{TCLMCF}.
\end{proof}

%\section{Greenwood  and Reed--Frost epidemic models} \label{GrRF}
\subsection{Greenwood model} \label{Greenwood}
This is a  basic model which was introduced by Greenwood \cite{greenwood31IV} to study measles epidemics in United Kingdom.
It  is an $SIR$ epidemic in a finite population of size $N$.
The latent period is fixed and equal $1$ with infectiousness confined to a single time point.   At the moment of infectiousness of any given infective,
the chance of contact with any specified susceptible, sufficient or adequate to transmit the infection is $ p=1-q$. Infected
%the chance of contact with any specified susceptible, sufficient or adequate to transmit the infection is$p=1-q$.
individuals are removed from the infection chain. At time $0$,  assume that the number of Susceptible $S_0$ and Infected $I_0$  verify $S_0+I_0=N$.\\
 Denote by $S_n,I_n$ the number of Susceptible and Infected at time $n$. Then,  for all $n\geq 0$,
\begin{equation}\label{GrenSI}
S_n= I_{n+1}+S_{n+1},
\end{equation} %the number of Susceptible %Therefore,
and, at each generation the actual number of new cases has a Binomial distribution depending on the parameter $p$.
 %After it, susceptible individuals can be infected with probability $p$ ($0<p<1$).
In the Greenwood model, the chance of a susceptible of being infected depends only on the presence of some infectives and not on their actual number. Hence,  if $I_n= 0$, the epidemic terminates immediately since there is no further infectives. If $I_n\geq1$,
  %the distribution of $I_{n+1}$ given $(S_n,I_n)$ is $Bin(S_n,p)$ and  $S_1= S_0-I_1$. \\ Similarly,
  the conditional distribution of $I_{n+1}$ given the past
$\mathcal{ F}_n=\sigma ((S_i, I_i), i=0,\dots n)$ is
$$\mathcal{ L}( I_{n+1}|\mathcal{ F}_n)=\mathrm{Bin}(S_n,p) \quad \mbox{and } S_{n+1}= S_n- I_{n+1}.
$$
%P( I_{n+1}=i_{n+1}| S_n=s_n,I_n=i_n) = \frac{s_n!}{s_{n+1}!i_{n+1}!} p ^{i_{n+1}} q ^{s_{n+1} }.
The process keeps going on up to the time where there is no longer Infected in the population.
%Therefore, if $\mathcal{ F}_n=\sigma (S_i, i=0,\dots n)$.
Noting that  $\mathcal{ F}_n=\sigma (S_i, i=0,\dots n)$, $(S_n)$ is a Markov chain on $\{0,\dots,S_0\}$ with transition matrix
% \begin{equation}\label{Qgreen}
%\P((S_{n+1}= s_{n+1}| \mathcal{ F}_n)=\begin{pmatrix} s_n \\ s_n-s_{n+1} \end{pmatrix} \; p^{s_n-s_{n+1}} (1-p)^{s_{n+1}} .
%\end{equation}
\begin{equation}
Q_p(i,j)= \begin{pmatrix}i\\{i-j}\end{pmatrix} p^{i-j} (1-p)^{j} \mbox { if }
0 \leq j\leq i \leq S_0; \quad
Q_p(i,j)= 0  \mbox{ otherwise.}
\end{equation}
% \mbox{   if }
%	 i_{n+1}+s_{n+1}= s_n, \\
%&=& 0  \mbox{ otherwise.}
%\P(I_{n+1}= s_n- s_{n+1}| S_n=s_n)
 %s_{n+1}>s_n  \mbox{  or if  } i_{n+1}+s_{n+1} \neq s_n,\\
%Note that in this model $\mathcal{ F}_n=\sigma (S_i, i=0,\dots n)$ and that
%$(S_n)$ is a Markov chain on $\{0,\dots,S_0\}$ with transition
%\noindent
%\underline{Some relevant questions}: The successive numbers of susceptibles $(s_0,s_1,\dots,s_n)$ have been observed.\\
%- Is it possible to  estimate $p$ on the basis of these observations?\\
%- What is the duration of the epidemics (which depends on $p$)?\\
\underline{Parametric inference} \\
Assume that the successive numbers of Susceptible $(s_0,s_1,\dots,s_n)$ have been observed up to time $n$.
In this model,  $(S_n)$ decreases  with $n$, and  extinction occurs after a geometric number of generations. Therefore,
the inference framework is to assume that $S_0$ (hence $N$) $\rightarrow \infty$.  \\
Let $\P_p$ the probability associated to the Markov chain with  transition $Q_p$ and initial condition $s_0$.
The likelihood associated with parameter $p$ and observations $(s_1,\dots,s_n)$ is, if
$s_0\geq s_1 \dots \geq s_n$,
\begin{equation}
L_n(p;s_1,\dots,s_n) =
 \prod_{k=1}^{n} \P_p(S_{k}=s_{k}| S_{k-1}=s_{k-1})=C(s_0,\dots,s_n )p^{s_0-s_n} (1-p)^{\sum_{k=1}^n s_n}.
\end{equation}

%\begin{eqnarray*}
%L_n(p;s_0,s_1,\dots,s_n) &=& \P_p(S_0=s_0,\dots, S_n=s_n)\\
%  &=& \P_p(S_0=s_0)\Pi_{k=1}^{n} \P_p(S_{k}=s_{k}| S_{k-1}=s_{k-1})\\
%%&=& \P_p(S_0=s_0)\Pi_{n=0}^{N-1}C_{s_n}^{s_n-s_{n+1}} p ^{s_n-s_{n+1}}(1-p)^{s_{n+1}}\\
%&=& C(s_0,\dots,s_n )p^{s_0-s_n} (1-p)^{\sum_{k=1}^n s_n}.
%\end{eqnarray*
All the quantities independent of $p$  have been gathered in the term $C(s_0,,\dots,s_n)$. They depend on the model and the observations, and therefore have no influence on the estimation of $p$.
%Setting $\ell_n(p; s_0,\dots,s_n) =\log L_n(p;s_0,\dots,s_n)= \ell_n(p)$,
Elementary computations yield that the value of $p$ that maximizes the likelihood is
$$\hat{p}_n= \frac{s_0-s_n}{\sum_{k=0}^{n-1} s_k}=  \frac{1}{ \mbox{``mean time to infection''}}.$$
%If we assume that we observe up to extinction time $T$, then $\displaystyle{\hat{p}_T= $.

Another approach for estimating parameters of a stochastic process is the Conditional Least Squares (CLS) method.
This is the analog of the traditional Least Squares method for i.i.d.\ observations.
It is  especially  relevant when computing the likelihood is intractable.
%(see Section \ref{??} in the Appendix).
Noting that $\E_{p}(S_k |\mathcal{ F}_{k-1})= (1-p)S_{k-1}$, it reads as
%The CLS estimator is then defined as a minimizer of the Conditional lLS .
%If  $(X_n)$ is a Markov chain on $\N$ with transition kernel $Q_{\theta}(.,\cdot)$, the CLS is defined as
\begin{equation}    \label{CLSGreen}
U_n(p,S_1,\dots,S_n)= \sum_{k=1}^n\;\left( S_k- \E_{p}(S_k |\mathcal{ F}_{k-1})\right)^2=  \sum_{k=1}^n\;\left( S_k- (1-p) S_{k-1})\right)^2.
\end{equation}
%Since here $E_{p}(S_k |\mathcal{ F}_{k-1})= (1-p)S_{k-1}$, we get
 %\begin{equation}
%	U_n(\theta,S_0,\dots,S_n)= \sum_{k=1}^n \;\left( S_k- (1-p) S_{k-1}\right)^2.
%\end{equation}
The  associated Conditional Least Squares  estimator is
\begin{equation}\label{pCLS}
	\tilde{p}_n= 1- \frac{\sum_{k=1}^n s_{k-1}s_k}{\sum_{k=1}^n s_{k-1}^2}.
\end{equation}
A concern in statistics is to answer the question: how does such an estimator (or other ones) behave according to the asymptotic framework
%(when the observation time increases (whi
(here   $S_0 \rightarrow \infty$).  Is one of these two estimators better?

\subsection{Reed--Frost model}\label{RF}
It is also a chain  Binomial $SIR$ model relevant to model the evolution of an ordinary influenza in a small group of individuals.
The latent period is long with respect to a short infectious period and new infections occur at successive generations separated by latent periods. It is assumed that latent periods are  equal to $1$, contacts between Susceptibles and Infected are independent, and  that
the probability of contact between a Susceptible and an Infected is  $p=1-q$.
Therefore the probability of a Susceptible escaping infection given $I$ Infected is $q^I$, and if $\mathcal{ F}_n= \sigma((S_0,I_0),\dots,(S_n,I_n))$,
 %If at time $n$, $(S_n=s_n, I_n=i_n)$,
%Assume that
%Then,
% Denote by  $(S_n, I_n)$ the number of Susceptible and Infective individuals at time $n$.
%Assume that $p$ is now the probability of contact between a Susceptible and an Infective individual and let $q=1-p$ (probability of no contact between a Susceptible and an Infective). Assume moreover that the probability that a Susceptible remains Susceptible at time $n+1$ is $q ^{i_n}$ (probability of no contact with the $i_n$ infectives).
%and the probability of infection for a Susceptible is $p_n=(1- q^{i_n})$.
%the  distribution of $I_{n+1}$ given $\mathcal{ F}_n= \sigma((S_0,I_0),\dots,(S_n,I_n))$ is
$$\mathcal{ L}(I_{n+1}| \mathcal{ F}_n)= \mathrm{Bin}(S_n,p_n ) \; \mbox{ with } p_n=1-q^{I_n}  \quad \mbox{and  } S_{n+1}= S_n-I_{n+1}.$$
Then $(S_n,I_n)$ is a Markov chain on $\N^2$ with probability transitions,
\begin{eqnarray*}
 Q_q((s_n,i_n), (s_{n+1},i_{n+1}))&=& \begin{pmatrix} s_n \\ s_{n+1} \end{pmatrix} (q^{i_n})^{s_{n+1}}(1-q^{i_n})^{i_{n+1}} \;  \mbox{ if } s_{n+1}+i_{n+1}= s_n,\\
 &=& 0 \; \mbox{ otherwise}.
\end{eqnarray*}
%Contrary to the Greenwood model, the sequence of r.v. $(S_n)$ is no longer Markov.
%As before, the questions are how long before the end of the epidemics, what is the total number of infected individuals during the epidemics.
%Since extinction also occurs in finite time a.s., the inference framework is to assume that $S_0,N\rightarrow \infty$.
%The related random variables are:
%$T= inf\{n, I_n=0\}$, and $Z=\sum_{n\geq 1} I_n$.\\
%\noindent
%\subsection{Likelihood approach}
\underline{Parametric inference} \\
Assume that the successive numbers of Susceptible and Infected
have been observed up to time $n$ and consider the estimation of $q=1-p$.
Denote $\P_q$ the probability associated with the Markov chain with transition $Q_q$ and initial condition $(s_0,i_0)$.
%In this model,  $(S_n)$ decreases  with $n$, and  extinction occurs in finite time a.s.(after a geometric number of generations). Therefore,
%the inference framework is to assume that $S_0$ (or $N$) $\rightarrow \infty$.  \\
Then, if $s_{k+1}+i_{k+1}= s_{k}$ for $k=0,\dots,n-1$,
\begin{equation}\label{LnRF}
L_n(q;(s_1,i_1\dots, (s_n,i_n))= \prod_{k=0}^{n-1} \begin{pmatrix} s_k \\s_{k+1} \end{pmatrix} (q^{i_k})^{s_{k+1}}
(1-q^{i_k})^{i_{k+1}}.
\end{equation}
%\begin{eqnarray*}
%L_n(q;(s_0,i_0),&\dots&, (s_n,i_n))=
%\P_q((S_0,I_0)= (s_0,i_0),\dots (S_K,I_K)= (s_K,i_K))\\
%&=& \P_q((S_0,I_0)= (s_0,i_0)) \prod_{n=0}^{K-1}\P_q((S_{n+1},I_{n+1})= (s_{n+1},i_{n+1})/(s_n,i_n))\\
%&=& \P_q((S_0,I_0)=(s_0,i_0))\;\Pi_{n=0}^{N-1} \begin{pmatrix} s_n \\s_{n+1} \end{pmatrix} (q^{i_n})^{s_{n+1}}
%(1-q^{i_n})^{i_{n+1}}.
%\end{eqnarray*}
Therefore,
$ \log L_n(q) = C((s_k,i_k)) + \sum _{k=0}^{n-1} (s_{k+1}i_ k\log q + i_{k+1} \log(1-q^{i_k})).$ \\
Differentiating with respect to $q$ yields
\begin{equation*}
	\frac{d\log L_n}{dq}= \frac{1}{q}\sum _{k=0}^{n-1}\frac{i_k}{1-q^{i_k}}(s_{k+1}-s_kq^{i_k}).
\end{equation*}
%\begin{equation*}
%	\frac{d\log L_k}{dq}= \sum _{k=0}^{n-1} (\frac{s_{k+1}i_k}{q}-\frac{i_{k+1}i_k q^{i_k-1}}{1-q^{i_k}})
%= \frac{1}{q}\sum _{k=0}^{n-1}\frac{i_k}{1-q^{i_k}}(s_{k+1}-s_kq^{i_k}).
%\end{equation*}
The maximum likelihood estimator $\hat{q}_n$  of $q$ is a solution of the equation
$$\sum _{k=0}^{n-1}\frac{i_k}{1-q^{i_k}}(s_{k+1}-s_k q^{i_k})=0.$$
 Its properties can be studied as the number of observations increases (implying that the initial population grows to infinity).
 %Since extinction occurs in finite time, a classical asymptotic framework for statistical inference is to assume that the initial population grows to infinity.
%Note that in the  Reed--Frost model, neither $(S_n)$ nor $(I_n)$ is a Markov chain.\\

Here a problem which  occurs in practice already appears in this simple model, the case of  ``Partially Observed Markov Processes'': it
corresponds to the fact that both coordinates  $(S_n,I_n)$ are not  observed, but  only  the successive numbers of Infected individuals   $(I_{k},  k=0,\dots ,n-1)$
 are available.
%It  corresponds to the statistical problem of inference for partially observed Markov models.
In the special case of Hidden Markov Models (see  the Appendix, Section \ref{Miscstat} for the definition of H.M.M.),
  the theory for inference  is  now well known (\cite{cap05IV}, \cite{vanh08IV}), while there is no general theory for partially observed Markov processes. Many methods and algorithms have been proposed to deal with it  in practice (see e.g.\ \cite{dou01IV}, \cite{fea12IV}, \cite{ion11IV}). For applications specific to epidemics, many authors have addressed this problem (see e.g.\ \cite{cau04IV}, \cite{cha19IV}, \cite{hoh05IV},  \cite{ion11IV}, together with  the development of   packages (see R package POMP  \cite{kin08IV})

 \subsection{Birth and death chain with re-emerging}\label{BDC}
We consider now the example of an epidemic model with re-emerging in  a large infinite population.
It can be  described by a birth and death chain on $\N$ with reflection at $0$.
This models for instance  farm animals epidemics  when infection can also be produced by the environment.
Let $p,q$ denote the birth rate and death rates  with $\{0<p,q<1,p+q<1\}$. We assume that  $I_0= i_0 \geq 1$ and that $(I_n)$, the number of infected at time $n$,
%We moreover assume that there is re-emerging
% an epidemic model can be described  This description corresponds for instance to an epidemic where individuals
%(e.g.\ farm animals) are infected by the environment only) and recovery is obtained by vaccination,
%assuming that only one animal can be vaccinated by time unit. Hence, we assume that there is still some possible infection when there are no longer Infected individuals.
evolves as follows:
\begin{enumerate}
\item[-] if $k \geq 1 $, then $\P(I_{n+1}=k+1|I_n=k)=p$, $\P(I_{n+1}=k-1|I_n=k)=q $,  $ \P(I_{n+1}=k | I_n=k)=1-p-q$; \\
\item[-] if $k=0$, then $\P(I_{n+1}=1|I_n=0)=p$,  $\P(I_{n+1}=0|I_n=0)= 1-p$ (re-emerging probability).
\end{enumerate}

The Markov chain $(I_n)$ is irreducible aperiodic on $\N$ and, if $p < q $, $(I_n)$  is
positive recurrent with stationary distribution
$$\lambda_{(p,q)} (i)= \left(1-\frac{p}{q}\right) \left(\frac{p}{q}\right)^i.$$

\noindent\underline{Parametric inference} \\
Let $\Theta= \{(p,q), 0<p<q<1 \mbox{ with  } p+q<1\}$ and let $\theta_0$ be the true parameter value.
  Assume that  $I_0 = i_0 > 0$ is non-random and fixed and
consider the estimation of $\theta= (p,q) \in \Theta$ based on the observation of the successive numbers of Infected
 up to time $n$. \\
 Let  $(Q_{\theta}(i,j), i,j \in \N)$ denote the transition kernel $(I_n)$:
\begin{enumerate}
\item[-] if $i\neq  0$, then $Q_{\theta}(i,j)= p \delta_{\{i+1\}}(j)+q \delta_{\{i-1\}}(j) + (1-p-q)  \delta_{\{i\}}(j)$,\\
\item[-] if $i=0$, then $Q_{\theta}(0,j)= p \delta_{1}(j)+(1-p) \delta_{0}(j) $.
\end{enumerate}
Noting that for $j \neq\{ i-1,i, i+1\}$, $N_n^{ij}=0$,  the loglikelihood $\ell_n(\theta)$  satisfies
\begin{align*}
\ell_n(\theta)&=  \sum_{i,j  \in\ N} N_n^{ij} \log Q_{\theta}(i,j)\\
&=    B_n \log p + D_n \log q + R_n \log (1-p-q) + N_n^{0,0} \log(1-p), \mbox{ with}
\end{align*}
\begin{equation}\label{BDR}
B_n= \sum_{i\geq 0} N_n^{i,i+1}, \; D_n=\sum_{i\geq 1} N_n^{i,i-1}, \; R_n=\sum_{i\geq 1}N_n^{i,i}.
\end{equation}
Since the Markov chain $(I_n,I_{n+1})$ is positive recurrent  on $\N^2$ with stationary measure $ \lambda_{\theta}(i) Q_{\theta}(i,j)$, we can study directly the limit behaviour of $\ell_n(\theta)$.
Applying the ergodic theorem to $(I_n,I_{n+1})$ yields that,  almost surely under $\P_{\theta_0}$,
\begin{align*}
\frac{1}{n} N_{n}^{i,i+1} &\rightarrow  p_0 \lambda_{\theta_0}(i)\mbox{ for }i \geq  1,\\
\frac{1}{n} N_{n}^{i,i-1} &\rightarrow q_0 \lambda_{\theta_0}(i),\\
\frac{1}{n} N_{n}^{i,i}  &\rightarrow  r_0 \lambda_{\theta_0}(i),\\
\frac{1}{n} N_{n}^{0,0} &\rightarrow (1-\frac{p_0}{q_0}) (1-p_0).
\end{align*}

Therefore, using \eqref{BDR},
\begin{align*}
\frac{1}{n} B_n  &\rightarrow  p_0,\\
\frac{1}{n} D_n &\rightarrow  q_0 \times \frac{p_0}{q_0}= p_0,\\
\frac{1}{n}R_n &\rightarrow  \frac{r_0 p_0}{q_0},\\
\frac{1}{n} N_{n}^{0,0} &\rightarrow (1-\frac{p_0}{q_0}) (1-p_0).
\end{align*}
Joining these results, under $P_{\theta_0}$, as $n \rightarrow \infty$,
$$\frac{1}{n} \ell_n(\theta) \rightarrow p_0 \log p + p_0 \log q +\frac{r_0 p_0}{q_0}\log r + (1-\frac{p_0}{q_0}) (1-p_0) \log(1-p) := J(\theta_0,\theta).$$
We can check directly that $\theta \rightarrow J(\theta_0,\theta)$ possesses a unique global  maximum at $\theta_0$.
The associated maximum likelihood estimator $ \hat{\theta}_n$ is
\begin{equation}\label{MLEBD}
\hat{p}_n= \frac{1}{n}B_n;\quad \hat{q}_n= \frac{B_n}{D_n+R_n} (1-\frac{1}{n}B_n).	
\end{equation}
Successive applications of the ergodic theorem yield that $(\hat{p}_n,\hat{q}_n)$ converges $\P_{\theta_0}$ a.s. to $(p_0,q_0)$.\\

To study the limit distribution of $(\hat{p}_n,\hat{q}_n)$, we use
  the results of Section \ref{MLEMC}  in the Appendix.
Let $Q= (Q(i, j)) $ denote the (unnormalized) transition kernel on $\N\times \N$:
\begin{align*}
Q(i, i + 1) &= 1= Q(i,i )\mbox{ for }i \in \N\mbox{ and}\\
Q(i, i-1) &= 1\mbox{ for }i \geq 1.
\end{align*}
According to  \eqref{Qf},  the family $(Q_{\theta}, \theta \in \Theta)$ is dominated by $Q $
 with associated function $f_{\theta} (i,j)$:
\begin{align*}
f_{\theta}(i,i+1) &=p\mbox{ for } i \geq 0,\\
f_{\theta}(i,i-1)&=q\mbox{ for }  i \geq 1,\\
f_{\theta}(i,i) &=1-p-q,\\
f_{\theta}(0,0) &=1-p.
\end{align*}

Except the compactness assumption of $\Theta$ (only  required for the consistency of the MLE),
the Markov chain satisfies Assumptions (H1)--(H8) of Section \ref{MLEMC},
Therefore,  under $\P_{\theta_0}$,
$$\sqrt{n}  \begin{pmatrix} \hat{p}_n-p_0\\  \hat{q}_n-q_0\end{pmatrix} \rightarrow_{\mathcal{ L}} \mathcal{ N}(0,I^{-1}(\theta_0)),$$
%Hence, applying the ergodic theorem, we get that $(\hat{p}_n,\hat{q}_n)$ converges $\P_{\theta_0}$ a.s.
%to $(p_0,q_0)$.
%It is straightforward to check that this Markov chain satisfies Assumptions (H1) -(H8) of Section \ref{MLEMC} .
with, using  Definition (\ref{FisherCh3}),
$$ I(\theta_0)= \sum_{i\geq 0} \lambda_{\theta_0}(i)\sum_{j\geq 0} \frac{\nabla_{\theta}f_{\theta_0}(i,j) \nabla_{\theta}^* f_{\theta_0}(i,j)}{f_{\theta_0}(i,j)^2} Q_{\theta_0}(i,j).$$
Hence $I(\theta_0)$ can be explicitly computed:  for $\theta=(p,q)$, we get
\begin{equation}
I(p,q)=\begin{pmatrix}
\frac{r+p^2}{p(1-p)r}& \frac{p}{qr}\\
\frac{p}{qr}&\frac{p(1-p)}{rq^2}
\end{pmatrix}
 \Rightarrow  I^{-1}(p,q)= \begin{pmatrix} p(1-p)& -pq\\
-pq&\frac{q^2(p^2+r)}{p(1-p)}
\end{pmatrix}.
\end{equation}

\subsection{Modeling an infection chain in an Intensive Care Unit}\label{ICU}
This example is taken from Chapter 4 of \cite{die13IV}.
%Diekmann, Heesterbeck, Britton \cite{diekmannheesterbeekbritton-book}.
 It aims at describing nosocomial infections (i.e.\  infections acquired in a hospital). The incidence of these infections is highest in  an Intensive Care Unit, which is characterized by a small number of beds (about 10 beds at most) and rapid turnover of patients by way of admission and discharge. %This  concerns a finite population of size $N$
%($N$ small) but with high turnover (most patients stay only a couple of days).
There are two routes for infection (colonization) for a patient.
\begin{enumerate}
\item[-] The endogenous route ($\alpha$ mechanism): bacteria are already present in a newly admitted patient  but at low undetectable levels and resistant bacteria  develop
because of antibiotic treatments during the stay. Let  $e^{-\alpha}=(1-a) $ the probability per individual per time unit of getting infected by this route.
\item[-] The exogenous route ($\beta$ transmission):  it  models the probability of infection  of a Susceptible  by an Infected  in the ICU per time unit,
 $e^{-\beta}= (1-b)$.
 \end{enumerate}

To describe the  composition of the ICU  in terms of Infected and Susceptible individuals on long time intervals, a Markov chain model  can be  used  as follows.
% we consider the probabilities of the various compositions of the ICU
Each patient has probability $ d$ of being discharged by unit of time. Discharge and admission take place every day at noon; new admitted individuals are susceptible.
 Observations are obtained from a bookkeeping scheme that
concerns  the state of the ICU immediately after discharge (12h05).\\

Consider the simplest example, an ICU with two beds.
 It corresponds to three possible states: State $0$ (both patients are Susceptible), State $1$ (one Susceptible, one  Infected) and $2$ (both are Infected).
 Denote by $X_n$ the composition of the ICU at time $n$.
%Let  $X_n$ denote the state of the ICU after discharge (at 12h05) and introduce
%$\bar{X}_{n+1}$ the state of the ICU just before discharge (at 11h55)
% on the next day.
 Let us compute  according to $\theta= (a,b,d)$ the  transition  matrix $Q_{\theta}$ of $(X_n)$.
Introduce
$\bar{X}_{n+1}$ the state of the ICU just before discharge (at 11h55)
 on the next day.
If   $X_n=0$,   $\P(\bar{X}_{n+1}=0)=a^2$, $\P(\bar{X}_{n+1}=1)=2 a (1-a)$ and $\P(\bar{X}_{n+1}= 2)= (1-a)^2$. If  $X_n=1$, $\P(\bar{X}_{n+1}=1)= ab$, and $\P(\bar{X}_{n+1}=2)= (1-ab)$.
Finally, if  $X_n=2$, $\P(\bar{X}_{n+1}=2)= 1$.
This yields that, after discharge (12h05), $X_{n+1}=  0,1,2$ with respective probabilities,
%\begin{equation*}
%Q_{\theta}= \begin{pmatrix}(a+(1-a)d)^2& 2(1-a)^2 d(1-d)+2a(1-a)(1-d) &(1-a)^2(1-d)^2\\
%abd+(1-ab)d^2& 2(1-ab)d(1-d)+ab (1-d)& (1-ab)(1-d)^2\\
%d^2& 2d(1-d)& (1-d)^2
%\end{pmatrix}.
%\end{equation*}
\begin{equation*}
Q_{\theta}= \begin{pmatrix}(a+(1-a)d)^2& 2(1-a)(1-d)(a+(1-a)d)&(1-a)^2(1-d)^2\\
abd+(1-ab)d^2& 2(1-ab)d(1-d)+ab (1-d)& (1-ab)(1-d)^2\\
d^2& 2d(1-d)& (1-d)^2
\end{pmatrix}.
\end{equation*}

 Let $\Theta= (0,1)^3$. Assume that the states $(X_i)$ of the ICU after discharge have been observed up to time $n$.
 The maximum likelihood estimator of  $ \theta$ reads,
%Observing after discharge the state  $(X_n)$ of the allows to estimate $\theta$.
%the Loglikelihood is,
using (\ref{Nkl}),
 \begin{eqnarray*}
 \ell_n(\theta) &=&\sum _{k=1}^n \log Q_{\theta}(X_{k-1},X_{k})=
\sum _{i, j \in\{0,1,2\}}N_n^{ij} \log Q_{\theta}(i,j),\\
\hat{\theta}_n&=& \mbox{ argsup}_{\theta \in \Theta}\;  \ell_n(\theta).
\end{eqnarray*}
Since $(X_n)$ is a positive recurrent Markov chain on $\{0,1,2\}$, we can apply  the results stated in the Appendix, Section \ref{ApStatMC}. The MLE $\hat{\theta}_n$ is consistent and converges at rate
$\sqrt{n}$ to a Gaussian law $\mathcal{ N}_3(0, I^{-1}(\theta))$, where $I(\theta)$ is the Fisher information matrix defined in \eqref{FisherCh3}.\\

Assume now that  there is no systematic control of  the exact status of the patients  after discharge, but that each  patient is tested with probability p.
%Therefore, the state of the ICU is observed  $X_n
%This can be modeled  as follows.  Assume that each patient is tested with probability p.
Then, the observations are no longer $(X_n)$, but $(Y_n)$, with conditional transition matrix  $ (T_p(i,j)= P(Y_n=j |X_n=i), 0\leq i, j \leq 2)$,
$$ T_p= \begin{pmatrix}
(1-p)^2& 2p(1-p)& p^2\\
p(1-p)& p^2+ (1-p)^2&p(1-p)\\
p^2&2p(1-p)& (1-p)^2
\end{pmatrix}.
$$
If only $(Y_n)$ is observed, we have to deal with a Hidden Markov Model $(X_n, Y_n)$ as  defined in the Appendix, Section \ref{Miscstat}.
The estimation of $\theta$ or $(\theta,p)$ has to take into account this additional noise to be efficient (see e.g \cite{cap05IV}).

\section{Two extensions to continuous state and continuous time  Markov chain models}
\subsection{A simple model for population dynamics.} \label{AR1details}
The $AR(1)$ model is a classical model of population dynamics with continuous state space and allows us to illustrate explicitly various inference questions.
  Consider the autoregressive process on $\R$s introduced in Section \ref{AR1}  defined by
%$X_0 = x_0$ and for $i\geq 1$,
$$ X_0= x_ 0 ; \quad  \mbox{ and for } i\geq 1,  X_{i} = a X_{i-1} + \gamma \epsilon_{i},$$
where  $(\epsilon_{i})$ is a sequence of  i.i.d.\ random variables on $\R$ with distribution  $f_{\theta}(x) dx$, independent of $X_0$.\\
This is a Markov chain on $(\R,\mathcal{ B}(\R))$ with transition kernel
$Q_{\theta,a}(x,dy)=  f_{\theta}(y-ax) dy$. If  $X_0$ is known, choosing as dominating kernel  the Lebesgue measure on $\R$, the
likelihood reads as
$$L_n(a,\theta )=  \prod_{i=1}^n \; f_{\theta}(X_i-aX_{i-1}).$$

The Gaussian $AR(1)$ corresponds to  $\epsilon_i \sim  \mathcal{ N}(0, \gamma^2)$:
\begin{eqnarray*}
Q_{a,\gamma^2}(x,dy) &=& \frac{1}{\gamma \sqrt{2\pi}} \exp-(\frac{1}{ 2\gamma ^2}( y-ax)^2) dy,\mbox{ and}\\
L_n(a,\gamma^2)&= & \prod_{i=1}^n\; \frac{1}{\gamma \sqrt{2\pi}} \exp(-\frac{1}{ 2\gamma ^2}( X_i-aX_{i-1})^2),\\
\ell_n(a, \gamma^2)&=&
-(n/2) \log \gamma^2 - 1/(2\gamma^2) \sum_{i=1}^n (X_i-aX_{i-1})^2.
\end{eqnarray*}
 The properties of the MLE have been presented in Chapter \ref{chap1:intro}.
 %these notes.
\subsection{Continuous time Markov epidemic model}
We just recall here results for the $SIR$ Markov jump process (see Section \ref{sec:likelihood}).
%Let us compare the results for the SIR epidemic Dynamics. \\
 Assume that the jump process $(\mathcal{Z}^N(t))$ is continuously observed on $[0,T]$.
 Its dynamics is described by the two parameters $(\lambda,\gamma)$.
% The ODE associated with the limit behaviour of the normalized process ${(\cal Z}^N(t)/N)$  is\\
% $\frac{ds}{dt} = -\lambda s(t)i(t); \frac{di}{dt}=  \lambda s(t)i(t)+\gamma i(t)$.\\
 %Let  $s(t)= s( \lambda_0, \gamma_0,t); i(t)= i(\lambda_0, \gamma_0,t)$ the solution of the ODE.
The  Maximum Likelihood Estimator $ (\hat{\lambda},\hat{\gamma})$ is  explicit (see  \cite{and00IV} or  Section \ref{sec:likelihood}).
Indeed, let  $(T_i)$ denote   the successive jump times  and set   $J_i=0$ if we have an infection and $J_i=1$ if we have a recovery.
 Let $K_N(T)= \sum_{i\geq 0}1_{T_i\leq T}$. Then
 \begin{align*}
\hat{\lambda}_N&= \frac{1}{N} \frac{ \sum_{i=1}^{K_N(T)} (1-J_i)}{\int_0^T S^N(t) I^N(t) dt}= \frac{1}{N} \frac{\mbox{ \# Infections} }{\int_0^T S^N(t) I^N(t) dt},\\
 \hat{\gamma}_N&= \frac{1}{N} \frac{ \sum_{i=1}^{K_N(T)} J_i}{\int_0^T I^N(t) dt} = \frac{\mbox{\# Recoveries}}{\mbox{``Mean infectious period''}}.
 \end{align*}
As the population size $N$ goes to infinity, $(\hat{\lambda}_N,\hat{\gamma}_N)$  is consistent and
  $$\sqrt{N} \begin{pmatrix}\hat{ \lambda}_N-\lambda \\ \hat{\gamma}_N-\gamma \end{pmatrix} \rightarrow \mathcal{ N}_2\left(0, I^{-1}(\lambda,\gamma)\right), \mbox{ where } I(\lambda,\gamma)= \begin{pmatrix} \frac{\int_0^T s(t)i(t) dt}{\lambda} &0\\0& \frac{\int_0^Ti(t) dt} {\gamma}
\end{pmatrix} , $$
and  $(s(t),i(t) )$  is the solution of the ODE associated with the limit behaviour of the normalized process $(\mathcal{Z}^N(t)/N)$:
 $\frac{ds}{dt} = -\lambda s(t)i(t); \frac{di}{dt}=  \lambda s(t)i(t)-\gamma i(t)$.\\

The matrix $I(\lambda,\gamma) $ is the Fisher information matrix of this statistical model.\\
% $$I(\lambda,\gamma)= \begin{pmatrix} \frac{\int_0^T s(t)i(t) dt}{\lambda} &0\\0& \frac{\int_0^Ti(t) dt} {\gamma}
%\end{pmatrix}.$$

\section{Inference for Branching processes} \label{Chapter: Statbranching}

%\label{intro}
At the early stage of an outbreak, a good approximation for the epidemic dynamics is to consider that the population of Susceptible is infinite and that
Infected individuals evolve according to a branching process (see Section \ref{sec-early-stage} of Part I).
We present here some classical statistical results in this domain.
This Markov chain model is an example of  non-ergodic processes which leads  to different statistical results.
%\pagebreak

\subsection{Notations and preliminary results}
Some basic facts
% the probabilistic behaviour of such processes is given in Part 1of these notes (Chapter 5, Section 1) facts
 on discrete time branching processes (or Bienaym\'e--Galton--Watson processes) are given in  Part I, Section \ref{TB-EP_sec_Br-proc} of these notes (see also the classical monographs on branching processes  \cite{ath72IV} or \cite{jag75IV}).
%We consider discrete time  branching processes (Bienaym\'e-Galton--Watson processes).
 We complete these facts with some properties useful for the inference.
%Let $\{\xi_{k,i}, k\geq 0, i\geq 0\}$ be independent identically distributed random variables taking values in $\N$ with offspring distribution $G$.

Consider an ancestor $Z_0=1$ has $\xi_0$ children according to an offspring law $G$ defined by
$$\P(\xi_0=k)= p_k ,\;\; k \geq 0  \;\mbox{ and   } \sum_{k \geq 0} p_k=1.$$
Let $m=\E(\xi_0)$ and  $g(s)= \E(s^{\xi_0 })$. The $i$-th of those children has $\xi_{1,i}$ children, where the random variables $\{\xi_{k,i}, k\geq 0, i\geq 1\} $
are i.i.d.\ with distribution $G$.
Let $Z_n$ denote the number of individuals in generation $n$. Then,
\begin{equation} \label{ZGW}
	Z_{n+1}= \sum_{i=1}^{Z_n}\xi_{n,i}.
\end{equation}
Denote by  $E= \{ \exists \; n ,\;  Z_n=0\}$ the set of extinction. \\
If $ m\leq 1$ and  if $p_1 \neq 1$, the process $Z_n$ has a probability $q=1$ of extinction.\\
If $m>1$, the process is supercritical and has a probability  of extinction $q<1$, which is the smallest solution of the equation
$g(s)=s$ on $[0,1]$. The set $E^c$ is equal  to $\{\omega,Z_n(\omega )\rightarrow \infty\}$.\\

% Let $E$ denote the set of extinction: $E= \{lim_n\;  Z_n=0\}$.
This extinction probability is an important parameter for the early stages of an epidemic. It corresponds to the probability of a minor outbreak.\\
We complete the results given in  Part I, Section \ref{TB-EP_sec_Br-proc}.
Let $\mathcal{ F}_n=\sigma (Z_0,\dots, Z_n) $ and define $ W_n= m^{-n}Z_n$.
Then  $(W_n)$ is a $\mathcal{ F}_n$-martingale.
 %uniformly bounded in $L^2$
%We assume that the BGG process is supercritical: $m>1$
\begin{theorem}\label{W}
Assume that $m>1$ and that the offspring law $G$ has finite variance $\sigma^2$. Then,
 %(W_n)$ is a $\mathcal{ F}_n$-martingale uniformly bounded in $L^2$, and t
 there is a non-negative random variable $W$ such that
 \begin{enumerate}
\item[(i)]  $W_n \rightarrow W$ as $n\rightarrow \infty$  a.s. and in $L^2$.
\item[(ii)] $\{W>0\}= \{Z_n\rightarrow \infty \} =E^c$ and  $\{W=0\}=\{\lim_n\;  Z_n=0\}=E$.
\item[(iii)] Moreover,  $\E W =1, \; \mathrm{var}(W)= \frac{\sigma^2}{m(m-1)}$.
\end{enumerate}
\end{theorem}
\noindent
%Note that the usual condition $\sum_{k\geq 1} k p_k \log p_k<\infty$ is ensured by  the assumption $\Var\;  G< \infty$.
  %Define from $(Z_i)$ the random variables
%\begin{equation}\label{Sn}
%S_n= \sum_{i=0}^n Z_i 	
%\end{equation}
%By the Toeplitz lemma and some algebra
\begin{corollary} \label{CorSnSnp}
If  $m>1$, then, almost surely  	
\begin{equation}\label{Sn}
 \frac{1}{m^n} \sum_{i=1}^n Z_i  \rightarrow  \frac{m}{m-1}W; \quad  \frac{1}{m^n} \sum_{i=1}^n Z_{i-1} \rightarrow \frac{1}{m-1} W.
 \end{equation}
\end{corollary}
%(i) $\displaystyle{ \frac{m-1}{m^n-1}{S_n} \rightarrow W}$ $\P$ a.s.\\
%(ii) $\displaystyle{\frac{S_n}{m^n} \rightarrow \frac{m}{m-1}}$ a.s.
\begin{proof}
We  write $ \sum_{i=1}^n Z_i= \sum_{i=1}^n m^i \frac{Z_i}{m^i}$. Using Theorem \ref{W}, $\frac{Z_n}{m^n} \rightarrow W$  a.s.
An application of the Toeplitz lemma stated below and some algebra yield the two results.
\begin{lemma}\label{Toeplitz}
(Toeplitz Lemma)	Let $ (a_n)$ a sequence of non-negative real numbers and $(x_n)$ a sequence on $\R$. If
	 $ \; \sum_{i=1}^n\;a_i \rightarrow \infty$ and if $(x_n)\rightarrow  x \in \R $ as $n\rightarrow \infty$, then
	$$\frac{\sum_{i=1}^n a_i x_i}{\sum_{i}^n\;a_i} \rightarrow x  \mbox{ as } n\rightarrow \infty. $$
\end{lemma}
\end{proof}

Assume that the offspring distribution  $G_{\theta}(\cdot)$ depends  on a parameter $\theta$ with finite mean $m(\theta)>1$ and finite variance
$\sigma^2(\theta)$.
Denote by $\P_{\theta}$ the law on $(\N^{\N}, \mathcal{ B}(N^{\N}))$ of the branching process $(Z_n) $  with offspring law  $G_{\theta}(\cdot)$.
%Denote by $m(\theta)$ and  $\sigma^2(\theta)$ the mean and the variance  of the offspring distribution.
Then  $(Z_n, n \geq 0)$ is a Markov chain with state space $\N$, initial condition $Z_0=1$  and transition matrix,
\begin{equation}\label{transbranching}
Q_{\theta}(i,j)= G_{\theta}^{\star i}(j),
\end{equation}
 where $\star$ denotes the convolution product of two functions and  $f^{\star i}$ is the $i$-fold convolution product of $f(\cdot)$.
%$\theta \in \Theta$, with $\Theta$ a (compact) subset of $\R^q$.
Let $\mu_i$ denote the measure  $\mu_i(k)=1$ for all $ k \in \N$, $\lambda_n=\otimes_{i=1}^n\mu_i$.
Then, the likelihood reads as
\begin{equation}\label{LnBGW}
	\frac{d\P_{\theta}}{d\lambda_n}(Z_0,\dots,Z_n)= L_n(\theta)= \prod_{i=1}^n G_{\theta}^{\star Z_{i-1}}(Z_i); \quad \ell_n(\theta)= \sum _{i=1}^n \log (G_{\theta}^{\star Z_{i-1}}(Z_i)).
	\end{equation}
Under this expression, studying the likelihood for general offspring laws is intractable.  We detail in the next section a framework where it is possible to study this likelihood, and in the next section another method  based on Weighted Conditional Least Squares.

\subsection{Inference  when the offspring law belongs to an exponential family}
Among parametric families of distributions, exponential families of distributions, widely used in statistics, provide here  a nice framework
to study this likelihood. A short recap is given in the Appendix  Section \ref{Miscstat}  (see e.g. the classical monograph  \cite{bic15IV} for the complete exposition).
%\section{Preliminary results}
% Let $X$ be a random variable in $\R^k$ (or $\Z^k$) with distribution $P_{\theta}$ and density $p(\theta,x)$, with $\theta \in \Theta$, subset of $\R^q$.
%\begin{definition}\label{expfam1}
%The family $\{P_\theta, \theta \in \Theta\}$ is an exponential family if there exist  $q$ functions $(\eta_1,\dots,\eta_q)$ and $\phi$ defined on $\Theta$, $q$ real functions $T_1,\dots, T_q$  and a function $h(\cdot)$ defined on $\R^k$ such that\\
%\begin{equation}\label{EF1}
%p(\theta,x)= h(x) \exp\{\sum _{j=1}^q \eta_{j}(\theta)T_j(x)-\phi(\theta)\}\; ; x \in \R^k.
%\end{equation}
%\end{definition}
%Then $T(X)= (T_1(X),\dots,T_q(X))$ is a sufficient statistic in the i.i.d.\ case.
%The random variable $X$ satisfies
%\begin{equation}\label{momentexp}
%	m(\theta)= \E_{\theta}(X)= \nabla_{\theta} \phi (\theta);\quad \sigma^2(\theta)= Var_{\theta}(X)= \nabla_{\theta}^2 \phi (\theta).
%\end{equation}

Assume that the offspring law is a power series distribution:\\
\begin{equation}\label{Psd}
	p(k)= A(\zeta)^{-1}\; a_k \;\zeta ^k, \quad \mbox{with  } A(\zeta)=\sum _{k\geq 0} a_k \zeta ^k.
\end{equation}
% $\Zeta= \{\zeta \in \R,\; A(\zeta )< \infty)$.\\
Setting $\theta= \log \zeta$, $\Theta=\{\theta \in \R,\; A(e^\theta )< \infty\}$, $h: k \rightarrow h(k)=  a_k$ and $\phi:\theta \rightarrow \phi(\theta)= \log A(\log ( e^{\theta}))$, we get that
it is a special case  of an exponential family of distributions on $\N$ with $T(X)=X$ and
\begin{equation}\label{pk}
	p(\theta,k)= h(k)\exp{(k \theta-\phi(\theta))}.
\end{equation}
The random variable $X$ satisfies that
\begin{equation}\label{momentexp}
m(\theta):= \E_{\theta}(X)= \nabla_{\theta} \phi (\theta);\quad \sigma^2(\theta):= Var_{\theta}(X)= \nabla_{\theta}^2 \phi (\theta).
\end{equation}
Moreover, if $X_1,\dots,X_n$ are i.i.d.\ with distribution \eqref{Psd},  then
\begin{equation}\label{densitysum}
\P(X_1+ \dots+X_n=k) =H(n,k) \exp(k\theta- n\phi(\theta))\; \mbox{ where } H(n,k)= h^{*n}(k).
\end{equation}
% where the function $ H: \N \times \N \rightarrow [0,1] $ is defined by  $H(i,k)= h^{*i}(k)$, .
Therefore for offspring distributions satisfying \eqref{Psd} or \eqref{pk}, the transition kernel is
$$Q_{\theta}(i,k)=  H(i,k) \exp(k\theta-i\phi(\theta)).$$

Let us note that several families of classical distributions on $\N$ are included in this set-up:
\begin{enumerate}
\item[\textbf{-}] Geometric distributions on $\N^*$ with parameter $p$ (i.e.\  $P(X=k)= p (1-p)^{k-1})$: \\
       $ \theta= \log(1-p)$; $h(k)=1$ and  $\phi(\theta)= \log\frac{e^{\theta}}{1-e^{\theta}}$.\\
\item[\textbf{-}] Binomial distributions $(\mathcal{ B}(N,p), p\in (0,1))$ with  $N$ fixed: \\
  $\theta= \log\frac{p}{1-p}$, $h(k)= \frac{N!}{k!(N-k)!}$ and  $\phi(\theta)= N \log(1+e^{\theta}) $.\\
\item[\textbf{-}]  Poisson distributions $\mathcal{ P}(\lambda)$: $\theta= \log \lambda$, $h(k)= \frac{1}{k!}$ and
$\phi(\theta)= e^{\theta}$.  \\
\item[\textbf{-}] Negative Binomial distributions $ (\mathcal{ NB}(r,p), p\in (0,1))$  with $r$ fixed (i.e.   $P(X=k)= \frac{\Gamma(k+r)}{\Gamma(r) k!}p^r (1-p)^k$:\\
$\theta= \log(1-p)$, $h(k)= \frac{(k+r)\dots (r+1)}{k!}$ and $\phi(\theta)= r\log(1-e^{\theta})$).
\end{enumerate}

%\section{Parametric Inference for supercritical branching processes}
Let us come back to the likelihood (\ref{LnBGW}).  Let $\theta_0 \in\Theta$ be the true value of the parameter.
%Denote by  $\P_{\theta}$ the distribution $(\N^{\N}, \mathcal{ B} (N^{\N}) $ of the Galton -Watson process with  initial condition$Z_0=1$ and offspring distribution  $F_{\theta}(\cdot)$ on $\N$ and consider the statistical model $(\N^{\N}, \mathcal{ B} (N^{\N}),(\P_{\theta}, \theta  \in \Theta))$.
%Denote by $m(\theta) $ and $\sigma^2(\theta) $ the mean and the variance of the offspring distribution.
We assume
\begin{enumerate}[(A4)]
\item[\textbf{ (A1)}] The offspring distribution $G_{\theta}$ belongs to  an exponential power series family:
For  all  $k \in \N, \; G_{\theta}(k)= h(k)\exp{(k \theta-\phi(\theta))} $.
 \item[\textbf{ (A2)}] $\Theta$ is a compact subset of $\{\theta, \sum_{k\geq 0} h(k) e^{\theta k}< \infty\}$, $\theta_0 \in \mbox{Int}(\Theta)$.
 \item[\textbf{ (A3)}]  For all $ \theta \in \Theta$, $m(\theta)>1$ and $\sigma ^2(\theta)$ finite.
 \item[\textbf{ (A4)}] There exists a $\delta >0$ such that   $ E (Y^{2+\delta}) = \mu_{2+\delta} < \infty$  where  $Y \sim G_{\theta}$.
\end{enumerate}
Consider the estimation of $\theta$ when the successive generation sizes $(Z_1, \dots Z_n)$ are observed.
Under (A1)--(A3),  the loglikelihood is, using  \eqref{densitysum},
\begin{equation}\label{lnexp}
	\ell_n(\theta)= C(Z_0,\dots,Z_n)+  \sum_{i=1} ^n (\theta Z_i-\phi(\theta)Z_{i-1}),
	\end{equation}
with  $C(Z_0,\dots,Z_n)=  \sum_{i=1}^n \log H(Z_{i-1},Z_i).$
The constant $C(Z_0,.\dots,Z_n)$ depends only on the observations and brings no  information on $\theta$.

The M.L.E $\hat{\theta}_n$, defined as any solution of $\nabla_{\theta}{\ell}_n(\theta)=0$, satisfies
$$\sum _{i=1}^n Z_i-\nabla_{\theta}\phi(\hat{\theta}_n) \sum_{i=1}^n Z_{i-1}=0.$$
Using that $\nabla_{\theta}\phi(\theta)= m(\theta)$ (see (\ref{momentexp})), $\hat{\theta}_n$ satisfies
\begin{equation}\label{mtheta}
m(\hat{\theta}_n)= \frac{\sum _{i=1}^n Z_i}{\sum_{i=1}^n Z_{i-1}}.
\end{equation}
%Let $\theta_0$ be the true value of the parameter.\\
By Theorem \ref{W},  $m(\theta_0)^{-n}Z_n$ converges a.s. and in $L^2$ under $\P_{\theta_0}$ to  a random variable $W$ such that $W>0$ on $E^c$, the non-extinction set, which satisfies $\P_{\theta_0}(E^c)= 1-q >0$   under  (A3).
%Let $W$ denote the random variable defined in Theorem \ref{W} as $m(\theta_0)^{-n} Z_n \underset{n\rightarrow \infty} \rightarrow  W$ a.s.
\begin{theorem}\label{MleGW}
Assume (A1)--(A4).%that  the branching process  has an offspring law $F_{\theta}$ satisfying \eqref{Assumpexp} with mean
%$m(\theta)>1$ and finite variance $\sigma^2(\theta)$.
Then, on  $E^c$, $m(\hat{\theta}_n)$ satisfies
\begin{enumerate}
\item[\textbf{ (i)}] $m(\hat{\theta}_n) \rightarrow m(\theta_0)$ a.s. under $\P_{\theta_0}.$
\item[\textbf{ (ii)}] $ m(\theta_0)^{n/2} (m(\hat{\theta}_n) -m(\theta_0)) \rightarrow _\mathcal{ L}  \sqrt{(m(\theta_0) -1)\sigma^2(\theta_0) } \;\eta^{-1} \;  N $,
where  $\eta, N$ are independent r.v.s,
$N \sim\mathcal{ N}(0,1)$, and $ \eta$ is the positive  variable defined by $\eta^2=W$ on $E^c$.
\end{enumerate}
\end{theorem}
Clearly, $m(\theta)$ is the parameter that is naturally estimated here.
\begin{proof}
 Let us write
$$m(\hat{\theta}_n)= \frac{\frac{\sum_{i=1}^n Z_i}{m^n}}  { \frac{\sum_{i=1}^nZ_{i-1}}{m^n}}.$$
Using  Corollary \ref{CorSnSnp},
both terms of the above fraction converge a.s. so  that $ m(\hat{\theta}_n)\rightarrow m(\theta_0)$ a.s.\\
%Hence $m(\hat{\theta}_n)$ is a strongly consistent estimator of $m(\theta_0)$.
Let us prove (ii). The score function reads as
$$\nabla_{\theta}{\ell}_n(\theta)= \sum_{i=1}^n Z_i - m(\theta) \sum_{i=1}^n Z_{i-1}=   \sum_{i=1}^n (Z_i - m(\theta)  Z_{i-1}). $$
Under $\P_{\theta_0}$,  $\nabla_{\theta}{\ell}_n(\theta_0)$ is a centered $\mathcal{ F}_n$-martingale $(M_n)$ with increments $X_i= Z_i -m(\theta_0) Z_{i-1}$.
%Note that  it satisfies $E_{theta}\nabla_{\theta}{\ell}_n(\theta)= - E_{theta}\nabla_{\theta}^2{\ell}_n(\theta)$.
Conditionally on $\mathcal{ F}_{i-1}$,  $X_i$
%$X_i= M_i-M_{i-1}= Z_i -m(\theta_0) Z_{i-1}$ $X_i$
 is the sum of $Z_{i-1}$ independent centered random variables so that
%Its increments $X_i= M_i-M_{i-1}= Z_i -m(\theta_0) Z_{i-1}$ satisfy, using that, conditionally on $\mathcal{ F}_{i-1}$, $X_i$ is the sum of $Z_{i-1}$ independent centered random variables
$$\E_{\theta_0}(X_i^2|\mathcal{ F}_{i-1})= \sigma^2 (\theta_0) Z_{i-1} ; \quad    \langle M\rangle_n=\sigma^2 (\theta_0)\sum_{i=1}^n Z_{i-1} .$$
%%\pagebreak
Hence
$$ s_n^2 (\theta_0)= \E_{\theta_0} (\langle M\rangle_n) = \sigma^2 (\theta_0) \sum_{i=1}^{n}m(\theta_0)^{i-1}= \sigma^2(\theta_0) \frac{m(\theta_0)^{n }-1}{ m(\theta_0)-1}.$$
Therefore  $s_n^2(\theta_0) \rightarrow \infty $ as $n \rightarrow  \infty$ and
\begin{equation} \label{snBGW}
\frac{s_n^2(\theta_0)}{m(\theta_0)^n} \rightarrow \frac{\sigma^2(\theta_0)}{m(\theta_0)-1}.
\end{equation}
Let us check the conditions of the Central limit theorem for martingales  (see \ref{TCLMart1})
recalled in the Appendix
%thChapter 3, Section \ref{AppenCh3}.\\
Under (A3), $(M_n)$ is a square integrable  centered $\mathcal{ F}_n$-martingale such that
 $\E_{\theta_0}(\langle M_n\rangle)= s_n(\theta_0)^2 \rightarrow \infty$.
 Let us  check (H2). We have
% (3) There exists a positive random variable $\xi$ such that $\frac{1} {s_n(\theta_0)^2} <M_n> \rightarrow  \eta$ in $\P_{\theta_0}$-probability .\\
  $$\frac{1} {s_n(\theta_0)^2} \langle M_n\rangle =  \frac{m^n(\theta_0)}{ s_n^2(\theta_0)}\frac{\sigma^2(\theta_0) }{m^n(\theta_0)} \sum_{i=1}^n Z_{i-1}.$$
  Hence according to Corollary \ref{CorSnSnp} and  Theorem \ref{W},$\frac{1} {s_n(\theta_0)^2} \langle M_n\rangle \rightarrow  W $ in probability  under $\P_{\theta_0}$
  % According to Theorem \ref{W}, $
 with $ W>0$ on $E^c$ and $E_{\theta_0}(W)=1$.  Therefore we can set $ W= \eta^2$ and obtain  (H2). \\
  It remains to check the asymptotic negligibility Assumption  (H1').
  %For this assume that the offspring distribution  $G_{\theta} $  satisfies (A4).\\
  %has a finite moment of order $(2+\delta)$ with $\delta >0$. :\\
  % \textbf{ (A4)}: There exists $\delta >0$ such that   $ E (Y^{2+\delta}) = \mu_{2+\delta} < \infty$  where  $Y \sim F_{\theta}$.\\
 % Conditionally on $\mathcal{ F}_{i-1}$,  $X_i$ is the sum of $Z_{i-1}$ i.i.d.\ random variables $\{(Y_{i-1,k} -m(\theta_0)), k=1, \dots, Z_{i-1}\}$.
 We have, for $X_i= Z_i-m(\theta_0)Z_{i-1}$,
 \[\E_{\theta_0}(|X_i|^{2+\delta }|\mathcal{ F}_{i-1})= Z_{i-1}\E_{\theta_0}(|Y-m(\theta_0)|^{2+\delta}).\]
  Under (A4), using that $\E_{\theta_0}(|Y-m(\theta_0)|^{2+\delta}) \leq C  (\mu_{2+\delta} + m(\theta_0)^{2+\delta}) <\infty$ yields
  \begin{equation}\label{checkH1}
  \frac{1}{s_n^{2+\delta}} \sum _{i} ^n \E_{\theta_0}(|X_i|^{2+\delta }|\mathcal{ F}_{i-1})=  (\frac{1} {s_n^{\delta}})  \E_{\theta_0}(|Y-m(\theta_0)|^{2+\delta})
 ( \frac{1}{s_n^2} \sum_{i=1} ^n Z_{i-1}).
 \end{equation}
 Using   Corollary \ref{CorSnSnp} and   \eqref{snBGW} yields that the last term of \eqref{checkH1}  is bounded in probability under $\P_{\theta_0}$.
 Since $\delta>0$, the first term of \eqref{checkH1} tends to $0$, which achieves the proof of  (H1').\\
 Therefore, we get that on the non-extinction set ,  under $\P_{\theta_0}$,
 %applying to $(M_n)$ Theorem \ref{LimitTheo} of the Appendix yields that,
 % under $\P_{\theta_0}$, on the non-extinction set,
\begin{equation}\label{TCLbranch}
	(\frac{M_n}{s_n},\frac{\langle M\rangle_n}{s_n^2}) \rightarrow_\mathcal{ L} (\eta N,\eta^2),
\end{equation}
with  $\eta,N$  independent, $\eta= W^{1/2}$  and $N\sim \mathcal{ N}(0,1)$.\\
%According to Theorem \ref{W}, the random variable $\eta$ is positive on $E^c$, the non-extinction set of $(Z_n)$.
To study the limit distribution of $\hat{m}_n$,  we write
$$ \hat{m}_n-m(\theta_0)= \frac{ \sum_{i=1}^n (Z_i-m(\theta_0) Z_{i-1})}{\sum_{i=1}^n Z_{i-1}} = \sigma^2(\theta_0) \frac{M_n}{\langle M_n\rangle}.$$
This yields that
$$m(\theta_0)^{n/2}(\hat{m}_n-m(\theta_0)) = \sigma^2(\theta_0)   \frac{m(\theta_0)^{n/2}}{s_n} \frac{ \frac{M_n}{s_n}}{ \frac{\langle M_n\rangle}{s_n^2}}.
$$
Using  \eqref{snBGW} and \eqref{TCLbranch} achieves the proof of (ii).
\end{proof}
%yields that  under $\P_{\theta_0}$, on $ E^c$  %$$m(\theta_0)^{n/2} (\hat{m}_n-m(\theta)) \rightarrow _\mathcal{ L} (m(\theta_0)-1)^{1/2} \sigma (\theta_0)\eta^{-1} N.$$

Let us stress that here, contrary to the previous models,  the Fisher information,
$ E_{\theta} \langle M\rangle_n=  \sigma^2(\theta) \frac{m(\theta)^{n}-1}{ m(\theta)-1}$ converges to infinity at a much faster rate than ``usually''  for  $m(\theta)>1$.
Indeed, the information contained in $(Z_1,\dots,Z_n)$ is of the same order as the information in the last observation $Z_n$. In that respect, the model is explosive in terms of growth of information.

%\pagebreak

Note that this result could  be obtained using the MLE Heuristics presented in the Appendix, substituting $\sqrt{n}$ by $s_n$ and using that\\
$\nabla_{\theta}^2 \ell_n(\theta)= - \nabla_{\theta}^2\phi(\theta)\sum_{i=1}^{n}Z_{i-1}= - \sigma^2(\theta)  \sum_{i=1}^{n}Z_{i-1}=-   \langle M(\theta)\rangle_n$.\\

Now we have estimated $m(\theta)$ instead of  $\theta$. To estimate $\theta$, we just have to consider the application $ \theta  \rightarrow m (\theta)$. Assuming that there exists $\phi$ differentiable  such that $\phi(y) =\theta= m^{-1} (y)$, an application of Theorem \ref{phimle} yields the result for $\theta$.

\subsection{Parametric inference for general Galton--Watson processes}
We assume now that the offspring distribution  $G(\cdot)$ has mean  $m$ and  finite variance $\sigma^2 $ and
consider
 the Galton--Watson process with  initial condition $Z_0=1$ and offspring distribution $G$. We assume
 %$(F_{\theta}, \theta  \in \Theta).$.
 %Denote by $m$ and $\sigma^2 $ the mean and the variance of the offspring distribution.
\begin{enumerate}[(B2)]
\item[\textbf{ (B1)}]  The offspring law $G$ satisfies $m>1$ and  $\sigma^2 <\infty$.
\item[\textbf{ (B2)}]  The offspring law $G$ has  a finite moment of  order $4$: $ E (Y^4) = \mu_{4} < \infty$  where  $Y \sim G$.
\end{enumerate}

 On the basis on the successive population  sizes $(Z_1,\dots,Z_n)$, we are concerned with the estimation of $\theta=(m, \sigma^2)$.
Denote by $\P_{\theta}$ the distribution on $(\N^{\N} ,\mathcal{ B}(N^{\N} ) )$ of  $(Z_n)$.
Under (B1), the branching process is supercritical ($m>1$) and the non-extinction set $E^c$  has a positive probability.
Clearly, studying estimators based on (\ref{LnBGW}) is intractable. Therefore, we had rather study estimators based on Conditional Least Square methods.
The conditional mean and variance of $Z_n$ with respect to $\mathcal{ F}_{n-1}$ write
\begin{equation}\label{conditmv}
	\E_{\theta}(Z_n|\mathcal{ F}_{n-1})= m Z_{n-1};\quad Var_{\theta}(Z_n|\mathcal{ F}_{n-1})= \sigma^2 Z_{n-1}.
\end{equation}
On the non-extinction set $E^c$, let us consider  the contrast function (which is a weighted Conditional Least Square method):
\begin{equation}\label{WCLS}
	U_n(\theta)= \sum_{i=1}^n \frac{1}{Z_{i-1}}(Z_i-m Z_{i-1})^2.
	\end{equation}
Note that $U_n(\theta)$ only depends on $m$ and therefore $\sigma^2$ cannot be estimated using $U_n$.\\
 Define $ \tilde{m}_n$ as
 %(\tilde{m}_n,\tilde{\sigma^2}_n)$
  a solution of
$$U_n(\tilde{m}_n)=  min_{\theta \in \Theta} U_n(\theta).$$
Hence it satisfies $\nabla_{\theta}U_n(\tilde{m}_n)=0$, which yields
\begin{equation} \label{mtilde}
 \tilde{m}_n=\frac{\sum_{i=1}^n Z_i}{\sum_{i=1}^n Z_{i-1}}.
 \end{equation}
%Note that,  as a function of the observations, $\tilde{m}_n$  has the same expression that the MLE  defined in \eqref{mtheta}.
 The simplest approach for  estimating $\sigma^2$  is to use the residual variance:
\begin{equation}\label{sigmatilde}
\tilde{\sigma}^2_n=\frac{1}{n}\sum_{i=1}^n \frac{1}{Z_{i-1}}(Z_i-\tilde{m}_n Z_{i-1})^2.
\end{equation}
Then the following holds.
%By Theorem \ref{W} $m^{-n}Z_n$ converges a.s. and in $L^2$  to  a random variable $W$, which satisfies $W>0$ on $E^c$, the non-extinction set which, under {\b
%(B1)} satisfies $\P(E^c)= 1-q >0$.
%Let $W$ denote the random variable defined in Theorem \ref{W} as $m(\theta_0)^{-n} Z_n \underset{n\rightarrow \infty} \rightarrow  W$ a.s.
\begin{theorem}\label{MceGW}
Assume (B1)--(B2). Then, on the non-extinction set $E^c$, the estimators $(\tilde{m}_n, \tilde{\sigma}^ 2_n)$
defined in \eqref{mtilde}--\eqref{sigmatilde} satisfy, as $n\rightarrow \infty$,  under $\P_{\theta}$,
\begin{enumerate}
\item[\textbf{ (i)}] $\tilde{m}_n \rightarrow m$ almost surely.
\item[\textbf{ (ii)}] $ m^{n/2} (\tilde{m}_n -m) \rightarrow _\mathcal{ L}  \sqrt{(m-1) \sigma^2 } \;\eta^{-1} \;  N $,
where  $\eta, N$ are independent r.v.s, $N \sim\mathcal{ N}(0,1)$, $ \eta$ is the positive  variable defined by $\eta^2=W$.
\item[\textbf{ (iii)}]  $\tilde{\sigma}^2_n \rightarrow \sigma^2 $ in probability under $\P_{\theta}.$
\item[\textbf{ (iv)}] $\sqrt{n} ( \tilde{\sigma}^2_n-\sigma^2) \rightarrow _{\mathcal{ L}} \mathcal{ N}(0,2 \sigma^4) .$
\end{enumerate}
\end{theorem}

\begin{proof} The study of the asymptotic properties of $\tilde{m}_n$  is similar to the previous section, since $\tilde{m}_n $
has the same expression with respect to the observations that $\hat{m}_n(\theta)$.
The proofs of (iii) and (iv) are derived from  \cite{gut91IV}, Chapter 3.
%voir demonstration page 111 Guttorp.
Let us prove (iii). We have $\tilde{\sigma}^2_n -\sigma^2 =\frac{1}{n} (A_n^1+ A_n^2+A_n^3)$ with
\begin{align*}
A_n^2&= (m-\tilde{m}_n)^2 (\sum_{i=1}^n Z_{i-1}), \\
A_n^3&= 2 (m-\tilde{m}_n) \sum_{i=1}^n (Z_i-mZ_{i-1})\mbox{ and}
\end{align*}
\begin{equation}
A_n^1= \sum_{i=1}^n  X_i \quad \mbox { with  } X_i= \frac{1}{Z_{i-1}}(Z_i- mZ_{i-1})^2 -\sigma^2.
\end{equation}
Let us  study the first term $A_n^1$. %Using that $Z_i= \sum_{k=1}^{Z_{i-1}}Y_{i-1,k} $ with $Y_{i-1,k}$ i.i.d.,
It is a centered $\mathcal{ F}_n$-martingale under $\P_{\theta}$. The computation of $\E_{\theta}(X_i^2|\mathcal{ F}_{i-1})$ relies on the property that, for i.i.d.\ random variables $Y_i$ with $\E(Y_i)=m, \Var Y_i=\sigma^2$ and finite fourth moment $\E(Y^4)= \mu_4$,  $ \bar{Y} = \frac{1}{n} \sum_{i=1}^nY_i$ satisfies
$$\E(\bar{Y}- m)^4 =  \frac{3 \sigma ^4}{n^2}+ \frac{1}{n^3}(\mu_4- 3\sigma ^4).$$
%$\E(\bar{X}- \mu)^4 =  \frac{3 \sigma ^4}{n^2}+ \frac{1}{n^3}(\mu_4- 3\sigma ^4)$, with $\mu_4= EX^4$. Hence,
Hence  on the non-extinction set $E^c$,
\begin{equation} \label{crochetAn}
\langle A^1\rangle_n= \sum_{i=1}^n (2 \sigma ^4 + \frac{1}{Z_{i-1}}(\mu_4-3 \sigma ^4)).
\end{equation}
Hence $ \Var(A_1^n) \leq 2 n(\sigma^4+\mu^4)$. Therefore, applying a strong law of large numbers for martingales (\cite{hal80IV}, Theorem 2.18) yields
$\frac{1}{n} A_n^1 \rightarrow 0 \; \P_{\theta}$-a.s.\\
  The second term is $A_n^2=  \left(m^n(\tilde{m}_n -m)^2\right) \left(\frac{1}{m^n} \sum_{i=1}^nZ_{i-1}\right)$. By (ii) and  Corollary \ref{CorSnSnp}, we
get that these two terms converge in distribution  so that  $A_n^2$ is bounded in probability. \\
%Hence $ \frac{1}{n}A_n^2 \rightarrow 0$ in probability
Noting that  $M_n= \sum_{i=1}^n(Z_{i-1}-mZ_{i-1})$ is the martingale studied in the previous section yields that
$A_n^3=  [m^{n/2} (m-\tilde{m}_n)] [ \frac{1}{m ^{n/2} }M_n]$. By (ii)  $[m^{n/2} (m-\tilde{m}_n)]$ converges in distribution. The CLT for $(M_n)$ stated in  \eqref{TCLbranch}  yields that  $m ^{n/2} M_n $ converges in distribution. Hence, $A_n^3$ is also bounded in probability.
 Joining these results  yields that $\frac{1}{n} (A_n^1+ A_n^2+A_n^3) \rightarrow 0$, which
  achieves  the proof of (iii).

Let us prove (iv). The previous computations yield that $n^{-1/2} A_n^2$ and  $ n^{-1/2} A_n^3$ both converge to $0$.
The martingale $(A_n^1)$ is centered square integrable and  $s_n^2= E_{\theta}\langle M\rangle_n$ satisfies  $ \frac{1}{n} s_n^2 \rightarrow 2 \sigma ^4$.
Condition (H1') is satisfied assuming the existence of a moment of order  $4+\delta$ with $\delta>0$ for  the offspring law $G$.
 Therefore, the CLT for martingales (see Theorem \ref{TCLMart1}) yields that  $\frac{1}{n}\langle M\rangle_n \rightarrow  2 \sigma^4$  a.s.
Joining these results achieves the proof of (iv).
 \end{proof}
%$$\sqrt{n} (\tilde{\sigma}^2_n -\sigma^2) \rightarrow \mathcal{ N} (0,2 \sigma^4)$$
With similar arguments, one can prove the asymptotic independence of $(\tilde{m}_n,\tilde{\sigma}_n^2)$.\\

The extinction probability  is an important parameter in many applications. In the early stages of an epidemic, the extinction probability corresponds to the probability of a minor outbreak.
However, unless the extinction probability $q$ is a function of $m$ and $\sigma^2$ only, it cannot be consistently estimated observing the generation sizes.
A parametric setting $(G_{\theta} ,\theta \in \Theta) $ is required for the offspring law.
Let $g(\theta,s)$ denote the generating function of $G_{\theta}$
%: $g(\theta,s) = \sum_{i\geq 0} s^k G_{\theta}(k)$
 and define
$$\tilde{q}_n= \inf\{s, g(s,\tilde{\theta}_n)=s\}.$$
Then, according to \cite{gut91IV}, under additional regularity assumptions,  $\tilde{q}_n$ is consistent if  $\tilde{ \theta}_n $ is consistent,
%converges at rate $ m(\theta_0)^{n/2}$ to $\theta$,
and  converges at the same rate $ m(\theta_0)^{n/2}$ as  $\tilde{ \theta}_n $.
% converges at rate $ m(\theta_0)^{n/2}$,
%then the same rate holds true for $\tilde{q}_n$ .
%\begin{equation}

\subsection{Examples}
\textbf{ Example 1.} Let us consider the supercritical branching process with offspring law $\mathcal{ P}oi(\lambda)$ with $\lambda>1$ and initial condition $Z_0=1$.
Theorem  \ref{MleGW} applies here and yields that, under $\P_{\lambda_0}$, on the non-extinction set $E^c$,
\begin{eqnarray*}
\hat{\lambda_n} &=& \frac{\sum_{i=1}^n Z_i}{\sum_{i=1}^n Z_{i-1}}  \rightarrow \lambda_0 \quad  \mbox{ a.s.},\\
 \lambda_0 ^{n/2} (\hat{\lambda_n} -\lambda_0 )& \rightarrow &  \sqrt{(\lambda_0-1)\lambda_0}\;  \eta^{-1} N, \mbox{ with }  \eta, N  \mbox{ independent }
 N \sim \mathcal{ N}(0,1) , \eta^2= W,
 \end{eqnarray*}
where  $W>0$ on $E^c$, $EW=1, \Var W=\frac{1}{\lambda_0-1}$.\\

\textbf{ Example 2.} Consider now the supercritical branching process with offspring law the Geometric distribution $ G$ on $N^*$ with parameter $p$ ($G(k)= p(1-p)^{k-1}, k\geq 1$).
First, note that  $\P_p(E) =0 $ and if $ Y \sim G$ , $\E(Y)= 1/p$ and $\Var Y= \frac{1-p}{p^2}$.
Assume that $0<p<1$ and that $Z_0=1$.  Theorem  \ref{MleGW}  yields that, under $\P_{p_0}$,
\begin{eqnarray*}
 \frac{1}{\hat{p}_n}  &=& \frac{\sum_{i=1}^n Z_{i}} {\sum_{i=1}^n Z_{i-1}} \rightarrow \frac{1}{p_0} \quad  \mbox{ a. s.},\\
{p_0}^{-n/2} (\frac{1}{ \hat{p}_n}-\frac{1}{p_0})& \rightarrow &\sqrt{\frac{(1-p_0)^2}{p_0^3}}\; \eta^{-1} N \mbox{ with }  \eta^2= W, \E(W)=1, \Var W=1.
\end{eqnarray*}
To estimate $p$, an application of Theorem \ref{phimle} with  $\phi(y)=1/y$ yields that, under $\P_{\theta_0}$,
$$\hat{p}_n=  \frac{\sum_{i=1}^n Z_{i-1}} {\sum_{i=1}^n Z_{i}}\rightarrow p_0,\quad
p_0^{-n/2} ( \hat{p}_n - p_0) \rightarrow \sqrt{p_0}(1-p_0) \;\eta^{-1} W.$$

\textbf{ Example 3.} Consider the  general fractional linear branching process with offspring law $G$:
$G(0)=a, G(k)= (1-a)p (1-p)^{k-1}, k\geq 1$.
Then the mean offspring is $m=\frac{1-a}{p}$ and $\sigma^2= \frac{(1-a)}{p^2} (1-p+a)$. Assume that $m>1$, the extinction set has probability $q= \frac{a}{1-p}$.
On $E^c$,
$m$ is estimated at rate $m^{n/2}$ while $\sigma^2$ is estimated at rate $\sqrt{n}$. Therefore,  $\hat{q}_n $, which  depends on $\hat{m}_n$
and $\hat{\sigma}^2_n$, is estimated at rate $\sqrt{n}$.

\subsection{Variants of Branching processes}
A large class of branching processes  are used for modeling Epidemic dynamics. It encompasses subcritical or critical branching processes (a),  branching processes with immigration (b), multitype branching processes with immigration, Crump-Mode-Jagers branching process, which are continuous time branching processes  which are no longer Markov if the time between successive generations is not exponential (c).

Case (a) can be studied either assuming that the initial population size   $\{Z_0 \rightarrow  \infty\}$ or conditionally on late  extinction (leading to quasi-stationary
distributions). Cases (b) and (c) can be studied along similar lines than the ones in the previous section.
%(see e.g.\ \cite{wei??}).\\
Stating all these results is beyond the scope of these notes. We had rather choose to present accurately the simplest case, which already contains many
problems arising in these other models.

\chapter[Inference Based on the Diffusion Approximation of Epidemic Models]{Inference Based on the Diffusion Approximation of Epidemic Models}\label{Diffusions}
\chaptermark{Inference Based on Diff.\ Approx.\ of Epidemic Models}
\section{Introduction}
\label{Diffusions:intro}

The contents of this chapter is mainly based on the three papers \cite{guy14IV}, \cite{guy15IV} and  \cite{guy16IV}.\\

In the first part of these notes, several mathematical models have been proposed to describe  Epidemic dynamics  in a closed homogeneous community.
%Recall that, for these diseases, individuals are classified as being susceptible, infected, an individual is either latent (exposed but not yet infectious), infectious or  recovered and immune .
%These states are categorized as susceptible ($S$), exposed ($E$), infectious ($I$) and recovered/immune ($R$).)
 The properties of the stochastic  $SEIR$ model have been studied in the first part of these notes.
Several mathematical formalisms were proposed to describe transitions of individuals between states: ODE/PDE (\cite{die13IV}),
difference equations and continuous or discrete-time stochastic processes (see Part I, Sections  of these notes and also \cite{dal01IV}, \cite{die13IV}),  such as point processes, Pure jump processes, renewal processes, branching processes, diffusion processes. When data are available, key parameters can be estimated using these models through likelihood-based or M-estimation methods sometimes coupled to Bayesian methods (see e.g. \cite{die13IV}). However, these data are most often partially observed (e.g.\ infection and recovery dates are not observed for all individuals during the outbreak, not all the infectious individuals are reported) and also temporally and/or spatially aggregated. In this case, estimation via likelihood-based approaches is rarely straightforward, regardless to the mathematical formalism.\\

For instance, the natural modeling of  epidemics by pure jump processes  presents systematically the drawback that inference for such models
requires that all the jumps are observed. Since these data are rarely available in practice, statistical methods rely on data augmentation in order to complete the data and add in the analysis all the missing jumps. For moderate to large populations, the complexity increases rapidly, becoming the source of additional problems.
Various approaches were developed during the last years to deal with partially observed epidemics. Data augmentation and likelihood-free methods such as the Approximate Bayesian Computation
(ABC) opened some of the most promising pathways for improvement (see e.g.\cite{bre09IV}, \cite{mck09IV}). Nevertheless, these methods do not completely circumvent the issues related to incomplete data. As stated also in \cite{cau12IV}, \cite{bri16IV}, there are some limitations in practice, due to the size of missing data and to the various tuning parameters to be adjusted (see also \cite{and00IV}, \cite{one10IV}).
Moreover, identifiability issues are rarely addressed.\\

\noindent
In this context, it appears that diffusion processes, satisfactorily approximating epidemic dynamics (see e.g.\ \cite{fuc13IV}, \cite{ros09IV}), can be profitably used for inference of model parameters from epidemiological data.  In Part I, Sections \ref{TB-EP_sec_LLN} and \ref{TB-EP_sec_CLT}, the  Markov jump process $(\mathcal{Z}^N(t))$ in a closed population of size $N$, when normalized by $N$, $(Z^N(t)= N^{-1} \mathcal{Z}^N(t))$ satisfies an ODE  as the population size $N$ goes to infinity.  In Section \ref{TB-EP_sec_DiffusApprox},  it is proved the  Wasserstein $L_1$-distance between  $(Z^N(t))$  and a multidimensional diffusion process with   diffusion coefficient proportional to $1/\sqrt{N}$ is of order $o(N^{-1/2})$ on a finite interval $[0,T]$.
Hence, epidemic dynamics can be described  using  multidimensional diffusion processes
$(X^N(t))_{t\geq 0}$ with a small diffusion coefficient proportional to $1/\sqrt{N}$. Since epidemics are usually observed over limited time periods, we consider in what follows the parametric inference based on observations of the epidemic dynamics on a fixed interval $[0,T]$. Let us stress that this approach assumes  a major outbreak
in a large community.\\

Historically, statistics for diffusions were developed for continuously observed processes which renders possible getting an explicit formulation of the likelihood (\cite{kut84IV}, \cite{lip01IV}).  In this context, two asymptotics exist for estimating  parameters in the drift coefficient of a diffusion continuously observed on a time interval $[0,T]$: $T\rightarrow \infty$ for recurrent diffusions and \{$T$ fixed and the diffusion coefficient tends to $0$\}. As mentioned above, in practice, epidemic data are not continuous, but partial, with various mechanisms underlying the missingness and leading to intractable likelihoods: trajectories can be discretely observed with a sampling interval (low frequency or high frequency observations, i.e.\  $n \rightarrow \infty$); discrete observations can correspond to integrated processes; some coordinates can be unobserved. Since the 1990s, statistical methods associated to the first two types of data have been developed (e.g.\ \cite{gen93IV}, \cite{gen00IV}, \cite{glo01IV}), \cite{kes00IV}). Recently proposed approaches for multidimensional diffusions are based on the filtering theory (\cite{fea08IV}, \cite{gen06IV}). Concerning diffusions with small diffusion coefficient from discrete observations, it was first studied in  \cite{gen90IV}, \cite{glo09IV}, \cite{sor03IV}, and more devoted to epidemic dynamics in \cite{guy14IV}, \cite{guy15IV}.
%The asymptotic properties of estimators of this type of diffusions were largely studied over the two last decades (e.g.\ \citet{lar90,gen90,gen02,sor03,glo09}) in various contexts (uni- and multidimensional cases, observations sampled at low and high frequency, discrete sampling of the state space).
Statistical inference for diffusion processes entails some special features, that we recall for sake of clarity in  \ref{Recapdiff}.
It reveals that, in the context of discrete observations, it is  important to distinguish parameters in the drift and parameters in the diffusion  coefficients because they are not estimated at the same rate.
 We detail and extend here some  recent work (\cite{guy14IV}, \cite{guy15IV}, \cite{guy16IV}) where  we focus on the parametric  inference in the drift coefficient
$b(\alpha,X^\epsilon(t))$
and in the diffusion
coefficient $\epsilon \sigma( \beta,X^\epsilon(t))$ of a multidimensional diffusion model
$\left( X^\epsilon(t) \right)
_{t\geq 0}$ with small diffusion coefficient,
when it is observed at
discrete times on a fixed time interval in the asymptotics $\epsilon \rightarrow 0$.

Section \ref{sec:diff_approx}   presents the diffusion approximation of the Markov jump process describing the epidemic dynamics starting from its  $Q$- matrix and detail these approximations for several epidemic models  studied in Part I of these notes, where another method is used to get these approximations (see  Part I, Sections \ref{TB-EP_sec_CLT} and \ref{TB-EP_sec_DiffusApprox}).
% \ref{TP-EP_sec_CLT}, \ref{TP-EP_sec_DiffusApprox}).
 We then consider the parametric inference when the epidemic dynamics is observed at discrete times on a finite interval, which corresponds to one outbreak  of the  epidemics.
The inference is studied for small sampling intervals (Section \ref{HFO}) and fixed sampling intervals (Section \ref{LFO}).
On simulated data sets of two epidemic models, the $SIR$ and the $SIRS$ with seasonal forcing (see \cite[Chapter 5]{kee11IV}), we study  the properties
of our estimators based on discrete observations of these two jump Markov processes, and compare our results to the optimal inference for these jump processes, which is obtained when all the jumps (i.e. observations of all the times of infection and recovery within the population) are observed (Section \ref{sec:simul}).

It often occurs that in practice some components of the epidemics are not observed. In the $SIR$  epidemics, the successive numbers of Susceptible for instance might be unobserved and the data consist of the successive increments of the number of Infected on each time interval.
We study in Section \ref{sec:partial}  the inference when one coordinate of  the process is observed at discrete times. We detail the results on two examples, the 2-dimensional Ornstein--Uhlenbeck diffusion process and the  diffusion approximation of the $SIR$-model when only the successive numbers of Infected are available (Section \ref{sec:SIRpart}). Finally,  Section \ref{sec:SIRSgrippe} is devoted to the estimation  based on the real data set on Influenza epidemics, which is described by an $SIRS$ epidemic model.

\section{Diffusion approximation of jump processes modeling epidemics}
\label{sec:diff_approx}

This section starts from the definition of the stochastic epidemic model by a Pure jump Markov process  $(\mathcal{Z}^N(t))$ on $\Z^d$ specified by its  $Q$ - matrix.
 We detail how to get the diffusion approximation of $(\mathcal{Z}^N(t))$ from this description, which is another way for getting  the diffusion process obtained  in  Part I, Section \ref{TB-EP_sec_DiffusApprox} of these notes.
%a general approach for building multidimensional and obtain iffusion processes with small diffusion coefficient by approximating Markov jump processes, based on Guy et al.. (2015)
Using limit theorems for stochastic processes, we characterize the limiting Gaussian process. Then, based on the theory of  small perturbations of dynamical systems
(\cite{fre84IV}), we link the normalized process to a diffusion process with small diffusion coefficient.
% (\ref{generic}).
These approximations are then applied to $SIR$, $SEIR$, and $SIRS$ models for epidemic dynamics.
%(\ref{SIR_appli}).

%------------------------------------------------------------------------------------------------------
\subsection{Approximation scheme starting from the jump process $Q$-matrix} \label{generic}

Let $(\mathcal{Z}^N(t))$   a multidimensional Markov jump process  with state space $ E \subset  \Z^p$ which describes the epidemic dynamics in a closed population of size $N$, the integer ``$ p$'' corresponding to the number of health states  in the infection dynamics model.
%is a multidimensional Markov jump process $(Z(t), t \geq 0)$ with state space $ E \subset  \Z^p$, where

This process is described by an initial distribution  on $E$ and  a collection of non-negative functions $(\beta_j(t, \cdot): E \rightarrow \R^+)$  indexed by $j \in \Z^p$, $j \neq (0,\dots,0)$, that satisfy,
\begin{equation}\label{H:bound_alpha}
\forall i \in E, 0 <\sum _{j \in\Z^p}
\beta_j(t,i)  =\beta(t,i)< \infty.
\end{equation}
These functions are the transition rates of the process   $(\mathcal{Z}^N(t))$  with $Q(t)$-matrix having as elements
\begin{equation}\label{trans}
q_{i, i+j}(t)=  \beta_j(t,i) \; \mbox{ if } \;  j\neq 0,\; \mbox{ and  } q_{i,i}(t)=-\beta(t,i)  \; \mbox{  for  } \; i ,i+j  \in E.
\end{equation}
Another useful description of  $(\mathcal{Z}^N(t))$   is based on the joint distribution of its jump chain and holding times.
The process stays in  each state $i \in E$ during an exponential time $\mathcal{ E}(\beta(t,i))$, and then jumps to the state $i+j$ according to a Markov chain $(X_n)$ with transition probabilities ${\mathbb P}(X_{n+1}=i+j\; | \; X_n=i)= \beta_j(t,i)/\beta(t,i).$ \\

\noindent
We consider  the class  of  density dependent Markov jump processes $(\mathcal{Z}^N(t))$  which possess a limit behaviour when normalized by the population size $N$.
Let us define the two sets
\begin{equation} \label{E}
E= \{0,\dots,N\}^p \quad E^-=\{-N,\dots, N\}^p.
\end{equation}
The state space of  $(\mathcal{Z}^N(t))$  is $E$ and  its jumps belong to $E^-$. \\
From the original jump process $(\mathcal{Z}^N(t))$  on $E= \{0,\dots, N\}^p$, let  %$(Z_N(t), t \geq 0) $ where
\begin{equation}\label {ZN}
 Z^N(t)=\frac{\mathcal{Z}^N(t)}{N} \mbox{ with state space } E_N=\{N^{-1}i,  i \in E\}.
\end{equation}
Its jumps are now   $y=j/N$  and  transition rates from $z \in E_N$ to $z+j/N$ at time $t$ defined using \eqref{trans},
\begin{equation}\label{barZN}
q^{N}_{z,z+y}(t)= \beta_{Ny}(t,Nz). % =\alpha_l(k) \mbox{ if } y=k/N, z=l/N).
\end{equation}

\noindent
Denote for  $x =(x_1,\dots, x_p) \in {\mathbb R}^d$, $[x] = ([x_1], \dots, [x_p ] )\in {\mathbb Z}^p$, where  $[x_i]$ is the integer part of $x_i$.\\

\noindent
We assume in the sequel  that  $(\mathcal{Z}^N(t))$   is density dependent,  i.e.\  there exist a collection of functions $\beta_j : \R^+ \times [0,1]^p \rightarrow \R^+$ such that,
\begin{enumerate}[(H2)]
\item[\textbf{ (H1)}] $\forall j,\; \forall z \in  [0,1]^p\;\;$
$\frac{1}{N} \beta_{j}(t,[Nz]) \rightarrow \beta_j(t,z)$  as $N \rightarrow \infty$  locally uniformly in $t$.
\noindent
\item[\textbf{ (H2)}] $\forall  j\in E^-, \; \beta_j(t,z) \in C^2(\R^+, [0,1]^p).$
\end{enumerate}

\noindent
%A jump process $(Z(t)) $ with state space $E$ (see (\ref{E})) is density dependent if it satisfies (H1).\\
Then, define the two functions $b^N (t,z) $ and $ b (t,z): \R^+\times [0,1]^p \rightarrow {\mathbb R}^p$ and  the two $p\times p$ positive symmetric matrices $\Sigma^N$ and $\Sigma$  (with the notation $M^{\star}$ for the transposition of a matrix or of a  column vector $j$ in  $E$),
\begin{equation}\label{defb}
 b^N(t,z)= \frac{1}{N}\sum_{j \in E^-} \beta_j(t,[Nz]) j ; \quad  b(t,z) = \sum_{j \in E^-} \beta_j(t,z) j  ;
\end{equation}
\begin{equation}\label{def:Sigma}
\Sigma^N (t,z)=  \frac{1}{N}\sum _{j\in E^-} \beta_j(t, [Nz]) j j^{\star};\quad  \Sigma(t,z)=  \sum _{j \in E^-} \beta_j(t,z) j j^{\star}.
\end{equation}
\noindent
Under (H1)  the functions $b(t,z)$ and $ \Sigma(t,z)$ are well defined and $b(t,z)$ is Lipschitz under  (H2). Therefore,
there exists  a unique smooth  solution $z(t)$ to the ODE
\begin{equation}\label{ODEPARTIV}
\frac{dz}{dt}= b(t, z(t))dt\; ; \quad z(0)=x.
\end{equation}
Let $\nabla_z b(t,z)$ denote the gradient of $b(t,z)$
\begin{equation}\label{gradb}
\nabla_z b(t,z) = \bigl(\frac{\partial b_i}{\partial z_j}(t,z)\bigr)_{1\leq i,j\leq p} \; .
\end{equation}
The  resolvent matrix  $\Phi(t,u)$ associated with \eqref{ODEPARTIV} is defined as the solution
 \begin{equation}\label{def:phi}
		\frac{d\Phi}{dt}(t,s)=\nabla_z b(t, z(t))\Phi(t,s) ; \quad \Phi(s,s)=I_p .
\end{equation}

\noindent
Under (H1), (H2)  the following holds:  if   $Z^N(0)  \rightarrow x$ as $N\rightarrow \infty$,  then,  locally uniformly in $t$,
%using definitions  (\ref{ZN}, (\ref{def:Sigma}), (\ref{ODEPARTIV}) and (\ref{def:phi})
\begin{equation}\label{CVU}
	\forall t \geq 0, \lim_{N\rightarrow \infty}  \parallel
Z^N(t)-z(t) \parallel \; =0\;\mbox{  a.s.}
\end{equation}
where $z(t)$ is solution of \eqref{ODEPARTIV}.\\
Let $(D,\mathcal{ D})$ denote the space of ``cadlag''
functions $\{f: \R^+ \rightarrow \R^p\}$ endowed with the Skorokhod topology. Then,
\begin{equation}\label{G}
\sqrt{N}(Z^N(t) -z(t))_{t\geq 0} \rightarrow  (G(t))_{t\geq 0}  \; \mbox{in distribution in } (D,\mathcal{ D}),
\end{equation}
\noindent
where  $(G(t))$ is a centered  $p$-dimensional Gaussian process with covariance matrix
 \begin{equation}\label{CovG}
\mathrm{Cov}(G(t),G(r))= \int_0 ^{t\wedge r}\Phi(t,u) \Sigma(u,z(u))\; \Phi^{\star}(r,u) du.
\end{equation}
The proofs of these results are given under a general form in  Part I, Sections \ref{TB-EP_sec_LLN}  and \ref{TB-EP_sec_CLT} of these notes,
and based on this presentation in  \cite{guy14IV}, \cite{guy15IV}.\\

\noindent
Heuristically,  there is an approach which yields the diffusion approximation of $(Z^N(t)) $; it rests on an expansion of the generator  $\mathcal{ A}_N$ of  $(Z^N(t))$ (\ref{ZN}).
%is a Markov process with generator $\mathcal{ A}_N$ defined
%\sout{on $C^2(\R^p,\R)$, where $\mathcal{ A}_N$ can be written }
For $f \in C^2(\R^+\times\R^p,\R) $, it reads as
$$\mathcal{ A}_N f(t,z)= \sum_{j\in E ^-} \beta_j(t, Nz)(f(t,z+\frac{j}{N})-f(t,z)).$$
A Taylor expansion of $\mathcal{ A}_N f(t,z)$ yields, using (H1), (H2) and (\ref{defb}), for $j= (j_1,\dots,j_p)^*  \in E ^-$,
\begin{align*}
\mathcal{ A}_N f (t,z) =& \sum_{j\in E ^-}  N \beta_j(t, z) (f(t,z+\frac{j}{N})-f(t,z))+ o(1/N)\\
 = &\, (\nabla_z f(t,z))^{\star} \; b(t,z) +  \frac{1}{2N}  \left( \sum_{j\in E ^-}
\beta_j(t,z) \sum_{k,l= 1}^d   j_k  \; j_l \; \nabla^2_{z_k z_l } f(t,z) \right)\\
&\phantom{\,(\nabla_z f(t,z))^{\star} \; b(t,z)}+ o(1/N) \\
 =&\, (\nabla_z f(t,z))^{\star} \; b(t,z) + \frac{1}{2N}\sum_{k,l= 1}^d \Sigma_{kl}(t,z) \;\nabla^2_{z_k z_l } f(t,z) + o(1/N),
\end{align*}
%frac{\partial^2 f}{\partial z_k \partial z_l}(t,z)
where the last equality is obtained using (\ref{def:Sigma}). The first two terms of the last expression correspond to the generator of a diffusion process on $\R^p$ with drift coefficient $b(t,\cdot)$ and diffusion matrix  $\frac{1}{N} \Sigma(t,\cdot)$,
 %Approximation diffusion de $Z_N(t)$: diffusion de m\^eme g\'en\'erateur
\begin{equation}\label{XN}
dX^N(t)= b(t,X^N(t)) dt + \frac{1}{\sqrt{N} }\sigma (t,X^N(t)) dB(t) \; ; \; \; X^N(0)=x,
\end{equation}
%	\; X_N(0)=Z_N(0)$
where   $(B(t)_{t\geq 0})$ is a Brownian motion on $\R^p$ defined on a probability space $(\Omega,(\mathcal{ F}_{t})_{t \geq 0}, {\mathbb P})$ independent of $X^N(0)$, and
$\sigma(t,\cdot)$ is a square root of $\Sigma(t,\cdot)$:
  $\sigma(t,z)\; \sigma(t,z)^{\star} =\Sigma(t,z).$\\
%Let us note that, contrary to the first approach, this does not yield an approximation result between the sample paths of $(Z_N)$ and $(X_N)$. Indeed, no similar limit theorem can justify that $(X_N)$ approximates $(Z_N)$, since the $(Z_N)$ are not expressible in terms of any kind of limiting process (see Ethier \& Kurtz , chapter 11, section 3). A coupling theorem of Kolmos-Major-Tusnady is required to get a direct comparison. The main interest of this second approach lies in
%the fact that it allows all the mathematical developments available for diffusion processes to be used in our framework.\\

\noindent
These approaches can be connected together  a posteriori  using  the theory of random perturbations of dynamical systems  (\cite{aze82IV}, \cite{fre84IV}) and the following theorem.

\begin{theorem} \label{th_Taystoch}
% Assume that  $Z^N(0) \rightarrow x $ as  $N\rightarrow \infty$.
Setting $\epsilon= 1/\sqrt{N}$, the paths of $X^N(\cdot)$ satisfy, as $\epsilon \rightarrow 0$,
\begin{equation}\label{TS}
	X^N (t) =X^{\epsilon}(t)= z(t)+ \epsilon g(t)+ \epsilon^2 R^{\epsilon}(t),\; \mbox{ with } \sup_{t\leq T} \parallel \epsilon R^{\epsilon}(t)\parallel \rightarrow 0 \mbox{  in probability,}
\end{equation}
where  $z(t)$ is the solution of \eqref{ODEPARTIV}, $B(t)$ is a $p$-dimensional Brownian motion and $(g(t))$ is the process  satisfying  the  SDE
$$ dg(t)= \nabla_z b(t,z(t)) g(t)dt +\sigma(t,z(t)) dB(t),\;\; g(0)=0. $$
\end{theorem}

\noindent
This stochastic differential equation can be solved explicitly and we get, using (\ref{def:phi}), that
%the closed form  for $g$,
\begin{equation}\label{eq:gtstoch}
	g(t)= \int_0^t\Phi(t,s)\sigma(s, z(s)) dB(s).
\end{equation}
Hence, $(g(t))$ is a centered  Gaussian process having the same
covariance matrix (\ref{CovG}) as the process $(G(t))$  defined in (\ref{G}).
Therefore, for
$\epsilon= 1/\sqrt{N}$, $\sqrt{N}(Z^N_t-z(t))_{t\geq 0}$
and $\epsilon^{-1} \left(Z^{\epsilon}(t)-z(t)\right)_{t\geq 0}$ converge to a Gaussian process
having the same distribution.\\

It is moreover proved in Part I, Section \ref{TB-EP_sec_DiffusApprox} of these notes,
that the  Wasserstein $L^1$ distance between $(Z^N(t))$ and $(X^N(t))$ converges to $0$.
%\noindent
%%\begin{rem}
%For sake of clarity, we have detailed above these approximations    for diffusion processes with drift and diffusion coefficients depending  only on the state space variable $x$.However, this expansion of $X_{\epsilon}(\cdot)$ is still true for time-dependent jump processes leading to diffusion processes  $(X_N(\cdot))$  with drift term and diffusion matrix  respectively $b(t,x)$  and $\frac{1}{N}\Sigma(t,x)$: the corresponding ODE and Gaussian process are respectively,
%\begin{equation}\label{xtgt}
% x(t)= x_0 +\int_0^t b(u, x(u))du, \quad g(t)= \int_0^t\Phi(t,s)\sigma(s, x(s)) dB(s).
%\end{equation}
%\noindent
\subsection{Diffusion approximation of some epidemic models }\label{sec:approxdiif}
\subsubsection{The diffusion approximation applied to the SIR epidemic model}
\label{SIR_appli}
We apply first the generic method leading successively to $b(\cdot)$, $\Sigma (\cdot)$ and $ (X^N)$ described in \ref{generic} to the $SIR$ model introduced i%n Section \ref{SIRS_appli} and
 in Part I,  Chapter \ref{TB-EP_chap_StochMood} of these notes  through the 2-dimensional continuous-time Markov jump process $\mathcal{Z}^N(t)=(S(t),I(t))$ to  build the associated $SIR$ diffusion process. Along to its initial state
$\mathcal{Z}^N(0)= (S(0), I(0))$, the Markov jump process is characterized by two transitions,
$(S,I)\stackrel{\frac{\lambda}{N}SI}{\longrightarrow}(S-1,I+1)$ and
$(S,I)\stackrel{\gamma I}{\longrightarrow}(S,I-1)$.
Parameters $\lambda$ and $\gamma=1/d$ represent the transmission rate and the recovery rate (or the inverse of the mean infection duration $d$), respectively. The rate $\lambda SI / N$ translates two main assumptions: the population mix homogeneously (same $\lambda$ for each pair between one $S$ and one $I$) and the transmission is proportional to the fraction of infectious individuals in the population, $I/N$ (frequency-dependent formulation of the transmission term).\\

\noindent
%\sout{following the three-step algorithm introduced here above}.
The diffusion approximation of the process$( \mathcal{Z}^N(t))$ describing the epidemic dynamics can be summarized by three steps.
The original $SIR$ jump process in a closed population has state space $\{0,\dots, N\}^2$, the  jumps $j$ are $(-1,1)$ and
$(0,-1)$ with transition rates,
$$q_{(S,I),(S-1,I+1)} = \lambda S\;\frac{I}{N}=\beta_{(-1,1)}(S,I);\quad q_{(S,I),(S,I-1)} = \gamma I =\beta_{(0,-1)}(S,I).$$
Normalizing $(\mathcal{Z}^N(t))$ by the population size $N$, we obtain, setting  $z=(s,i) \in [0,1]^2$, as $N \rightarrow \infty$,
$$\frac{1}{N}\beta_{(-1,1)}([Nz]) \rightarrow \beta_{(-1,1)}(s,i)= \lambda s i;\quad \frac{1}{N}\beta_{(0,-1)} ([Nz]) \rightarrow \beta_{(0,-1)}(s,i)= \gamma i.$$
These two limiting functions clearly satisfy (H1)--(H2).
Finally, the two functions given  in (\ref{defb}), (\ref{def:Sigma}) are well defined and now depend on $(\lambda,\gamma)$. \\
Set $\theta= (\lambda,\gamma)$ and denote by $b(\theta,z)$ and $\Sigma(\theta,z)$  the associated functions.
 We get
\begin{equation}\label{diffSIR}
b(\theta,(s,i))= \begin{pmatrix}
-\lambda si \\ \lambda si -\gamma i \end{pmatrix};\quad
\Sigma(\theta,(s,i))
=  \begin{pmatrix} \lambda si & -\lambda si \\ - \lambda si & \lambda si
+ \gamma i \end{pmatrix}.\end{equation}

Assume that $ \mathcal{Z}^N(0)$ satisfies $(N^{-1}S(0),N^{-1}I(0)) \rightarrow x= (s_0,i_0)$  with $s_0>0$, $i_0>0$, $s_0+i_0 \leq 1$ as $N\rightarrow \infty$.
Then the associated ODE is, using \eqref{ODEPARTIV},
\begin{equation} \label{ODESIRIV}
\frac{ds}{dt}= - \lambda si \;; \quad \frac{di}{dt}= \lambda si -\gamma i \; ; \quad  (s(0), i(0))=(s_0,i_0).
\end{equation}
The diffusion approximation of the $SIR$ epidemics obtained in (\ref{XN}) is the solution of the SDE
%On a filtered probability space $(\Omega,\mathcal{ F}=(\mathcal{ F}_t), {\mathbb P})$, let $ \displaystyle{B(t)= \begin{pmatrix}B_1(t)\\B_2(t)\end{pmatrix}}$ denote a standard two-dimensional Brownian motion and consider the Choleski decomposition of $\Sigma (\theta,y)$, \\
%$$\sigma(\theta,(s,i))= \begin{pmatrix} \sqrt{\lambda si} & 0 \\ -\sqrt{\lambda si} & \sqrt{\gamma i} \end{pmatrix}.$$
% $ X_N (t)= \begin{pmatrix}S_N (t)\\ I_N (t) \end{pmatrix}$ satisfies the stochastic differential equation defined by $X_{N}(0)=x_0$ and \\
\smallskip

\resizebox{0.99\linewidth}{!}{
  \begin{minipage}{\linewidth}
\begin{align*}
dS^N (t)  &= -\lambda S^N (t) I^N (t) dt+\frac{1}{\sqrt{N}}\sqrt{\lambda S^N (t) I^N (t) } \; dB_1(t),  \quad  S^N(0) = s_0,\\
dI^N (t)  &=  (\lambda S^N (t) I^N (t)-\gamma I^N (t)) dt -\frac{1}{\sqrt{N}}\left(\sqrt{\lambda S^N (t) I^N (t)}\;dB_1(t)-\sqrt{\gamma\; I^N (t) }\;  dB_2(t)\right),\\
&\qquad I^N(0) = i_0,
\end{align*}
\end{minipage}
}
\smallskip

\noindent where $(B(t))$  is  standard two-dimensional Brownian motion and  $\sigma(\theta,z)$ corresponds to the Choleski decomposition of $\Sigma (\theta,z)= \sigma(\theta,z)\sigma^{\star}(\theta,z)$, \\
$$\sigma(\theta,(s,i))= \begin{pmatrix} \sqrt{\lambda si} & 0 \\ -\sqrt{\lambda si} & \sqrt{\gamma i} \end{pmatrix}.$$

\noindent
In order to visualize the influence of the population size $N$ on the sample paths of the normalized jump process $Z_N(t)= \mathcal{Z}^N(t)/N$, several trajectories
have been simulated using an $SIR$ model with parameters $(\lambda,\gamma)=(0.5, 1/3)$,  so that $R_0= \lambda/\gamma=1.5$. Results are displayed in Figure \ref{fig:Nvar}. We observe that, as the population size increases, the stochasticity of sample paths decreases. However, it   still keeps a non-negligible stochasticity for a large population size ($N=10000$). Since the peak of $I^N(t)$ is quite small
(about $0.08$ here), this can be explained by a  moderate size of the ratio  ``signal over noise''  even for large $N$ (here of order $0.08/0.01$).

\begin{figure}[ht]
\centering
\includegraphics[width=0.9\textwidth]{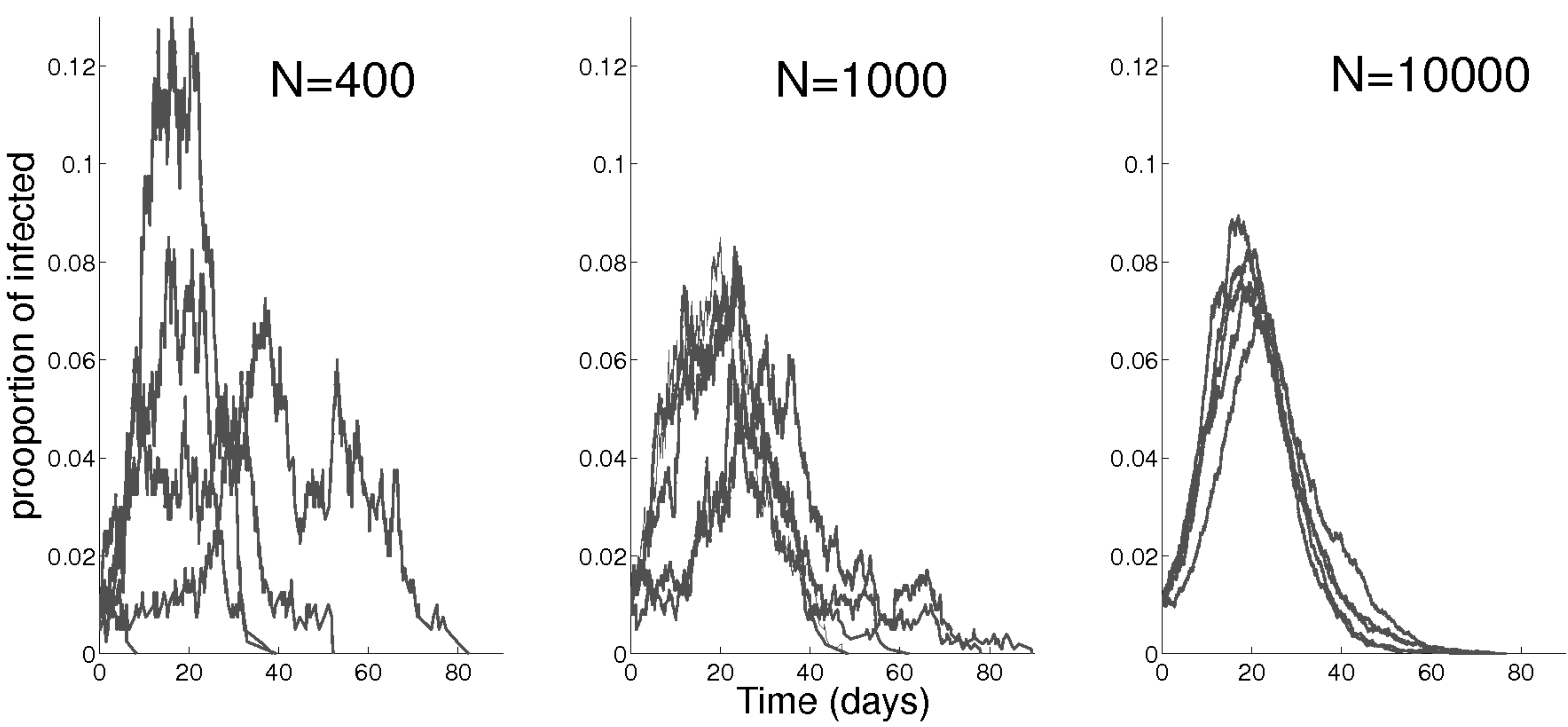}\\
\caption{Five simulated trajectories of the proportion of infectious individuals over time using the $SIR$ Markov jump process for $(s_0,i_0)=(0.99,0.01)$ $(\lambda,\gamma)=(0.5,1/3)$ and for each $N=\{400,1000,10000\}$ (from left to right).}
\label{fig:Nvar}
\end{figure}

%------------------------------------------------------------------------------------------------------
\subsubsection{The diffusion approximation applied to the $SIRS$  epidemic model with seasonal forcing}
\label{SIRS_appli}

Another important class of epidemics models
%introduced in Section \ref{sec:intro}
 is the $SIRS$ model, which allows possible reinsertion of removed individuals into $S$ class. The additional transition reads as $(S,I)\stackrel{\delta (N-S-I) }{\longrightarrow}(S+1,I)$, where $\delta$ is the average rate of immunity waning. To mimic recurrent epidemics, additional mechanisms need to be considered.
 Indeed, to avoid that successive epidemics cycles die out, one way is to introduce an external immigration flow.
 Hence, one possible model to describe recurrent epidemics is the $SIRS$ model with seasonal transmission (at rate $\lambda(t)$), external immigration flow in the $I$ class (at rate $\eta$) and, when the time-scale of study is large, demography (with birth and death rates equal to $\mu$ for a stable population of size $N$). Seasonality in transmission is captured using a time non-homogeneous transmission rate, expressed under a periodic form \begin{equation}\label{lambdat}\lambda(t):=\lambda_0(1+\lambda_1
\sin(2\pi t/T_{per}))\end{equation}
where $\lambda_0$ is the baseline transition rate, $\lambda_1$ the intensity of the seasonal effect on transmission and $T_{per}$ is introduced to model an annual or  t seasonal trend (see \cite{kee11IV}, Chapter 5). Typically for modeling Influenza epidemics, we fixed it at $T=365$.\\

Assuming again a constant population size, we obtain a new two-dimensional system with four transitions for the corresponding Markov jump process:
\begin{eqnarray*}
(S,I)\stackrel{\frac{\lambda(t)}{N} S(I+N\eta)}{\longrightarrow}(S-1,I+1) &; & \quad (S,I)\stackrel{\mu S}{\longrightarrow}(S-1,I) ;\\
(S,I)\stackrel{(\gamma+\mu)I}{\longrightarrow}(S,I-1) &;& \quad (S,I)\stackrel{\mu N+\delta(N-S-I)}{\longrightarrow}(S+1,I).
\end{eqnarray*}

\noindent
Figure \ref{fig:SIRSbif} illustrates the dynamics of the $SIRS$ model (in ODE formalism) which is forced using sinusoidal terms. In particular, given the parameter values we have chosen, we can notice two distinct regimes: one with annual cycles (top graph) and the other with biennial dynamics (middle graph). The qualitative changes in model dynamics are explored by modifying a control parameter or \textit{bifurcation parameter} (here $\lambda_1$) and constructing a \textit{bifurcation diagram}.
\begin{figure}[ht]
\centering
	\includegraphics[width=0.8\textwidth]{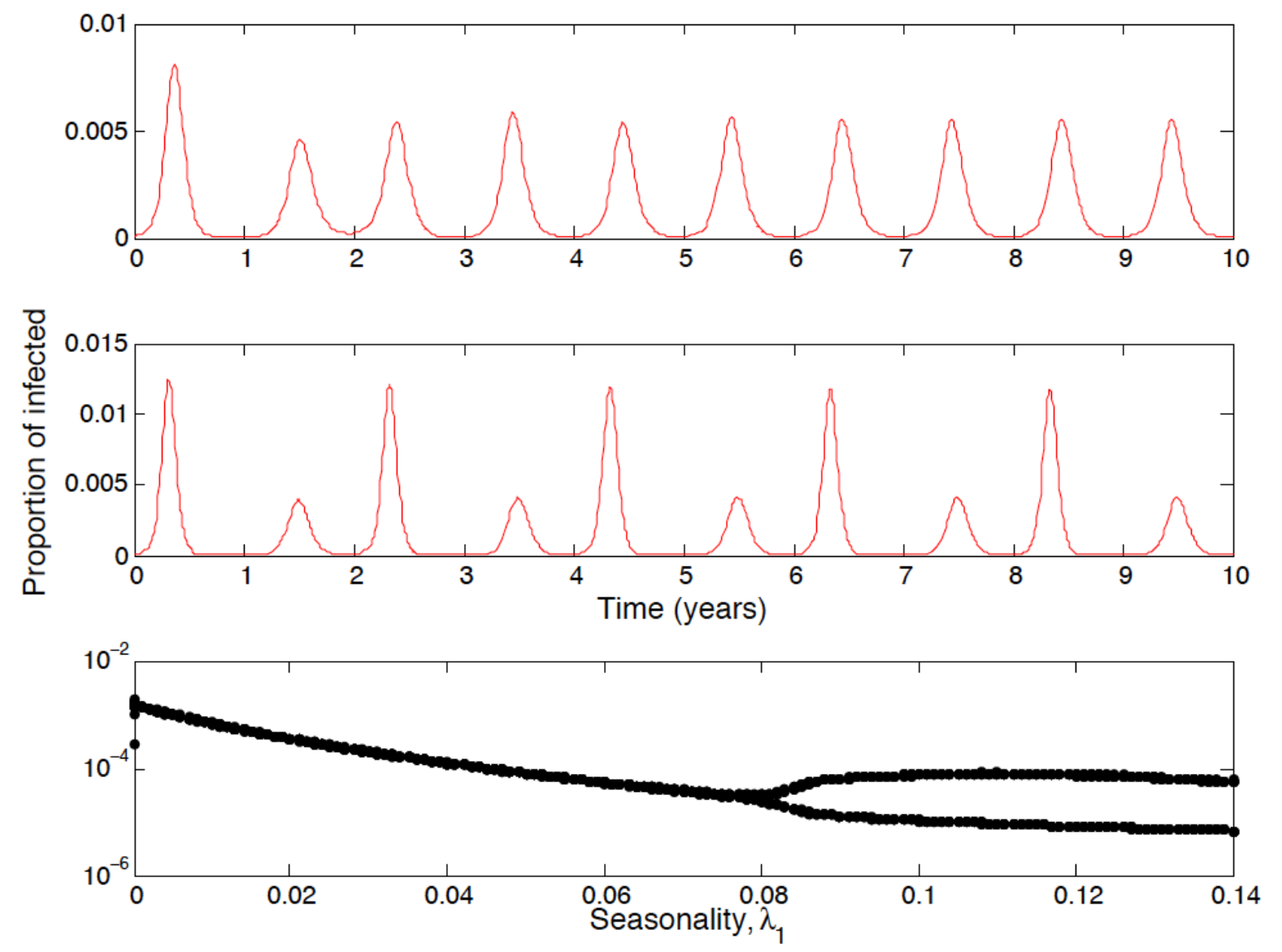}
\caption{Proportion of infected individuals, $I(t)$, over time (top and middle panels) simulated using the ODE variant of the $SIRS$ model with $N=10^7$, $T_{per}=365$, $\mu=1/(50 \times T_{per})$, $\eta=10^{-6}$, $(s_0,i_0)=(0.7,10^{-4})$ and $(\lambda_0,\gamma,\delta)=(0.5,1/3,1/(2\times 365))$. The top panel corresponds to $\lambda_1=0.05$, the middle panel to $\lambda_1=0.1$. The bottom panel represents the bifurcation diagram with respect to $\lambda_1$.}
\label{fig:SIRSbif}
\end{figure}

The diffusion approximation is built following the same generic scheme of Section \ref{generic} as for the $SIR$ model in Section \ref{SIR_appli}.
The four jumps $j$ corresponding to  functions $\beta_j$ are $j^*=(-1,1); (-1,0); (0,-1); (1,0)$ leading to
\begin{align*}
\beta_{(-1,1)}(t, S,I)= \frac{\lambda(t)}{N}S(I+N\eta),& \quad  \beta_{(0,-1)}(t, S,I)= (\gamma+\mu )S,\\
\beta{(0,-1)}(t, S,I)= (\gamma+\mu )S,& \quad\;\; \beta_{(1,0)}(t, S,I)= \mu N+\delta(N-S-I) S.
\end{align*}
The jump process is time-dependent and so we have to check (H1b)--(H2).
Straightforward computations yield that they are satisfied since, for $(s,i)
\in [0,1]^2$,
\begin{align*}
\beta_{(-1,1)}(t,(s,i))&=\lambda(t)s(i+\eta);\quad \beta_{(-1,0)}(t,(s,i))=\mu s;\\
\beta_{(0,-1)}(t,(s,i))&=(\gamma +\mu)i;\quad\quad \beta_{(1,0)}(t,(s,i))= \mu+\delta(1-s-i).
\end{align*}
%Clearly, (H1b) and (H2b) are satisfied.
Finally, setting $\theta=(\lambda_0,\lambda_1,\gamma,\delta,\eta,\mu)$, the associated drift function
$b(\theta,t,(s,i))$ and diffusion matrix $\Sigma(\theta,t, (s,i))$ are
\begin{equation}\label{bSIRStempIV}
	\begin{array}{ll}b(\theta,t, (s,i))
&=\begin{pmatrix}-\lambda(t) s(i+\eta) +\delta(1-s-i)+\mu(1-s)\\ \lambda(t) s(i+\eta)-(\gamma+\mu) i \end{pmatrix}
\end{array},
\end{equation}
\begin{equation}
	\begin{array}{lll}\Sigma(\theta,t, (s,i))&=
\begin{pmatrix}\lambda(t) s(i+\eta)+\delta(1-s-i)+\mu(1+s)&-\lambda(t) s(i+\eta)\\-\lambda(t) s(i+\eta)&\lambda(t) s(i+\eta)+(\gamma+\mu) i\end{pmatrix}\end{array}.
\end{equation}
Therefore, the  associated ODE  is, using \eqref{bSIRStempIV},
\begin{eqnarray*}
\frac{ds}{dt}&=& -\lambda(t) s(i+\eta) +\delta(1-s-i)+\mu(1-s),\quad  s(0)=s_0;\\
 \frac{di}{dt}&=& \lambda(t) s(i+\eta)-(\gamma+\mu) i,\quad  i(0)=i_0.
 \end{eqnarray*}
%\begin{pmatrix} \frac{ds}{dt}\\ \frac{di}{dt}\end{pmatrix}= b(\theta,t, (s,i))dt\;; \quad (s(0),i(0)= (s_0,i_0).
%\end{equation}
\noindent
Choosing $\sigma(\theta, t, (s,i))$ such that $\sigma (\theta, t, (s,i))  \sigma(\theta, t, (s,i))^{\star}= \Sigma (\theta,t ,(s,i))$, we obtain that the approximating diffusion $X_N(t)$ satisfies
\begin{equation}\label{Xsirs}
	dX^N(t)= b(\theta,t, (S_N,I_N))dt+ {\frac{1}{\sqrt N}}\sigma(\theta,t (S^N,I^N));\;\; X^N(0)= x.
	\end{equation}

\subsubsection{A Minimal model for Ebola Transmission with temporal transition rate}
According to \cite{cam15IV},
a  basic  model for Ebola dynamics  consists in a $SEIR$ model with  temporal transmission rate.
In a rough approximation, assuming homogeneous mixing in a size $N$ community yields, setting  $ \mathcal{Z}^N(t)= (S,E,I) $,
 %$SEIR$ model with  temporal transmission rate
%\includegraphics[width=0.5\textwidth]{Images/SEIR-rates.pdf}
\begin{align*}
&(S,E,I)\overset{\lambda (t)\frac{SI}{N}}{\longrightarrow }\; (S-1,E+1,I);\\
&(S,E,I)\;\overset{\nu E}{\longrightarrow} (S,E-1, I+1);\\
&(S,E,I)\overset{\gamma I}{\longrightarrow} (S,E, I-1).
\end{align*}
The diffusion approximation has drift and diffusion matrix given by, for $z=(s,e,i)$,
$$b(\theta,t,z)= \begin{pmatrix} -\lambda(t) si\\ \lambda(t) s i-\nu e\\ \nu e -\gamma i \end{pmatrix} ;\quad
\Sigma(\theta,t,z)=  \begin{pmatrix} \lambda(t) s i&-\lambda(t) s i&0 \\-\lambda(t) s i&\lambda(t) s i+\nu e& -\nu e\\
0& -\nu e& \nu e+\gamma i \end{pmatrix}.$$

Two questions concerning the inference arise in this model:  the non-parametric estimation of $\lambda(\cdot)$ and the presence of random effects since the dynamics are observed in different locations.

 \subsubsection{Two variants of the $SEIRS$ model with demography}
 In  Part I, Chapter \ref{chap_MarkovMod}  of these notes, an example of $ SEIRS$ model with demography is proposed (see Example \ref{ex:SEIRS}).
 Removed individuals loose their immunity at rate $\delta$; there is an influx of susceptible at rate $\mu N$ and individuals, whichever type, die at rate $\mu$.
 Hence, 9 jumps are present in this model, for $(s,e,i,r)$, which yields for $Z=(S,E,I)$,
 $$(S,E,I)\overset{\lambda \frac{SI}{N}}{\longrightarrow } (S-1,E+1,I),\quad (S,E,I)\overset{\nu E}{\longrightarrow} (S,E-1, I+1),$$
$$  (S,E,I)\overset{ \mu N +\delta(N-S-E-I)} {\longrightarrow} (S+1,E, I), \quad (S,E,I)\overset{ \mu I +\gamma I} {\longrightarrow} (S,E, I-1).$$
 $$ (S,E,I)\overset{ \mu S} {\longrightarrow} (S-1,E, I),\quad
  (S,E,I)\overset{ \mu E} {\longrightarrow} (S,E-1, I),\quad
(S,E,I) \overset{ \mu(N-S-E-I)}{\longrightarrow} (S,E, I).$$
% $$\begin{pmatrix} -1\\1\\0\\0 \end{pmatrix};
% \begin{pmatrix}   0\\-1\\1\\0 \end{pmatrix}; \begin{pmatrix}  0\\0\\-1\\1 \end{pmatrix}; \begin{pmatrix} 1\\0\\0 \\-1 \end{pmatrix};
%\begin{pmatrix} 1\\0\\0\\0 \end{pmatrix}; \begin{pmatrix} -1\\0\\0\\0 \end{pmatrix}; \begin{pmatrix}  0\\-1\\0\\0 \end{pmatrix}; \begin{pmatrix}  0\\0\\-1\\0 \end{pmatrix};
%\begin{pmatrix} 0\\0\\0\\-1 \end{pmatrix}$$
 This yields, setting  $z=(s,e,i)$ and $\theta=(\lambda, \nu,\gamma,\delta,\mu )$
\begin{eqnarray*}
b(\theta,z)&=& \begin{pmatrix}-\lambda s i +\mu (1-s) + \delta  (1- s- i- e) \\  \lambda s i-(\mu+\nu) e \\ \nu e -(\gamma+\nu) i \end{pmatrix}; \\
\Sigma(\theta,z)&=&\begin{pmatrix}\lambda s i +\mu (1+s)+ \delta (1- s- i- e) & - \lambda s i &0\\
- \lambda s i &  \lambda s i +(\mu+\nu)e& -\nu e\\
 0& -\nu e&  \nu e+(\gamma+\nu)i \end{pmatrix}.
\end{eqnarray*}
%Concerning the variant of this model (Example 2.2.3), easy computations yield that the approximating diffusion is also the one given above.
%  j_2= \begin{pmatrix}   0\\-1\\1\\0 \end{pmatrix}; j_3= \begin{pmatrix}  0\\0\\-1\\1 \end{pmatrix};  j_4= \begin{pmatrix} 1\\0\\0 \\-1 \end{pmatrix};$$
%
%$$j_5=\begin{pmatrix} 1\\0\\0\\0 \end{pmatrix}; j_6= \begin{pmatrix} -1\\0\\0\\0 \end{pmatrix}; j_7= \begin{pmatrix}  0\\-1\\0\\0 \end{pmatrix}; j_8 = \begin{pmatrix}  0\\0\\-1\\0 \end{pmatrix};
%j_9=\begin{pmatrix} 0\\0\\0\\-1 \end{pmatrix}$$

\section{Inference for discrete observations of diffusions on [0,T]}
Our concern  here is parametric inference for these models.
Statistical inference for discretely observed diffusion processes present some specific properties (see Section \ref{Recapdiff} in the Appendix)  that lead us to consider distinct parameters in the drift coefficient (here $\alpha$) and in the diffusion coefficient ($\beta$).
% although  we have proved that  diffusion approximations of Epidemic Dynamics  have the property that  the same parameter is present in both coefficients ($\alpha\equiv \beta$).
The state space of the diffusion is $\R^p$,
and the  parameter set $\Theta$ is a subset of $\R^a \times \R^b$, with $\alpha \in \R^a, \beta \in \R^b$.
For instance, the  $SIR$ diffusion approximation corresponds  to $p=2$ and $\alpha= \beta= (\lambda,\gamma)$.

%\noindent
In order to deal with general epidemics, we consider time-dependent diffusion processes on $\R^p$ with small diffusion coefficient $\epsilon=1/\sqrt{N}$ satisfying the
 stochastic differential equation (SDE):
\begin{equation}\label{SDEGen}
dX(t)= b(\alpha,t,X(t)) dt +\epsilon \sigma(\beta,t,X(t))\; dB(t) \;; \quad X(0)=x,
\end{equation}
where  $(B(t)_{t\geq 0})$ is  a $p$-dimensional  Brownian motion defined on a probability space $(\Omega,(\mathcal{ F}_{t})_{t \geq 0}, {\mathbb P})$,
$b(\alpha,t,\cdot):\R^p \rightarrow \R^p$ and $\sigma(\beta,t,\cdot):  \R^p \rightarrow \R^p\times \R^p $ and $x$ is non-random fixed.\\
%\begin{equation}\label{SDEGen}
%dZ^{\epsilon}(t)= b(\alpha,t,Z^{\epsilon}(t)) dt +\epsilon \sigma(\beta,t,Z^{\epsilon}(t))\; dB(t) \;; \quad Z^{\epsilon}(0)=x,
%\end{equation}
%where  $(B(t)_{t\geq 0})$ is  a $p$-dimensional  Brownian motion defined on a probability space $(\Omega,(\mathcal{ F}_{t})_{t \geq 0}, {\mathbb P})$,
%\sout{ independent of $x_0$},

\noindent
Since epidemic dynamics are usually observed at discrete times,
%$(t_k, k=1,\dots,n)$,
 our aim is to study the estimation of $\theta= (\alpha,\beta)$ based on the observations
 %the observations %$ (Z^{\epsilon}(t_k),k=1,\dots,n)$.
\begin{equation}\label{observations}
 (X(t_k), k=1 \dots n) \mbox{ with } t_k=k\Delta;\ T=n\Delta  \quad (\mbox{sampling interval }\Delta).
\end{equation}

For observations on a fixed time interval, $[0,T]$,  there are distinct asymptotic results according to $\Delta$.
\begin{enumerate}[(2)]
\item[\textbf{ (1)}]  \underline{High frequency sampling
$ \Delta=\Delta_n \rightarrow 0$}: The number of observations $n=T/\Delta_n$ goes to $\infty$  while $T= n\Delta_n$ is fixed. There is a double asymptotic framework: $\epsilon \rightarrow 0$ and $\Delta \rightarrow 0$ (or $n= T/\Delta \rightarrow \infty $) simultaneously.  Let us stress that we shall use both notations for this second asymptotics $n\rightarrow \infty$ or $\Delta\rightarrow 0$.  Although it might be confusing, it is sometimes better to state  results according to the number of observations and sometimes according to the sampling interval $\Delta$.
\item[\textbf{ (2)}] \underline{Low frequency sampling  $ \Delta$ is fixed}:
 It leads to a finite number of observations $n=T/\Delta$. Results are obtained in the asymptotic framework $\epsilon\rightarrow 0$.
\end{enumerate}

At first glance, the low frequency sampling  seems a priori a suitable framework  for epidemic data. However, both high and low frequency observations could be appropriate in practice because the choice of the statistical framework depends more  on the relative magnitudes between $T$, $\Delta$ and the population size $N$ ($= \epsilon ^{-2}$) than on their accurate values.\\
% We present successively results obtained for the high frequency sampling, where the asymptotics is $ \epsilon=1/\sqrt{N} \rightarrow 0$, $\Delta_n=T/n \rightarrow 0$ (Section \ref{HFO}), and for the low frequency sampling, $ \epsilon=1/\sqrt{N} \rightarrow 0$, $\Delta$ fixed (Section \ref{LFO}).
%
%From now on, we set   $t_k=k\Delta$ with $T= n \Delta$ and  assume that $ X(0)= x$ is non-random known. Therefore, the observations consist of the $n$-tuple  $ (X(t_k), k=1,\dots, n)$ and  denote by
%PLUS TARD
% $\P_{\theta}^{ n}$ its distribution on $( (\R^p)^n, \mathcal{ B}((\R^p)^n))$.
% Clearly, the distributions of this $n$-tuple corresponding to different values of
%$(\alpha,\beta)$ are absolutely continuous.

From a statistical point of view,  the sequence $(X(t_k))$ is  a time-dependent Markov chain and therefore the likelihood  depends on its transition probabilities.
%The main difficulty  lies in the fact that the likelihood is intractable
However, the link between the
parameters  present in the SDE and the transition probabilities of $(X(t_k))$ is generally not explicit, which leads to intractable likelihoods. This is a well known problem for discrete observations of diffusion processes.  Alternative approaches based on M-estimators or contrast functions (see e.g.\ \cite{vaa00IV} for independent random variables, \cite{kes12IV} for stochastic processes) have to be investigated (see also  the recap  presented in Section \ref{Recapdiff} in the  Appendix of this part).\\

After the statement of some preliminary results, we present successively the statistical inference for high frequency sampling, where the asymptotics is $ \epsilon=1/\sqrt{N} \rightarrow 0$, $\Delta_n=T/n \rightarrow 0$ (Section \ref{HFO}), and for the low frequency sampling, $ \epsilon=1/\sqrt{N} \rightarrow 0$, $\Delta$ fixed (Section \ref{LFO}).
\subsection{Assumptions, notations and first results}
Let $\theta_0$ be the true value of the parameter and $\Theta $ the parameter set.
Denote by $\mathcal{ M}_p(\R)$ the set of $p \times p$ matrices.
%For inference, we need the following assumptions
We first assume that $b(\alpha,t,z)$ and $\sigma(\beta,t,z)$ are measurable in $(t,z)$, Lipschitz continuous with respect to the second variable and satisfy a linear growth condition:
for all $t\geq 0, z,z' \in \R^p$, there exists a global constant $K$ such that
%\textbf{(S0)}: $\Theta= K_a \times K_b$ is a compact set of $\R^{a+b}$, $\theta_0 \in  \mbox{Int}(\Theta)$.\\
\begin{enumerate}[(S3):]
\item[\textbf{(S1)}:]  $\forall  \theta \in \Theta$, $\parallel b(\alpha, t,z)-b(\alpha,t,z')\parallel + \parallel \sigma(\beta, t,z)-\sigma(\beta,t,z') \parallel \leq K \parallel z-z' \parallel$.
\item[\textbf{(S2)}:] $\forall (\alpha,\beta) \in \Theta, \parallel b(\alpha; t,z)\parallel^2+\parallel \sigma(\beta;t,z)\parallel^2
\leq K(1+\parallel z\parallel^2).$
\item[\textbf{(S3)}:] $\forall (\beta,t, z), \;\Sigma(\beta;t,z)= \sigma(\beta;t,z) \sigma^{\star}(\beta;t,z)$ is non-singular.
\end{enumerate}
%If $b(\alpha,t,z)$ and $\sigma(\beta,t,z)$ are measurable in $(t,z)$, Lipschitz continuous with respect to $z$,
Assumptions (S1)--(S3) are classical assumptions that  ensure that, for all $\theta$, (\ref{SDEGen}) admits a unique strong  solution (see e.g.\ \cite[Chapter 5.2.B]{kar00IV}).\\

Another set of assumptions is required for the inference:
\begin{enumerate}[(S6):]
\item[\textbf{(S4)}:] $\Theta= K_a \times K_b$ is a compact set of $\R^{a+b}$, $\theta_0 \in  \mbox{Int}(\Theta).$
\item[\textbf{(S5)}:] For all $t \geq 0$,  $b(\alpha;t,z) \in C^3(K_a \times  \R^+ \times \R^p,\R^p)$ and $\sigma(\beta;t,z)\in C^2(K_b \times  \R^+ \times \mathcal{ M}_p(\R)) $.
\item[\textbf{(S6)}:]  $ \alpha \neq \alpha' \Rightarrow b(t;\alpha,z(\alpha,t)) \not \equiv b(t;\alpha',z(\alpha',t))$.
\item[\textbf{(S7)}:] $  \beta \neq \beta' \Rightarrow \Sigma(t;\beta,z(\alpha_0,t)) \not \equiv \Sigma(t;\beta',z(\alpha_0,t))$.
\end{enumerate}

\noindent
Assumptions \textbf{(S4)--(S5)} are classical for the inference for diffusion processes.
Usually, it is sufficient in  \textbf{(S5)} to deal with $C^2$ functions. The additional
 differentiability condition comes from  regularity conditions required on $\alpha \rightarrow
 \Phi(\alpha, t,s)$. Indeed, \textbf{(S5)} on $b(\alpha,t,z)$ ensures  that the function
$\Phi(\alpha,t,t_0)$ belongs to $ C^2(K_a\times [0,T]^2, \mathcal{ M}_p)$.
 %(See Appendix \ref{rem:Phi} for the proof)\\
Assumption \textbf{(S6)} is the usual identifiability assumption
for a diffusion continuously observed on
$[0,T]$ and \textbf{(S7)} is an identifiability assumption for parameters in the diffusion coefficient.\\

Since  $(X(t))$ is a diffusion process on $(\Omega,(\mathcal{ F}_{t})_{t \geq 0}, {\mathbb P})$, the space of observations is
 $ (C_T= C([0,T],\R^p),\mathcal{ C}_T)$ where $\mathcal{ C}_T$ is  the Borel $\sigma$-algebra on  $C([0,T],\R^p)$.
 Let $\P_{\theta}= \P_{\alpha,\beta}$  the probability distribution on  $( C_T,\mathcal{ C}_T)$ of   $(X(t)), 0\leq t \leq T)$ satisfying (\ref{SDEGen}).
  Let $\mathcal{ G}_k^n$ denote the $\sigma$-algebra $ \sigma(X(s), s \leq \frac{kT}{n})$.\\
%and  $X(t, . ) $ the coordinate function : for  $x= x(\cdot) \in C([0,T],\R^p), \;  X(t, x)=x(t)$.
%%As for Markov chains models, let us define the associated canonical process.
 %Since  $(Z^{\epsilon})$ is a diffusion process on $(\Omega,(\mathcal{ F}_{t})_{t \geq 0}, {\mathbb P})$, the space of observations is
%it follows that %$ \{\forall \omega, t \rightarrow Z^{\epsilon}(t, \omega) \}$ is  a continuous function  from $[0,T] $ to  $\R^p$.The observation space is
%Let $\P_{\theta}^{\epsilon}= \P_{\alpha,\beta}^{\epsilon}$
%Let $\P_{\theta}}= \P_{\alpha,\beta}$  the probability distribution on  $( C_T,\mathcal{ C}_T)$  image of  $\P$ by  $(Z^{\epsilon}(t)), 0\leq t \leq T)$ satisfying (\ref{SDEGen}):
%%for $A_i
%for $A_i $ Borelian sets of $\R^p$ and $A = \{x(\cdot) \in C_T, x(t_i) \in A_i, i=1 \dots, k\}$, \\
%$\P( (Z^{\epsilon}(t)), 0\leq t \leq T) \in A)= \P_{\theta}^{\epsilon}( X(t_i) \in A_i, i=1 \dots , k)$.
%%%%%%%%%%%%%%%%%%%%%%%%%%%
%%%%%%%%%%%%%%%%%%%%%%%
%%%%%%%%%%%%%%%%%%%%%%%%%%
%Under some additional regularity assumptions,
%Let us define some notations used in the sequel.
%For a matrix $A$, $A^*$ is the transposition of $A$,  det$A$, is the determinant of $A$, and  Tr$A$ its trace.

For $g(\theta,t,z): \Theta \times [0, T] \times \R^p \rightarrow \R^p$,
$\nabla_z g (\cdot)$ is  the $\mathcal{ M}_{p}$ matrix
\begin{equation}\label{gradf}
 \nabla_{z} g(\cdot) = (\frac{\partial g_i}{\partial z_j}(\theta,t, z))_{ 1\leq i,j\leq p} \quad \mbox{and } \quad  \nabla_{\theta} g(\cdot) = (\frac{\partial g_i}{\partial \theta_j}(\theta,t, z))
%= (\frac{\partial g _i}{\partial \theta_k}(\theta,t, z), p
 \end{equation}
If $z=z(\theta,t)$, then
\begin{equation} \label{gradcompos}
\nabla_{\theta}\left(g(\theta,t,z(\theta,t))\right)= \nabla_{\theta}g(\cdot) + \nabla_z g(\cdot)\nabla_{\theta}z(\cdot).
\end{equation}
 Quantities are indexed by $\theta$ (resp.\ $\alpha$ or $\beta$)   when they  depend on both $\alpha,\beta $ (resp.$\alpha$ or $\beta$).
 Introducing  the dependence with respect to $t,\theta$ in \eqref{ODEPARTIV}, \eqref{def:phi}, \eqref{eq:gtstoch} yields,
\begin{equation}\label{TSparamxg}
	\begin{array}{rl}
\frac{ \partial z}{\partial t}(\alpha,t) =&  b(\alpha,t, z(\alpha,t) );\quad  z(\alpha,0)= x_0 ,\\
g(\alpha,\beta, t)= &\int_0^t\Phi(\alpha,t,u)\sigma(\beta,u,z(\alpha,u)) dB(u),\quad \mbox {with } \Phi(\alpha,\cdot) \mbox{  such that}\\
\frac{\partial \Phi}{\partial t}  (\alpha,t,u)=&\nabla_z b(\alpha,t, z(\alpha,t))\Phi(\alpha,t,u) \; , \; \Phi(\alpha,u,u)=I_p.
\end{array}
\end{equation}
The expansion (\ref{TS}) holds for time-dependent diffusion processes.
% (see e.g.\ \cite{aze82}).
\begin{proposition}\label{TStheta}
Assume (S1)--(S5). Then, under $\P_{\theta}$,  $(X(t), 0\leq t\leq T)$ satisfies that,  uniformly with respect to $\theta$,
\begin{equation}\label{TSparam}
	\begin{array}{rl}
X(t)=& z(\alpha,t)+\epsilon g(\theta,t)+\epsilon^2 R^{\epsilon}( \theta, t),\quad \mbox{with}\\
 lim_{\epsilon\rightarrow 0,r \rightarrow \infty}&  \P_{\theta}( \sup_{t\leq T}\parallel  R^{\epsilon}(\theta,t)\parallel >r)=0\\
 \sup_{t\leq T}\parallel  R^{\epsilon}(\theta,t)\parallel &\;\mbox{has uniformly bounded moments}.
 % \rightarrow 0 \mbox{ in  }\P_{\theta}^{\epsilon}   \mbox{-probability  as  } \epsilon \rightarrow 0.
%\P_{\theta}^{\epsilon}&( \sup_{t\leq T}\parallel  R^{\epsilon}(\alpha,\beta,t)\parallel >r) \rightarrow 0 \mbox{ as } \epsilon\rightarrow 0,r \rightarrow \infty
\end{array}
\end{equation}
%Moreover,  $\sup_{t\leq T}\parallel  R^{\epsilon}(\theta,t)\parallel $ has uniformly bounded moments.
% (\cite{aze82}).
%\sup_{t\leq T}& \parallel \epsilon R^{\epsilon}(\alpha,\beta,t)\parallel \rightarrow 0 \mbox{ in  }\P_{\theta}^{\epsilon}   \mbox{-probability  as  } \epsilon \rightarrow 0.
\end{proposition}
\begin{proposition}\label{gtheta}
Under (S1)--(S5), the process $(g(\theta,t))$ satisfies  using (\ref{TSparamxg}) \\
$$\forall s < t,\quad   g(\theta,t)= \Phi(\alpha,t,s) g(\theta,s)+ \int_s ^t \Phi(\alpha, t,u) \sigma(\beta,u,z(\alpha,u)) dB(u),$$
where the two terms of the r.h.s.\ above are independent random variables.
\end{proposition}
\begin{proposition}\label{Rteps}
Assume (S1)--(S2). If moreover $b(\alpha,\cdot)$
and $\sigma(\beta,\cdot)$ have
uniformly bounded derivatives, there exist constants
only depending on $T$ and $\theta $ such that
\begin{enumerate}
\item[(i)] $\forall t \in[0,T]$,
$\E_{\theta}( \norm{R^{\epsilon}(\theta,t)}^{2} < C_1 $,
\item[(ii)] $ \forall t \in[0,T]$, as $h \rightarrow 0$,
$\E_{\theta}({ \norm{R^{\epsilon}(\theta, t+h)-R^{\epsilon}(\theta,t)}^{2}})
 < C_2 h$.
 \end{enumerate}
\end{proposition}
We refer  to \cite{aze82IV}, \cite{fre84IV}, and \cite{glo09IV} for the proofs of these propositions for $\theta$ fixed. Assumption (S4) allows us to extend these results  to $\theta \in \Theta$.

 \subsection{Preliminary results}
 Let us define using  (\ref{TSparamxg}) the random variables,
\begin{equation}\label{BkTS}
 B_k(\alpha, X)=X(t_k)-z(\alpha, t_{k})-\Phi(\alpha,t_k,t_{k-1})
\left[X(t_{k-1})-z(\alpha,t_{k-1})\right].
\end{equation}
%Let $\mathcal{ G}_k^n$ denote the $\sigma$-algebra $ \sigma(B(s), s \leq \frac{kT}{n})$.
 Then the following holds.

\begin{lemma}\label{BkX}
Assume (S1)--(S4). Then, under $\P_{\theta}$, %$\P_{\theta}^{\epsilon}$,
as $\epsilon \rightarrow 0$,
\begin{eqnarray*}
B_k(\alpha, X)&=& \epsilon \sqrt{\Delta} \;T_k(\theta)+ \epsilon ^2  D^{\epsilon} _k( \theta,\Delta) \quad \mbox{with  }
 \sup_{k}  \E_{\theta}||D^{\epsilon} _k( \theta,\Delta)||^2\leq C\Delta,\; \mbox{ and} \\
 T_k(\theta)&=& \frac{1}{\sqrt{\Delta}}\int_{t_{k-1} }^{t_k} \Phi(\alpha, t_k,u)\sigma(\beta,u ,z(\alpha,u)) dB(u),
 \end{eqnarray*}
 where $C$ a constant independent of $\theta, \epsilon,\Delta$.
  \end{lemma}
 Therefore, the random variables
   $T_k(\theta)$ are  $p$-dimensional  $\mathcal{ G}_k^n $-measurable independent Gaussian random variables with covariance matrix
 \begin{equation}\label{SkTS}
S_k(\alpha,\beta) = S_k(\theta)= \frac{1}{\Delta}\int_{t_{k-1}}^{t_k} \Phi(\alpha, t_k,s)\Sigma (\beta,s, z(\alpha,s))\Phi^{\star}(\alpha, t_k,s)ds.
\end{equation}

\begin{proof}  Using  Propositions \ref{TStheta} and \ref{gtheta} yields that
\begin{eqnarray*}
D^{\epsilon} _k( \theta,\Delta) &=& R^{\epsilon}(\theta, t_k)- \Phi(\alpha,t_k,t_{k-1}) R^{\epsilon}(\theta,t_{k-1})\\
&=&  R^{\epsilon}(\theta, t_k)- R^{\epsilon}(\theta,t_{k-1} ) -(\Phi(\alpha,t_k,t_{k-1})-I_p) R^{\epsilon}(\theta,t_{k-1})\\
&=& R^{\epsilon}(\theta, t_k)- R^{\epsilon}(\theta,t_{k-1} ) - \Delta \nabla_z b(\alpha,t_{k-1},z(\alpha,t_{k-1}) )R^{\epsilon}(\theta,t_{k-1})\\
& & + \Delta ^2 O_P(1).
\end{eqnarray*}
 An application of Proposition \ref{Rteps} together with (S4) yields the result.
 \end{proof}

Define  the two random matrices
\begin{equation}\label{Sigkbeta}
\Sigma_{k}(\beta)= \Sigma(\beta,t_k,X(t_k)), \quad \sigma_k(\beta)= \sigma(\beta,t_k,X(t_k)).
\end{equation}
 Then, for small $\Delta $, we have using \eqref{Sigkbeta}
\begin{lemma}\label{SigmaS}
Assume (S1)--(S5). Then, under $\P_{\theta}$, %$\P_{\theta}^{\epsilon}$,
as $\epsilon, \Delta \rightarrow 0$,
$$ ||S_k(\theta)-\Sigma_{k-1}(\beta) || \leq K \epsilon \sup_{\theta,t\leq T}||g(\theta,t)||+ \Delta \sup_{\theta, z} ||\nabla_z\Sigma(\beta,s,z)||\leq \epsilon  C_1 O_P(1)+C_2 \Delta.$$
\end{lemma}
The proof is straightforward  using (S1), (S5) and Proposition (\ref{TSparam}).
%\E({ \norm{R^{\epsilon}(\theta, t+h)-R^{\epsilon}(\theta,t)}^{2}})$B_k(\alpha, X)= \epsilon \sqrt{\Delta}$
%These two Lemmas \ref{BkX} and  \ref{SigmaS} yield that   the random variables $ \frac{1}{\epsilon \sqrt{\Delta} }B_k(\alpha_0,X)$ are approximately  conditionally independent centered Gaussian random variables
%on  $\R ^p$  with covariance approximated by $\Sigma_{k-1}(\beta_0)$ under $\P_{\theta_0}$. Analogously to \cite{kes97} or \cite{glo09},
%we  introduce  a contrast function, using definitions \eqref{BkTS}, \eqref{Sigkbeta}, %approximating the observations law:
%\begin{equation}
%	\label{Ucheck}
%\check{U}_{\epsilon,\Delta} (\alpha,\beta)=\sum_{k=1}^{n}\log \det \Sigma_{k-1}(\beta)+
%\frac{1}{\epsilon^2 \Delta}\sum_{k=1}^{n} B_k(\alpha,X)^{*}\; \Sigma_{k-1}^{-1}
%(\beta)\; B_k(\alpha,X),
%\end{equation}
%where $\det A$ denotes  the determinant of a matrix $A$.
%
%The minimum contrast estimators are defined as any solution of
%\begin{equation}\label{def:estimateurs_chech}
%(\check{\alpha}_{\epsilon,\Delta},\check{\beta}_{\epsilon,\Delta})
%=\underset{(\alpha,\beta)\in \Theta }{argmin} \ \check{U}_{\epsilon,\Delta}(\alpha,\beta).
%\end{equation}
%\section{Preliminary lemmas}

Let us now state some preliminary results on the random variables $B_k(\alpha, X)$  defined in (\ref{BkTS}) useful for the inference.\\
%\noindent
%\subsection{ Preliminary lemmas}

Under $\P_{\theta_0} $,  Proposition \ref{TStheta}  yields that $B_k(\alpha, X)$ converges to $B_k(\alpha, z(\alpha_0,\cdot))$ and  that $B_k(\alpha_0, X)$  converges to $B_k(\alpha_0, z(\alpha_0,\cdot))= 0$ a.s.
 Let us define on $[0,T]$
\begin{equation} \label{def:Gamma}
      	\Gamma(\alpha_0,\alpha, t)=b(\alpha_0,t, z(\alpha_0,t))-b(\alpha,t, z(\alpha,t))
-\nabla_z b(\alpha, t, z(\alpha,t))(z(\alpha_0,t)-z(\alpha,t)).
      \end{equation}
%Recall that  $\Delta= \Delta_n$  with $ n\Delta= T$ fixed.
The sequence $B_k(\alpha, z(\alpha_0,\cdot))$ satisfies:
\begin{lemma}\label{BGamma}
Assume (S1), (S2), (S4). Then, as $\Delta \rightarrow 0$, there exists  a constant $C$  such that
\[\frac{1}{\Delta}B_k(\alpha,z(\alpha_0,\cdot))= \Gamma( \alpha_0,\alpha, t_{k-1})+ \Delta\norm{\alpha-\alpha_0} r_k(\alpha_0,\alpha)\]
with $ \displaystyle{\sup_{k,\alpha \in K_a }  \norm{r_k(\alpha_0,\alpha)}} \leq C$.
%$ uniformly bounded on $K_a$.
%$$\sup_{1 \leq k \leq n;\;  \alpha \in K_a}
%\norm{\frac{B_k(\alpha, z(\alpha_0,\cdot))}{\Delta}- \Gamma( \alpha_0,\alpha, t_{k-1})} \rightarrow 0 .$$
\end{lemma}
\begin{proof}  Using (\ref{BkTS}), (\ref{def:Gamma}) and  that $\Phi(\alpha,t_k,t_{k-1})= I_p + \Delta \nabla_z b(\alpha,t_{k-1} ,z(\alpha,t_{k-1}) ) +\Delta^2 O(1)$, yields

\begin{eqnarray*}
B_k(\alpha, z(\alpha_0,\cdot))&=&
\int_{t_{k-1}} ^{t_k} (b(\alpha_0,s,z(\alpha_0,s))-b(\alpha,s,z(\alpha,s))) ds\\
& &\phantom{B_k(\alpha, z(\alpha_0,\cdot))=}+ \left(I_p - \Phi(\alpha,t_k,t_{k-1} )\right) (z(\alpha_0,t_{k-1})-z(\alpha,t_{k-1}))\\
&=& \Delta  \Gamma(\alpha_0,\alpha,t_{k-1})+ \Delta^2 \norm{\alpha-\alpha_0}r_k(\alpha_0,\alpha).
\end{eqnarray*}
Assumptions (S1), (S2) and (S4) ensure  that the remainder term has order $\Delta ^2$ uniformly in $k, \alpha $.
%\left( z(\alpha_0, t_k)-z(\alpha_0,t_{k-1})\right)-\left(z(\alpha, t_k)-z(\alpha,t_{k-1}) \right)
\end{proof}

Consider now the random variables $B_k(\alpha,X)$
\begin{lemma}\label{BXa}
Assume (S1)--(S5).  Then, under $\P_{\theta_0}$, as $ \epsilon,\Delta \rightarrow 0$, the following holds for all $ k \leq n$,
\begin{eqnarray*}
\frac{1}{\Delta}( B_k(\alpha,X) - B_k(\alpha_0,X))& =&\frac{1}{\Delta} B_k(\alpha, z(\alpha_0,\cdot))+ \epsilon ||\alpha-\alpha_0|| \eta_k(\alpha_0,\alpha,\epsilon, \Delta)\\
&=& \Gamma(\alpha_0,\alpha,t_{k-1}) + ||\alpha-\alpha_0|| ( \Delta O(1) + \epsilon O_{P}(1)),
\end{eqnarray*}
where $\eta_k =\eta_k(\alpha_0,\alpha,\epsilon, \Delta)$ is $\mathcal{ G}_{k-1}^n$-measurable and uniformly bounded in  probability .
\end{lemma}
\begin{proof}  Using (\ref{TSparam}) and (\ref{BkTS}), we get that
\begin{align*}
&B_k(\alpha,X) -  B_k(\alpha_0,X) =  \\
&B_k(\alpha, z(\alpha_0,\cdot))+\epsilon  (\Phi(\alpha_0,t_k,t_{k_1})- \Phi(\alpha, t_k,t_{k_1})) (g(\theta_0,t_{k-1})+\epsilon  R^{\epsilon}(\theta_0, t_{k-1})).
\end{align*}
By (S1)--(S5),
\begin{align*}
&\norm{\Phi(\alpha_0,t_k,t_{k_1})- \Phi(\alpha, t_k,t_{k_1})}\\
&\quad\quad\leq 2 \Delta \norm{\nabla_z b(\alpha_0,,t_{k-1},z( \alpha_0,t_{k-1}))-
\nabla_z b(\alpha,,t_{k-1}, z (\alpha,t_{k-1}))}.
\end{align*}
Now, this term is  bounded by
$K \Delta ||\alpha-\alpha_0||$ since
 $(t,\alpha) \rightarrow \nabla_z b(\alpha,z(\alpha,t))$ is uniformly continuous on $[0,T] \times K_a$. Using now Proposition \ref{TStheta} yields that  $(g(\theta_0,t_{k-1})
+\epsilon  R^{\epsilon}(\theta_0, t_{k-1}))$ is bounded in $\P_{\theta_0}$-probability and  $\mathcal{ G}_{k-1}^n$-measurable.
\end{proof}

The next lemma concerns the properties of $B_k(\alpha_0,X)$.
\begin{lemma}\label{BXa0}
Assume (S1)--(S5). Then, using \eqref{Sigkbeta}, as $ \epsilon,\Delta \rightarrow 0$, under $\P_{\theta_0}$,
\begin{enumerate}
\item[\textbf{(i)}] $B_k(\alpha_0,X)= \epsilon \sigma_{k-1}(\beta_0) (B(t_k)-B(t_{k-1}))+ E_{k} (\theta_0,\epsilon, \Delta),$
where $ E_k=  E_k(\theta_0,\epsilon ,\Delta)$ satisfies that, for $m \geq 2$, $\E_{\theta_0}( ||E_k||^m | \mathcal{ G}_{k-1}^n) \leq C \epsilon^m \Delta^m $.
\item[\textbf{ (ii)}] If $(V_k )$ is a sequence of $\mathcal{ G}_{k-1}^n$-measurable random variables  in $\R^p$ uniformly bounded in  probability, then
$\displaystyle{\sum_{k=1}^n V_k^{*} B_k(\alpha_0,X) \rightarrow 0}  $ in probability.
\end{enumerate}
\end{lemma}
%\pagebreak

\begin{proof}  Let us first prove (i) and study the term $E_k$. We have $E_k=E_k^1+E_k^2$ with
\begin{align*}
E_k^1=&\int_{t_{k-1}}^{t_k}\left(b(\alpha_0,t,X(t))-b(\alpha_0,t,z(\alpha_0,t)) \right)dt\\
&\quad +\left(I_p- \Phi(\alpha_0,t_k,t_{k-1}) \right) \left( X(t_{k-1})- z(\alpha_0,t_{k-1})\right)\mbox{ and}\\
E_k^2=&\epsilon\int_{t_{k-1}}^{t_k}\left(\sigma(\beta_0,s, X(s))-\sigma(\beta_0,s, X(t_{k-1}))\right)dB(s).
\end{align*}
 Set in  (\ref{TSparam}), $R_1^{\epsilon}(\theta, t)= g(\theta,t)+\epsilon R^{\epsilon}(\theta,t)$.
Using that $(t,x) \rightarrow b(\alpha,t, x)$ is uniformly Lipschitz, we obtain,
\begin{align*}
\norm{E_k^1}\leq& \Delta C \sup_{t \in [t_{k-1};t_k]}\norm{X(t)-z(\alpha_0,t)}\\
&\quad +\Delta\epsilon\norm{(\int_{0}^{1} \nabla_z b(\alpha_0,z(\alpha_0,t)))\Phi(\alpha_0,t, t_{k-1})dt)
R_1^{\epsilon}(\theta_0, t_{k-1})}\\
\leq& C' \epsilon\Delta \sup_{t \in [t_{k-1};t_k]}\norm{R_1^{\epsilon}(\theta_0,t)}
\end{align*}
The proof for $E_k^2$ follows the sketch given in \cite[Lemma 1]{glo09IV}.
We first prove this result based on the stronger condition $\Sigma$ and $b$ bounded.
Then, similarly to \cite[Proposition 1]{glo09IV}, this  assumption  can be relaxed.
We use sequentially the Burkh\"{o}lder--Davis--Gundy (see e.g.\ \cite{kar00IV}) and Jensen inequalities to obtain
\begin{align}
&\E_{\theta_0}(\norm{E_k^2}^m | \mathcal{ G}_{k-1}^n)\nonumber\\
& \leq C \epsilon^m \E_{\theta_0}\left( (\int_{t_{k-1}}^{t_k} \norm{\sigma(\beta_0,s,X(s))-
\sigma(\beta_0,t_{k-1} ,X(t_{k-1})}^2 ds)^{m/2}|\mathcal{ G}_{k-1}^n\right)
\end{align}
  \begin{equation}\
  \leq C \epsilon^m \Delta^{m/2-1} \int_{t_{k-1}}^{t_k} \E_{\theta_0}\left(\norm{\sigma(\beta_0,s,X(s))-\sigma(\beta_0,t_{k-1}, X(t_{k-1}))}^m |\mathcal{ G}_{k-1}^n \right) |ds.
   \end{equation}
% Then Ito's formula leads to
% $\Ec{\norm{\sigma(\beta_0,X_s)-\sigma(\beta_0,X_{t_{k-1}})}^m}{k-1}=\Ec{\norm{\som{i,j=1}{p}\integ{t_{k-1}}{s}\left(\deriv{\sigma}{x_i}(\beta_0,X_u)b(\alpha_0,X_u)+\epsilon^2\deriv{^2\sigma}{x_i\partial x_j}(\beta_0,X_u)\Sigma(\beta_0,X_u)\right)du+\epsilon \deriv{\sigma}{x_i}(\beta_0,X_u)\sigma(\beta_0,X_u)dB_u}^m}{k-1}ds $
Then,  using that  $(t,x)  \rightarrow \sigma(\beta,t, x)$ is Lipschitz, we obtain:
\begin{align*}
&\E_{\theta_0}^{\epsilon} (\norm{E_k^2}^m |\mathcal{ G}_{k-1}^n)  \leq C'\epsilon^m\Delta^{m/2-1}\int_{t_{k-1}}^{t_k}\E_{\theta_0}^{\epsilon}(\norm{X(s)- X(t_{k-1})}^m)ds\\
&\leq C' \epsilon^m\Delta^{m/2-1}\int_{t_{k-1}}^{t_k}\E_{\theta_0}^{\epsilon} \left[\norm{\int_{t_{k-1}}^{s} (b(\alpha_0,u,X(u)) du+\epsilon \sigma(\beta_0,u, X(u))dB(u))} \right]^m ds.
\end{align*}
Since $b$ is  bounded on $\mathcal{U}$, $\norm{ \int_{t_{k-1}}^{s} b(\alpha_0,u, X(u)) du} \leq K  |s-t_{k-1}| $ and It\^{o}'s isometry yields
\begin{align*}
\E_{\theta_0} \left(\norm{\int_{t_{k-1}}^{s} \sigma(\beta_0,u, X(u))  dB(u)} ^m\right) &\leq \E_{\theta_0} \left(\norm{\int_{t_{k-1}}^{s}\Sigma(\beta_0,u, X(u))du }\right)^{m/2}\\
 &\leq K |s-t_{k-1}|^{m/2}.
\end{align*}
Thus, $\E_{\theta_0} (\norm{E_k^2(\theta_0)}^m |\mathcal{ G}_{k-1}^n)
\leq C^{''}\epsilon^m\Delta^{m/2-1}\int_{t_{k-1}}^{t_k} |s-t_{k-1}|^{m/2}ds \leq C^{'''}\epsilon^m\Delta^{m}
$.\\

The proof of  (ii) relies on the  Lemma \ref{VGCJJ} for triangular arrays stated in Section \ref{LimitTheo}.
%Set $$\xi_k^n= \zeta_k^n - E(\zeta_k^n | \mathcal{ G}{k-1}^n),\quad  B_n = \sum_{k=1}^{n}\xi_k^n
Set $\zeta_k^n = V_k^{*} B_k(\alpha_0,X) $. Using  (i), we have %Lemma \ref{BXa},  $V_k$ is $\mathcal{ G}_{k-1}^n$-measurable and by Lemma \ref{BXa0},
\[\E_{\theta_0}(\zeta_k^n |\mathcal{ G}_{k-1}^n) =  V_k^{*}  \E_{\theta_0}^{\epsilon}(E_k (\theta_0,\epsilon,\Delta)|\mathcal{ G}_{k-1}^n)\] with
 $\E_{\theta_0}^{\epsilon}(\norm{E_k (\theta_0,\epsilon,\Delta)} |\mathcal{ G}_{k-1}^n) \leq C \epsilon \Delta$.\\
 %Let  us  apply Lemma \ref{VGCJJ} for triangular arrays stated in Section \ref{sec:appenCh5}.
Since  $\sup_{k\leq n}\norm{V_k} $ is bounded in probability,
$\sum_{k=1} ^n \E_{\theta_0}(\zeta_k^n |\mathcal{ G}_{k-1}^n) \leq  C\epsilon T
\rightarrow 0$.\\
Therefore   condition (i) of Lemma \ref{VGCJJ} is  satisfied with $U=0$.\\
Now, $\displaystyle{\E_{\theta_0}[(\zeta_k^n)^2 |\mathcal{ G}_{k-1}^n)]= V_k^{*} \E_{\theta_0}(B_k (\alpha_0,X)B_k^{*}(\alpha_0,X)|\mathcal{ G}_{k-1}^n) V_k}$.\\
Using (i) of Lemma \ref{BXa0}   yields that
\begin{align*}
\E_{\theta_0}(B_k (\alpha_0,X)B_k^{*}(\alpha_0,X) |\mathcal{ G}_{k-1}^n)&= \epsilon ^2 \Delta \Sigma_{k-1}(\beta_0)+  \E_{\theta_0}(E_k E_k ^{*}|\mathcal{ G}_{k-1}^n)\\
&\leq  K_1 \epsilon^2 \Delta + C_2 \epsilon^2 \Delta ^2.
\end{align*}
Hence, $\sum_{k=1} ^n \E_{\theta_0}((\zeta_k^n)^2 |\mathcal{ G}_{k-1}^n) \rightarrow 0$.
Therefore, applying  Lemma \ref{VGCJJ} achieves the proof.
%oining theses two results yields (ii) % that  $ \sum_{k=1} ^n   V_k^{*} B_k(\alpha_0,X)  \rightarrow 0$  in $\P_{\theta_0}^{\epsilon}$-probability.
\end{proof}
A last lemma concerns the terms  $(\nabla_{\alpha_i} B_k)$ for $ i=1, \dots, a$.
\begin{lemma}\label{dNkdb}
Assume (S1)--(S6). Then, under $\P_{\theta_0}$, for all $i,j \leq a$, for all $\alpha \in  K_a$, as $\epsilon,\Delta \rightarrow 0$,
\begin{enumerate}
\item[(i)] $\frac{1}{\Delta}\nabla_{\alpha_i}B_k(\alpha,X)= -\nabla_{\alpha_i } b(\alpha,t_{k-1}, z(\alpha, t_{k-1}))
+M^i_k(\alpha)[ (z(\alpha_0,t_{k-1})- z(\alpha, t_{k-1})) + \epsilon Z^{\epsilon}_{k-1} (\theta_0)] +\Delta O_P(1)$,
 where  $M^i_k(\alpha)$ are uniformly bounded matrices and   $Z^{\epsilon}_{k-1} (\theta_0)$ are $\mathcal{ G}_{k-1}^n$-measurable r.v.s uniformly bounded in probability.\\
\item[(ii)]  For all  $k\leq n$,    $\frac{1}{\Delta}\norm{\nabla^2_{\alpha_i ,\alpha_j }B_k(\alpha,X)}$ is bounded  uniformly in probability.
\end{enumerate}
 \end{lemma}
%(i) $\frac{1}{\Delta}\nabla_{\alpha_i}B_k(\alpha,X)= -\nabla_{\alpha_i } b(\alpha,t_{k-1}, z(\alpha, t_{k-1}))+\epsilon \zeta_{k}^i (\theta_0, \alpha,\epsilon,\Delta) +
%\Delta r_{k}^i(\alpha_0,\alpha, \Delta)$ where \\
% the r.v. $ \zeta_{k}^i (\theta_0,\alpha,\epsilon,\Delta) $ are $\mathcal{ G}_{k-1}^n$-measurable and  $\sup_{k\leq n} \norm{\zeta_{k}^i}$ is bounded in probability,\\
%% random vectors of $\R^{p}$  and
%the $r_{k}^i (\alpha_0,\alpha, \Delta)$ are deterministic  vectors
%such that  $\sup_{k\leq n} \norm{r_k^i (\theta_0,\alpha)}$ is  uniformly bounded. \\
%%$\sup_{k\leq n} \norm{\zeta_{k}^i}$ is bounded in $\P_{\theta_0}^{\epsilon}$-probability
%Moreover, $\norm{\zeta_{k}^i (\theta_0,\alpha,\epsilon,\Delta) -
%\zeta_{k}^i (\theta_0,\alpha_0,\epsilon,\Delta) }\leq \norm{\alpha-\alpha_0} Z_{k,i}(\theta_0)$, with $Z_{k,i}(\theta_0,\epsilon)$ bounded in
%probability  and $ \norm{r_k^i (\theta_0,\alpha)- r_k^i (\theta_0,\alpha_0)}  \leq C \norm{\alpha-\alpha_0}$ uniformly in $k$. \\
%%$ \alpha \rightarrow r_k^i (\theta_0,\alpha)$  is continuous and  $\sup_{k\leq n} \norm{r_k^i (\theta_0,\alpha)}$ is  uniformly bounded. \\
%(ii)  For all $\alpha \in K_a$, $1\leq i,j\leq a$ and $1\leq k\leq n$,    $\frac{1}{\Delta}\nabla^2_{\alpha_i ,\alpha_j }B_k(\alpha,X)$ is bounded  uniformly in probability .
%\end{lemma}
%
\begin{proof} We have, using (\ref{TSparam}) and (\ref{BkTS}),
\begin{align*}
B_k(\alpha,X)=&(X(t_k)-X(t_{k-1})) -(z(\alpha,t_k)-z(\alpha, t_{k-1} ))\\
&\quad +(I_p-\Phi(\alpha, t_k, t_{k-1}))(X(t_{k-1})-z(\alpha,t_{k-1})).
\end{align*}
Therefore,
\[\nabla_{\alpha_i}B_k(\alpha,X)=E_{k,i}+ \epsilon \Delta  M^i_k(\alpha)Z_{k-1}(\theta_0)\]
with $M^i_k(\alpha)= -\frac{1}{\Delta}\nabla_{\alpha_i} \Phi(\alpha, t_k, t_{k-1})$,
$Z_{k-1}(\theta_0)= g(\theta_0,t_{k-1})+ \epsilon  R^{\epsilon}(\theta_0, t_{k-1})$
and
\begin{align*}
E_{k,i}=& - \nabla_{\alpha_i}(z(\alpha,t_k)-z(\alpha, t_{k-1})) +(\Phi(\alpha, t_k, t_{k-1})-I_p) \nabla_{\alpha_i}z(\alpha,t_{k-1})\\
&-\nabla_{\alpha_i} \Phi(\alpha, t_k, t_{k-1})(z(\alpha_0,t_{k-1})-z(\alpha,t_{k-1})).
\end{align*}
Proposition \ref{TStheta} yields the result for $Z_k(\theta_0)$.
%\pagebreak

Now, $ \Phi(\alpha,t_k,t_{k-1})= \exp \{\int_{t_{k-1} }^{t_k} \nabla_z b(\alpha, s,z(\alpha,s)) ds\}$ so that
\[M^i_k(\alpha)=-\nabla_{\alpha_i} \nabla_z b(\alpha, t_{k-1},z(\alpha ,t_{k-1}))+\Delta \mbox{O}(1).\]
To study $E_{k,i}$, we use that $\displaystyle{\Phi(\alpha, t_k, t_{k-1})-I_p= \Delta  \nabla_z b(\alpha ,t_{k-1},z(\alpha, t_{k-1})) +  \Delta^2 O(1)}$ and
  $\displaystyle{\nabla_{\alpha_i} \left(b(\alpha,t, z(\alpha,t))\right)
= \nabla_{\alpha_i}b(\alpha,t,z(\alpha,t)) + \nabla_z b(\alpha,t,z(\alpha,t)) \nabla_{\alpha_i} z(\alpha,t)}$.\\
Therefore
$E_{k,i}=-  \nabla_{\alpha_i}b(\alpha, t_{k-1},z(\alpha,t_{k-1}))+ M^i_k(\alpha)(z(\alpha_0,t_{k-1})-z(\alpha,t_{k-1})) + \Delta \mbox{O}(1).$\\
This achieves the proof of (i).

Let us prove (ii). We have
$\nabla^2_{\alpha_i\alpha_j}B_k(\alpha,X)= f_k^{ij}(\alpha_0,\alpha,\Delta)+\xi^{ij}_k(\theta_0,\alpha,\epsilon,\Delta)$ with\\
$ \xi^{ij}_k=  (\nabla_{\alpha_i\alpha_j}^2 \Phi(\alpha,t_k, t_{k-1}))[X(t_{k-1})-z(\alpha_0,t_{k-1})]$ and
\begin{align*}
f_k^{ij}(\alpha_0,\alpha,\Delta)=&-\left(\nabla^2_{\alpha_i\alpha_j}z(\alpha,t_k )-\Phi(\alpha,t_k,t_{k-1}) \nabla^2_{\alpha_i\alpha_j}z(\alpha,t_{k-1}) \right)\\
&+ \nabla_{\alpha_i}\Phi(\alpha,t_k,t_{k-1}) \nabla_{\alpha_j}z(\alpha,t_{k-1}) \\
&+ \nabla_{\alpha_j}\Phi(\alpha, t_k,t_{k-1}) \nabla_{\alpha_i}z(\alpha,t_{k-1})\\
&- \nabla_{\alpha_i \alpha_j}^2 \Phi(\alpha,t_k, t_{k-1})(z(\alpha,t_{k-1})-z(\alpha_0,t_{k-1})).
\end{align*}
The result is obtained using  Proposition \ref{TStheta} and the property that $\frac{1}{\Delta} \norm{\nabla_{\alpha_i} \Phi (\cdot)}$ and
$\frac{1}{\Delta} \norm{\nabla^2_{\alpha_i,\alpha_j} \Phi (\cdot)}$ are uniformly bounded.
%that $\frac{1}{\Delta} \norm{f_k^{ij}(\alpha_0,\alpha,\Delta)}$ is uniformly bounded as $\Delta \rightarrow 0$
%and that  $ \xi^{ij}_k$ is uniformly  bounded in probability under $\P_{\theta_0}^{\epsilon}$.
\end{proof}

 \section{Inference based on high frequency observations on $[0,T]$} \label{HFO}
We assume that both $\epsilon$ and $\Delta= \Delta_n$ go to $0$.
%The sample path $(X(t))$ is observed at times $t_k=k\Delta$. Hence, t
The number of observations $n$ goes to infinity. We study the estimation of $\theta=(\alpha,\beta)$ based on $(X(t_k), k=1,\dots,n)$. \\
%ENLEVER OU DEPLACER\\
%%The inference in the case $\sigma(\beta, x) \equiv 1$ was first investigated for one dimensional diffusions by \cite{gen90}, using expansion (\ref{TS}).
%% Previous results obtained in \cite{glo09} require conditions linking $\epsilon $ and $\Delta $ that do not fit epidemic data, where generally the parameter $\epsilon$ is small, and orders of magnitude for $N$ and $n$ satisfy  $N>>n$  so that  $\Delta$ is comparatively large with respect to $\epsilon$.
%% We proposed  in \cite{guy14} another method based on Theorem \ref{Taylorstoch} of $X(\cdot)$ which extends results obtained in \cite{gen90}.
%FIN

Lemmas \ref{BkX} and  \ref{SigmaS} yield that   the random variables $ \frac{1}{\epsilon \sqrt{\Delta} }B_k(\alpha_0,X)$ are approximately  conditionally independent  centered Gaussian random variables in  $\R ^p$  with covariance approximated by $\Sigma_{k-1}(\beta_0)$.
%under $\P_{\theta_0}$.
 Analogously to \cite{kes97IV} or \cite{glo09IV},
we  introduce  a contrast function, using definitions \eqref{BkTS}, \eqref{Sigkbeta}, %approximating the observations law:
\begin{equation}
	\label{Ucheck}
\check{U}_{\epsilon,\Delta} (\alpha,\beta)=\sum_{k=1}^{n}\log \det \Sigma_{k-1}(\beta)+
\frac{1}{\epsilon^2 \Delta}\sum_{k=1}^{n} B_k(\alpha,X)^{*}\; \Sigma_{k-1}^{-1}
(\beta)\; B_k(\alpha,X).
\end{equation}
%where $\det A$ denotes  the determinant of a matrix $A$.
The minimum contrast estimators are defined as any solution of
\begin{equation}\label{def:estimateurs_chech}
(\check{\alpha}_{\epsilon,\Delta},\check{\beta}_{\epsilon,\Delta})
=\underset{(\alpha,\beta)\in \Theta }{\mathrm{argmin}} \ \check{U}_{\epsilon,\Delta}(\alpha,\beta).
\end{equation}
%Let $\P_{\theta}^{ n}$ the  distribution on $( (\R^p)^n, \mathcal{ B}((\R^p)^n))$ of $(X(t_k),k=1,\dots,n)$.
% Clearly, the distributions of this $n$-tuple corresponding to different values of
%$(\alpha,\beta)$ are absolutely continuous.
%% Let us define using  (\ref{TSparamxg}) the random variables,
%\begin{equation}\label{BkTS}
% B_k(\alpha, X)=X(t_k)-z(\alpha, t_{k})-\Phi(\alpha,t_k,t_{k-1})
%\left[X(t_{k-1})-z(\alpha,t_{k-1})\right]
%\end{equation}
%Let $\mathcal{ G}_k^n$ denote the $\sigma$-algebra $ \sigma(B(s), s \leq \frac{kT}{n})$.
% Then , the following holds.
%
%%\pagebreak

\subsection{Properties of the estimators}
In what follows, we use to describe the asymptotics with respect to $\Delta= \Delta_n$  either $\Delta \rightarrow 0$ or $n\rightarrow \infty $.
 Indeed, it is more explicit to state results according to the number of observations  $n$ rather than in terms of the size of the sampling interval  $\Delta=\Delta_n$.
Results are obtained when $\epsilon\rightarrow 0 $ and  $\Delta \rightarrow 0$ (or $n\rightarrow \infty$) simultaneously.

Define, for $\theta=(\alpha,\beta) \in \Theta$ with $\Theta$ a compact subset of $\R^a\times \R^b$, the matrices $I_b(\theta)=(I_b(\theta)_{ij}, 1\leq i,j \leq a)$,
 $I_{\sigma}(\theta)=(I_{\sigma} (\theta)_{ i,j}, 1\leq i,j \leq b)$ and $I(
 \theta)$ by
 \begin{equation}\label{Ibtheta}
 (I_b(\theta))_{ij}= \int_{0}^T( \nabla _{\alpha_i}b(\alpha, t, z(\alpha,t)))^{*} \Sigma^{-1} (\beta,t,z(\alpha,t)) \nabla_{\alpha_j} b(\alpha, t, z(\alpha,t)) dt,
 \end{equation}

\begin{align}
&(I_{\sigma} (\theta))_{i,j}= \label{Isigma}\\
&\resizebox{0.93\linewidth}{!}{
\begin{minipage}{\linewidth}
$\displaystyle\frac{1}{2T}\int_0 ^T  \mathrm{Tr} \left( \nabla_{\beta_i} \Sigma(\beta,t, z(\alpha,t)) \Sigma^{-1}(\beta,t, z(\alpha,t))  \nabla_{\beta_j}\Sigma(\beta,t, z(\alpha,t))\Sigma^{-1}(\beta,t, z(\alpha,t))  \right) dt.$
\end{minipage}}\nonumber
\end{align}

\smallskip

\begin{equation}  \label{Itheta}
I(\theta)= \begin{pmatrix} I_b(\theta))&0 \\
0 &I_{\sigma} (\theta)
\end{pmatrix}.
\end{equation}

\noindent
Recall that $A^*$ denotes the  transpose  of  a matrix  $A$ and $ \mathrm{Tr}(A)$ its  trace.
%In what follows, we either use to describe the asymptotics with respect to $\Delta= \Delta_n$  either $\Delta \rightarrow 0$ or $n\rightarrow \infty $.
% Indeed, it is more explicit to state results according to the number of observations  $n$ rather than in terms of the size of the sampling interval  $\Delta=\Delta_n$.

\begin{theorem}\label{theo:normalite}
Assume that $(X(t))$ satisfying \eqref{SDEGen} is observed at times $t_k=k\Delta_n$ with $T=n\Delta_n$ fixed.
Assume  \textbf{(S1)--(S7)} and that $I_b(\theta_0)$ is non-singular. Then,   as $\epsilon \rightarrow 0, \Delta= \Delta_n \rightarrow 0$,
\begin{enumerate}
\item[(i)] $(\check{\alpha}_{\epsilon,\Delta},\check{\beta}_{\epsilon,\Delta}) \rightarrow (\alpha_0,\beta_0)$ in $\P_{\theta_0}$-probability.\\
\item[(ii)] If moreover $I_{\sigma}(\theta_0)$ is non-singular,
% Assume \textbf{(S1)--(S7)}. If $I_b(\alpha_0,\beta_0),I_\sigma(\alpha_0,\beta_0)$ %defined in \eqref{Ib}, \eqref{Isigma}
 % are invertible, then, as $\epsilon,\Delta \rightarrow 0$, we have, under $\mathbb{P}_{\theta_0}^{\epsilon}$, in distribution
 \[ \begin{pmatrix}
    \epsilon^{-1}\left(\check{\alpha}_{\epsilon,\Delta}-\alpha_0\right)\\
\sqrt{n}\left(\check{\beta}_{\epsilon,\Delta}-\beta_0\right)
   \end{pmatrix}
\rightarrow  \mathcal {N}_{a+b} \left(0,\begin{pmatrix}
                	I_b^{-1}(\alpha_0,\beta_0) & 0 \\
0 & I_\sigma^{-1}(\alpha_0,\beta_0)
                \end{pmatrix}
\right)
\]
in distribution under $\P_{\theta_0}$.
\end{enumerate}
\end{theorem}

Note that results are obtained without any condition linking $\epsilon$ and $\Delta$ (or $n$). Indeed,  the
 previous results obtained in \cite{glo09IV} require conditions linking $\epsilon $ and $\Delta $ that do not fit epidemic data, where generally the  orders of magnitude for $N$ and $n$ satisfy  $N>>n$  so that  $\Delta$ is comparatively large with respect to $\epsilon= 1/\sqrt{N}$.
 We proposed  in \cite{guy14IV} another method based on Theorem \ref{th_Taystoch} which extends results obtained in \cite{gen90IV}, where the inference in the case $\sigma(\beta, x) \equiv 1$ was  investigated for one-dimensional diffusions using expansion (\ref{TS}).\\

Since estimators of parameters in the drift (here $\alpha$) and in the diffusion coefficient (here $\beta$) converge at distinct rates, $\epsilon^{-1}$ and $\sqrt{n}= \Delta^{-1/2}$ respectively, the study of asymptotic properties has to be  performed according to the successive steps:
\begin{enumerate}[Step (4):]
\item[Step (1):]  Consistency of  $\check{\alpha}_{\epsilon,\Delta}$ (Proposition \ref{CV-contraste}).
\item[Step (2):] Tightness of the sequence $\epsilon^{-1} ( \check{\alpha}_{\epsilon,\Delta}-\alpha_0)$ with respect to $\beta$ (Proposition \ref{borneproba}).
\item[Step (3):] Consistency of $\check{\beta}_{\epsilon,\Delta}$ (Proposition \ref{CV-beta}).
\item[Step (4):]  Asymptotic normality for both estimators (Theorem \ref{theo:normalite},(ii)).
\end{enumerate}
The proof is  technical and detailed  in a separate section. Before  this proof, let us state some comments.
\subsubsection{Comments} \label{comments}
\textbf{(1)} \emph{Efficiency of estimators}: Note that the matrix $I_b(\theta)$ is equal to the Fisher information matrix associated to the estimation of $\alpha$ when $(X(t))$ is  continuously observed on $[0,T]$ (see e.g.\ \cite{kut84IV} and  Section \ref{Recapdiff} in the Appendix).
Therefore $\check{\alpha}_{\epsilon,\Delta}$ is efficient for this statistical model. \\

\noindent\textbf{ (2)} \emph{Comparison with estimation based  on  complete observation of the jump process $(\mathcal{Z}^N(t))$}:
 Coming back to epidemics, we can compare the estimation of the parameters of the pure jump process $(\mathcal{Z}^N(t))$ and $(Z^N(t)= \frac{1}{N}\mathcal{Z}^N(t))$ describing the epidemic dynamics in a finite population of size $N$ and
 the estimators  built from its diffusion approximation.
 Let us stress that there is a main difference between these two approaches. Statistical inference for  $(\mathcal{Z}^N(t))$ is based on the observations of all the jumps, which
 implies here the observation of all the times of infection and recovery for each individual in the population, while for the diffusion $(X(t))$, we consider discrete observations $(X(t_k), k=1,\dots,n)$.\\

\noindent\textbf{ (3)}  \emph{Comparison of estimators for  the SIR epidemic dynamics}:
 Assume that the jump process $(\mathcal{Z}^N(t))$ is continuously observed on $[0,T]$.
 Its dynamics is described by the two parameters $(\lambda,\gamma)$. Set $Z^N(t)= (S^N((),I^N(t))^*$, and assume that $Z^N(0) \rightarrow  x_0=(s_0,i_0)^*$, with
 $s_0>0, i_0>0$.
 Let  $s(t)= s( \lambda_0, \gamma_0,t); i(t)= i(\lambda_0, \gamma_0,t)$ the solution of the corresponding ODE.

The  Maximum Likelihood Estimator $ (\hat{\lambda},\hat{\gamma})$ is  explicit (see  \cite{and00IV} or  Chapter \ref{chappointproc} of this part).
Indeed, let  $(T_i)$ denote   the successive jump times  and set   $J_i=0$ if we have an infection and $J_i=1$ if we have a recovery.
 Let $K_N(T)= \sum_{i\geq 0}1_{T_i\leq T}$. Then
 \begin{align*}
\hat{\lambda}_N&= \frac{1}{N} \frac{ \sum_{i=1}^{K_N(T)} (1-J_i)}{\int_0^T S^N(t) I^N(t) dt}=\frac{1}{N} \frac{\mbox{ \# Infections} }{\int_0^T S^N(t) I^N(t) dt},\\
\hat{\gamma}_N&= \frac{1}{N} \frac{ \sum_{i=1}^{K_N(T)} J_i}{\int_0^T I^N(t) dt} = \frac{\mbox{\# Recoveries}}{\mbox{``Mean infectious period''}}.
 \end{align*}
As the population size $N$ goes to infinity, $(\hat{\lambda}_N,\hat{\gamma}_N)$  is consistent and
  $$\sqrt{N} \begin{pmatrix}\hat{ \lambda}_N-\lambda \\ \hat{\gamma}_N-\gamma \end{pmatrix} \rightarrow \mathcal{ N}_2\left(0, I^{-1}(\lambda,\gamma)\right), \mbox{ where } I(\lambda,\gamma)= \begin{pmatrix} \frac{\int_0^T s(t)i(t) dt}{\lambda} &0\\0& \frac{\int_0^Ti(t) dt} {\gamma}
\end{pmatrix}  \;.
$$
The matrix $I(\lambda,\gamma) $ is the Fisher information matrix of this statistical model.\\
% $$I(\lambda,\gamma)= \begin{pmatrix} \frac{\int_0^T s(t)i(t) dt}{\lambda} &0\\0& \frac{\int_0^Ti(t) dt} {\gamma}
%\end{pmatrix}.$$

 Consider now the $ SIR$ diffusion approximation  $X(t)$  described in Section \ref{SIR_appli}.
 We have
 \begin{equation}
b(\theta,(s,i))= \begin{pmatrix}
-\lambda si \\ \lambda si -\gamma i \end{pmatrix};\quad
\Sigma(\theta,(s,i))
=  \begin{pmatrix} \lambda si & -\lambda si \\ - \lambda si & \lambda si
+ \gamma i \end{pmatrix}.\end{equation}

Therefore,
$$ \nabla_{\theta} b(\theta, (s,i))= \begin{pmatrix} -si& 0\\si&-i\end{pmatrix};
\Sigma^{-1}(\theta, (s,i))=\frac{ 1}{\lambda \gamma s i} \begin{pmatrix}\lambda s+\gamma& \lambda s\\
\lambda s& \lambda s \end{pmatrix}.
$$

The matrix $I_b(\theta)$ defined in \eqref{Ibtheta} is
$$I_b(\lambda,\gamma)= \begin{pmatrix} \frac{1}{\lambda} \int_0 ^T s(t) i(t) dt& 0\\
0& \frac{1}{\gamma}  \int_0 ^T i(t) dt\end{pmatrix}.$$

Therefore, we obtain the same information matrix  in both cases.\\

Consider the $SIRS$  model with immunity  waning $\delta$.  We  have $ \theta=(\lambda,\gamma,\delta)$
The diffusion approximation satisfies
$$b(\theta,(s,i))= \begin{pmatrix}
-\lambda si +\delta(1-s-i) \\ \lambda si -\gamma i \end{pmatrix};\quad
\Sigma(\theta,(s,i))
=  \begin{pmatrix} \lambda si + \delta(1-s-i) & -\lambda si \\ - \lambda si & \lambda si
+ \gamma i \end{pmatrix}.$$
Hence,
\[ \nabla_{\theta} b(\theta,s,i)= \begin{pmatrix} -si& 0& (1-s-i)\\si&-i &0\end{pmatrix},\;\]
\[I_b(\theta)= \int_0^T \nabla_{\theta}^* b(\theta,s(t),i(t))\Sigma^{-1} (\theta,s(t),i(t)) \nabla_{\theta} b(\theta,s(t),i(t)) dt.\]
Then $I_b(\theta) $ can   be computed  and compare to the Fisher information matrix derived from the statistical model corresponding to complete observation of the $SIRS$ jump process.\\

\subsection{Proof of Theorem \ref{theo:normalite}}
Recall the notations: for a matrix  $A$, $A^{*}$  the transposition of $A$, $\mbox{det}(A)$ the determinant of $A$ and $\mathrm{Tr}(A)$ the trace of $A$.
\subsubsection{Step (1):  Consistency of  $\check{\alpha}_{\epsilon,\Delta}$}
% (Proposition \ref{CV-contraste}).\\
%\subsection{ Study  of the contrast function $\check{U}_{\epsilon,\Delta} (\alpha,\beta)$}

%Set
%\begin{equation}\label{Sigkbeta}
% \Sigma_k(\beta)= \Sigma(\beta,t_k,X(t_k)).
%\end{equation}
 Let us define, using (\ref{def:Gamma}),
 \begin{equation}\label{def:K1}
K_1(\alpha_0,\alpha,\beta)=\int_{0}^{T} \Gamma(\alpha_0,\alpha, t)^{*}
\Sigma^{-1}(\beta, t, z(\alpha_0,t))\Gamma(\alpha_0,\alpha,t)dt.
\end{equation}
By  Assumption (S4),  if $\alpha \neq \alpha_0 $, $b(\alpha,t,z(\alpha,t)) \not \equiv b(\alpha_0,t,z(\alpha_0))$, therefore  the function
$\Gamma(\alpha_0,\alpha,\cdot) \not \equiv 0$, which implies that
$K_1(\alpha_0,\alpha,\beta)$ is non-negative and equal to 0 if and only
if  $\alpha=\alpha_0$.
%Let us study the consistency of $\check{\alpha}_{\epsilon,\Delta}$ (Step 1).

The contrast function  $\check{U}_{\epsilon,\Delta} (\alpha,\beta)$ defined in (\ref{Ucheck}) satisfies
\begin{proposition}\label{CV-contraste}
Assume (S1)--(S6). Then, as $\epsilon,\Delta \rightarrow 0$, the following
convergences hold.
\begin{enumerate}
\item[(i)] $ \sup_{\theta\in \Theta} | \epsilon^2 \left(\check{U}_{\epsilon,\Delta}(\alpha,\beta)-\check{U}_{\epsilon,\Delta}(\alpha_0,\beta)\right)-K_1(\alpha_0,\alpha,\beta)|\rightarrow 0$ in $\P_{\theta_0}$-probability.
\item[(ii)]  $\check{\alpha}_{\epsilon,\Delta} \rightarrow \alpha_0 $ in probability under $\P_{\theta_0}$.
\end{enumerate}
\end{proposition}

\begin{proof}
Let us prove (i). We have, by (\ref{Ucheck}) and \eqref{Sigkbeta},
\[\epsilon^2 (\check{U}_{\epsilon,\Delta}(\alpha,\beta)-\check{U}_{\epsilon,\Delta} (\alpha_0,\beta))= T_1+T_2\]
 with
\begin{eqnarray*}
T_1&=&2\sum_{k=1}^{n} \frac{\left(B_k(\alpha,X) - B_k(\alpha_0,X)\right)^{*}}{\Delta}\; \Sigma^{-1}_{k-1}(\beta)\;B_k(\alpha_0,X),\\
T_2&=&\Delta \sum_{k=1}^{n} \frac{\left(B_k(\alpha,X) - B_k(\alpha_0,X)\right)^{*} }{\Delta} \; \Sigma^{-1}_{k-1}(\beta)\;\frac{\left(B_k(\alpha,X) - B_k(\alpha_0,X)\right)}{\Delta}.
\end{eqnarray*}
 By Lemma \ref{BXa}, $\displaystyle{ \frac{\left(B_k(\alpha,X) - B_k(\alpha_0,X)\right)}{\Delta} }$ is bounded, and  (ii) of Lemma  \ref{BXa0} yields that  $T_1$ goes to 0 in $\P_{\theta_0}$-probability.
Using now  (\ref{def:Gamma}), we have by Lemma \ref{BXa}, setting $\zeta_{k}=\Delta r_{k} + \epsilon \norm{\alpha-\alpha_0}\eta_k$,
\begin{eqnarray*}
T_2&=&  \Delta \sum_{k=1}^{n} (\Gamma(\alpha_0,\alpha,t_{k-1})+\zeta_{k-1})^{*}  \Sigma^{-1}_{k-1}(\beta)\;(\Gamma(\alpha_0,\alpha,t_{k-1}) +\zeta_{k-1})\\
&=&  \Delta \sum_{k=1}^{n} \left(\Gamma(\alpha_0,\alpha,t_{k-1})^{*} \Sigma_{k-1}^{-1}(\beta) \Gamma(\alpha_0,\alpha,t_{k-1}) + R_k(\theta_0,\theta,\epsilon,\Delta)\right).
\end{eqnarray*}
The first term of the above formula as a Riemann sum converges by Lemma \ref{BGamma}
to the function $K_1(\alpha_0,\alpha,\beta)$ defined in \eqref{def:K1} as $\Delta \rightarrow 0$. This convergence is uniform with respect to the parameters.
The remainder term  is
\begin{align*}
R_k(\theta_0,\theta,\epsilon,\Delta)=& \Gamma(\alpha_0,\alpha,t_{k-1})^*( \Sigma^{-1}_{k-1}(\beta) -\Sigma^{-1}(\beta,t_{k-1}, z(\alpha_0,t_{k-1})))\Gamma(\alpha_0,\alpha,t_{k-1})\\
&+  \Delta R^1_k(\theta_0,\theta,\epsilon,\Delta) + \epsilon R_k^2 (\theta_0,\theta,\epsilon,\Delta).
\end{align*}
Using Proposition  \ref{TStheta}  and Lemma \ref{BXa},   it is straightforward to get that  $\sup_{k}\norm{R_k(\theta_0,\theta,\epsilon,\Delta)} \rightarrow 0$ in
$\P_{\theta_0}$-probability uniformly with respect to $\theta$.
Hence, $T_2$ converges to $ K_1(\alpha_0,\alpha,\beta)$ in $\P_{\theta_0}$-probability
 uniformly with respect to $\theta$.\\
 %since the remainder terms go to 0 uniformly since $r_k$ and $\eta_k$ are uniformly bounded.
%The first term of the above formula is a Riemann sum which converges by Lemma \ref{BGamma}
%to the function $K_1(\alpha_0,\alpha,\beta)$ defined in \eqref{def:K1} as $\Delta \rightarrow 0$. This convergence is uniform with respect to the parameters.
Let us prove (ii). Noting that, for all $\beta$, $ K_1( \alpha_0,\alpha_0,\beta) =0$, we have

\begin{align*}
0 \leq &\,  K_1(\alpha_0,\check{\alpha}_{\epsilon,\Delta}, \check{\beta}_{\epsilon,\Delta})- K_1(\alpha_0,\alpha_0, \check{\beta}_{\epsilon,\Delta})\\
 \leq &\, [\epsilon^2 (\check{U}_{\epsilon,\Delta}(\alpha, \check{\beta}_{\epsilon,\Delta}) - \check{U}_{\epsilon,\Delta}(\alpha_0, \check{\beta}_{\epsilon,\Delta}) )
 - K_1(\alpha_0,\alpha, \check{\beta}_{\epsilon,\Delta}) ]\\
  & +  [K_1(\alpha_0,\check{\alpha}_{\epsilon,\Delta}, \check{\beta}_{\epsilon,\Delta})-
  \epsilon^2 (\check{U}_{\epsilon,\Delta}(\check{\alpha}_{\epsilon,\Delta}, \check{\beta}_{\epsilon,\Delta})-\check{U}_{\epsilon,\Delta}(\alpha_0 ,\check{\beta}_{\epsilon,\Delta}))]\\
 & + \epsilon^2 [\check{U}_{\epsilon,\Delta}(\check{\alpha}_{\epsilon,\Delta}, \check{\beta}_{\epsilon,\Delta}) -\check{U}_{\epsilon,\Delta}(\alpha, \check{\beta}_{\epsilon,\Delta})]\\
  \leq &\, 2 \sup_{\beta \in K_b} |\epsilon^2 [\check{U}_{\epsilon,\Delta}(\alpha, \beta) -\check{U}_{\epsilon,\Delta}(\alpha_0, \beta)] -K_1(\alpha_0,\alpha,\beta)|,
  \end{align*}
where the last inequality is obtained using that the minimum of $\check{U}_{\epsilon,\Delta}(\alpha, \beta)$ is  attained at $( \check{\alpha}_{\epsilon,\Delta}, \check{\beta}_{\epsilon,\Delta})$.
By Proposition \ref{CV-contraste} (i), we finally get that
\[|K_1(\alpha_0,\check{\alpha}_{\epsilon,\Delta}, \check{\beta}_{\epsilon,\Delta})- K_1(\alpha_0,\alpha_0, \check{\beta}_{\epsilon,\Delta})|  \rightarrow 0,\]
which yields by Assumption \textbf{ (S6)} that $\check{\alpha}_{\epsilon,\Delta} \rightarrow \alpha_0$
in $\P_{\theta_0}$-probability as $\epsilon,\Delta \rightarrow 0.$
\end{proof}

\subsubsection{Step (2): Tightness of the sequence $\epsilon^{-1} ( \check{\alpha}_{\epsilon,\Delta}-\alpha_0)$ }
This step is crucial in the presence of different rates of convergence for $\alpha$ and $\beta$ and concerns  results that hold for all $\beta \in K_b$.

%with respect to $\beta$ (Proposition \ref{borneproba}).\\Let us now  study Step 2, i.e.\  the tightness of the sequence $\epsilon^{-1}(\check{\alpha}_{\epsilon,\Delta}- \alpha_0)$ with respect to $\beta$.
% Define, for $\theta=(\alpha,\beta) \in \Theta$, the matrix $I_b(\alpha,\beta)$  for $1\leq i,j \leq a$,
% \begin{equation}\label{Ibtheta}
% (I_b(\alpha,\beta))_{ij}= \int_{0}^T( \nabla _{\alpha_i}b(\alpha, t, z(\alpha,t)))^{*} \Sigma^{-1} (\beta,t,z(\alpha,t)) \nabla_{\alpha_j} b(\alpha, t, z(\alpha,t)) dt.
% \end{equation}
\begin{proposition}\label{borneproba}
 Assume \textbf{(S1)--(S4)} and that  $I_b(\alpha_0,\beta_0)$ defined in (\ref{Ibtheta})  is non-singular. Then,  as $\epsilon,\Delta \rightarrow 0$,  $\sup_{\beta \in K_b}\norm{\epsilon^{-1}\left( \check{\alpha}_{\epsilon,\Delta}-\alpha_0\right)}$ is bounded in $\mathbb{P}_{\theta_0}$-proba\-bility.
\end{proposition}

\begin{proof}
Recall the notation: for $f$ a twice differentiable real function,
$\nabla^2_\alpha f = (\frac{\partial ^2 f}{\partial  \alpha_i \partial \alpha_j})_{1\leq i,j\leq a}$.\\
%We have $0=\nabla_{\alpha}\check{U}_{\epsilon,\Delta}(\check{\alpha}_{\epsilon,\Delta}, \check{\beta}_{\epsilon,\Delta})$.
Under (S5), $\check{U}_{\epsilon,\Delta}(\alpha,\beta)$ is $C^2$ and a Taylor expansion  of $\nabla_{\alpha}\check{U}_{\epsilon,\Delta}$ at point $(\alpha_0,\check{\beta}_{\epsilon,\Delta})$  w.r.t.\ $\alpha$ yields,
\begin{equation}\label{Expansionalpha}
0= \epsilon  \nabla_{\alpha}\check{U}_{\epsilon,\Delta}(\check{\alpha}_{\epsilon,\Delta}, \check{\beta}_{\epsilon,\Delta}) =  \epsilon \nabla_{\alpha}\check{U}_{\epsilon,\Delta}(\alpha_0, \check{\beta}_{\epsilon,\Delta})+ \epsilon^2 N_{\epsilon,\Delta}(\check{\alpha}_{\epsilon,\Delta}, \check{\beta}_{\epsilon,\Delta})
\frac{(\check{\alpha}_{\epsilon,\Delta}-\alpha_0)}{\epsilon},
\end{equation}
\begin{equation}\label{Neps}
\mbox{with} \quad N_{\epsilon,\Delta}(\alpha,\beta) =\int _0^1 \nabla^2_{\alpha} \check{U}_{\epsilon,\Delta}
(\alpha_0+ t (\alpha-\alpha_0),\beta) dt.
\end{equation}
The proof relies on two properties:  under $\P_{\theta_0}$, as $\epsilon,\Delta \rightarrow 0$, for all $\beta \in K_b$,  $(\epsilon \nabla_{\alpha} \check{U}_{\epsilon,\Delta}(\alpha_0,\beta))$ converges in distribution to a Gaussian law and the sequence
%\begin{equation}\label{nabla2U}
$\epsilon^2  \nabla^2_{\alpha} \check{U}_{\epsilon,\Delta}(\alpha_0,\beta) $ converges almost surely.

Let us study  $- \epsilon \nabla_{\alpha}\check{U}_{\epsilon,\Delta}(\alpha_0, \beta) $.
Define the $ a\times a$ matrix
\begin{equation}\label{Jbeta}
J( \theta_0, \beta)= \int_0^T (\nabla_{\alpha} b(\alpha_0, t,z(\alpha_0,t)))^*\Xi( \theta_0,\beta, t)\nabla_{\alpha} b(\alpha_0, t,z(\alpha_0,t)) dt, \; \mbox {with}
 \end{equation}
 \begin{equation} \label{defXi}
\Xi( \theta_0,\beta, t)=\Sigma^{-1} (\beta, t, z(\alpha_0,t))  \Sigma (\beta_0, t, z(\alpha_0,t))   \Sigma^{-1} (\beta, t, z(\alpha_0,t)).
\end{equation}
 The following holds.
\begin{lemma}\label{Convcki}
Assume (S1)--(S5). Then, as $\epsilon,\Delta \rightarrow 0$,
$$ - \epsilon \nabla_{\alpha}\check{U}_{\epsilon,\Delta}(\alpha_0, \beta)  \rightarrow \mathcal{ N}(0, 4  J(\theta_0,\beta)) \mbox{ in distribution under } \P_{\theta_0}.$$
\end{lemma}
\begin{proof}
We have, using the notations of Lemma \ref{dNkdb}
and setting
\begin{equation}\label{Hik}
 H^i_k(\alpha_0,\beta)= \Sigma^{-1} (\beta,t_{k-1},z(\alpha_0,t_{k-1}) )\nabla_{\alpha_i} b(\alpha_0, t_{k-1},z(\alpha_0,t_{k-1}))
  % \begin{pmatrix} H^1_k(\alpha_0,\beta)\\ \cdots \\H^a_k(\alpha_0,\beta) \end{pmatrix}
%+ \Delta r_k^i(\alpha_0,\alpha, \Delta)),H^i_k(\alpha_0,\beta)=
\end{equation}
%\begin{equation}\label{Hik}
%H^i_k(\alpha,\beta)= \Sigma^{-1} (\beta, t_{k-1}, z(\alpha_0,t_{k-1})) (\nabla_{\alpha_i} b(\alpha, t_{k-1},z(\alpha,t_{k-1}))+ \Delta r_k^i(\alpha_0,\alpha, \Delta)),
%\end{equation}
\begin{equation*}
-\epsilon \nabla_{\alpha_i}\check{U}_{\epsilon,\Delta}(\alpha_0, \beta)= -\frac{2}{\epsilon \Delta}  \sum_{k=1}^n   (\nabla_{\alpha_i} B_k(\alpha_0,X)) ^{*}  \Sigma^{-1}_{k-1}(\beta) B_k(\alpha_0,X)= A_n^i+A^{\prime,i}_n
+ A^{\prime\prime,i}_n,
\end{equation*}
with
 \begin{eqnarray*}
A_n^i&=&  \frac{2}{\epsilon}  \sum_{k=1}^n  H^i_k(\alpha_0,\beta)^* B_k(\alpha_0,X),\\
%\mbox H^i_k(\alpha_0,\beta)= \Sigma^{-1} (\beta,t_{k-1},z(\alpha_0,t_{k-1}) \nabla_{\alpha_i} b(\alpha_0, t_{k-1},z(\alpha_0,t_{k-1}))\\
 A^{\prime,i}_n&=&-2 \sum_{k=1}^n (M_k^i (\alpha_0) Z_{k-1}(\theta_0))^*\Sigma^{-1}_{k-1}(\beta) B_k(\alpha_0,X),\\
 A^{\prime\prime,i}_n&=&2  \sum_{k=1}^n \frac{\nabla_{\alpha_i} B_k(\alpha_0,X)}{\Delta} ^{*} \frac{ \Sigma^{-1}_{k-1}(\beta)-
 \Sigma^{-1} (\beta, t_{k-1}, z(\alpha_0,t_{k-1}))} {\epsilon} B_k(\alpha_0,X).
\end{eqnarray*}
By Lemma \ref{BXa0} (ii), Lemma \ref{dNkdb} and Theorem \ref{TStheta},  $A^{\prime,i}_n$  and $ A^{\prime\prime,i}_n$ tend to 0 in $\P_{\theta_0}$-probability
as $\epsilon,\Delta \rightarrow 0$.\\
To study $A_n^i$, we write, using the notations of Lemma \ref{BkX},
\begin{align}
B_k(\alpha_0,X) = \epsilon\sqrt{\Delta} T_k(\theta_0)&+\epsilon^2 (R(\theta_0,t_k)- R(\theta_0,t_{k-1}))\label{BkRk}\\
& +\epsilon^2(I_p- \Phi(\alpha_0, t_k,t_{k-1})) R(\theta_0,t_{k-1}) .\nonumber
\end{align}
Hence, $A_n^i= D^i_{n}+ C^{i}_{n}+C^{\prime,i} _{n}$ where, using \eqref{Hik},
\begin{equation} \label{Dni}
D_n ^i=2 \sqrt{\Delta} \sum_{k=1}^n  (H^i_k(\alpha_0,\beta))^*  T_k(\theta_0),
\end{equation}
\[C^{i}_{n}= 2 \epsilon \sum_{k=1}^n  (H^i_k(\alpha_0,\beta))^*(R(\theta_0,t_k)- R(\theta_0,t_{k-1}))\mbox{ and}\]
\[C^{\prime,i} _{n}=2 \epsilon \Delta \sum_{k=1}^n  (H^i_k(\alpha_0,\beta))^* \frac{1}{\Delta} (I_p- \Phi(\alpha_0,t_k,t_{k-1})) R(\theta_0,t_{k-1}).\]
Let us first study $C^{\prime,i} _{n}$.
Noting that
\[\frac{1}{\Delta} (I_p- \Phi(\alpha_0,t_k,t_{k-1}))= \nabla_z b( \alpha_0, z(\alpha_0,t_{k-1})) + \Delta O(1),\]
we have
$|C^{\prime,i} _{n}| \leq \epsilon n C(\theta_0) $, with $C(\theta_0) $ bounded in probability.

To study  $C^{i}_{n}$, we first apply an Abel transform to the sequence and get
\[C^{i}_{n}=  2\epsilon \sum_{k=1}^n  (H^i_{k-1}(\alpha_0,\beta)- H^i_{k}(\alpha_0,\beta)) ^* R(\theta_0,t_{k-1})+ \epsilon H^n_{k}(\alpha_0,\beta)^* R(\theta_0,t_{n}).\]
The continuity assumptions ensure that $\sup_{k \leq n} \frac{1}{\Delta} \norm {H^i_{k-1}(\alpha_0,\beta)- H^i_{k}(\alpha_0,\beta))} $ is bounded. Hence $C^i_n \rightarrow 0$ since $\norm{R (\theta_0,t_k)}$ is uniformly bounded.

It remains to study the main term  $D_{n}= (D_{n}^i)_{1\leq i \leq a}$ defined in \eqref{Dni}. Let
$$H_k(\alpha_0,\beta)= \Sigma_{k-1}^{-1}(\beta)\nabla_{\alpha} b(\alpha_0,t_{k-1},z(\alpha_0, t_{k-1})).$$
Then $(D_n)$ is a multidimensional triangular array which reads as
$D_n=  \sum_{k=1}^n \zeta_k^n $ with
$\zeta_k^n =   \sqrt{\Delta} H_k(\alpha_0,\beta)^*  T_k(\theta_0) \in \R^a$.\\
Note that $D_n$ does not depend on $\epsilon$  and convergence results are obtained for $\Delta_n\rightarrow 0.$
% %Then  $D_n=  \sum_{k=1}^n \zeta_k^n $
%%with $D_{n}^i = \sqrt{\Delta} \sum_{k=1}^n  H_k^i(\alpha_0,\beta)^*  T_k(\theta_0)$. \\
% It is a multidimensional triangular array  with increments $\zeta_k^n $.\\
%=   \sqrt{\Delta} H_k(\alpha_0,\beta)^*  T_k(\theta_0)$.\\
%The multidimensional  triangular array  $D_n=\sum_{k=1}^n  \zeta_k^n$ satisfies that, as $\epsilon,\Delta \rightarrow 0$, \\
To apply   to $(D_n)$ a theorem of convergence in law for triangular arrays (Theorem \ref{CLTTT} in the Appendix  or \cite{jac12IV} Theorem 2.2.14), we have to prove  that,
\begin{enumerate}
\item[(i)] $\sum_{k=1}^n  \E_{\theta_0}(\zeta_k^{n}|\mathcal{ G}_{k-1}^n)=0$,

\item[(ii)] $\sum_{k=1}^n  \E_{\theta_0}(\zeta_k^{n,i}\zeta_k^{n,j} |\mathcal{ G}_{k-1}^n) \rightarrow  J_{ij}( \theta_0, \beta) $ (see Definition  \ref{Jbeta} below),
\item[(iii)] $\sum_{k=1}^n  \E_{\theta_0}( (\zeta_k^{n,i})^4 |\mathcal{ G}_{k-1}^n)  \rightarrow 0$.
\end{enumerate}
Since $T_k(\alpha_0)$ is centered,  (i) is clearly satisfied.
For (ii), consider for $ 1 \leq i, j \leq a$,
\begin{align*}
\E_{\theta_0}( \zeta_k^{n,i}  \zeta_k^{n,j}|\mathcal{ G}_{k-1}^n)&= \Delta H_k^i(\alpha_0,\beta)^* \E_{\theta_0}(T_k(\theta_0)T_k^* (\theta_0)) H_k^j(\alpha_0,\beta)\\
&= \Delta H_k^i(\alpha_0,\beta)^* S_k(\alpha_0,\beta_0) H_k^j(\alpha_0,\beta).
\end{align*}
Noting that   $\norm{S_k(\theta_0)-\Sigma( \beta_0,t_{k-1},z(\alpha_0,t_{k-1} ) } \leq   C \Delta $ yields, using Definition \ref{defXi},
\begin{align*}
&\E_{\theta_0}(\zeta_k^{n,i}\zeta_k^{n,j} |\mathcal{ G}_{k-1}^n)\\
& =\Delta (\nabla_{\alpha_i} b(\alpha_0, t_{k-1},z(\alpha_0,t_{k-1})))^*
\Xi( \theta_0,\beta, t_{k-1})\nabla_{\alpha_j} b(\alpha_0, t_{k-1},z(\alpha_0,t_{k-1}))\\
&\quad +\Delta^2 O(1).
\end{align*}
%+ \Delta ^2  R^{ij}_k (\theta_0,\beta,\epsilon,\Delta)}$ where  the terms $ R^{ij}_k (\theta_0,\beta,\epsilon,\Delta) $ are  uniformly bounded and where
%$\Xi$ is the matrix defined by, for $t \in[0,T]$
%\begin{equation} \label{defXi%\Xi( \theta_0,\beta, t)=\Sigma^{-1} (\beta, t, z(\alpha_0,t))  \Sigma (\beta_0, t, z(\alpha_0,t))   \Sigma^{-1} (\beta, t, z(\alpha_0,t)).
%\end{equation}
Therefore,  as a Riemann sum,
\begin{align*}
&\sum_{k=1}^n  \E_{\theta_0}(\zeta_k^{n,i}\zeta_k^{n,j} |\mathcal{ G}_{k-1}^n)\\
&\quad \rightarrow  \int_0^T (\nabla_{\alpha_i} b(\alpha_0, t, z(\alpha_0,t)))^*\Xi( \theta_0,\beta, t)\nabla_{\alpha_j} b(\alpha_0, t,z(\alpha_0,t)) dt.
\end{align*}
%= \Delta \sum_{k=1}^n (\nabla_{\alpha_i} b(\alpha, t_{k-1},z(\alpha_0,t_{k-1})))^*
%\Xi( \theta_0,\beta, t_{k-1})(\nabla_{\alpha_j} b(\alpha, t_{k-1},z(\alpha,t_{k-1})))+  R( \theta_0,\beta,\epsilon,\Delta)$,
%\%Sigma^{-1} (\beta, t_{k-1}, z(\alpha_0,t_{k-1}))  \Sigma (\beta_0, t_{k-1}, z(\alpha_0,t_{k-1})   \Sigma^{-1} (\beta, t_{k-1}, z(\alpha_0,t_{k-1}))
%with $ \norm{R( \theta_0,\beta,\epsilon,\Delta) }\leq n\Delta^2 C_2 $.\\
Checking (iii) is  easily obtained  since
$\E_{\theta_0}( (\zeta_k^{n,i})^4 |\mathcal{ G}_{k-1}^n) \leq \Delta^2 \sup _{k,\beta} \norm{ H_k(\alpha_0,\beta)}$.\\
Joining these results achieves the proof of Lemma \ref{Convcki}.
\end{proof}
%Let us define, using \eqref{defXi} the $ a\times a$ matrix
%\begin{equation}\label{Jbeta}
%J( \theta_0, \beta))= \int_0^T (\nabla_{\alpha} b(\alpha_0, t,z(\alpha_0,t)))^*\Xi( \theta_0,\beta, t)\nabla_{\alpha} b(\alpha_0, t,z(\alpha_0,t)) dt.
% \end{equation}
%We have proved that,  under $\P_{\theta_0}^{\epsilon}$
%\begin{equation}\label{ConvCki}
%- \epsilon \nabla_{\alpha} \check{U}_{\epsilon,\Delta}(\alpha_0,\beta)) \rightarrow \mathcal{ N}_a(0, 4 J(\theta_0,\beta))
%\end{equation}
%To obtain  the tightness of the sequence $\epsilon^{-1} (\check{\alpha}_{\epsilon,\Delta}-\alpha_0)$,

%Using  (\ref{Expansionalpha}) and \ref{Neps} yields
%$\epsilon^{-1}(\check{\alpha}_{\epsilon,\Delta}-\alpha_0)= -(\epsilon^2 N^{-1}_{\epsilon,\Delta}(\check{\alpha}_{\epsilon,\Delta},\check{\beta}_{\epsilon,\Delta}) (\epsilon \nabla_{\alpha} \check{U}_{\epsilon,\Delta} (\alpha_0, \check{\beta}_{\epsilon,\Delta}))$.

Using  (\ref{Expansionalpha}) and \ref{Neps}, it remains to study
the term
\[\epsilon^2 \nabla^2_{\alpha} \check{U}_{\epsilon,\Delta}
(\alpha_0+ t (\check{\alpha}_{\epsilon,\Delta}-\alpha_0),\check{\beta}_{\epsilon,\Delta})\]
We have
%The second -order derivative of the contrast function is given by
%-\epsilon \nabla_{\alpha}\check{U}_{\epsilon,\Delta}(\alpha_0, \check{\beta}_{\epsilon,\Delta})= \left( \int_0^1 \epsilon^2 \nabla^2_{\alpha} \check{U}_{\epsilon,\Delta}
%(\alpha_0+ t (\check{\alpha}_{\epsilon,\Delta}-\alpha_0),\check{\beta}_{\epsilon,\Delta}) dt \right) \frac{1}{\epsilon}(\check{\alpha}_{\epsilon,\Delta}-\alpha_0).
%,3\end{equation}
%\begin{equation}\label{d2U}
$\epsilon^2 \nabla^2_{\alpha_i \alpha_j}  \check{U}_{\epsilon,\Delta}(\alpha,\beta) =\sum_{l=1}^4 A^{ij}_l  \quad \mbox{with}$
%\end{equation}
%% \epsilon^2 \nabla^2_{\alpha_i \alpha_j}  \check{U}_{\epsilon,\Delta}(\alpha,\beta)&=& \frac{2}{\Delta}
% \sum_{k=1}^n (\nabla_{\alpha_i} B_k(\alpha))^{*} \Sigma^{-1}_{k-1}(\beta)  \nabla_{\alpha_j}B_k(\alpha)
% +\frac{2}{\Delta} \sum_{k=1}^n
%(\nabla^2_{\alpha_i \alpha_j} B_k(\alpha))^* \Sigma^{-1}_{k-1}(\beta) B_k(\alpha)\\
%&=& \sum_{l=1}^4 A_i
% \end{eqnarray*}
 \begin{eqnarray*}
A^{ij}_1&=& \frac{2}{\Delta} \sum_{k=1}^n (\nabla_{\alpha_i} B_k(\alpha_0))^{*} \Sigma^{-1}_{k-1}(\beta) \nabla_{\alpha_j}B_k(\alpha_0),\\
A^{ij}_2&=& \frac{2}{\Delta} \sum_{k=1}^n (\nabla_{\alpha_i} B_k(\alpha)-\nabla_{\alpha_i} B_k(\alpha_0))^{*} \Sigma^{-1}_{k-1}(\beta)
(\nabla_{\alpha_j}B_k(\alpha)+\nabla_{\alpha_j}B_k(\alpha_0)),\\
A^{ij}_3&=& 2   \sum_{k=1}^n
\frac{1}{\Delta} (\nabla^2_{\alpha_i \alpha_j} B_k(\alpha))^* \Sigma^{-1}_{k-1} (\beta) B_k(\alpha_0),\\
A^{ij}_4&=& 2 \Delta \sum_{k=1}^n
\frac{1}{\Delta} (\nabla^2_{\alpha_i \alpha_j} B_k(\alpha,X))^* \Sigma^{-1}_{k-1}(\beta) \frac{1}{\Delta} (B_k(\alpha,X) - B_k(\alpha_0,X)).
\end{eqnarray*}
 By Lemmas \ref{BGamma},   \ref{dNkdb} and \ref{BXa},  $A^{ij}_2$  and $A^{ij}_4$ satisfy   $\norm{A^{ij}_l} \leq CT \norm{\alpha-\alpha_0}$.
% An application of Lemma  \ref{BXa0} (ii)  and \ref{dNkdb},  yields that $A^{ij}_2$ and $A^{ij}_4\rightarrow 0$.
 Lemma  \ref{dNkdb} (ii)  and Lemma \ref{BXa0} (ii) yield that  $A_3^{ij} \rightarrow 0.$

The main term  $A^{ij}_1$ satisfies, by Lemma \ref{dNkdb} (i),
\begin{align*}
A^{ij}_1 =&\, 2 \Delta \sum_{k=1}^n ( \nabla_{\alpha_i} b(\alpha_0,t_{k-1},z(\alpha_0,t_{k-1})))^*\Sigma_{k-1}^{-1}(\beta) \nabla_{\alpha_j} b(\alpha_0,t_{k-1}, z(\alpha_0,t_{k-1}))\\
& +  \epsilon O_P(1).
\end{align*}
Theorem \ref{TStheta} yields that, under $\P_{\theta_0}$, $\Sigma_{k-1}^{-1}(\beta)= \Sigma^{-1}(\beta, t, z(\alpha_0,t)) + \epsilon O_P(1)$.  Therefore, as a Riemann sum, we get, using \eqref{Ibtheta}, that
$A^{ij}_1\rightarrow (I_b(\alpha_0,\beta))_{ij} $ in $\P_{\theta_0}$-probability as $\epsilon,\Delta \rightarrow  0$.
Joining these results, we get  that, under $\P_{\theta_0} $, as $\epsilon,\Delta \rightarrow 0$, for all $\beta$,
%\begin{equation} \label{limitD2U}
$\epsilon^2 \nabla^2_{\alpha} \check{U}_{\epsilon,\Delta} (\alpha_0,\beta) \rightarrow 2 I_b(\alpha_0,\beta)$.
%\int_0^T(\nabla_{\alpha_i} b(\alpha_0,t, z(\alpha_0,t)))^* \Sigma^{-1}(\beta, t, z(\alpha_0,t)) \nabla_{\alpha_j} b(\alpha_0,t, z(\alpha_0,t)) dt.
%\end{equation}
Using now the consistency of $\check{\alpha}_{\epsilon,\Delta}$ yields that
 \begin{equation} \label{borneD2U}
 \sup _{\beta \in K_b} \norm{\epsilon^2 \nabla^2_{\alpha} \check{U}_{\epsilon,\Delta} (\alpha_0+ t (\check{\alpha}_{\epsilon,\Delta} -\alpha_0),\beta)-\epsilon^2 \nabla^2_{\alpha} \check{U}_{\epsilon,\Delta} (\alpha_0,\beta)} \leq K \norm{\check{\alpha}_{\epsilon,\Delta} -\alpha_0)}.
 \end{equation}
% \begin{equation} \label{limitD2U}
%\epsilon^2 \nabla^2_{\alpha_i\alpha_j} \check{U}_{\epsilon,\Delta} (\alpha_0,\beta) \rightarrow 2 \int_0^T(\nabla_{\alpha_i} b(\alpha_0,t, z(\alpha_0,t)))^* \Sigma^{-1}(\beta, t, z(\alpha_0,t)) \nabla_{\alpha_j} b(\alpha_0,t, z(\alpha_0,t)) dt.
%\end{equation}
Coming back to (\ref{Expansionalpha}),  it remains to prove that $N_{\epsilon,\Delta}(\check{\alpha}_{\epsilon,\Delta},\beta)$ is non-singular. Under (S3), $\Sigma(\beta,t,z)$ is  non-singular. Hence,
% Hence,under  the assumption that $ I_b(\theta_0)$ is non-singular
\smallskip

\resizebox{0.97\linewidth}{!}{
  \begin{minipage}{\linewidth}
\begin{align*}
\inf_{\beta \in K_b} &\det& \left(\left [\int_0^T \nabla_{\alpha_i} b(\alpha_0,t, z(\alpha_0,t))^* \Sigma^{-1}(\beta, t, z(\alpha_0,t)) \nabla_{\alpha_j} b(\alpha_0,t, z(\alpha_0,t) dt \right]_{1\leq i,j\leq a}\right)\\
& \geq & c \det \left(\left[\int_0^T \nabla_{\alpha_i} b(\alpha_0,t, z(\alpha_0,t))^* \nabla_{\alpha_j} b(\alpha_0,t, z(\alpha_0,t) dt \right]_{1\leq i,j\leq a}\right) >0.
\end{align*}
\end{minipage}
}
\smallskip

\noindent Now, the consistency of $\check{\alpha}_{\epsilon,\Delta}$  implies that, using (\ref{Neps}), $\P_{\theta_0}^{\epsilon}( \det( \epsilon^2 N_{\epsilon,\Delta}(\check{\alpha},\beta) )>0)$ tends to 1. Therefore (\ref{Expansionalpha}) yields
\[\epsilon^{-1}(\check{\alpha}_{\epsilon,\Delta}-\alpha_0)= -(\epsilon^2 N^{-1}_{\epsilon,\Delta}(\check{\alpha}_{\epsilon,\Delta},\check{\beta}_{\epsilon,\Delta}) (\epsilon \nabla_{\alpha} \check{U}_{\epsilon,\Delta} (\alpha_0, \check{\beta}_{\epsilon,\Delta}))\] is tight.
\end{proof}

\subsubsection{Step (3): consistency of  $\check{\beta}_{\epsilon,\Delta}$}
 Let us now study the estimation for  the diffusion parameter. Set
% Recall that , under (S?), $\beta \in  K_b$, compact subset.of $\R^b$.
\begin{equation}\label{def:K2}
\begin{array}{lll}
K_2(\alpha_0,\beta_0,\beta)&=&\frac{1}{T} \int_{0}^{T} \mathrm{Tr} \left( \Sigma^{-1}(\beta,t,z(\alpha_0, t))\Sigma (\beta_0,t, z(\alpha_0,t))\right) dt \\
&-& \frac{1}{T} \int_{0}^{T} \log \det \left(\Sigma^{-1}(\beta, t, z(\alpha_0,t)) \Sigma(\beta_0,t, z(\alpha_0,t))\right) \; dt   \;- p \\
\end{array}
\end{equation}
Using  the following inequality for invertible symmetric $p\times p$ matrices $A$,
$\log \det A+ p\leq \mathrm{Tr}(A)$,
$ K_2(\alpha_0,\beta_0,\beta) \geq 0$  and  $ K_2(\alpha_0,\beta_0,\beta) =0 $ if and only if
\[\{ \forall t \in[0,T], \Sigma(\beta_0,t, z(\alpha_0,t))=\Sigma (\beta,t,z(\alpha_0,t)),\]
which implies $\beta=\beta_0$ by (S7).

\begin{proposition}\label{CV-beta}
 Assume \textbf{(S1)--(S7)}. Then, if $I_b(\alpha_0,\beta_0)$ defined in \eqref{Ibtheta} is non-singular, the following holds in $\P_{\theta_0}$-probability, using \eqref{Ucheck},
  \eqref{def:estimateurs_chech} and \eqref{def:K2}
 \begin{enumerate}
\item[(i)]  $\sup_{\beta\in K_b}|\frac{1}{n}\left(\check{U}_{\Delta,\epsilon}(\check{\alpha}_{\epsilon,\Delta},\beta) -\check{U}_{\Delta,\epsilon}(\check{\alpha}_{\epsilon,\Delta},\beta_0)\right) -K_2(\alpha_0,\beta_0,\beta)| \rightarrow 0$ as $\epsilon,\Delta \rightarrow 0.$\\
\item[(ii)] $\check{\beta}_{\epsilon,\Delta} \rightarrow \beta_0 $ as $\epsilon,\Delta \rightarrow 0.$
\end{enumerate}
\end{proposition}

\begin{proof}
Let us first prove (i). Using \eqref{Ucheck} and \eqref{Sigkbeta},we get \\
$\frac{1}{n}\left(\check{U}_{\Delta,\epsilon}(\alpha,\beta) -\check{U}_{\Delta,\epsilon}(\alpha,\beta_0)\right)=A_1(\beta_0,\beta)+A_2(\alpha,\beta_0,\beta)$ with\\
\begin{equation}\label{A1}
A_1(\beta_0,\beta)=\frac{1}{n}\sum_{k=1}^{n}log \det(\Sigma_{k-1}(\beta) \Sigma_ {k-1}^{-1}(\beta_0)),
\end{equation}
\begin{equation}\label{A2}
A_2(\alpha,\beta_0,\beta)=\frac{1}{n \Delta \epsilon^2}\sum_{k=1}^{n}B_k(\alpha,X)^{*}
(\Sigma_{k-1}^{-1}(\beta)-\Sigma_{k-1}^{-1}(\beta_0)) B_k(\alpha,X).
\end{equation}
Using that, under \textbf{(S5)}, $z \rightarrow \log\left(\det\left[\Sigma(\beta,t,z )\Sigma^{-1}(\beta_0,t,z)\right]\right)$ is differentiable,
%on $U$,
an application of Proposition \ref{TStheta}  yields that, under  $\P_{\theta_0}$,
\smallskip

\resizebox{0.97\linewidth}{!}{
  \begin{minipage}{\linewidth}
\begin{align*}
&A_1(\beta_0,\beta)\\
&=\frac{\Delta}{T} \left( \sum_{k=1}^{n}\log\left(\det\left[\Sigma(\beta, t_{k-1}, z(\alpha_0, t_{k-1}))\Sigma^{-1}(\beta_0,t_{k-1},z(\alpha_0, t_{k-1}))\right]\right)+\epsilon R^{1,\epsilon}_{\theta_0,\beta}(t_{k-1})\right),
\end{align*}
\end{minipage}
}
\smallskip

\noindent with $\norm{R^{1,\epsilon}_{\alpha_0,\beta,\beta_0}}$ uniformly bounded in  probability.
Hence, $A_1(\beta_0,\beta)$, as a Riemann sum, converges to $ \frac{1}{T} \int_{0}^{T}\log\left(\det\left[\Sigma(\beta,t, z(\alpha_0,t))\Sigma^{-1}(\beta_0,t,z(\alpha_0, t))\right]\right)dt\;$ as $\epsilon,\Delta \rightarrow 0$.\\
Applying Lemma \ref{BXa0} to $B_k(\alpha_0,X)$  and the notations therein yields
\begin{equation}\label{A2b}
 A_2(\theta_0,\theta) = \frac{\Delta }{T} \sum_{k=1}^{n} Z_k^{*} M_k Z_k+\sum_{i=1}^4\Lambda^i (\theta_0,\theta),
 \end{equation}
 with
\[Z_k=\frac{1}{\sqrt{\Delta}}\left(B(t_k)-B(t_{k-1})\right),\,   T_k= \Sigma_{k-1}^{-1}(\beta) -\Sigma_{k-1}^{-1}(\beta_0),\,
 M_k= \sigma_{k-1}^{*}(\beta_0)T_k\; \sigma_{k-1}(\beta_0),\]
 and
\begin{align*}
\Lambda_1(\alpha, \theta_0) &= \frac{2  \sqrt{\Delta}}{\epsilon} \sum_{k=1}^{n} E_k^{*} T_k\; Z_k,\\
\Lambda_2(\alpha, \theta_0) &= \frac{1}{T\epsilon^2} \sum_{k=1}^{n} E_k^{*} E_k,\\
\Lambda_3(\alpha, \theta_0) &= \frac{2}{T\epsilon^2}  \sum_{k=1}^{n} (B_k^*(\alpha,X) -B_k^*(\alpha_0,X))T_k \; B_k(\alpha_0,X),\mbox{ and} \\
\Lambda_4(\alpha, \theta_0) &= \frac{1}{T\epsilon^2} \sum_{k=1}^{n} (B_k^*(\alpha,X) -B_k^*(\alpha_0,X)) T_k\;  (B_k(\alpha,X)- B_k(\alpha_0,X)).
\end{align*}
 The random vectors $Z_k$ are $\mathcal{N}\left(0,I_p\right)$ independent of $\mathcal{ G}_{k-1}^n$ and $M_k$ is $\mathcal{ G}_{k-1}^n$-measurable. Using that for $Z \sim \mathcal{N}\left(0,I_p\right)$,
$E(Z^{*}M Z)=\mathrm{Tr}(M)$, we get
$$\E_{\theta_0}^{\epsilon}( Z_k^{*} \;M_k\;  Z_k| \mathcal{ G}_{k-1}^n)= \mathrm{Tr}(M_k)=\mathrm{Tr}\left(\Sigma_{k-1}^{-1}(\beta)\Sigma_{k-1}(\beta_0)-I_p\right).$$
Hence, the first term of $A_2(\alpha_0,\beta_0,\beta)$ converges to
\[\frac{1}{T} \int_{0}^{T}\mathrm{Tr}\left(\Sigma^{-1}(\beta,t,z(\alpha_0, t))\Sigma(\beta_0,t,z(\alpha_0, t))\right)dt -p.\]

It remains to study the other terms of $A_2(\alpha_0,\beta_0,\beta)$.
% $\displaystyle{\frac{1}{T\epsilon^2} \Lambda_i(\alpha,\theta_0)}$.
 %Set  $T_k(\beta_0, \beta,X)= \left(\Sigma^{-1}(\beta,t_{k-1},X(t_{k-1}))-\Sigma^{-1}(\beta_0,t_{k-1},X(t_{k-1}))\right)$. Under $\P_{\theta}$, it is  $\mathcal{ G}_{k-1}^n$-measurable and bounded in probability and,
 To study $\Lambda_1$, let  $\zeta_{k}^{n}= \frac{\sqrt{\Delta}}{\epsilon} E_k^{*} T_k \;Z_k$.

We have, by Lemma   \ref{BXa0} that, in $\P_{\theta_0} $-probability,
\begin{align*}
\E(\zeta_{k}^{n}|\mathcal{ G}_{k-1}^n)&\leq  \frac{\sqrt{\Delta}}{\epsilon} \sup\norm{T_k }(\E(\norm{E_k}^2|\mathcal{ G}_{k-1}^n))^{1/2}\leq   C \Delta^{3/2},\; \mbox{ and}\\
\E((\zeta_{k}^{n})^2 |\mathcal{ G}_{k-1}^n) &\leq   \frac{\Delta}{\epsilon^2} \sup \norm{T_k }^2 \Delta^2 \epsilon^2\leq C \Delta ^3.
\end{align*}
 Therefore, by Lemma \ref{VGCJJ},  $\Lambda_1(\alpha, \theta_0) \rightarrow 0 $ in $\P_{\theta_0}$-probability as $\epsilon,\Delta \rightarrow 0$.
 Similar arguments yield that  $\Lambda_2(\alpha, \theta_0) \rightarrow 0 $ in $\P_{\theta_0}$-probability.\\
For $\Lambda_3(\alpha, \theta_0)$,  set $\displaystyle{\zeta_{k}^{n}= \frac{1}{\epsilon^2} (B_k^*(\alpha,X) -B_k^*(\alpha_0,X))T_k \; B_k(\alpha_0,X)}$
Using Lemma \ref{BXa} yields that
\begin{align*}
E(\zeta_{k}^{n}|\mathcal{ G}_{k-1}^n) &\leq  \frac{\norm{\alpha-\alpha_0)}}{\epsilon} \Delta^2  O_P(1),\;\mbox{ and}\\
\E((\zeta_k^n)^2|\mathcal{ G}_{k-1}^n) &\leq  \frac{\norm{\alpha-\alpha_0}^2}{\epsilon^2} \Delta^3 O_P(1),
\end{align*}
so that
%$\displaystyle \E((\zeta_k^n)^2|\mathcal{ G}_{k-1}^n) \leq $
$\displaystyle{\sum \E(\zeta_k^n|\mathcal{ G}_{k-1}^n) \leq  \Delta \frac{\norm{\check{\alpha}_{\epsilon,\Delta}-\alpha_0)}}{\epsilon}}$.
% \frac{1}{\epsilon^2} (B_k^*(\alpha,X) -B_k^*(\alpha_0,X))T_k(\beta_0,\beta,X)
%E(E_k)| \mathcal{ G}_{k-1}^n) \leq \frac{\Delta}{\epsilon^2} \norm{\alpha-\alpha_0)} \epsilon \Delta= \Delta ^2 \frac{\norm{\alpha-\alpha_0)}}{\epsilon}$.
By Proposition \ref{borneproba}, the sequence  $ (\epsilon^{-1}\norm{\check{\alpha}_{\epsilon,\Delta}-\alpha_0})$ is uniformly bounded in probability, so that
$\displaystyle{\sum \E(\zeta_k^n|\mathcal{ G}_{k-1}^n)  \rightarrow 0}$ and $\displaystyle{\sum \E((\zeta_k^n)^2|\mathcal{ G}_{k-1}^n) \rightarrow 0}$.\\

 Another application of Lemma \ref{VGCJJ} yields that $ \displaystyle{\Lambda_3(\check{\alpha}_{\epsilon,\Delta}, \theta_0)  \rightarrow 0}$. For $\Lambda_4$, the result is straightforward  since
 $\displaystyle{|\Lambda_4| \leq n\Delta^2 (\frac{\norm{\check{\alpha}_{\epsilon,\Delta}-\alpha_0}}{\epsilon}) ^2}$. This achieves the proof of (i).\\
Let us study (ii). We have,  using \eqref{def:K2},
\begin{eqnarray*}
0&\leq&K_2(\alpha_0,\beta_0,\check{\beta}_{\epsilon,\Delta}) \leq [K_2 (\alpha_0,\beta_0,\check{\beta}_{\epsilon,\Delta})
-\frac{1}{n}(\check{U}_{\Delta,\epsilon}(\check{\alpha}_{\epsilon,\Delta} ,\check{\beta}_{\epsilon,\Delta}) -\check{U}_{\Delta,\epsilon}(\check{\alpha}_{\epsilon,\Delta},\beta_0))] \\
&+& \frac{1}{n}(\check{U}_{\Delta,\epsilon}(\check{\alpha}_{\epsilon,\Delta} ,\check{\beta}_{\epsilon,\Delta}) -\check{U}_{\Delta,\epsilon}(\check{\alpha}_{\epsilon,\Delta},\beta_0)).
\end{eqnarray*}
Noting that the last term of the above  inequality is non-negative, (i) yields that $K_2(\alpha_0,\beta_0,\check{\beta}_{\epsilon,\Delta})\rightarrow 0$,
which ensures, by Assumption (S5), that $\check{\beta}_{\epsilon,\Delta} \rightarrow \beta_0$ in $\P_{\theta_0}$-probability.
\end{proof}

\subsubsection{Step (4): Asymptotic normality}
 Let us now study the asymptotic properties of these estimators and achieve the proof of Theorem \ref{theo:normalite}.
 % Define, using definition \eqref{Ibtheta}, the two matrices, $I(\alpha,\beta) $ and
Let us define for $\theta=(\alpha,\beta)$,
\begin{equation} \label{Lepsn}
\Lambda_{\epsilon,n}(\theta)= \begin{pmatrix}\epsilon \nabla_{\alpha} \check{U}_{\epsilon,\Delta}(\alpha,\beta)\\
\frac{1}{\sqrt{n}} \nabla_{\beta} \check{U}_{\epsilon,\Delta}(\alpha,\beta) \end{pmatrix} \quad \mbox{and}
\end{equation}
\begin{equation}\label{Depsn}
D_{\epsilon,n}(\theta)= \begin{pmatrix}\epsilon^2 \left(\nabla^2_{\alpha_i,\alpha_j} \check{U}_{\epsilon,\Delta}(\theta)\right)_{1\leq i,j\leq a} &
\frac{\epsilon}{\sqrt{n}}\left( \nabla^2_{\alpha_i\beta_j}\check{U}_{\epsilon,\Delta} (\theta)\right)_{1\leq i\leq a,1\leq j\leq b}\\
\frac{\epsilon}{\sqrt{n}}\left( \nabla^2_{\alpha_i\beta_j}\check{U}_{\epsilon,\Delta} (\theta)\right)_{1\leq i\leq a,1\leq j\leq b}&
\frac{1}{n}\left( \nabla^2_{\beta_i\beta_j}\check{U}_{\epsilon,\Delta} (\theta)\right)_{1\leq i,j\leq b}
\end{pmatrix}.
\end{equation}
A Taylor expansion at point $\theta_0$ yields,
% setting $ \theta_t= \theta_0+t(\check{\theta}_{\epsilon,\Delta}- \theta_0)$,
\begin{align}
\begin{pmatrix} 0\\0
\end{pmatrix} &= \Lambda_{\epsilon,n} (\check{\alpha}_{\epsilon,\Delta}, \check{\beta}_{\epsilon,\Delta})\label{dUadUb}\\
&=    \Lambda_{\epsilon,n} (\theta_0)+ \int_0^1 D_{\epsilon,n}(\theta_0+t(\check{\theta}_{\epsilon,\Delta}- \theta_0)) dt \; \begin{pmatrix}
\epsilon^{-1}(\check{\alpha}_{\epsilon,\Delta}-\alpha_0) \\
\sqrt{n}(\check{\beta}_{\epsilon,\Delta}-\beta_0)
\end{pmatrix}.\nonumber
\end{align}
Therefore, we have to prove that,  under $\P_{\theta_0}$, as $\epsilon,\Delta \rightarrow 0$ (or  $n=  \Delta^{-1/2} \rightarrow \infty $),
\begin{enumerate}
\item[(i)] $-  \Lambda_{\epsilon,n} (\theta_0) \rightarrow  \mathcal{ N}( 0, 4 I(\theta_0))$ in distribution,
\item[(ii)] $ \sup_{t \in[0,1]} \norm{D_{\epsilon,n}(\theta_0+t(\check{\theta}_{\epsilon,\Delta}- \theta_0))- 2 I(\theta_0)} \rightarrow 0 $ in probability.
\end{enumerate}

\begin{proof}
Let us prove (i).
We have that, for $1\leq i \leq a$,
\begin{equation}\label{dUa}
 - \epsilon \nabla_{\alpha_i} \check{U}_{\epsilon,\Delta}(\alpha_0,\beta_0)= \sum_{k=1}^n \xi_k^i(\theta_0)\; \mbox{ with }\;
\xi_k^i(\theta_0)= -\frac{2}{\epsilon \Delta}B_k^{*}(\alpha) \Sigma ^{-1}_{k-1}(\beta_0) \nabla_{\alpha_i} B_k(\alpha_0).
 \end{equation}
Using that, for a positive symmetric matrix $M(x)$,
\[\frac{d}{dx} (\log \det M(x))= \mathrm{Tr}\left(M^{-1}(x) \frac{d}{dx} M(x)\right)\]
 and \eqref{Sigkbeta},  set
 \begin{equation}\label{Mkj}
 M_k^j(\beta)=\Sigma_k^{-1}(\beta) \nabla_{\beta_j}\Sigma_k (\beta).
\end{equation}
Then  $\frac{1}{\sqrt{n}}\nabla_{\beta_{j} }\check{U}_{\epsilon,\Delta}(\alpha_0,\beta_0)= \sum_{k=1}^n \eta_k^i(\theta_0)$ with
\begin{equation}\label{ekj}
 \eta_k^j(\theta_0)= \frac{1}{\sqrt{n}}[\mathrm{Tr}(M_{k-1}^j(\beta_0)) -\frac{1}{\epsilon^2\Delta} B_k^{*}(\alpha_0) M_{k-1}^j(\beta_0) \Sigma ^{-1}_{k-1}(\beta_0) B_k(\alpha_0,X)].
 \end{equation}

The proof that  $\displaystyle{-\epsilon\nabla_{\alpha} \check{U}_{\epsilon,\Delta}(\alpha_0,\beta_0)} $  converges to
the Gaussian distribution  $\mathcal{ N} (0,I_b(\theta_0))$  is obtained by substituting  $ \beta$ with $\beta_0$ in the proof of Proposition \ref{borneproba}.

Let us study
$\displaystyle{-\frac{1}{\sqrt{n}} \nabla_{\beta}\check{U}_{\epsilon,\Delta}(\alpha_0,\beta_0)}$. Let us first prove
\begin{lemma}\label{traceM}
If $M$ is a $\mathcal{ G}_{k-1}^n$-measurable random matrix, then
\begin{equation}\label{BkBk}
\frac{1}{\epsilon^2 \Delta}\E\left(B_k^*(\alpha_0)M \Sigma ^{-1}_{k-1}(\beta_0)  B_k(\alpha_0,X)|\mathcal{ G}_{k-1}^n\right)= \mathrm{Tr} (M)+\Delta R_k(\epsilon,\Delta)
 \end{equation}
 with $\sup_k |R_k(\epsilon,\Delta)| $ uniformly bounded in $\P_{\theta_0} $-probability.
 \end{lemma}
\begin{proof}%Let us prove \eqref{BkBk}.
Using Lemma \ref{BXa0},
\begin{align*}
&\E(B_k^*(\alpha_0)M \Sigma ^{-1}_{k-1}(\beta_0) B_k(\alpha_0)|\mathcal{ G}_{k-1}^n)
= \sum_{l,l' =1}^p \left(M \Sigma ^{-1}_{k-1}(\beta_0)\right)_{ll'} \E(B_k^{l}(\alpha_0) B_k^{l'}(\alpha_0)|\mathcal{ G}_{k-1}^n)\\
&=\epsilon^2 \Delta \sum_{l,l' =1}^p \left(M \Sigma ^{-1}_{k-1}(\beta_0)\right)_{ll'} (\Sigma_{k-1}(\beta_0))_{l'l} +
\sum_{l,l' =1}^p \left(M \Sigma ^{-1}_{k-1}(\beta_0)\right)_{ll'} \E(E_k^l E_k^{l'} |\mathcal{ G}_{k-1}^n)\\
&= \epsilon^2 \Delta  \mathrm{Tr}(M) +R_k(\epsilon,\Delta)
\end{align*}
with $|R_k(\epsilon,\Delta)| \leq C \epsilon^2 \Delta^2 $ in probability.
\end{proof}

Let us  study the convergence of the triangular array $\sum_{k=1}^n \E( \xi_k^i(\theta_0))$. By  Lemma \ref{traceM}, we have
for  $j \leq b$,
\[\sum_{k=1}^n \E( \eta_k^j(\theta_0)|\mathcal{ G}_{k-1}^n) = \frac{1}{\epsilon^2\Delta \sqrt{n}} \sum _{k=1}^n R_k(\epsilon,\Delta) \leq\frac{ CT}{\sqrt{n}} \rightarrow 0.\]
Consider now, for $ j_1,j_2 \leq b$, $ \sum_{k=1}^n  \E( \eta_k^{j_1}(\theta_0) \eta_k^{j_2}(\theta_0)|\mathcal{ G}_{k-1}^n)$.

We have
\begin{align*}
&\E( \eta_k^{j_1}(\theta_0) \eta_k^{j_2}(\theta_0)|\mathcal{ G}_{k-1}^n)\\
&= \frac{1}{n} [\mathrm{Tr} (M_{k-1}^{j_1}(\beta_0) M_{k-1}^{j_2} (\beta_0))- 2 \mathrm{Tr} M_{k-1}^{j_1}(\beta_0)) \mathrm{Tr} M_{k-1}^{j_2}(\beta_0)+ C_{k}^{j_1,j_2}(\epsilon,\Delta)+ \Delta O_P(1)],
\end{align*}
with
\smallskip

\resizebox{\linewidth}{!}{
\begin{minipage}{\linewidth}
\begin{align*}
&C_k^{j_1j_2}(\epsilon,\Delta)\\
&= \frac{1}{\epsilon^4 \Delta^2}\E\left(B_k^*(\alpha_0)M_{k-1}^{j_1}(\beta_0)\Sigma ^{-1}_{k-1} (\beta_0) B_k(\alpha_0)
B_k^*(\alpha_0)M_{k-1}^{j_2}(\beta_0) \Sigma ^{-1}_{k-1}(\beta_0)  B_k(\alpha_0)|\mathcal{ G}_{k-1}^n\right).
\end{align*}
\end{minipage}
}
\smallskip

Therefore, omitting the parameters when there is no ambiguity,
\begin{align*}
&C_k^{j_1j_2}(\epsilon,\Delta)\\
&=\sum_{l_1,l_2,l_3,l_4}(M_{k-1}^{j_1}\Sigma ^{-1}_{k-1})_{l_1l_2}(M_{k-1}^{j_2} \Sigma ^{-1}_{k-1})_{l_3l_4}\E\left(B_k^{l_1}(\alpha_0)B_k^{l_2}(\alpha_0)B_{k}^{l_3}(\alpha_0)B_{k}^{l_4}(\alpha_0) |\mathcal{ G}_{k-1}^n\right).
\end{align*}

Based on the  property that, if $Z$ is a $p$-dimensional Gaussian random variable $\mathcal{ N}(0,\Sigma)$,
$\displaystyle{E(Z_{l_1}Z_{l_2}Z_{l_3} Z_{l_4})= \Sigma_{l_1l_2}\Sigma_{l_3l_4}+ \Sigma_{l_1l_3}\Sigma_{l_2l_4}+\Sigma_{l_1l_4}\Sigma_{l_2l_3}}$,
we get  that
\[C_k^{j_1j_2}(\epsilon,\Delta)= \left( \mathrm{Tr}(M^{j_1}_{k-1} M^{j_2}_{k-1})+ 2 \mathrm{Tr} M^{j_1}_{k-1} \mathrm{Tr} M^{j_2}_{k-1}+  \Delta O_P(1)\right).\]
Therefore  $\sum_{k=1}^n \E(\eta_k^{j_1}(\theta_0) \eta_k^{j_2}(\theta_0)|\mathcal{ G}_{k-1}^n)=  \frac{2 }{n} \sum_{k=1}^n \mathrm{Tr}(M^{j_1}_{k-1} M^{j_2}_{k-1})
+\Delta O_P(1)$.\\
Now, under $\P_{\theta_0}$,
$M_k^j( \beta_0) = \Sigma^{-1}( \beta_0, t_k, z(\alpha_0,t_k) )\nabla_{\beta_j} \Sigma( \beta_0, t_k, z(\alpha_0,t_k))+ \epsilon O_P(1)$ so that,
using \eqref{Isigma}, as $\epsilon,\Delta \rightarrow 0$,
$$\ \sum_{k=1}^n  \E( \eta_k^{j_1}(\theta_0) \eta_k^{j_2}(\theta_0)|\mathcal{ G}_{k-1}^n) \rightarrow 4 (I_{\sigma}(\theta_0))_{j_1 j_2}.$$
The proofs that  $\sum_{k=1}^n \E( \norm{\eta_k^i(\theta_0)}^4|\mathcal{ G}_{k-1}^n) \rightarrow 0$,
$\sum_{k=1}^n \E( \xi_k^i(\theta_0) \eta_k^j(\theta_0) |\mathcal{ G}_{k-1}^n) \rightarrow 0$ are similar and omitted.
%$\sum_{k=1}^n \E( \norm{\eta_k^i(\theta_0)}^4|\mathcal{ G}_{k-1}^n) \rightarrow 0$,\\
Finally, applying the theorem of convergence in law for triangular arrays recalled in Section \ref{LimitTheo} yields that
  $\displaystyle{\sum_{k=1}^n  \eta_k^i(\theta_0) \rightarrow \mathcal{ N} (0,4 I_{\sigma}(\theta_0))}$.
  Joining these results achieves the proof of  (i).
  \end{proof}

It remains to study $D_{\epsilon,n}(\theta)$ defined in \eqref{Depsn}.
\begin{proof}
We have already proved  that
\[\sup_{t\in[0,1]} \norm{\epsilon^2 (\nabla^2 _{\alpha_i,\alpha_j}  \check{U}_{\epsilon,\Delta}(\theta_0+t (\check{\theta}_{\epsilon,\Delta}- \theta_0))- 2(I_b(\theta_0))_{ij}} \rightarrow 0\]
in probability.
Consider now the term $ \frac{1}{n}\nabla^2_{\beta_i,\beta_j} \check{ U}_{\epsilon,\Delta}(\alpha,\beta)$. It reads as
%\pagebreak
\begin{align*}
&\nabla^2_{\beta_i\beta_j} \check{ U}_{\epsilon,\Delta}(\alpha,\beta)\\
 &\quad = \sum_{k=1}^n  \left(\mathrm{Tr}\left(\nabla_{\beta_i} M_{k-1}^j(\beta) \right)
-\frac{1}{\epsilon^2 \Delta}  B_k(\alpha)^*  ( \nabla_{\beta_i} M_{k-1}^j(\beta) )\Sigma_{k-1}^{-1} (\beta) B_k(\alpha)\right)\\
&\quad\quad +  \frac{1}{\epsilon^2 \Delta} \sum_{k=1}^n B_k(\alpha)^*  M_{k-1}^j(\beta) M_{k-1}^i(\beta)\Sigma_{k-1}^{-1}  (\beta)B_k(\alpha).
\end{align*}
Let us define the matrices, for $1\leq i,j\leq b$,
\begin{equation} \label{Mithetat}
M^i (\alpha,\beta,t)= \Sigma ^{-1}(\beta, t,z(\alpha,t))\nabla_{\beta_i}\Sigma (\beta, t,z(\alpha,t)),  \quad \mbox{and}
\end {equation}
\begin{equation}\label{Tijbeta}
T^{ij}_k(\beta)=[ M_{k}^j(\beta) M_{k}^i(\beta)-  \nabla_{\beta_i} M_{k}^j(\beta) ] \Sigma_{k-1}^{-1} (\beta).
\end{equation}
Using (\ref{BkBk})  yields that  the first term of  $\nabla^2_{\beta_i\beta_j} \check{ U}_{\epsilon,\Delta}(\alpha_0,\beta_0)$ is uniformly bounded in probability and that
the second term satisfies  $\sum_{k=1}^n (\mathrm{Tr}\left(M_{k-1}^j(\beta_0) M_{k-1}^i(\beta_0)\right) + \Delta  O_P(1))$. Hence,
%$\frac{1}{n}\nabla^2_{\beta_i\beta_j} \check{ U}_{\epsilon,\Delta}(\alpha_0,\beta_0)=
%frac{1}{n}\sum_{k=1}^n \mathrm{Tr}\left(M_{k-1}^j(\beta_0) M_{k-1}^i(\beta_0)\right) + \Delta  O_P(1).$\\
$$\frac{1}{n}\nabla^2_{\beta_i\beta_j} \check{ U}_{\epsilon,\Delta}(\alpha_0,\beta_0) \rightarrow  -\frac{1}{T}\int_0^T \mathrm{Tr}(M^j(\alpha_0,\beta_0,t)M^i(\alpha_0,\beta_0,t))dt.$$

It remains to prove that, under  $\P_{\theta_0}$,
%\begin{equation} \label{U2thetatheta0}
\[\sup_{t\in[0,1]} \frac{1}{n}\norm{\nabla^2_{\beta_i \beta_j} \check{ U}_{\epsilon,\Delta}(\theta_t)- \nabla^2_{\beta_i\beta_j} \check{ U}_{\epsilon,\Delta}(\theta_0)} \rightarrow 0\]  with  $\theta_t = \theta_0 + t( \check{\theta}_{\epsilon,\Delta} -\theta_0)$
 and that the terms \[\frac{\epsilon}{\sqrt{n}}(\nabla^2_{\alpha_i\beta j} \check{ U}_{\epsilon,\Delta}(\alpha,\beta)-
   \nabla^2_{\alpha_i\beta_j} \check{ U}_{\epsilon,\Delta}(\alpha_0,\beta_0))\rightarrow 0.\]
These two proofs rely on similar tools and are omitted.
\end{proof}

\section{Inference based  on low frequency observations} \label{LFO}
Consider now the case where the sampling interval $\Delta$ is fixed  and the time interval for observations is fixed.
It follows that  the number of observation points
$n=T/\Delta$ is finite.  %It corresponds to the term "Low frequency observations".
We prove  that only parameters in the drift  function can be consistently estimated. This agrees with the previous results where the rate of estimation of parameter $\beta$ in the diffusion coefficient is $\sqrt{n}$ in the high frequency set-up. Sometimes,  when modeling  epidemic dynamics, a parameter is
added in the $SIR$ model  to take account of larger fluctuations, substituting  the term $\sqrt{SI}$  by $(S(t)I(t))^{a}$  in the diffusion term. While in the ``High frequency''  set-up, this parameter $a$  can be consistently estimated, this is no longer true for a fixed sampling interval.

In order to illustrate that $\beta$ cannot be consistently estimated in this set-up,  we study the inference on a simple example, the one-dimensional Brownian motion with drift on $[0,T]$.
\subsection{Preliminary result on a simple example }
Let us consider the estimation of $(\alpha,\beta)$ as $\epsilon\rightarrow 0$ and $n=T/\Delta$ finite,  for the process
 %This fact can be easily illustrated considering the one-dimensional Brownian motion with drift on $[0,T]$.
 \begin{equation}\label{BMD}
 d X(t)= \alpha dt + \epsilon \beta dB(t) ; \quad X(0)= 0.
 \end{equation}
The observations are $(X(t_k), k=1,\dots,n)$.
 The $n$ random variables $(X(t_k)-X(t_{k-1}))$ are independent Gaussian with distribution  $\mathcal{ N}( \alpha \Delta, \epsilon ^2 \beta ^2 \Delta)$.
The likelihood is explicit and the maximum likelihood estimators are
\begin{equation}\label{MLEBMD}
\hat{\alpha}_{\epsilon}= \frac{X(T)}{T};\quad  \hat{\beta}^2_{\epsilon}= \frac{1}{n\Delta \epsilon^2} \sum_{k=1}^n ( X(t_k)-X(t_{k-1})- \Delta \hat{\alpha}_{\epsilon,\Delta})^2.
\end{equation}
Under $\P_{\theta_0}$, $\displaystyle{\hat{\alpha}_{\epsilon}= \alpha_0+\epsilon \beta_0 \frac{B(T)} {T}}$. Therefore, as $\epsilon \rightarrow 0$,
$\displaystyle{\hat{\alpha}_{\epsilon} \rightarrow \alpha_0}$  and $\epsilon^{-1} (\hat{\alpha}_{\epsilon}-\alpha_0)=  \beta_0 \frac{B(T)}{T}$  is  a Gaussian random variable $\displaystyle{\mathcal{ N} (0, \frac{ \beta_0^2 }{T})}$.

The MLE  of $\beta_0^2$
is $\displaystyle{\hat{\beta}^2_{\epsilon}=\beta_0^2 (\frac{1}{n} \sum_{k=1}^n Z_k^2- \frac{1}{n} \frac{B(T)^2}{T})}$, where $(Z_k, k=1,\dots,n)$ are i.i.d.\ $\mathcal{ N}(0,1)$.\\ Hence,  since $n$ is  fixed,
$\hat{\beta}^2_{\epsilon}$ is a fixed random variable which does not depend on $\epsilon$ with expectation $\displaystyle{\beta_0^2 (1-\frac{1}{n} )\neq  \beta_0^2}$, implying that it is  a biased estimator of $\beta_0^2 $.

This simple example  shows that parameters in the diffusion coefficient cannot be estimated as $\epsilon \rightarrow 0$.

\subsection{Inference for diffusion approximations of epidemics}
Considering equation \eqref{SDEGen}, three cases might occur: $\beta $ unknown; $\beta$ known
or $\Sigma(\beta,x)= \phi(\beta) \Sigma(x)$ (with $\phi(\beta)$ a known real function on $\R^+$);  $\beta$ present in the drift coefficient (e.g.\ $\beta= \varphi(\alpha)$ with $\varphi$ a known function).
 This last case systematically occurs for the diffusion approximation of epidemic dynamics:  the parameters ruling the jump process modeling the epidemic dynamics are both present in the drift and in the diffusion coefficients, i.e.\  $ \beta \equiv \alpha$.
 For example, the diffusion approximation of the $SIR$, we have, setting $\alpha = (\lambda,\gamma)$, that  the drift term is $ b(\alpha,z)$ and the diffusion term is
 $\Sigma(\alpha,z)$

%Considering equation \eqref{SDEGen}, three cases might occur: $\beta $ unknown; $\beta$ known
%or $\Sigma(\beta,x)= \phi(\beta) \Sigma(x)$ (with $\phi(\beta)$ a known real function on $\R^+$);  $\beta$ present in the drift coefficient (e.g.\ %$\beta= \varphi(\alpha)$ with $\varphi$ a known function).
% This last case systematically occurs for the diffusion approximation of epidemic dynamics:  the parameters ruling the jump process modeling the epidemic dynamics are both present in the drift and in the diffusion coefficients, i.e.\  $ \beta \equiv \alpha$.
%
Having in mind epidemics, we study here this  case and  assume that, under $\P_{\alpha}$,
\begin{equation} \label{SDEalpha}
dX(t)= b(\alpha,t,X(t)) dt + \epsilon \sigma(\alpha,t,X(t)) dB(t), \quad X(0) =x.
\end{equation}
 %Let us introduce the assumptions for this statistical model associated to $\Delta$ fixed.
 The time interval is $[0,T]$, the sampling interval is $\Delta$ with $T=n\Delta$, and both $T, \Delta, n$ are fixed.

 The observations consist of the $n$ random variables
 $(X(t_k), k=1,\dots,n)$ with $t_k=k\Delta$.
 As in the previous section,  the inference is based  on the  random variables $B_k(\alpha,X)$ defined in \eqref{BkTS}, which satisfy using Lemma \ref{BkX}
 %expansion \eqref{Ts param} where
 \begin{equation}\label{BkD}
 B_k(\alpha,X)= \epsilon\sqrt{\Delta} T_k(\alpha) + \epsilon^2 D_k^{\epsilon}(\alpha),\mbox{ with }
 D_k^{\epsilon}= R^{\epsilon}(\alpha,t_k)- \Phi(\alpha, t_k,t_{k-1}) R^{\epsilon}(\alpha,t_{k-1}).
 \end{equation}
 \begin{equation} \label{Tka}
  T_k(\alpha)=\frac{1}{\sqrt{\Delta}} \int_{t_{k-1}}^{t_k}  \Phi(\alpha,t_k,u) \sigma(\alpha,u,z(\alpha,u)) dB(u),
  \end{equation}
  \begin{equation}\label{Ska}
  S_k(\alpha)= \frac{1}{\Delta}  \int_{t_{k-1}}^{t_k}  \Phi(\alpha,t_k,u) \Sigma(\alpha,u,z(\alpha,u))\Phi^*(\alpha,t_k,u)du.
  \end{equation}

This leads   to define the contrast function depending now on $ (X(t_1),\dots,X(t_n))$,
%using \eqref{Tka}, \eqref{Ska}
\begin{equation}\label{Ueps}
  \bar{U}_{\epsilon}\left(\alpha,(X_{t_k})\right) =  \bar{U}_{\epsilon}(\alpha)= \sum_{k=1}^n \log \det S_k(\alpha) + \frac{1}{\epsilon^2 \Delta}\sum_{k=1}^n  B_k^*(\alpha,X)S_k^{-1}(\alpha) B_k(\alpha,X).
 \end{equation}
Denote by $\alpha_0$ the true value of the parameter and $\Theta$ the parameter set.
 We assume
 \begin{enumerate}[(S6b):]
 \item[\textbf{ (S4b}):]  $\Theta$ a compact set of $\R^a$ ; $\alpha \in \mathrm{Int} (\Theta)$.
 \item[\textbf{ (S5b}):] Assumption (S5) on  $b(\alpha,t,z)$ and $\sigma(\alpha,t,z)$.
 \item[\textbf{ (S6b}):] $\alpha \neq \alpha_0 \Rightarrow  \{ \exists  k,\;1 \leq k \leq n, \;   z(\alpha, t_k) \neq z(\alpha_0,t_k)\}$.
 \end{enumerate}
  %The random variables $B_k(\alpha,X)$ defined in \eqref{BkTS} satisfy Lemma \ref{BkX}, where, using (\ref{TSparamxg})
 %This leads   to define the contrast function depending now on $ (X(t_1),\dots,X(t_n))$,
%%using \eqref{Tka}, \eqref{Ska}
%\begin{equation}\label{Ueps}
%  \bar{U}_{\epsilon}\left(\alpha,(X_{t_k})\right) =  \bar{U}_{\epsilon}(\alpha)= \sum_{k=1}^n \log \det S_k(\alpha) + \frac{1}{\epsilon^2 \Delta}\sum_{k=1}^n  B_k^*(\alpha,X)S_k^{-1}(\alpha) B_k(\alpha,X).
% \end{equation}
 The estimator is defined as any solution of
\begin{equation}\label{baralpha}
\bar{\alpha}_{\epsilon}=\underset{\alpha\in K_a }{\mathrm{argmin}}\; \bar{U}_{\epsilon}\left(\alpha,(X_{t_k})\right).
\end{equation}
Let us  study the  properties of  $\bar{\alpha}_{\epsilon} $. For this, define, using \eqref{Ska}, the  $p\times a $ matrix $G_k(\alpha)= (G_k^1,\dots,G_k^a)$ and the $a\times a$ matrix $M(\alpha)$,
\begin{equation}\label{Ma}
M(\alpha)= \Delta \sum_{k=1}^n G_k(\alpha)^* S_k(\alpha)^{-1} G_k(\alpha),\mbox{ with}
\end{equation}
\begin{equation}\label{Gk}
G_k^i (\alpha) = \frac{1}{\Delta} ( - \nabla_{\alpha_i } z(\alpha,t_k) + \Phi(\alpha,t_k,t_{k-1})  \nabla_{\alpha_i} z(\alpha,t_{k-1})).
\end{equation}

Then, the following holds
\begin{theorem}  \label{LimLFO}
Assume (S1)--(S3), (S4b)--(S6b).  Then, as $\epsilon\rightarrow 0$, under $\P_{\alpha_0}$,
\begin{enumerate}
\item[(i)] $ \bar{\alpha}_{\epsilon} \rightarrow \alpha_0$ in probability.
\item[(ii)] If moreover $M(\alpha_0 )$ defined in \eqref{Ma} is non-singular, then
\[\epsilon^{-1}(\bar{\alpha}_{\epsilon} -\alpha_0) \rightarrow \mathcal{ N}_a(0,M^{-1}(\alpha_0))\] in distribution.
\end{enumerate}
%under $\mathbb{P}_{\alpha_0}^{\epsilon}$ .
\end{theorem}

\begin{proof}
Let us first prove (i). Define, using \eqref{BkD}, \eqref{Ska},
\begin{equation}\label{Kdelta}
\bar{K}_{\Delta}(\alpha_0,\alpha) = \frac{1}{\Delta}\sum_{k=1}^{n}B_k^*(\alpha, z(\alpha_0,\cdot))S_k^{-1}(\alpha)  B_k(\alpha, z(\alpha_0,\cdot)).
\end{equation}
Since $B_k(\alpha_0, z(\alpha_0,\cdot)) = 0$, $\bar{K}_\Delta(\alpha_0,\alpha)\geq 0$ and $\bar{K}_\Delta(\alpha_0,\alpha_0)=0$. Assume now that
  $\bar{K}_\Delta(\alpha_0,\alpha)=0 $. Then, for   all $ k  \in \{1,..n\}$,
\[z(\alpha,t_k)- z(\alpha_0,t_k)= \Phi( \alpha, t_k, t_{k-1})(z(\alpha, t_{k-1})- z(\alpha_0, t_{k-1})).\]
The matrix $\Phi(\alpha, t_k, t_{k-1})$ being non-singular,  the identifiability Assumption \textbf{ (S6b)} implies that $\alpha=\alpha_0$. \\
Since the  sum in \eqref{Ueps} is finite, we get,
using (\ref{BkTS}) and Proposition \ref{TStheta}, that
%\begin{lemma}
 %Assume {(S1)}-{ (S3)} and {(S4b)}-{(S6b)}. Then, under $\mathbb{P}_{\alpha_0}^{\epsilon}$, as $\epsilon \rightarrow 0$\\
 $ \underset{\alpha \in K_a }{ \sup}  |\epsilon^2 \bar{U}_{\epsilon}(\alpha) - \bar{K}_{\Delta}(\alpha_0,\alpha)| \rightarrow 0 $ in $\P_{\theta_0}$-probability as $\epsilon \rightarrow 0$.
%(ii)  $\bar{\alpha}_{\epsilon} \rightarrow \alpha_0$ in probability.
%\end{lemma}
% Since the  summation in \eqref{Ueps} is finite, the proof of (i) is immediate using (\ref{BkTS}) and Proposition \ref{TStheta}.
 Therefore, we have
 \begin{eqnarray*}
 0 &\leq& \bar{K}_{\Delta}(\alpha_0,\bar{\alpha}_{\epsilon})- \bar{K}_{\Delta}(\alpha_0,\alpha_0)\\
 & \leq&
  2 \underset{\alpha\in K_a }{ \sup}  |\epsilon^2 U_{\epsilon}(\alpha) -\bar{K}_{\Delta}(\alpha_0,\alpha)| +  \epsilon^2 |U_{\epsilon}(\bar{\alpha})- U_{\epsilon}(\alpha_0)|\\
  &\leq&   2 \underset{\alpha\in K_a }{ \sup}  |\epsilon^2 U_{\epsilon}(\alpha) -\bar{K}_{\Delta}(\alpha_0,\alpha)| .
 \end{eqnarray*}
Then  the proof of (i) is achieved  by means of the identifiability Assumption \textbf{ (S6b)}.

Let us now prove (ii).
To study the asymptotic properties of  $\bar{\alpha}_{\epsilon} $ as $\epsilon \rightarrow 0$, we write, for $i,j\leq a$,
\begin{align*}
0&= \epsilon \nabla_{\alpha_i}\bar{U}_{\epsilon} (\bar{\alpha}_{\epsilon} )\\
&= \epsilon \nabla_{\alpha_i} \bar{U}_{\epsilon} (\alpha_{0})+\epsilon^2
\sum_{j=1} ^a (\int_0 ^1 (\nabla^2_{\alpha_j \alpha_i} \bar{U}_{\epsilon} (\alpha_0 + t (\bar{\alpha}_{\epsilon} -\alpha_0))dt)\; (\frac{\bar{\alpha}^j_{\epsilon} -\alpha_0^j}{\epsilon}).
\end{align*}
Consider first $\epsilon \nabla_{\alpha} \bar{U}_{\epsilon} (\alpha_{0})$. Using \eqref{BkTS} and \eqref{Ska}, for $i=1,\dots,a$, it reads as
\begin{eqnarray*}\epsilon \nabla_{\alpha_i} \bar{U}_{\epsilon} (\alpha_{0})&=&\epsilon \sum_{k=1}^n \nabla_{\alpha _i}\log \det S_k(\alpha_0)
+ \frac{1}{\epsilon \Delta} \sum _{k=1}^n B^*_k(\alpha_0)\; \nabla_{\alpha_i} S_k^{-1} (\alpha_0) B_k(\alpha_0) \\
&+&  \frac{2}{\epsilon \Delta}   \sum_{k=1}^n (\nabla_{\alpha_i} B_k^*(\alpha_0)) \;  S_k^{-1} (\alpha_0) B_k(\alpha_0)
= A_1^i (\alpha_0) + A_2^i (\alpha_0) + A_3^i (\alpha_0).
\end{eqnarray*}
%with  $A_1^i (\alpha_0) = \epsilon \sum_{k=1}^n \nabla_{\alpha _i}\log \det S_k(\alpha_0)$,\\
%$A_2^i (\alpha_0)= \frac{1}{\epsilon \Delta} \sum _{k=1}^n B^*_k(\alpha_0)\; \nabla_{\alpha_i} S_k^{-1} (\alpha_0) B_k(\alpha_0)$ \\ and
%$A_3^i (\alpha_0)=\frac{2}{\epsilon \Delta}   \sum_{k=1}^n (\nabla_{\alpha_i} B_k^*(\alpha_0)) \;  S_k^{-1} (\alpha_0) B_k(\alpha_0)$.
Since  $\nabla_{\alpha_i} \log (\det S_k(\alpha_0))= \mathrm{Tr}(S_k^{-1} (\alpha_0) \nabla_{\alpha_i} S_k(\alpha_0))$,  $A_1^i (\alpha_0) $ is well defined and,  under the regularity assumptions, $A_1^i (\alpha_0) = n\epsilon O(1) $, which goes to 0 as $\epsilon \rightarrow 0$, $n$ being fixed.

Applying Lemma \ref{BkX} for the variables $T_k(\alpha_0)$, $D_k^{\epsilon}(\alpha_0) $  yields that
\begin{align*}
A_ 2 ^i  (\alpha_0)&= \epsilon \sum_{k=1}^n T^*_k(\alpha_0) \nabla_{\alpha_i} S_k^{-1} (\alpha_0)  T_k(\alpha_0)\\
&\quad + 2 \frac{\epsilon}{\sqrt{\Delta}}
\sum_{k=1}^n T_k^*(\alpha_0) \nabla_{\alpha_i} S_k^{-1} (\alpha_0)  ( \epsilon D_k^{\epsilon}(\alpha_0))\\
&\quad + \frac{\epsilon}{\Delta}  \sum_{k=1}^n ( \epsilon D_k^{\epsilon}(\alpha_0))^* \nabla_{\alpha_i} S_k^{-1}(\alpha_0) (\epsilon D_k^{\epsilon}(\alpha_0)).
\end{align*}
It follows from  Lemma \ref{BkX}, that $ \sup _k \norm{\epsilon D_k^{\epsilon}(\alpha_0)}$ is bounded.
%yields that the two last terms of $A_2^i (\alpha_0)$ go to 0.
Therefore, %the three terms of $A_ 2^i (\alpha_0)$  are equal to $\epsilon  Z_i$, where $Z_i$, as a finite sum of bounded random variables, is a bounded random variable. Hence,
$A_ 2^i (\alpha_0) \rightarrow 0$ in $\P_{\alpha_0}$-probability.

Let us study the main  term  $(A_3 ^i(\alpha)$ of $\epsilon \nabla_{\alpha_i} \bar{U}_{\epsilon} (\alpha_{0})$.\\
 Using Proposition \ref{TStheta}  and \eqref{BkTS}, \eqref{Gk} yields that,
%Let us define $ G_k(\alpha)= (G_k^i(\alpha))$
under $\P_{\alpha_0}$,
 \begin{equation}\label{dGk}
  \nabla_{\alpha_i} B_k(\alpha_0)= \Delta G_k^i (\alpha_0 ) -  \epsilon (\nabla_{\alpha_i}  \Phi(\alpha_0,t_k,t_{k-1})( g(\alpha_0,t_{k-1} ) + \epsilon R^{\epsilon}(\alpha_0, t_{k-1}))),
  \end{equation}
 where  $ \sup_k \norm{\epsilon R( \alpha, t_k)}$ is uniformly  bounded in probability.
 %+ \epsilon O_P(1).
 Therefore,
\[A_3^i (\alpha_0)= 2 \sqrt{\Delta}  \sum_{k=1}^n ((G_k^i  (\alpha_0))^*S_k^{-1} (\alpha_0) T_k(\alpha_0) + \epsilon R' _k(\alpha_0)),\]
 with $ R^{\prime}_k(\alpha_0)$ uniformly bounded in probability.
 By Lemma \ref{BkX}, $(T_k (\alpha_0)), k=1,\dots,n)$ are independent centered Gaussian random variables with covariance matrix  $S_k(\alpha_0)$.
% $ ((G_k (\alpha_0))^*S_k^{-1} (\alpha_0)   T_k (\alpha_0)), k=1,\dots,n$  are independent centered Gaussian random variables with covariance matrix  $G_k (\alpha_0))^*S_k^{-1} (\alpha_0) G_k (\alpha_0)$.
We  that   $A_3(\alpha_0) = (A_3^1(\alpha_0), \dots,A_3(\alpha_0))^* $ converges to  the Gaussian random variable $ \mathcal{ N}_a (0,4M(\alpha_0)) $. Joining all these results yields that
 $$- \epsilon \nabla_{\alpha} \bar{U}_{\epsilon} (\alpha_{0}) \rightarrow \mathcal{ N}_a(0, 4 M (\alpha_0)) \;   \mbox {with  }$$
 $$M(\alpha_0)= (M(\alpha_0))_{ij}=
  \Delta  \sum_{k=1}^n (G_k^i(\alpha_0))^*S_{k}^{-1}(\alpha _0)  G_k^j(\alpha_0).$$

% herefore $ \sum_{k=1}^n G_k (\alpha_0))^*S_k^{-1} (\alpha_0)   T_k (\alpha_0)$ is a Gaussian random variable Setting $A_3(\alpha)= (A_3^1(\alpha),\dots,A_3^a(\alpha))$,  we get that  the first term  of $A_3(\alpha_0)$
% % $ \norm{R_k(\alpha_0)} \leq 2 n\sqrt{\Delta} sup_k \norm{ T_k(\alpha_0)+\epsilon D_k(\alpha_0)}+ \norm $.
% %The first term of $A_3^i (\alpha_0)$
%  is a finite sum of independent centered Gaussian random variables with covariance matrix $\Delta  (G_k (\alpha_0))^*S_k^{-1} (\alpha_0)  Var(T_k(\alpha_0) )S_k^{-1} (\alpha_0) G_k  (\alpha_0)$.  Hence,
% $A_3(\alpha_0) $ converges to  the Gaussian random variable $ \mathcal{ N}_a (0,4M(\alpha_0) $ with . Therefore, using\eqref{Gk},
% $$- \epsilon \nabla_{\alpha_i} \bar{U}_{\epsilon} (\alpha_{0}) \rightarrow \mathcal{ N}_a(0, 4 M (\alpha_0)) \;   \mbox {with  } M(\alpha_0)= (M(\alpha_0))_{ij}=
%  \Delta  \sum_{k=1}^n (G_k^i(\alpha_0))^*S_{k}^{-1}(\alpha _0)  G_k^j(\alpha_0).$$

Consider $\epsilon^2 \nabla^2_{\alpha_j\alpha_i}\bar{U}_{\epsilon} (\alpha)$. Similar computations yield that
\[\epsilon^2 \nabla^2_{\alpha_j\alpha_i} \bar{U}_{\epsilon}(\alpha_0)=  2 \Delta \sum_{k=1}^n (G_k^i(\alpha_0))^*  S_{k}^{-1} (\alpha_0) G_k^j(\alpha_0)+n \epsilon O_P(1).\]
 Therefore, for all $1\leq i,j\leq a$,
 $$ \epsilon^2 \nabla^2_{\alpha_i\alpha_j}\bar{U}_{\epsilon} (\alpha_0) \rightarrow  2M_{ij}(\alpha_0) \quad \P_{\alpha_0} \mbox{a.s.  as } \epsilon \rightarrow 0.$$
It remains to study $\sup _{t \in [0,1]} |\epsilon^2  \nabla^2_{\alpha_j\alpha_i}\bar{U}_{\epsilon} (\alpha_0 + t( \bar{\alpha}_{\epsilon} -\alpha_0)) -
 \epsilon^2 \nabla^2_{\alpha_j\alpha_j}\bar{U}_{\epsilon} (\alpha_0 )|$.\\
 %\begin{equation} \label{supdU2}
%  \sup _{t \in [0,1]} |\epsilon^2  \nabla^2_{\alpha_j\alpha_i)}\bar{U}_{\epsilon} (\alpha_0 + t( \bar{\alpha}_{\epsilon} -\alpha_0)) -
% \epsilon^2 \nabla_{\alpha_j\alpha_j}\bar{U}_{\epsilon} (\alpha_0 )|\rightarrow 0.
%\end{equation}
We have
$\epsilon^2  \nabla^2_{\alpha_j\alpha_i}\bar{U}_{\epsilon} (\alpha)= \frac{1}{\Delta}(A_1^{ij} (\alpha)+ A_2^{ij} (\alpha))$,
where
\[A_1^{ij}(\alpha)= 2\sum _{k=1}^n \nabla_{ \alpha_i}  B_k^*(\alpha)S_k^{-1} (\alpha)\nabla_{ \alpha_j}  B_k(\alpha),\quad
 A_2^{ij}(\alpha)=  \sum _{k=1}^n Z_k^*(\alpha) B_k(\alpha)\] with
 \begin{align*}
Z_k^*(\alpha)= 2 \nabla_{ \alpha_j}  B_k^*(\alpha)\nabla_{\alpha_i}S_k^{-1} (\alpha)&+ B_k^*(\alpha)\nabla^2_{\alpha_i\alpha_j} S_k^{-1}(\alpha)+
 2 \nabla_{ \alpha_i}  B_k^*(\alpha)\nabla_{\alpha_j}S_k^{-1} (\alpha)\\
 &+ 2 \nabla^2_{\alpha_i\alpha_j}B_k^*(\alpha) S_k^{-1} (\alpha).
\end{align*}
%Clearly the random variables $(Z_k(\alpha), k\leq n, \alpha  \in \Theta$ are uniformly bounded in probability.
% with respect to $k\leq n, \alpha  \in \Theta$.
Similarly to the previous section, we need that, under $\P_{\alpha_0}$, the properties stated below hold.
\begin{equation}\label{P1}
\norm{B_k(\alpha)-B_k(\alpha_0)} \leq \norm{\alpha-\alpha_0} (C_1+ C_2 O_P(1))  \; \mbox{uniformly with respect to } k,\alpha;
\end{equation}
\begin{equation}\label{P2}
\norm{\frac{1}{\epsilon} B_k(\alpha_0)}   \mbox{  are uniformly bounded random variables};
\end{equation}
\begin{equation}\label{P3}
 \underset{ k\leq n,\alpha \in\Theta} \sup \norm{\nabla_{\alpha_i}B_k(\alpha)} = O_P(1); \quad \mbox{and }  \norm{\nabla_{\alpha_i}B_k(\alpha)-\nabla_{\alpha_i}B_k(\alpha_0} \leq C_1\norm{\alpha-\alpha_0} .
 \end{equation}
 The proofs of these properties  are similar to the previous section and  omitted.

 Therefore,
\[A_2^{ij}(\alpha)- A_2^{ij}(\alpha_0)=\sum_{k=1}^n  (Z_k^*(\alpha)- Z_k^*(\alpha_0)) B_k(\alpha_0)+ \sum_{k=1}^n  Z_k^*(\alpha)(B_k(\alpha)-B_k(\alpha_0) ).\]
Using  \eqref{P1}, \eqref{P2} we get
\[|A_2^{ij}(\alpha)- A_2^{ij}(\alpha_0)| \leq   \sup \norm{Z_k(\alpha)}(  2 n \epsilon \sup \norm{\frac{B_k(\alpha_0)}{\epsilon}}+  \norm{\alpha-\alpha_0} (C_1+ C_2 O_P(1)).\]

Consider now $A_1^{ij}(\alpha)- A_1^{ij}(\alpha_0) $. It  reads as
\begin{align*}
A_1^{ij}(\alpha)- A_1^{ij}(\alpha_0) =&\, 2\sum_{k=1}^n[ \nabla_{\alpha_i}B_k^*(\alpha) S_k^{-1}(\alpha)  (\nabla_{\alpha_j}B_k(\alpha)
-\nabla_{\alpha_j}B_k(\alpha_0))]\\
&+[ \nabla_{\alpha_j}B_k^*(\alpha) S_k^{-1}(\alpha)  (\nabla_{\alpha_i}B_k(\alpha)-\nabla_{\alpha_i}B_k(\alpha_0))]\\
&+[
 \nabla_{\alpha_i}B_k^*(\alpha) (S_k^{-1}(\alpha)- S_k^{-1}(\alpha_0))  \nabla_{\alpha_j}B_k(\alpha_0)].
\end{align*}
 Hence,  $\norm{A_1{ij}(\alpha)- A_1^{ij}(\alpha_0)} \leq  2n C\norm{\alpha-\alpha_0}$.

Using the consistency  $\bar{\alpha}_{\epsilon}$, we get that
\[\sup _{t \in [0,1]} |\epsilon^2  \nabla^2_{\alpha_j\alpha_i)}\bar{U}_{\epsilon} (\alpha_0 + t( \bar{\alpha}_{\epsilon} -\alpha_0)) -
 \epsilon^2 \nabla^2_{\alpha_j\alpha_j}\bar{U}_{\epsilon} (\alpha_0 )|\rightarrow 0.\]

This achieves the proof of (ii) and of Theorem \ref{LimLFO}.
\end{proof}

\subsubsection{Comments}
%\begin{remark}\label{remdet}
\textbf{ (1)}  The term $\sum_{k=1}^n \log \det S_k(\alpha) $ could have been omitted in the definition of $\bar{U}_{\epsilon} (\alpha)$. It has no influence on the asymptotic properties of $\bar{\alpha}_{\epsilon}$. However, we have observed in the simulation results that it yields better estimators (less biased).
An explanation lies in the fact that in practice $\epsilon$ is small, but probably not enough to compensate this first term.
%, which could explain that
 the observations of  less biased estimators non-asymptotically.\\
% \end{remark}

 %\begin{remark}\label{remguy14}
\noindent \textbf{ (2)} In \cite{guy14IV}, we considered  the case of an unknown parameter $\beta$ in the diffusion coefficient and therefore used  a Conditional Least Square
  estimator  based on $ U_{\epsilon} (\alpha)= \sum_{k=1}^nB_k^*(\alpha) B_k(\alpha)$. The CLS estimator obtained is consistent. It converges at the same rate, but with a larger covariance matrix  $J^{-1}_{\Delta}(\alpha) I_{\Delta}(\alpha)   J^{-1}_{\Delta}(\alpha) $ with
 $J^{ij}_{\Delta}= \sum_{k=1}^n  (G_k^i(\alpha))^*G_k^j(\alpha)$ and
 $I_{\Delta} (\alpha)= \sum_{k=1}^n(G_k^i(\alpha))^* S_k(\alpha) G_k^j(\alpha)$.\\
 %\end{remark}
%\begin{remark}\label{remeffi}

\noindent \textbf{ (3)} We can compare  the result of Theorem \ref{LimLFO} to the inference  of an unknown parameter in the drift coefficient for a continuously observed  diffusion on $[0,T]$ in the asymptotics $\epsilon\rightarrow 0$.  According to \cite{kut84IV}, assuming a known  diffusion coefficient  $\epsilon \sigma(x)$, the Maximum Likelihood Estimator is consistent  and the Fisher information matrix is
%s well known for continuously observed diffusion processes on a time interval $[0,T]$ (see \cite{kut84}). The Fisher information matrix in this model is
 \begin{equation}\label{Fisherkuto}
 (I_b(\alpha_0,\beta_0))_{ij}=\int_{0}^T (\nabla_{\alpha_i}  b(\alpha_0, z(\alpha_0,s)))^* \Sigma^{-1}(z(\alpha_0,s))
 \nabla_{\alpha_j}b(\alpha_0,z( \alpha_0,s)) ds.
 \end{equation}
To compare  the estimator $\bar{\alpha}_{\epsilon,\Delta}$ with the CLS estimator, we can study the limits of the two Information matrices when $\Delta$ goes to zero. Using that $z(\alpha,\cdot)$ satisfies the ODE \eqref{ODEPARTIV},  we have,
\begin{equation}\label{rel:controle_Dk}
G_k(\alpha_0)=-\nabla_{\alpha}b (\alpha_0,z(\alpha_0,t_{k-1}))+o_{\Delta}(1), \mbox{ as $\Delta$ goes to zero}.
  \end{equation}
This result together with  Lemma \ref{SigmaS}  implies that $I_\Delta(\alpha_0,\beta_0) \rightarrow I_b(\alpha_0,\beta_0) $ as $\Delta \rightarrow 0$. Since $I_b(\alpha_0,\beta_0)$ is the optimal information matrix for continuous time observation, this convergence provides some kind of optimality result for fixed $\Delta$.

Consider now the covariance matrix of the CLS estimator. We have, $\epsilon\rightarrow 0$,
\begin{align*}
(J_{\Delta}(\alpha))_{ij}& \rightarrow \int_0^T \nabla_{\alpha_i}b (\alpha_0,z(\alpha_0,t))^* \nabla_{\alpha_j}b (\alpha_0,z(\alpha_0,t))dt,\mbox{ and} \\
(I_{\Delta}(\alpha))_{ij}& \rightarrow \int_0^T \nabla_{\alpha_i}b (\alpha_0,z(\alpha_0,t))^* \Sigma(\beta_0, z(\alpha_0,t))\nabla_{\alpha_j}b (\alpha_0,z(\alpha_0,t))dt.
\end{align*}

This clearly differs from the optimal asymptotic variance and confirms that the  CLS estimator is not efficient. However,
it  might be easier to minimize the CLS function $\sum_{k=1}^n  G_k(\alpha)^*G_k(\alpha)$ than the actual contrast function
$\sum_{k=1}^n  G_k(\alpha)^*S_{k-1}^{-1}(\alpha)G_k(\alpha)$. Therefore this CLS estimator  can be useful to serve as an initialization for other computations or algorithms.
%\end{remark}

%  $M_\Delta(\alpha_0)\tend{\Delta}{0}M_0(\alpha_0)=\integ{0}{T}\deriv{b(\alpha_0,x_{\alpha_0}(s))}{\alpha}\trans{\deriv{b(\alpha_0,x_{\alpha_0}(s))}{\alpha}}ds$,\\
%$J_\Delta(\alpha_0,\beta_0)\tend{\Delta}{0} M_0(\alpha_0)\left(\integ{0}{T}\deriv{b(\alpha_0,x_{\alpha_0}(s))}{\alpha}\Sigma(\beta_0,x_{\alpha_0}(s))\trans{\deriv{b(\alpha_0,x_{\alpha_0}(s))}{\alpha}}ds\right)^{-1}\trans{M}_0(\alpha_0)$.\\
%As expected, this result confirms that it is better to use the second contrast $\tilde{U}_{\epsilon,\Delta}$, to estimate the parameter $\alpha$ when $\beta$ is known.,
%$\Sigma(\beta,x)=\phi(\beta) \Sigma(x)$ or $\beta=\alpha$.

%\section{Study of the two examples $SIR$ and $SIRS$ with seasonal forcing}
% \subsection{SIR epidemic Dynamics}
% We have $\theta=(\lambda,\gamma)$,
%$$ \nabla_{\theta} b(\theta))= \begin{pmatrix} -si& 0\\si&-i\end{pmatrix};
%\Sigma^{-1}(\theta)=\frac{ 1}{\lambda \gamma s i} \begin{pmatrix}\lambda s+\gamma& \lambda s\\
%\lambda s& \lambda s \end{pmatrix}.
%$$
%Let  $s(t)= s( \lambda_0, \gamma_0,t); i(t)= i(\lambda_0, \gamma_0,t)$ the solution of the corresponding ODE.
%Therefore, the matrix $I_b(\theta$ defined in \eqref{Ib} is,
%$$I_b(\lambda,\gamma)= \begin{pmatrix} \frac{1}{\lambda} \int_0 ^T s(t) i(t) dt& 0\\
%0& \frac{1}{\gamma}  \int_0 ^T i(t) dt\end{pmatrix}.$$
%
\section{Assessment of estimators on simulated data sets}
% of epidemics dynamics}
\label{sec:simul}

We consider two examples of epidemic dynamics, the $SIR$
 and the $SIRS$ presented in the first part of these notes and recalled in Section  \ref{sec:approxdiif} for the diffusion approximation. We used the Gillespie algorithm (see Part I of these notes) to simulate  the $SIR$ epidemic dynamics $(\mathcal{Z}^N(t), 0 \leq t \leq T)$
 %Simulated trajectories of epidemic dynamics by Markov jump processes are performed using the algorithm of \cite{gil77} for the $SIR$ model
 and, for the $SIRS$ model, the $\tau$-leaping method (\cite{cao05IV}), which is more efficient for large populations.\\

As pointed in the introduction, diffusion approximations are relevant in case of a major outbreak in a large community.
Therefore, we  keep only in the analysis what we called ``non-extinct trajectories'', chosen  according to a frequently used empirical criterion: we keep epidemic trajectories such that  the final epidemic size is larger than  the observed empirical size minus the standard empirical  error of the final epidemic size.

The inference  is based only on non-extinct trajectories .
%(chosen, according to a frequently used empirical criterion, such as the final epidemic size is larger than the observed empirical size minus the standard empirical  of the number of initial susceptibles).
Since we possess, for each simulation, the whole sample path of the epidemic process, we can compute the maximum likelihood estimator  (see Chapter \ref{chappointproc} of this part)
which depends on the whole  path of the jump process.
For instance, for the $SIR$ case, the MLE is
%$$\hat{\lambda}_N= \frac{1}{N} \frac{ \sum_{i=1}^{K_N(T)} (1-J_i)}{\int_0^T S^N(t) I^N(t) dt}= \frac{\mbox{ \# Infections} }{\int_0^T S^N(t) I^N(t) dt}, \quad
 %\hat{\gamma}_N= \frac{1}{N} \frac{ \sum_{i=1}^{K_N(T)} J_i}{\int_0^T I^N(t) dt} = \frac{\mbox{\# Recoveries}}{\mbox{"Mean infectious period"}}.$$
 \begin{equation}\label{MLEPJP}
\hat{\lambda}_N=\frac{1}{N} \frac{\mbox{ \# Infections}}{\int_0^T S^N(t) I^N(t) dt}; \quad
\hat{\gamma}_N= \frac{1}{N} \frac{\mbox{\# Recoveries}}{\int_0^T  I^N(t) dt}.
\end{equation}
%where $T_i $ are   the successive times of events,$ K_N(T)$  the number of events (infection or recovery),and  for each event,  $J_i=0$ if there is an infection and $J_i=1$ if it is a recovery.

We call this MLE  based on complete epidemic data  \textbf{ the reference estimator}.
This is the best result that can be achieved from these epidemic data. \\

 In order to investigate the influence of various parameters, we consider various scenarios. Each   scenario corresponds  to the choice of the model, the parameters $\theta$, the population size $N$,  the time interval of observation $[0,T]$ and the sampling interval $\Delta$. We proceeded to 1000 repetitions  for  each scenario. \\
Hence, we varied the total size of the population $N$,  the parameters ruling the $SIR$, $SIRS$ epidemics, the time interval for observations $[0,T]$.
 Then, we sampled with sampling $\Delta$  each path of the Markov jump process.  This sampling interval also varies.
Therefore the observations coming from the simulations are
$$  \frac{\mathcal{Z}^N(k\Delta)}{N}= Z^N(k\Delta) \;  k=1,\dots, n  \; \mbox{with  }T=n\Delta.$$

Each scenario corresponds  to the choice of the model, the parameters $\theta$, the population size $N$,  the time interval of observation $[0,T]$ and the sampling interval $\Delta$.
%We proceeded to 100 repetitions  for  each scenario.

We compare the estimators obtained with the method described in the two previous sections with the MLE \eqref{MLEPJP}.
The properties of our minimum contrast estimators are assessed and compared to reference estimators.
%For the case of discretely  observed process, this is performed based on simulated data (section \ref{SIRsimuCO} ), both on $SIR$ and $SIRS$ dynamics (models described in \ref{SIR_appli} and \ref{SIRS_appli} respectively).\\
%For partial observations, the $SIR$ case is explored on simulated data in section \ref{sec:SIRpart}, whereas the $SIRS$ dynamics are investigated on real data corresponding to influenza cases over several consecutive seasons in France in section \ref{sec:SIRSgrippe}. Point contrast estimates ($CE$), $95\%$ theoretical confidence intervals ($CI_{th}$, when available) and empirical ones ($CI_{emp}$, built on 1000 runs) are provided for each set of parameter values. \\

\noindent
 For parameters with dimension greater than two, confidence ellipsoids are projected on planes, by considering all pairs of parameters.
Theoretical confidence ellipsoids are built as follows. Let $V(\theta_0)$  denote the covariance matrix of the asymptotic normal distribution of parameters
estimation in drift term (i.e.\  $I_b^{-1}(\theta_0)$  defined in (\ref{Ibtheta}) and $M^{-1}(\theta_0)$ defined in (\ref{Ma}).
% and $\Lambda(\eta_0)^{-1}V(\theta_0)\Lambda(\eta_0)$ in (\ref{CVL}),
Since  $\epsilon^{-1}V(\theta_0)^{-1/2}(\hat{\theta}_{\epsilon,\Delta}-\theta_0)\rightarrow_{\mathcal{L}}\mathcal{N}(0,I_{k})$ (where $\hat{\theta}_{\epsilon,\Delta}$ represents $\check{\alpha}_{\epsilon,\Delta}$ obtained minimizing (\ref {Ucheck}) %and (\ref{estimepid}),
 or $\bar{\eta}_{\epsilon,\Delta}$ in (\ref{baralpha})
 Then, for $k=a$ (dimension of $\alpha$),
 %in (\ref{CVRSPA}) and (\ref{estimepid}), and to $a+1$ in (\ref{CVL})).
we have,
 \begin{equation}\label{khi2}
\frac{1}{\epsilon^{2}}\;(\hat{\theta}_{\epsilon,\Delta}-\theta_0)^{*} V(\theta_0)^{-1}(\hat{\theta}_{\epsilon,\Delta}-\theta_0)\rightarrow_{\mathcal{L}}\chi_2(k).
\end{equation}
The matrix $V(\theta_0)^{-1}$ being positive,  the quantity $(\hat{\theta}_{\epsilon,\Delta}-\theta_0)^{*} V(\theta_0)^{-1}(\hat{\theta}_{\epsilon,\Delta}-\theta_0)$ is the squared norm of vector $\hat{\theta}_{\epsilon,\Delta}-\theta_0$ for the scalar product associated to $V(\theta_0)^{-1}$. If we denote by $\chi^2_k(0.95)$
%q^{95}_{k}$
the $95\%$ quantile of the $\chi^2_k$ distribution, the relation \eqref{khi2} could be rewritten as $||(\hat{\theta}_{\epsilon,\Delta}-\theta_0)^2 M(\theta_0)^{-1}||\leq \epsilon^2\chi^2_k(0.95)$  and define an ellipsoid in $\mathbb{R}^{k}.$ \\
\noindent
Empirical confidence ellipsoids are based on the variance-covariance matrix of centered estimators (based on $1000$ independent estimations), whose eigenvalues define the axes of ellipsoids.

%------------------------------------------------------------------------------------------------------
%\section{The case of  discrete observations  of  $SIR$ and $SIRS$ models, simulated data} \label{SIRsimuCO}

\noindent
%In this section, we consider the case where all the components of the epidemic process are observed. \frac{\mathcal{Z}^N(k\Delta)}{N}, \;  k=1,\dots, n  \; \mbox{with  }T=n\Delta.$$
In the two epidemic models detailed below, we assume both components of $Z^N( t) =S^N(t), I^N(t)$ are observed  with sampling interval $\Delta$, $((S^N(k\Delta), I^N(k\Delta)), k =1,\dots n)$ with $T=n\Delta$.\\
\noindent
%Simulated trajectories of epidemic dynamics by Markov jump processes are performed using the algorithm of \cite{gil77} for the $SIR$ model and the $\tau$-leap method ({cao05}), more efficient for large populations, for the $SIRS$ model. Based on these simulated data, the accuracy of our estimators is investigated with respect to the population size $N$, the number of observations $n$ and some of the remaining parameters.
%
%Inference, using contrast (\ref{UDelta}) is based only on non-extinct trajectories (chosen, according to a frequently used empirical criterion, such as the final epidemic size is larger than $5\% $ of the number of initial susceptibles).
%
%
%Moreover, the intrinsic limits of the method are investigated by comparing $CI_{th}$ for different values of $n$, other parameters being fixed, with the theoretical variance co-variance matrix when $n \rightarrow \infty$. \\

%----------------------------------------------------------------------
\subsection{The SIR model}
The parameters of interest for epidemics are considered following a reparameterization: the basic reproduction number, $R_0=\frac{\lambda}{\gamma}$, which represents the average number of secondary cases generated by one infectious in a completely susceptible population, and the average infectious duration, $d=\frac{1}{\gamma}$. Two values were tested for $R_0=\{1.5,5\}$ and $d$ was set to $3$ (in days, an average value consistent with influenza infection). Three values for the population size $N=\{400,1000,10000\}$ and of the number of observations $n=\{5,10,40\}$ were considered, along with two values for the final time of observation, $T=\{20,40\}$ (in days). For each scenario defined by a combination of parameters, the analytical maximum likelihood estimator ($MLE$), calculated from the observation of all the jumps of the Markov process (see \ref{chappointproc}), was taken as reference.\\

\noindent
\textit{Effect of the parameter values $\{R_0,d\}$ and of the number of observations $n$}\\
\noindent
The accuracy of the two estimators $ \check{\alpha}_{\epsilon,\Delta}$ and $\bar{\alpha}_{\epsilon}$, for $N=1000$ and from trajectories with weak ($R_0=5$) and strong ($R_0=1.5$) stochasticity is illustrated in Figure \ref{fig:1000_1}.  We observe that $R_0$ and $d$ are moderately correlated (ellipsoids are deviated with respect to the $x$-axis  and $y$- axis). The shape of confidence ellipsoids depends on parameter values: for $R_0=5$, the $95\%$ confidence interval  is larger for $R_0$ than for $d$, whereas the opposite occurs for $R_0=1.5$. For $R_0=5$, all these confidence intervals  are almost superimposed, which suggests that the estimation accuracy is not altered by the fact that  all the jumps are not observed. However, for $R_0=1.5$ the shape of ellipsoids varies with $n$. Point estimates for $MLE$
derived for complete observation of $ (\mathcal{Z}^N(t)$ of the original jump process and the estimators  $ \check{\alpha}_{\epsilon,\Delta}$, $\bar{\alpha}_{\epsilon}$ are very similar for different values of $n$, which confirms the interest of using   these estimators when small number of observations is available.\\

\begin{figure}[!ht]
\includegraphics[width=0.9\textwidth]{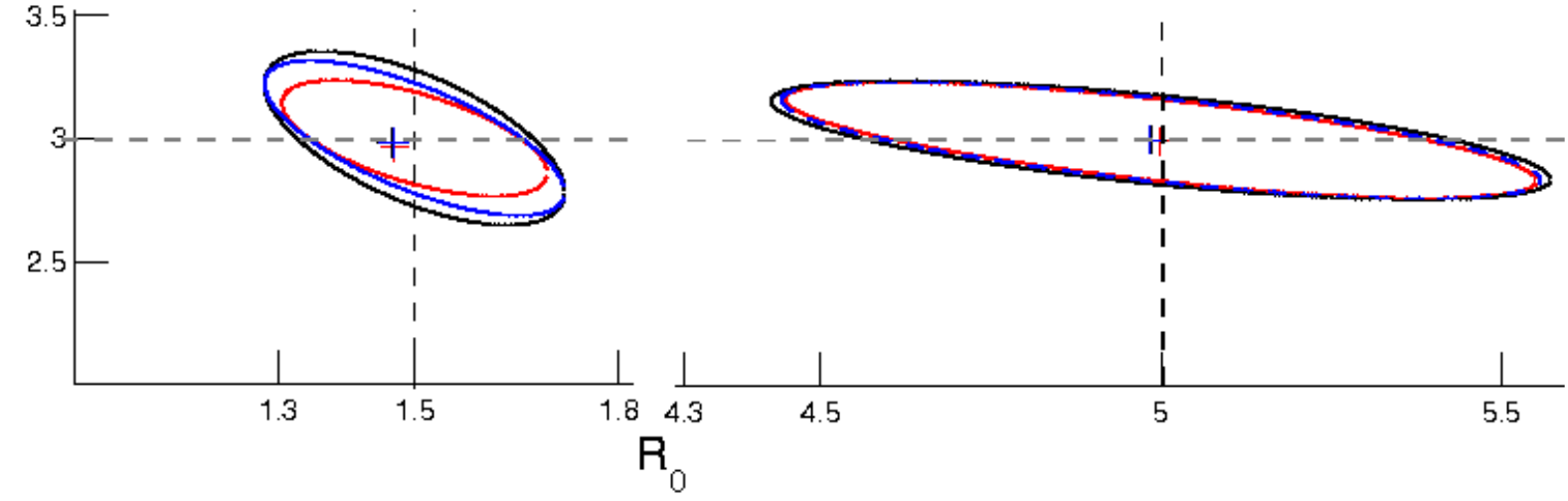}
\caption{Point estimators (+) are computed by averaging over 1000 independent simulated trajectories of the $SIR$ stochastic model (completely observed)  together  with their associated theoretical confidence ellipses centered on the true value: $MLE$ with complete observations (red), $CE$ for one observation/day, $n=40$ (blue) and $CE$ for $n=10$ (black). Two scenarios are illustrated: $(R_0,d,T)=\{(1.5,3,40);(5,3,20)\}$, with $N=1000$. For both scenarios $(S(0),I(0))=(0.99,0.01)$.The value of $d$ is reported on the y-axis. Horizontal and vertical dotted lines cross at the true value}.
\label{fig:1000_1}
\end{figure}

\begin{figure}[ht]
\includegraphics[width=0.9\textwidth]{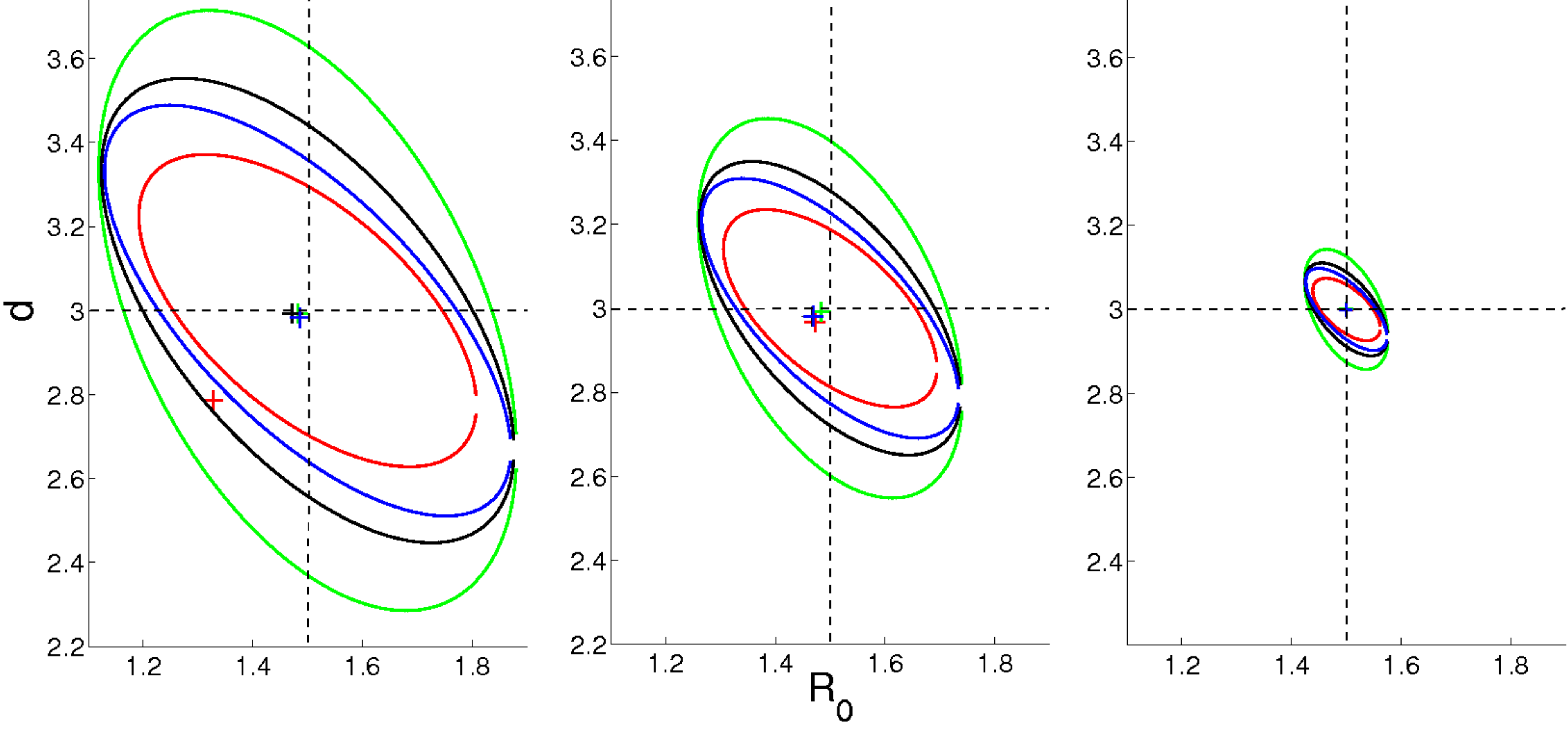}
\caption{Point estimators (+) computed by averaging over 1000 independent simulated trajectories of the $SIR$ stochastic model completely observed and their associated theoretical confidence ellipses centered on the true value: $MLE$ with complete observations (red), $CE$ for one observation/day, $n=40$ (blue), $CE$ for $n=10$ (black) and $CE$ for $n=5$ (green) for $(S(0),I(0))=(0.99,0.01)$, $(R_0,d)=(1.5,3)$ and $N=\{400,1000,10000\}$ (from left to right). Horizontal and vertical dotted lines cross at the true value.}
\label{fig:400-10000}
\end{figure}
\vspace{0.1cm}
\noindent
\textit{Effect of the parameter values $\{R_0,d\}$ and of the population size $N$}\
\noindent
From Figure \ref{fig:400-10000}, we can notice that $\sqrt{N}$ has an impact on estimation accuracy (the width of the confidence intervals decreases with $\sqrt{N}$). The case of very few observations ($n=5$) leads to the largest confidence intervals. The $MLE$ appears biased for $N=400$. This could be due to the fact that the $MLE$ is optimal when data represent a `typical' realization (i.e.\  a trajectory that emerges leading to a non-negligible number of infected individuals) of the Markov process, but could yield a bias when observations are far from the average behaviour. Our $CE$s seem robust to the departure from the `typical' behaviour (i.e.\  for noisy trajectories obtained either for small $N$ or small $R_0$).\\

%----------------------------------------------------------------------
\subsection{The SIRS model}
\noindent
For the $SIRS$ model introduced in Section \ref{SIRS_appli}, four parameters were estimated: $R$, $d$, $\lambda_1$ and $\delta$.
%(for numerical reasons, $10\lambda_1$ and $1/\delta T_{per}$ were rather estimated in practice).
Concerning the remaining parameters,$\mu$ was set to $1/50$ $\mbox{years}^{-1}$ (a value usually considered in epidemic models), $T_{per}$ was set to $365$ days (corresponding to annual epidemics) and $\eta$ was taken equal to $10^{-6}$ (which corresponds to $10$ individuals in a population size of $N=10^7$). We should notice that instead of estimating the real $R_0$ (more complicated to calculate for periodical dynamics), we prefer to estimate a parameter combination similar to the $R_0$ for $SIR$ model, $\lambda_0 / \gamma$, which was called here $R$.The performances of $CE$s were assessed on parameter combinations: $(R,d,\lambda_1,\delta)=\{(1.5,3,0.05,2)$ and $(1.5,3,0.15,2)\}$ and $T=20$ years, with $\lambda_1=0.05$ leading to annual cycles and $\lambda_1=0.15$ to biennial dynamics (Figure \ref{fig:SIRSbif}). Numerically, the scenarios considered are consistent with influenza seasonal outbreaks. The accuracy of estimation is relatively high, as illustrated in Figure \ref{fig:SIRSafterbif}, regardless of the parameter. For one observation per day (which can be assimilated to a limit of data availability), the accuracy is very similar to the one based on a complete observation of the epidemic process (blue and red ellipsoids respectively). Estimations based on one observation per week are less but still reasonably accurate.

\begin{figure}[ht]
\includegraphics[width=0.99\textwidth]{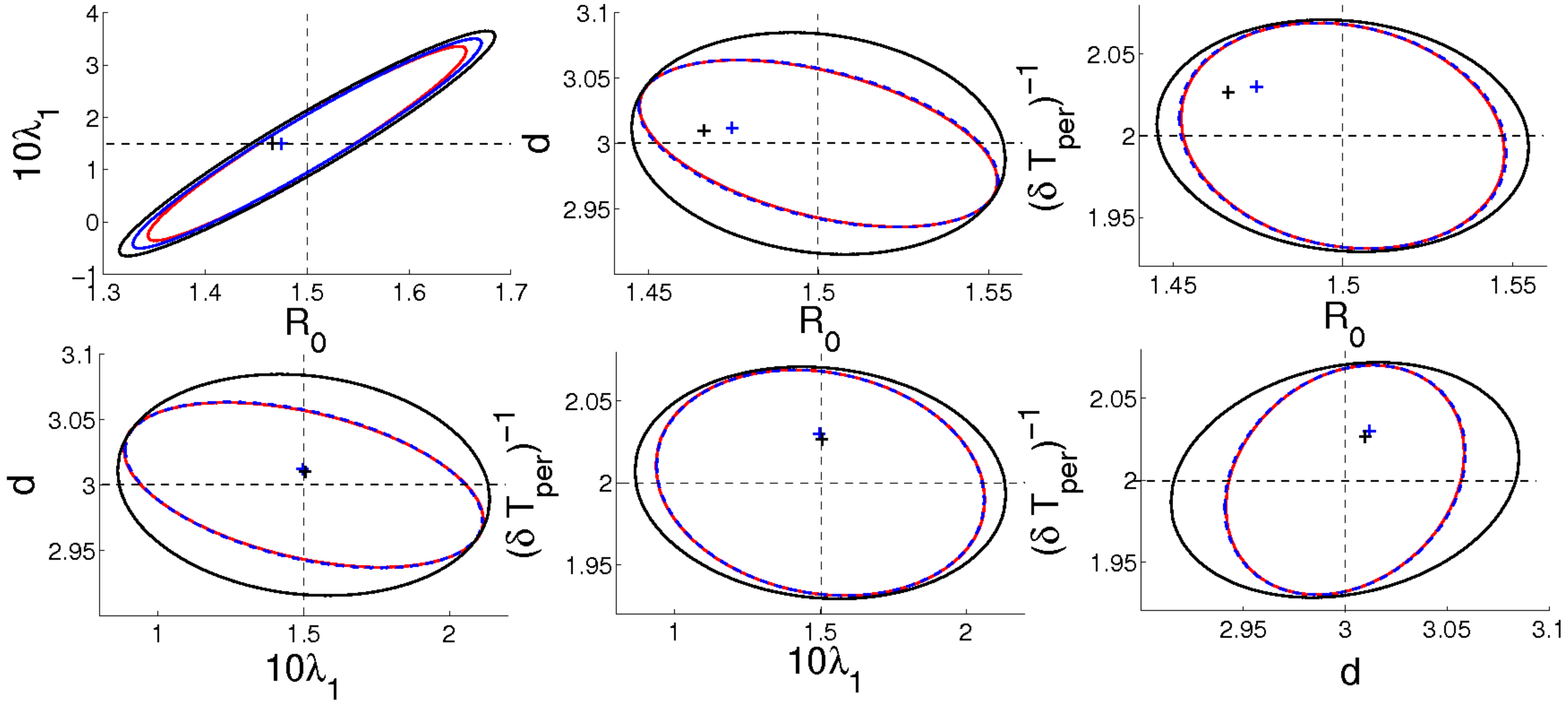}\\
\caption{Point estimators (+) computed by averaging over 1000 independent simulated trajectories of the $SIRS$ stochastic model with demography and seasonal forcing in transmission, completely observed (red), and their associated planar projections of theoretical confidence ellipsoids centered on the true value: $CE$ for one observation/day (blue) and for one observation/week (black) for $(R,d,\lambda_1,\delta)=(1.5,3,0.15,2)$, $T=20$ years and $N=10^7$. Asymptotic confidence ellipsoids ($n \rightarrow \infty$) are also represented (red,blue,black). Horizontal and vertical dotted lines cross at the true value.}
\label{fig:SIRSafterbif}
\end{figure}

%----------------------------------------------------------------%%%%%%%%%%%%%%%%%%%%%%%%%%%%%%%%%%

\section{Inference for partially observed epidemic dynamics}
\label{sec:partial}
In the case of epidemics, numbers of susceptible and infected individuals over time are generally not observed. In practice, (sometimes noisy) observations are often assumed to correspond to cumulated numbers, over the sampling interval $\Delta$, of newly infected individuals (i.e.\  $\int^{t_k}_{t_{k-1}} \lambda S(s)I(s)ds$). In the $SIR$ diffusion model, this corresponds to the recovered individuals $\{ (R(t_k)-R(t_{k-1})), k=1,\dots n\}$ for diseases  with short duration of the infected period. Hence, this situation can be  assimilated, as a first attempt, to the case where only one coordinate can be observed.\\

\noindent
In this section, we consider the case  of a two-dimensional diffusion process $X(t)=(X_1(t),X_2(t))^*$
 \begin{equation}\label{SDE2D}
dX(t)= b(\alpha,X(t)) dt +\epsilon \sigma( \beta, X(t) dB(t)) dt ; \quad X(0)=x ,
\end{equation}
where $B(t)$ is a Brownian motion on $\R^2$ and $x$ non-random.

 We assume that  only the first coordinate $X_1(t)$ is observed on a fixed time interval $[0,T]$ with sampling $\Delta$.
We consider the diffusion on $\R^2$  satisfying the stochastic differential equation
Therefore, the observations are now
\begin{equation}\label{Yk}
X_1(t_k),\; k= 1,\dots n,\quad  \mbox{with  }\;t_k=k\Delta,\quad T=n\Delta.
\end{equation}
For continuous observations of $(X_1(t))$ on a finite time interval $[0,T]$, two studies \cite{jam95IV}, \cite{kut94IV} are concerned with parametric inference in this statistical framework. Both  studied the maximum likelihood estimator of parameters in the drift function  for a diffusion matrix equal to $ \epsilon ^2 I_p$. This likelihood is difficult to compute since it relies on integration on the unobserved coordinate. \cite{jam95IV}, \cite{kut94IV} proposed filtering approaches to compute this likelihood, as it is done  for general Hidden Markov Models (see e.g.\ {\cite{cap05IV} , \cite{dou01IV}).
Here, we can take advantage of the  presence of $\epsilon$ and extend to partial observations the  method by  contrast processes and M- estimators that had been developed for complete observations  (\cite{gen90IV}, \cite{glo09IV}, \cite{guy14IV}), \cite{sor03IV}).\\

\noindent
We study the case of small (or high frequency) sampling interval, $\Delta=\Delta_n \rightarrow 0$, on a fixed time interval $[0,T]$ with $T=n\Delta$, which yields explicit results. This allows us to disentangle problems coming from discrete observations and those coming from the missing observation of one coordinate and hence provides a better understanding of the problems rising in this context.
The case of $\Delta$ fixed could be studied similarly, with more cumbersome notations and no such insights .

First, the notations required are introduced, results are then stated, and finally, to illustrate this approach, the example of a two-dimensional Ornstein--Uhlenbeck process, where all the computations are explicit is developed. The consequences on diffusion approximations of Epidemic models where computations are no longer explicit are detailed later.

%------------------------------------------------------------------------------------------------------
\subsection{Inference for  high frequency sampling of partial observations}
%\subsubsection{Notation for partial observations and assumptions}\label{Not_PO}
 Some specific notations need to be introduced.\\
  For $x \in \R^2$, $X^{\epsilon}(t)$, the diffusion process, $B(t)$ the Brownian motion, and $M $ a $2\times 2$ matrix,  we write
 \begin{equation}\label{yz}
 x=  \begin{pmatrix}x_1\\x_2\end{pmatrix};\;
X(t)=\begin{pmatrix}
X_1(t)\\
X_2 (t)
\end{pmatrix};\; B(t)=\begin{pmatrix}
B_1(t)\\B_2(t)
\end{pmatrix}; \; M=\begin{pmatrix}
M_{11} & M_{12}\\M_{21} & M_{22}
\end{pmatrix}.
\end{equation}
\noindent
For functions $f (\theta, x)$ defined for $x \in \R^2$, we use  \eqref{gradf} for differentiating with respect to $x$ and   \eqref{gradf}, \eqref{gradcompos}
for differentiation with respect to $\theta$.

\noindent
The observations are $(X_1(k\Delta),k=0,\dots n)$.
 Since $ x_2$ is not observed and unknown, we add it to the parameters. Therefore, setting  $x_2=\xi$, define using (S4),
\begin{equation}\label{theta}
 \eta=(\alpha, \xi) \in \R^{a+1};\quad \theta= (\alpha,\xi,\beta)= (\eta,\beta)\in \R^{a+b+1}.
\end{equation}
The quantities introduced in (\ref{TS}) depend on $\alpha$, $ \eta$ or  $ \theta$ and can be written, using \eqref{yz},
 %$z(\eta, t)= \begin{pmatrix} z_1(\eta, t)\\ z_2 (\eta, t)) \end{pmatrix}$,
%$x(\eta, t)= \begin{pmatrix}
%y(\eta, t)\\z (\eta, t)
%\end{pmatrix}$,\\
%$g (\theta,t)=\begin{pmatrix} g_1(\theta, t)\\g_2(\theta,t) \end{pmatrix}$,
%$R^{\epsilon}( \theta,t)=\begin{pmatrix} R_{1}^{\epsilon}(\theta,t)\\ R_2^{\epsilon}( \theta,t) \end{pmatrix}$, $\Phi(\eta, t,s) = \begin{pmatrix} \Phi_{11} &\Phi_{12}\\
%\Phi_{21}&\Phi_{22}\end{pmatrix}(\eta,t,s)$.
\noindent
The  expansion of $X(t) $  stated in  (\ref{TS})  yields that $X_1(t)$ satisfies, using notations (\ref{yz}),
  \begin{equation}\label{TaylorY}
 X_1(t)=z_1(\eta,t)+\epsilon g_1(\theta,t)+\epsilon^2 R_1^{\epsilon}(\theta,t) \mbox { with }
 \end{equation}
\begin{align}
g_1(\theta,t)&= \int_0 ^t \left( \Phi(\eta,t,u)\sigma(\beta,z(\eta,u)) \right)_{11} \;dB_1(u)\nonumber\\
&\quad + \left(\Phi(\eta,t,u)\sigma(\beta,z(\eta,u))\right)_{12} \;dB_2(u).\label{g1}
\end{align}
Using that $\Phi(t,u)=\Phi(t,s)\Phi(s,u) $ yields another expression for $g_1(\theta,t_k)$,
%as the sum of two independent random variables,%for $0\leq s\leq t$,

\begin{align}
g_1 (\theta,t_k) &=  \left( \Phi (\eta,t_k,t_{k-1}) g (\theta,t_{k-1})\right )_{1}\nonumber\\
 &\quad +  \int_{t_{k-1}}^{t_k}\left( \Phi(\eta,t,u)\sigma(\beta,z(\eta,u)) \right)_{11} \;dB_1(u)\nonumber\\
 &\quad +
\left(\Phi(\eta,t,u)\sigma(\beta,z(\eta,u))\right)_{12} \;dB_2(u).\label{g1bis}
\end{align}
%\subsection{Inference for  high frequency sampling of partial observations}
%\noindent
%In this case of partial observations, consistent and asymptotically normal contrast-based estimators are built. Parameter identifiability is also discussed.\\
\noindent
For estimating the unknown parameters, we use, instead of  a filtering approach, the stochastic expansion of $X(t)$, where the  unobserved component
$X_2(t) $ is substituted by its deterministic counterpart $z_2(\eta,t)$.
For building a tractable  estimation function, we also simplify the expression of $B_k(\alpha,X)$ (see (\ref{BkTS})) by replacing  $\Phi(\eta; t_k, t_{k-1})$ by its first-order approximation $I_2+ \Delta  \nabla_x  b(\alpha, z(\eta,t_{k-1}))$,
so that $\Phi_{11}(\eta, t_k,t_{k-1}) \simeq 1+\Delta \nabla_{x_1}  b_1(\alpha,z(\eta,t_{k-1}))$.

The path used in (\ref{BkTS}) is  $\begin{pmatrix}X_1(t)\\z_2(\eta,t)\end{pmatrix}$
leading, instead of $B_k( \alpha,X)$  to  $\displaystyle{\begin{pmatrix} A_k(\eta, X_1) \\ 0
\end{pmatrix} }$, with
%\sout{Using that $\Phi_{11}(t_k,t_{k-1}) \simeq 1+\Delta \frac{\partial b_1}{\partial y}(\alpha,x(\eta,t_{k-1}))$ yields}
\begin{equation}\label{Ak}
	A_k(\eta,X_1) = X_1({t_k})-z_1(\eta,t_k)-\left(1+\Delta \nabla_{x_1} b_1 (\alpha, z(\eta,t_{k-1}))\right)
	(X_1(t_{k-1})-z_1(\eta,t_{k-1})).
	\end{equation}
For a first approach, we  consider an estimation method  based on the Conditional Least Squares built on the $A_k(\eta,X_1)$'s defined in (\ref{Ak}).
\begin{equation}\label{BU}
{\bar U} _{\epsilon,\Delta}(\eta,X_1)= \frac{1}{\epsilon ^2 \Delta} \sum_{k=1}^n A_k(\eta, X_1)^2.
\end{equation}
This CLS functional does not depend on $\beta$, and therefore $\beta$  cannot be estimated using  ${\bar U} _{\epsilon,\Delta}$.
%and only $\eta=(\alpha, \xi)$ can be
estimated.
\noindent
The associated  estimators are then defined as
\begin{equation}\label{def:estimateurs_bar}
\bar{\eta}_{\epsilon,\Delta}=\underset{
\eta\in K_a\times K_z}{\mathrm{argmin}} \ \bar{U}_{\epsilon,\Delta}(\eta,X_1).
\end{equation}
%The case $\Delta \rightarrow 0$ is investigated. The case $\Delta$ fixed could be derived from the previous one, to be done in further study.
Note that this process could also be used  for estimating $\eta$ for fixed $\Delta$ and low frequency data, using  $\Phi_{11}(t_k,t_{k-1})$ instead of its approximation.\\

Assume that $\eta= (\alpha,\xi) \in \Theta$, with $\Theta$ compact set of $\R^a \times \R$.
% and denote by $K$ the parameter set of $\eta= (\alpha,\xi)$.
Denote by $\eta_0= (\alpha_0,\xi_0)$ the true parameter value and consider the estimation of $\eta$.
The distribution of $(X(t))$ satisfying \eqref{SDE2D} depends on $\theta=(\eta, \beta)$. Set $\theta_0= (\eta_0,\beta_0)$ and $\P_{\theta_0}$ the distribution of
$(X(t))$ on $(C([0,T],\R^2),\mathcal{ C}_T)$.

Let us first study $\bar{U}_{\epsilon,\Delta} (\eta,X_1)$.
\begin{lemma}\label{lemmeBU}
Assume (S1)--(S5). Then, the process $\bar{U}_{\epsilon,\Delta}(\eta,X_1)$ defined in (\ref{BU}) satisfies that, under ${\mathbb P}_{\theta_0}$,  as $\epsilon,\Delta \rightarrow 0$,
\begin{equation}\label{Kgamma}
	\epsilon ^2 \bar{U}_{\epsilon,\Delta}(\eta ,X_1) \rightarrow  J_T (\eta_0,\eta) =
\int_0^T (\Gamma_1(\eta_0,\eta;t))^2 dt \quad \mbox{a.s.  where}
\end{equation}
\begin{align}
 \Gamma_1(\eta_0,\eta; t)= b_1(\alpha_0,z(\eta_0, t))&-b_1(\alpha,z(\eta,t))\label{gamma1}\\
&-\nabla_{x_1}  b_1(\alpha, z(\eta,t))(z_1(\eta_0,t)-z_1(\eta,t)).\nonumber
\end{align}
\end{lemma}
\noindent

So, to get that $\bar{U}_{\epsilon,\Delta}(\eta,Y)$ is a contrast function for estimating $\eta=(\alpha,\xi)$,  we need an assumption that ensures that $\{\eta \neq \eta_0 \Rightarrow J_T(\eta_0,\eta)>0\}$. This leads to the additional identifiability assumption,\\
%using  (\ref{gamma1}), \\

\noindent
\textbf{ (S8)}:  $\eta \neq \eta_0  \Rightarrow  \{t \rightarrow \Gamma_1(\eta_0,\eta; t) \not\equiv 0 \}.$\\

\noindent
For deterministic systems, the notion of observability is used in the case of partial observations  (see e.g.\ \cite{poh78IV}, \cite{sed02IV}), which sums up to $\{\eta \neq \eta_0 \Rightarrow z(\eta,\cdot) \not\equiv z(\eta_0,\cdot)\}$. If the underlying deterministic system is not observable,  Assumption (S8) which makes reference to the identifiability of the model with respect to the parameters is not satisfied. But the converse is not true,  Assumption  (S8) being a bit  stronger.\\
\noindent
\begin{proof}
The proof of Lemma \ref{lemmeBU} is a repetition of  the proof of Lemma \ref{BkTS}. First, an application of the stochastic Taylor expansion yields that, as  $\epsilon\rightarrow 0, (X_1(t), 0 \leq t \leq T)
\rightarrow (z_1(\eta_0,t), 0 \leq t \leq T)$ almost surely under   ${\mathbb P}_{\theta_0}$. Second, letting $\Delta \rightarrow 0$, we get that,
there exists a constant $C>0$ such that
%uniformly with respect to $\eta = (\alpha,\xi )$,
\begin{equation}\label{lim Ak}
 \frac{1}{\Delta} A_k(\alpha, z_1(\eta_0,\cdot))= \Gamma_1(\eta_0,\eta,t_{k-1}) + \Delta \norm{ \eta-\eta_0} r_k(\eta_0,\eta),
\end{equation}
 with $\sup_k \sup_{\eta \in \Theta} \norm{r_k(\eta_0,\eta)} \leq C$.
%$$sup_{k=1,\dots n} ||\frac{A_k(\eta,y(\eta_0,\cdot))}{\Delta}-\Gamma_1(\eta_0,\eta; t_{k-1})||\rightarrow 0.$$
\end{proof}

To study the asymptotic behaviour of $\bar{\eta}_{\epsilon,\Delta}$, we have to introduce additional quantities.
First, we  define the vector $D(\eta, t) \in\R^{a+1}$, using  the notations defined in \eqref{gradf},
\begin{align}
D_i(\eta,t)
= &-(\nabla_{\alpha_i} b_1)(\alpha,z(\eta,t))-\nabla _{x_2} b_1(\alpha, z(\eta,t)) \nabla_{\alpha_i} z_2( \eta,t) \quad \mbox{ for } i=1,\dots,a,\nonumber \\
D_i(t)= &-\nabla_{x_2} b_1 (\alpha, z(\eta,t))\nabla_{\xi}  z_2 (\eta,t)\quad\mbox{ if } i=a+1,    \label{Di}
\end{align}
Then, built on the $D_i$'s, define the matrix $\Lambda(\eta)= (\Lambda_{ij}(\eta))$ by
\begin{equation}\label{lambda}
 \Lambda_{ij}(\eta)= 2
\int _0 ^T D_i(\eta,t)D_j(\eta, t)\; dt.
\end{equation}
\noindent
Finally, define the three functions for $\theta= (\alpha,\xi,\beta)$,
\begin{align}
v_1(\theta;t)&= \sigma_{11}^2(\beta, z(\eta,t))+\sigma_{12}^2(\beta, z(\eta,t))\nonumber\\
&= \Sigma_{11}(\beta,z(\eta,t)),\nonumber\\	
v_2(\theta;t,s)&=  \sigma_{11}(\beta,z(\eta,s))
\left(\Phi(\eta, t,s)\sigma(\beta,z(\eta,s))\right)_{21}\nonumber\\
&\quad + \sigma_{12}(\beta, z(\eta,s))
\left(\Phi(\eta,t,s)\sigma(\beta, z(\eta,s))\right)_{22}\nonumber\\
 &= \left(\Phi(\eta;t,s)\Sigma(\beta, x(\eta,s))\right)_{21},\nonumber\\
v_3(\theta,t,s )=& \int_0^{t \wedge s} \left(\Phi(\eta,t,u)\sigma( \beta,z(\eta,u))\right)_{11}
\left(\Phi(\eta,s,u)\sigma(\beta, z(\eta,u))\right)_{11} du\nonumber \\
  &\quad + \int_0^{t \wedge s} \left(\Phi(\eta,t,u)\sigma(\beta, z(\eta,u))\right)_{22}\left(\Phi(\eta,s,u)\sigma( \beta,z(\beta,u))\right)_{22} du.\label{vi}
\end{align}

\noindent
We can now state the main result of this section.

\begin{theorem}\label{limiteta}
 Assume (S1)--(S8). Then under $\P_{\theta_0}$, as $\epsilon,\Delta \rightarrow 0$,
\begin{enumerate}
\item[(i)] $ \bar{\eta}_{\epsilon,\Delta} \rightarrow \eta_0$ in probability .\\
\item[(ii)]  If moreover $ \epsilon^2 \Delta^{-1}= n\epsilon^2\rightarrow 0$ and
$\Lambda (\eta_0)$ defined in (\ref{lambda}) is invertible, then
\begin{equation}\label{CVL}
	\epsilon ^{-1}(\bar{\eta}_{\epsilon,\Delta}-\eta_0) \rightarrow
\mathcal{ N}(0,\Lambda (\eta_0)^{-1}V(\theta_0) \Lambda (\eta_0)^{-1}) \quad
\mbox{in distribution},
\end{equation}
where $V(\theta)= V^{(1)}(\theta)+ V^{(2)}(\theta)+ V^{(3)}(\theta)$ with, using (\ref{Di}), (\ref{vi}),
\begin{align}
&V_{ij}^{(1)}(\theta)= \int _0 ^T D_i(\eta,t)D_j(\eta, t) v_1(\theta,t) \; dt,\\
&V_{ij}^{(2)}(\theta) =   \int\int _{0\leq s \leq t \leq  T } \; D_i(\eta,s) D_j(\eta,t) \nabla_{x_2} b_1 (\alpha, z(\eta,s)) v_2(\theta,t,s)) ds\; dt, \\
&V_{ij}^{(3)}(\theta)=\\
&\int _0^T  \int _{ 0}^T  D_i(\eta,s) D_j(\eta, t) \nabla_{x_2}  b_1 (\alpha, z(\eta,s)) \nabla_{x_2} b_1 (\alpha, z(\eta,t)) v_3(\theta,t,s) ds\; dt.\nonumber
\end{align}
\end{enumerate}
\end{theorem}

Based on \eqref{Kgamma} and Assumption (S8), the proof of the consistency of $\bar{\eta}_{\epsilon,\Delta}$ is obtained by standard tools and omitted.

\noindent
For the proof of (ii), the main difficulty
%the proof of Theorem \ref{limiteta}
 lies in a precise study of  $\epsilon \nabla _i \bar{U}_{\epsilon,\Delta} (\eta_0,X_1)$, which is the sum of $n$ terms that are no longer conditionally independent. The three terms in the matrix $V(\theta_0)$ come from this expansion. Indeed,
\begin{equation}\label{DU}
\epsilon (\nabla _i \bar{U} (\eta_0,Y))_{i }, \rightarrow \mathcal{ N}_{a+1}\bigl(0, V(\theta_0)\bigr) \quad \mbox{ in distribution under  } \P_{\theta_0}.
\end{equation}
Then, studying $\epsilon^2\nabla_{ij} \bar{U} (\eta,Y)$ yields,
using  (\ref{BU}), (\ref{Di}), as $\epsilon,\Delta \rightarrow 0$,
\begin{equation}\label{nablaijU}
\epsilon^2	\nabla_{ij} \bar{U} (\eta_0,Y) \rightarrow  \Lambda_{ij}(\eta_0)= 2
\int _0 ^T D_i(\eta_0,t)D_j(\eta_0, t) dt \quad \mbox{a.s. under } \P_{\theta_0}.
\end{equation}
The proof is quite technical and is omitted.\\

%------------------------------------------------------------------------------------------------------
%\subsection{Two-dimensional  Ornstein--Uhlenbeck process: an explicit case}
Let us describe our method on a partially observed  two-dimensional Ornstein--Uhlenbeck  diffusion process $X(t)= (X_1(t), X_2(t))^*$ where all the  computations are explicit. Let \\
%Consider the diffusion $X(t)= (X_1(t), X_2(t))^*$ satisfying the SDE,
%$X(t)=\begin{pmatrix}
%X_1(t)\\X_2(t)
%\end{pmatrix}$
\begin{equation}\label{OU2D}
dX(t)= AX(t)dt+\epsilon \varsigma dB(t),\quad X(0)=\begin{pmatrix} x_1\\ x_2 \end{pmatrix} ,
\end{equation}	
with	$A=\begin{pmatrix}
a & b \\0 & a+h
\end{pmatrix}, \varsigma=\sigma\begin{pmatrix}
1 & 0 \\0 & 1
\end{pmatrix}.
$\\

\noindent
We assume that $h\neq 0$, $\sigma >0$. The parameter in the drift is $\alpha=(a,b,h)$. For  partial observations, we also need introducing $\eta= (a,b,h,\xi)$ and  $\theta= (a,b,h,\xi ,\sigma)$. The observations are
 $(X_1(t_k), k=1,\dots,n)$ with $t_k=k\Delta$, $T=n\Delta $ and $\Delta=\Delta_n\rightarrow 0$. \\
The solution of the ODE \eqref{ODEPARTIV} applied to the drift of diffusion process (\ref{OU2D}) is
\begin{equation}\label{OUODE}
z_1(\eta,t)=( z_1- \frac{\xi b}{h})e^{at}+\frac{\xi b}{h}e^{(a+h)t}; \quad z_2(\eta,t)=\xi e^{(a+h)t}.
\end{equation}
%Assumption (S1)--(S7) are satisfied. Looking at the analytical expression of $yz_1(\eta,t)$, we have that $b\xi= \tilde{b}\tilde{\xi}$ leads to identical solutions $z_1( \eta;t)$. Therefore, assumption (S8) is not satisfied and it is impossible to estimate $b$ and $\xi$ separately when observing one coordinate only. Moreover, this would also holds for the continuous observation case: the non-identifiability is an intrinsic problem to the partial observation case.
%So, we have to  define a new parameter $b' = b\xi$ and  consider that the parameter  to estimate is now $ \eta=( a,b',h)$. Then, it is easy to check that (S8) is now satisfied. \\
Let us compute  the matrix $\Phi(\alpha,t, u) = e^{(t-u)A}$, we have  $A=PDP^{-1}$, with
$$P=\begin{pmatrix}
1 & b/h \\ 0 & 1
\end{pmatrix},\; D=\begin{pmatrix}
a & 0\\0& a+h
\end{pmatrix}, \mbox{ so that }$$
$$\Phi(\alpha,t,s)=\begin{pmatrix}
e^{a(t-s)} & \frac{b}{h}\bigl(e^{(a+h)(t-s)}-e^{a(t-s)}\bigr) \\ 0 & e^{(a+h)(t-s)}
\end{pmatrix}.$$
%Then,
%$\Phi(\alpha,t,s)=\begin{pmatrix}
%e^{a(t-s)} & \frac{b}{h}\bigl(e^{(a+h)(t-s)}-e^{a(t-s)}\bigr) \\ 0 & e^{(a+h)(t-s)}
%\end{pmatrix}$.
The solution of (\ref{OU2D}) is therefore
$X(t)=  Pe^{tD}P^{-1}X(0)+\epsilon\sigma \int_{0}^{t}Pe^{(t-s)D}P^{-1}dB(s)$. Hence,
\begin{equation}\label{OUY}
X_1(t)= z_1(\eta,t)+\epsilon \sigma \left(\int_{0}^{t}\;e^{a(t-s)}dB_1(s) +\frac{b}{h} \int_{0}^{t}\;
(e^{(a+h)(t-s)}-e^{a(t-s)}) \;dB_2(s)\right).
\end{equation}
Using that $\nabla_{x_1} b_1(\alpha, z(\eta,t))= a$  and  (\ref{OUODE}) yields that
\begin{equation}\label{OU:Ak}
A_k(\eta,X_1) = X_1({t_k})-z_1(\eta, t_k)-(1+a\Delta)
\left(X_1({t_{k-1}})-z_1(\eta,t_{k-1})\right).
  \end{equation}
   %The function $\Gamma_1(\eta_0,\eta,t)$ writes
   $$ \Gamma_1(\eta_0,\eta,t)=  (a_0-a) z_1(\eta_0,t)+ b_0 \xi_0 e^{(a_0+h_0)t}- b\xi e ^{(a+h) t}.$$
 Assumptions (S1)--(S7) are satisfied. Looking at the analytical expression of $z_1(\eta,t)$, we have that $b\xi= \tilde{b}\tilde{\xi}$ leads to identical solutions
 $z_1( \eta,t)$. Therefore, Assumption (S8) is not satisfied and $b$, $\xi$ cannot be estimated separately when observing one coordinate only.
This is also true for the deterministic ODE and the non-identifiability is here an intrinsic problem to this  partial observation example.
%Moreover, this would also holds for the continuous observation case:

Therefore, we   define a new parameter $b' = b\xi$ and  consider that the parameter  to estimate is now $ \eta=( a,b',h)$. Then,
checking  (S8) is straightforward. \\
%\sout{Since only $bz_0$ is identifiable in this partial observation set-up, we define  $b'=bz_0$ and $\eta'= (a,b',h)$. Assumption $(S8)$ is now satisfied.}
The various quantities introduced in the previous section have a closed expression. Indeed,
the functions $D_i(\eta, t)$ defined in (\ref{Di}) write, using (\ref{nablaijU}), (\ref{OUODE}) with $ \eta=( a,b',h)$,
\begin{align*}
D_1(\eta , t)&=-(z_1- \frac{b'}{h})e^{at}-(\frac{b'}{h}+b't)e^{(a+h)t},\\
D_2(\eta,t)&=- e^{(a+h)t},\\
D_3 (\eta,t)&= -b' t e^{(a+h)t}.
\end{align*}
%\ D_a(\eta', t)= - z_1(\eta',t)-b't e^{(a+h)t}=-(z_1- \frac{b'}{h})e^{at}-(\frac{b'}{h}+b't)e^{(a+h)t}\; ; $$\\
%$$D_2(\eta',t) = D_{b'}(\eta', t)= - e^{(a+h)t}, \; D_3 (eta',t) = D_h(\eta',t)
%= -b' t e^{(a+h)t}.$$
The matrix $\Lambda (\eta)$ is defined as  $\Lambda (\eta)= (\Lambda_{ij}(\eta)) $ with $\Lambda_{ij}(\eta)= \int_0 ^T D_i(\eta,t) D_j(\eta, t) dt$
($= \int_0 ^T D(\eta,t) D^*(\eta, t) dt$.
%$ $(\Lambda_{ij}(\eta))_{1\leq i,j \leq 3}$, with
%$\Lambda_{ij}(\eta)= \int_0 ^T D_i(\eta,t) D_j(\eta, t) dt$. \\
The functions defined in (\ref{vi}) are, with $\theta=(a,b,h,\sigma)$,
$$v_1(\theta,t)= \sigma^2;\quad v_2(\theta,t,s)= 0;\quad
v_3(\theta,t,s)=\sigma ^2\left(\frac{e^{a|t-s|}}{2a} +\frac{e ^{(a+h) |t-s|}}{2(a+h)} \right).$$
Therefore,
%$v_3(\theta,t,s)=\frac{\sigma ^2}{2}  b^2\int_0^{t \wedge s} e^{(a+h)(t+s-2u)}du
%= \frac{\sigma ^2 b^2}{2(a+h)} e ^{2(a+h)(t\wedge s)}$. \\
\begin{align*}
V_{ij}(\theta)&= \sigma^2 \int _0 ^T D_i(\eta,t)D_j(\eta, t) dt \\
&\quad +\frac{\sigma^2 b ^2 }{2} \int_0^T \int_0^T D_i(\eta,s)D_j(\eta,t) \left(\frac{e^{a|t-s|}}{a} +\frac{e ^{(a+h) |t-s|}}{(a+h)} \right) ds dt.
\end{align*}
%$V_{ij}(\theta)= \sigma^2 \int _0 ^T D_i(\eta,t)D_j(\eta, t) dt + \frac{\sigma ^2 b^2}{2(a+h)}\int_0^T\int_0^T D_i(\eta,t)D_j(\eta, s)e ^{2(a+h)(t\wedge s)}ds\; dt.$\\
The estimator $\bar{\eta}_{\epsilon,\Delta}$ defined by (\ref{def:estimateurs_bar}) is a consistent estimator of $\eta_0=( a_0,b'_0,h_0)  $ and satisfies (\ref{CVL}) with the matrices $\Lambda(\eta_0)$ and $V(\theta_0)$ obtained above.
The asymptotic covariance matrix is therefore
\begin{align}
&\sigma ^2 \Lambda^{-1}(\eta)+ \label{loss}\\
&\frac{\sigma ^2 b^2}{2 }\Lambda^{-1}(\eta)\left(\int_0^T\int_0^T D_i(\eta,t)D_j(\eta, s)
\left(\frac{e^{a|t-s|}}{a} +\frac{e^{(a+h )|t-s|}}{a+h} \right) ds dt \right)_{ij}\Lambda^{-1}(\eta).\nonumber
\end{align}
%e ^{2(a+h)(t\wedge s)}ds dt
%REMARQUE:Previous section:
%$$I_b= \frac{1}{\sigma^2}\int_0 ^T \begin{pmatrix}z_1^2+z^2_2& z_1 z_2& z_2^2\\ z_1 z_2&z_2^2&0\\
% z_2^2&0 & z_2^2\end{pmatrix} dt.$$
%

In the case of complete discrete observations, the first term  of (\ref{loss}) is the asymptotic variance obtained with conditional least squares. Therefore, the loss of information coming from partial observations is measured by the second term of  (\ref{loss}) (added to the fact that only $bz_0$ is identifiable).

\subsection{Assessment of estimators on simulated and real data sets}
We first present the results on the  $SIR$  studied in the previous section but assuming partial observations.
Then we investigate the inference on the real data set of Influenza dynamics modeled with the $SIRS$ studied in the previous section.
\subsubsection{Inference for partial observation of the $SIR$ model  with sampling interval $\Delta$}
\label{sec:SIRpart}
\noindent
In this section, we consider the case where one component of the epidemic process $ X^N(t)= (S^N(t),I^N(t))$  is observed on $[0,T]$. The observations are the successive numbers of infected individuals
$$(I^N(k\Delta), k =1,\dots n) \mbox{ with sampling  }\Delta ; T=n\Delta.$$
\noindent

According to the notations of Section \ref{sec:partial}, we have to interchange the coordinates of  $S,I$ and set
 $ X(t) = (I(t),S(t))^*$; the drift term can be written as
 $$ X(t) = \begin{pmatrix}I(t)\\S(t) \end{pmatrix} ;\quad b((\lambda,\gamma),(i,s))=\begin{pmatrix}
\lambda si -\gamma i\\-\lambda si \end{pmatrix};\quad \Sigma(i,s)= \begin{pmatrix} \lambda si +\gamma i&-\lambda si \\
-\lambda si &\lambda si \end{pmatrix}.$$
 We assume that $I(0)= i_0, S(0)=s_0$.  Setting $ \xi= s_0$, then  the parameter defined in the previous section is  $\eta =(\lambda, \gamma,\xi)$.
 Denote by $z(\eta,t)= (i(\eta,t),s(\eta,t)) $ the solution of the  ODE
 $$di/dt= \lambda s i -\gamma i;   i(0)=i_ 0,\quad ds/dt=\lambda si;  s(0)=\xi.$$
 Then, the conditional least square method now reads as
$$ \bar{U}_{\epsilon,\delta}(\eta,I)= \frac{1}{\epsilon^2 \Delta} \sum _{k=1}^n (I(t_k)- i(\eta,t_k)-\left( 1+ \lambda s(\eta,t_{k-1})-\gamma\right)
(I(t_{k-1}) -i(\eta,t_{k-1}))^2.$$
Using definition \ref{gamma1}, the function $\Gamma_1(\eta_0,\eta,t)$ reads as
$$ \Gamma_1(\eta_0,\eta,t)= i(\eta_0,t)(\lambda_0 s(\eta_0,t)- \lambda s(\eta,t) -\gamma_0+\gamma).$$
To investigate the identifiability assumption, let us  check \textbf{ (S8)}. It reads as
$\eta \neq \eta_0 \Rightarrow  \{t\rightarrow \Gamma(\eta_0,\eta,t)\} \not\equiv 0.$\\

Assume  that we have observed that  the epidemic spreads, so  that we have $\forall t\in [0,T], i(\eta_0,t)>0$ .
Therefore, we have to prove that
\begin{equation}\label{rel:GammaSIR}
\{t \rightarrow  (\lambda_0 s(\eta_0,t)- \lambda s(\eta,t) -\gamma_0+\gamma)\equiv 0 \} \Rightarrow \{\eta=\eta_0\} .
\end{equation}
Differentiating this relation with respect to $t$ yields
\begin{equation}\label{rel:lambdas}
 \forall t, \lambda_0 ^2 s(\eta_0,t)i(\eta_0,t) - \lambda^2 s(\eta,t) i(\eta,t)=0 .
 \end{equation}
Using (\ref{rel:GammaSIR}), we get the second relation
$$\forall t, \frac{s(\eta,t)}{i(\eta_0,t)}  (\lambda i(\eta,t)- \lambda_0  i(\eta_0,t)) = \frac{\lambda_0(\gamma_0-\gamma)}{\lambda}.$$
 Differentiating this relation with respect to $t$ yields that \\
$$\lambda \frac{s(\eta,t) i(\eta,t)}{i(\eta_0,t)} (\lambda_0 i(\eta_0,t)-\lambda i(\eta,t) ) \equiv 0.$$
Since at time $0$,  $i(\eta,0)=i(\eta_0,0)=i_0$, we get that $\lambda=\lambda_0$. Using now \eqref{rel:lambdas} yields that, at time $0$,
$s(\eta, 0)= s(\eta_0,0)$ so that $\xi=\xi_0$. Finally, by relation \eqref{rel:GammaSIR}, we get $\gamma=\gamma_0$.

We conclude that  the two parameters $\lambda ,\gamma$  as well as the initial state $s_0$ are identifiable when observing $(I(t_k), k=0,\dots n)$.
The same holds true for $R_0=\lambda/\gamma$, $d=1/\gamma$ and $s_0$.

%----------------------------------------------------------------------
%\subsection{Comparison between theoretical and empirical ellipsoids}

Performances of estimators in the case of partially observed $SIR$ model are assessed on simulations obtained with the following parameters: $N=10000$, $R_0=1.5$, $d=3$, $s_0=0.97$, $T=40$. Observations are represented by vector $I^N(k \Delta)$. Estimations of parameters $(R_0,d,s_0)$ are performed on $1000$ simulated trajectories. Theoretical and empirical confidence ellipses are built as detailed in the introduction of Section \ref{sec:simul}. \\
\noindent
As shown in Figure \ref{fig:SIRpart_themp}, confidence ellipsoids are quite large in the case of partial data. However, they do not include unreasonable values from the epidemiological point of view.
Quantile based empirical $95\%$ confidence intervals are
still quite large.
%more narrow: $[1.38, 5.77]$ for $R_0$, $[2.64, 5.90]$ for $d$ and $[0.34, 0.99]$ for $s_0$.
\begin{figure}[ht]
\includegraphics[width=12cm, height=12cm, trim=1cm 6cm 1cm 6cm]{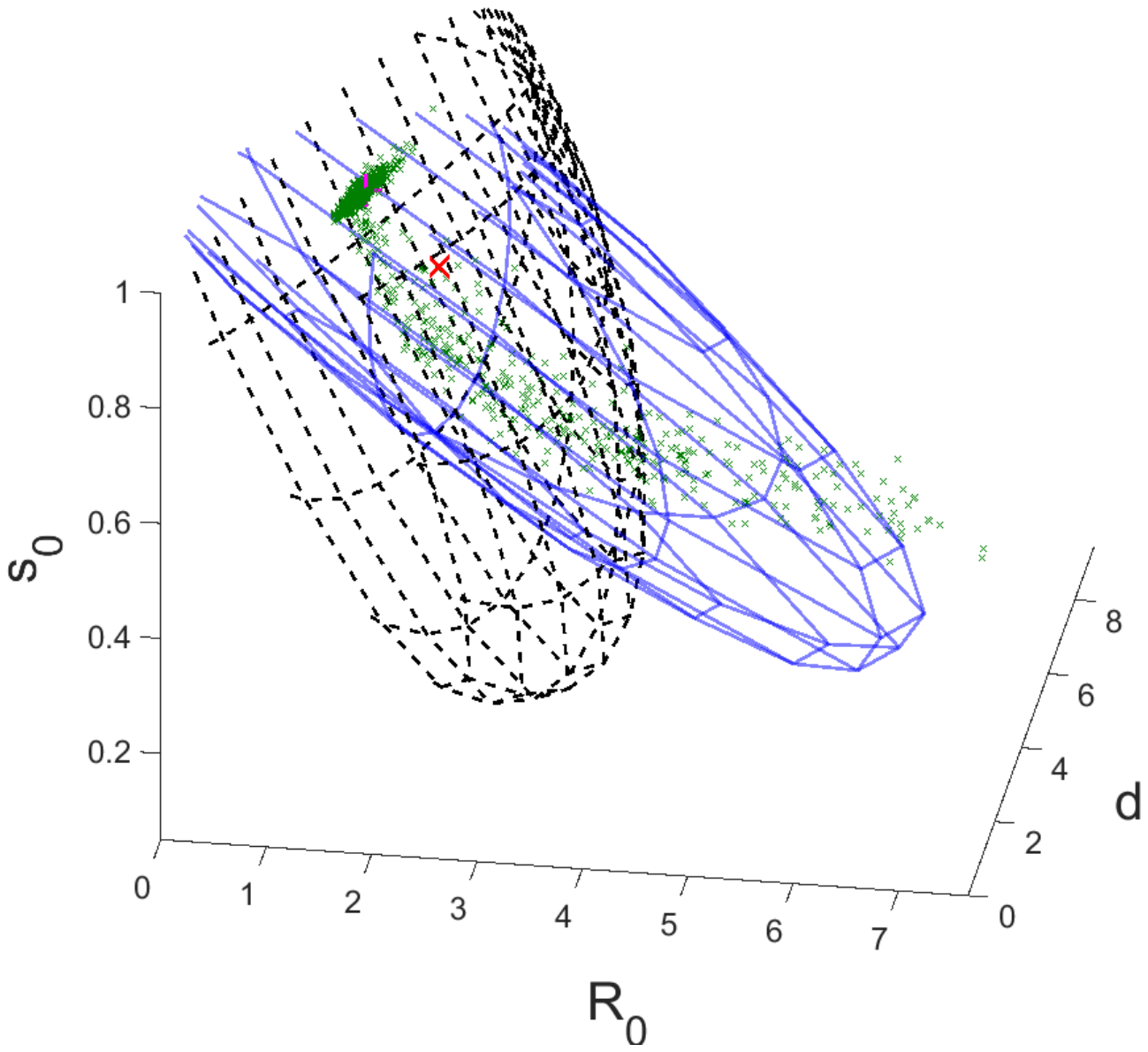}
\vspace{-1cm}
\caption{Point estimators (green) computed by averaging over 1000 independent simulated trajectories of the $SIR$ stochastic model, partially observed ($I(k\Delta)$ only)  for $(R_0,d,s_0)=(1.5,3,0.15,0.97)$, $T=40$ days and $N=10000$. Theoretical confidence ellipsoid (black), centered on the true value and empirical confidence ellipsoid (blue), centered on mean estimated value are provided. Both ellipsoids are truncated at plausible limits on each direction. Mean and median point estimator are $(R_0,d,s_0)=(1.89, 3.43, 0.88)$  (red cross) and $(1.54, 3.24, 0.99)$ (purple cross), respectively.}
\label{fig:SIRpart_themp}
\end{figure}

The relatively unexpected large volume of confidence ellipsoids, obtained despite theoretical identifiability of model parameters when observing only one component of the system (here $I_N(k \Delta)$) is probably due to the fact that the numerical variance-covariance matrix is ill-conditioned (the order of magnitude of the third eigenvalue is $100$ times smaller than that of the first two eigenvalues.  It probably corresponds to the notion of  ``Numerical Identifiability'', which does not
necessarily coincide with ``Theoretical Identifiability''.

Concerning point estimators, we successively considered the mean and the median of the estimators obtained for the $1000$  simulation experiments.
Assuming  the complete observation of both coordinates of the $SIR$ jump process yields, as expected, accurate values for $R_0,d$. Assuming  that only
$I(k\Delta), k=1,\dots n$ with $n=40$, we obtain for a true parameter value $(1.5,3,0.97)$ that the mean point estimator is $ (1.89,3.43, 0.88) $ and for the median estimator $(1.54,3.24, 0.99)$.

\subsubsection{Partial observations: $SIRS$ model, real data on influenza epidemics}
\label{sec:SIRSgrippe}
The performances of the contrast estimators for the case where only one coordinate of a diffusion process is observed are evaluated on data related to influenza outbreaks in France, collected by the French Sentinel Network (FSN), providing surveillance for several health indicators (www.sentiweb.org). These data are represented by numbers of individuals seeing a doctor during a given time interval, for symptoms related to influenza infection and are reported by a group of general practitioners (GP) voluntarily enrolled into the FSN. Several levels of errors of observation are associated to these data: (i) the state of individuals consulting a GP from the FSN is not exactly known: it can be assimilated to a new infection or to a new recovery, given that symptoms and infectiousness are not necessarily simultaneous and that a certain delay occurs between symptoms onset and consultation time (more correctly, the observed state is probably ``infected'' but not ``newly infected''); (ii) not all infected individuals go and see a GP; (iii) the GP's supplying the FSN database represent only a proportion of all French GP's; (iv) the exact dates of consultations are not known, data are aggregated over two-week time periods; (v) data are preprocessed by the FSN to produce observations with a daily time step.\\
\noindent
Here, we account partly for (i) on one hand and jointly for (ii) and (iii) on the other hand and assume that observations $Y({t_k})$ represent a proportion of daily (observation times $t_k=k\Delta$, with $\Delta=1$ day) numbers of newly recovered individuals: $Y({t_k})=\rho \gamma I(t_k)$, where $\rho$ can be interpreted as the reporting rate. Since data are available over several seasons of influenza outbreaks (data from 1990 to 2011, hence $[0,T]=[0,21.5]$ years), an appropriate model allowing to reproduce periodic dynamics is the $SIRS$ model described in Section \ref{SIRS_appli}.\\
%Assuming again a constant population size, we obtain a new two dimensional system with four transitions for the corresponding Markov jump process:
\begin{eqnarray*}
(S,I)\stackrel{\frac{\lambda(t)}{N} S(I+N\eta)}{\longrightarrow}(S-1,I+1) &; & \quad (S,I)\stackrel{\mu S}{\longrightarrow}(S-1,I) ;\\
(S,I)\stackrel{(\gamma+\mu)I}{\longrightarrow}(S,I-1) &;& \quad (S,I)\stackrel{\mu N+\delta(N-S-I)}{\longrightarrow}(S+1,I).
\end{eqnarray*}
The seasonality in transmission is modeled via $\lambda(t)= \lambda_0(1+\lambda_1 sin(2\pi t/T_{per}))$.

The parameter is  $\theta=(\lambda_0,\lambda_1,\gamma,\delta,\eta,\mu)$, the associated drift function
$b(\theta,t,(s,i))$ and diffusion matrix $\Sigma(\theta,t, (s,i))$ are
\begin{equation}\label{bsirs}
	\begin{array}{ll}b(\theta,t, (s,i))
&=\begin{pmatrix}-\lambda(t) s(i+\eta) +\delta(1-s-i)+\mu(1-s)\\ \lambda(t) s(i+\eta)-(\gamma+\mu) i \end{pmatrix}
\end{array},
\end{equation}
\begin{equation}\label{sigmasirs}
	\begin{array}{lll}\Sigma(\theta,t, (s,i))&=
\begin{pmatrix}\lambda(t) s(i+\eta)+\delta(1-s-i)+\mu(1+s)&-\lambda(t) s(i+\eta)\\-\lambda(t) s(i+\eta)&\lambda(t) s(i+\eta)+(\gamma+\mu) i\end{pmatrix}\end{array}.
\end{equation}
In summary, the data used are assumed to be discrete high frequency observations of one coordinate of the following two-dimensional diffusion with small variance:
\[\left\{\begin{array}{rl}
dS(t)& = -\lambda(t) S(t)  (I(t) +\eta) +\delta(1-S(t)-I(t)+\mu(1-S(t))) dt \\
&\quad + \frac{1}{\sqrt{N}} (\sigma_{11}dB_1(t)+ \sigma_{12}dB_2(t)) \\
dI(t ) & = (\lambda(t) S(t)(I(t)+\eta)-(\gamma+\mu) I(t) dt +\frac{1}{\sqrt{N}} (\sigma_{21}dB_1(t) + \sigma_{22}dB_2(t)).
\end{array}\right.\]

The vector of parameters to be estimated is $\alpha=(R=\lambda_0/\gamma,10 \lambda_1,d=1/\gamma,\delta_{per}=1/\delta T_{per},10\rho)$, where parameters are defined in equation \eqref{lambdat} and more generally in the entire Section \ref{SIRS_appli}. Parameters $\eta$, $\mu$ and $T_{per}$ are fixed at plausible values: $\eta=10^{-6}$, $\mu=\frac{1}{50}$ (years$^{-1}$) and $T_{per}=365$ days. The starting point  of the ODE system is unknown, but since we are interested in the stationary behaviour of this process, we fix ($r_{-20T_{per}}=0.27,i_{-20T_{per}}=0.0001$, see \cite{cau08IV} for example) and let the system evolve until $t=0$ for the tested set of parameter $\alpha$ to obtain our initial starting point. \\

\noindent
Estimation results are summarized in Figure \ref{fig:sentinelles_fit}, which represents multi-annual dynamics of influenza cases: observed dynamics (blue curve) and simulated ones (using the ODE version of the $SIRS$ model based on estimated parameter values; red curve).
%The estimator associated to the contrast process defined in \eqref{BU} is obtained using classical simplex algorithm for the minimisation steps (blue curve). Estimations based on another contrast expected to reduce the variance, derived from \eqref{BU}, with weighted sum (by a coefficient which could be taken equal to $(\Sigma^{-1})_{11}$), a correction term and accounting for $\beta=\alpha$ are also provided (red curve).
Estimators are associated to contrast process defined in \eqref{BU}. Point estimates of parameters  are: $(R,10 \lambda_1,d,\delta_{per},10\rho) = (1.47,1.94, 2.20, 5.66, 0.87)$. These values are in agreement with independent estimation based on data from the same database but using a different inference method, the maximum iterating filtering proposed by \cite{bre09IV}  (personal communication S.\ Ballesteros).
As shown in Figure \ref{fig:SIRpart_themp} for the $SIR$ model, widths of theoretical confidence intervals for each parameter should be larger than those corresponding to complete observations of the $SIRS$ model (drawn in Figure \ref{fig:SIRSafterbif}). In particular, for $\lambda_1$, the width of the confidence interval for partial observations will be larger than $0.35 * \sqrt{(10^7 / 6*10^7)} = 0.14$ (after correction for the population size, which is $N=10^7$ in Figure \ref{fig:SIRSafterbif} and $N=6*10^7$ in Figure \ref{fig:sentinelles_fit}).\\
%$(\frac{\lambda_0}{\gamma},\lambda_1,\frac{1}{\gamma},\frac{1}{T_{per}\delta},\rho)=(1.43,0.13,2.73,1.84,0.28)$ for contrast \eqref{BU} and $(\frac{\lambda_0}{\gamma},\lambda_1,\frac{1}{\gamma},\frac{1}{T_{per}\delta},\rho)=(1.98,0.16,2.16,1.81,0.25)$ for the modified contrast.
\noindent
We can notice from Figure \ref{fig:sentinelles_fit} that predicted trajectories correspond to a regime with bi-annual cycles, composed of two different peaks (red curve). The bifurcation diagram with respect to $\lambda_1$ (similar to Figure \ref{fig:SIRSbif}), when the remaining parameters are either set to fixed values (defined in this section) or to estimated values, exhibits the bifurcation from one annual cycle to bi-annual cycle at $\lambda_1=0.035$. This value is likely to belong to the confidence interval of estimated $\lambda_1=0.19$, since the width of this interval should be greater than 0.14. Hence, this can have some influence on estimation, influence which is not well characterized in the literature for models exhibiting bifurcation profiles, especially for trajectories corresponding to parameter values close to the bifurcation point.
We also observe that the smaller peak in the bi-annual cycles is underestimated, leading to almost no epidemic burst every other year. The presence of a bifurcation in the $SIRS$ ODE model   probably requires a better approximation of the original jump point process.

\begin{center}
\begin{figure}[ht]
\includegraphics[width=12cm, height=9cm, trim=2cm 7cm 1cm 7cm]{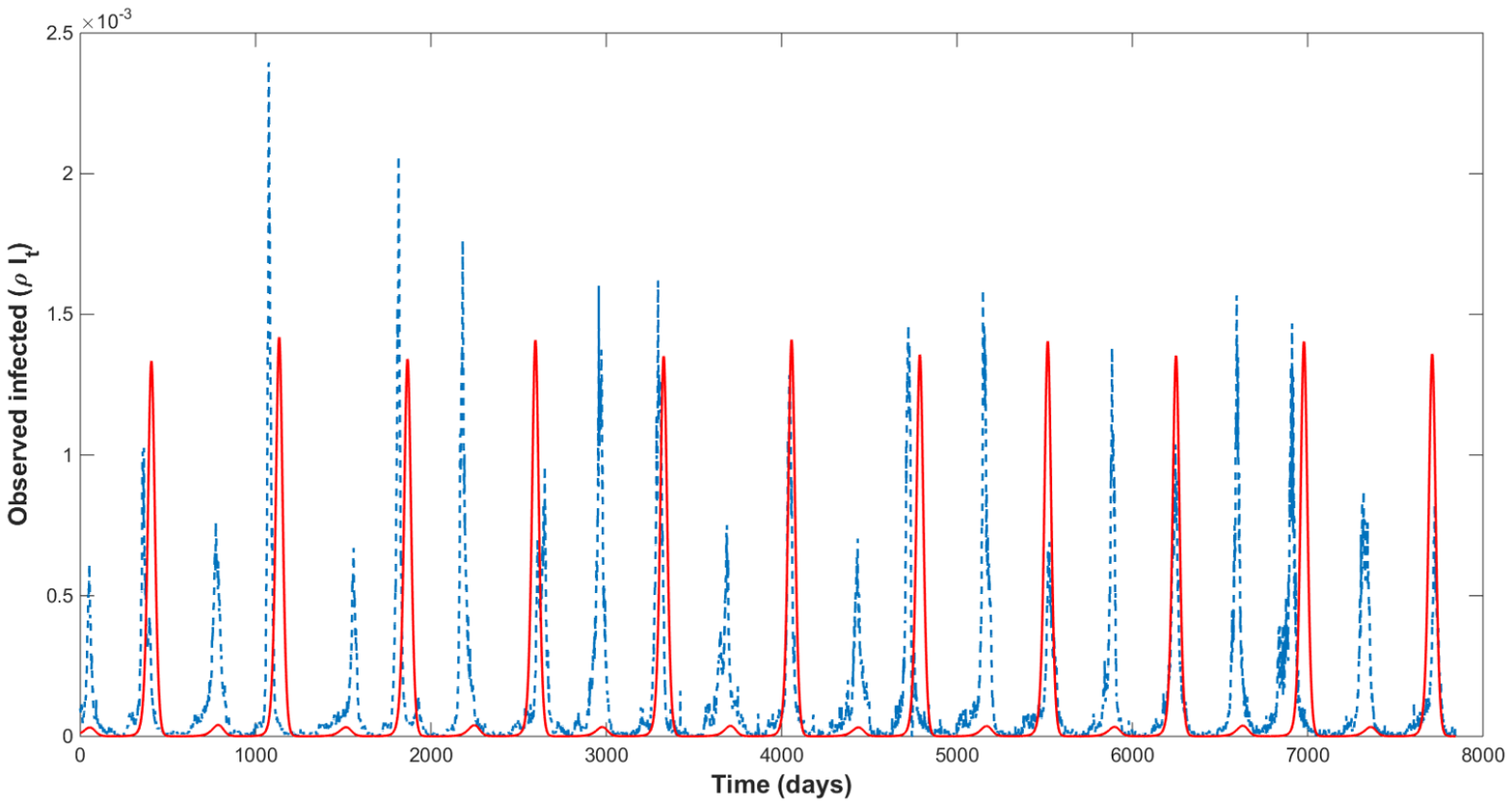}
\vspace{-1cm}
\caption{Time series of reported cases (expressed as a fraction of the total population in France) of influenza-like illness provided by the FSN (www.sentiweb.org) (blue curve) and deterministic trajectories (mean behaviour) predicted by the $SIRS$ model based on estimated parameters using contrast \eqref{BU} (red curve).}
\label{fig:sentinelles_fit}
\end{figure}
\end{center}
%

%\begin{figure}[ht]
%	\centering
%	\includegraphics{sfds.png}
%	\caption{Exemple d'insertion d'une image}
%	\label{fig::logo}
%\end{figure}

\subsubsection{Discussion and concluding remarks}
Several  extensions of this study are possible for partial observations. First, we have chosen to detail the case of high sampling interval. The study in the case of a fixed sampling interval $\Delta$ should  be obtained with similar tools, leading to similar results. Another extension concerns our choice of  a Conditional Least squares  for ${\bar U}_{\epsilon,\Delta}$. An estimation criterium similar to the one used in Section \ref{HFO} could be studied, using   $S_k(\alpha, \beta)$ (see \eqref{SkTS}) or substituting  $\Sigma (\beta, X(t_k))$ by
$\Sigma(\beta,x(\eta,t_k))$ for small sampling. This yields  the new process, using (\ref{Ak}),
\begin{equation}\label{BBU}
\bar{U}_{\epsilon,n} (\eta,(Y(t_k)))=
\sum_{k=1}^n log \;\Sigma(\beta,x(\eta,t_k))+\frac{1}{\epsilon^2 \Delta}
\Sigma(\beta,x(\eta,t_k))^{-1} (A_k(\eta, Y))^2.
\end{equation}
  The study of this process should yield estimators in the diffusion coefficient $\beta$ with probably additional assumptions linking $\epsilon $ and $\Delta$.
\noindent
Finally for fixed $\Delta$,  $S_k(\alpha,\beta)$ defined in (\ref{SkTS})  could be substituted by  $(S_k(\alpha,\beta))_{11}$ in the case of two distinct parameters in the drift and diffusion coefficient, and $ (S_k(\alpha)_{11}$ in the case corresponding to epidemics where the same parameters are present in the drift and
diffusion coefficients.
\noindent
Another extension of the method described in Section \ref{sec:partial}
is the case of a $p$-dimensional diffusion process where only the first $l$-coordinates are observed (for instance the $SEIR$ model with only Infected observed).

\chapter{Inference for Continuous Time SIR models} \label{chappointproc}

\noindent \textbf{by Catherine Lar\'edo and Viet Chi Tran}\\

\section{Introduction}\label{chap3:intro}
\label{intro}
Consider the $SIR$ epidemic model with exponential times in a finite population of size $N$ where  $S(t),I(t),R(t)$ denote the number of Susceptible, infected/infectious and Removed individuals at time $t$ with infection rate $\lambda$ and recovery rate $\gamma$ ($S(t)+I(t)+R(t)= N $ for all $t$). There are various ways of describing this
process using pure jump Markov processes. We refer to Chapter \ref{chap_MarkovMod} of Part I of these notes  and to Section \ref{LikPJP} of the Appendix for a recap on these processes.

%Another way of studying this process is to associate to this Pure Jump Process the multivariate counting processes $(N(t),R(t), t\geq 0)$ where $N(t)$ counts the number of infections occurring in the interval $(0,t]$ and $R(t)$ the number of removed individuals in $(0,t]$.
%Let $\mathcal{ F}_t$ the $\sigma$-algebra $\sigma(S(u),I(u); 0\leq u \leq t)$.

This description now  belongs to the domain of event time data, which are conveniently studied by the use of counting processes. We refer to Section \ref{LikPJP} of the Appendix for a short introduction to counting processes in continuous time.\\

At this point, we need an asymptotic framework to study the properties of these estimators.
Two frameworks have been proposed. \\
\noindent \textbf{Case (1):} Assume that the number of initially infected $I(0)= a$ remains fixed and that the number of initial Susceptible is $S(0)= n:= N-a$. We also assume for the sake of simplicity that $R(0)=0$. This leads to a total population size $N= n+a $ that goes to infinity.\\
\noindent \textbf{Case (2):} Assume that the population size $N \rightarrow \infty$ and that both $S(0), I(0)$ tend to infinity with $N$ such that $S(0)/N \rightarrow s_0 >0; I(0)/N \rightarrow i_0>0$ as $N \rightarrow \infty$.\\

%Let us denote by $N_n(\cdot), R_n(\cdot), \tau_n$ the counting processes and extinction times associated with the initial number of Susceptibles $n$. One has to study precisely  when a major outbreak occurs, {\it e.m.} if infinitely many susceptibles are infected  as $n\rightarrow \infty$.
%One has to study precisely  when a major outbreak occurs, {\it e.m.} if infinitely many susceptibles are infected  as $n\rightarrow \infty$.
\noindent
Case (1) has been studied by Rida \cite{ridaIV}, to which we refer for a detailed presentation. We focus here mainly on Case (2).

%We will focus here on Case (2). We refer to Becker (??) or \cite[Chapter 9]{andersonbritton}.\\
%Assume that $S_n(0)= n$ and $I_n(0)=\mu n$.
%Then, the normalized jump process $Z_n(t)=(S(t)/n,I(t)/n)$ converges a.s. ( and in sup norm) to the deterministic solution $(s(t),i(t))$ of the ODE:
%\begin{eqnarray*}
%ds(t) &=& -\lambda s(t)i(t) dt;\; s(0)=1,\\
%di(t) &=& (\lambda s(t)i(t)dt-\gamma i(t)) dt;\; i(0)=\mu
%\end{eqnarray*}
%Assume now that the process is complety observed to time $t$.
% The maximum likelihood estimators are the ones defined in (\ref{MLESIR}). Note that, if $R_0=\lambda/gamma>1$, the extinction $A_n$ has a probability that goes to $1$ as $n\rightarrow \infty$. They satisfy
%\begin{equation}\label{AsympSIR}
%	\sqrt{n} \begin{pmatrix}\hat{\lambda}_t-\lambda \\ \hat{\gamma}_t-\gamma \end{pmatrix} \rightarrow
%\mathcal{ N}_2(0, V^{-1}(t)) \quad \mbox{in distribution under } P_{\lambda,\gamma},
%\end{equation}
%where $V(t)= V_{ij}(t)$ is the $2\times 2$ matrix such that $V_{12}(t) = V_{21}(t)=0$ and,
%\begin{equation}\label{V}
%V_{11} (t) = \frac{\int_0^t s(u)i(u)du}{\lambda} = \frac{1-s(t)}{\lambda^2} ;\;
%V_{22}(t) = \frac{\int_0^t i(u)du}{\gamma}= \frac{1+\mu-s(t)-i(t)}{\gamma^2}.
%\end{equation}

\section{Maximum likelihood in the SIR case}\label{sec:likelihood}

To ease notation, we work here on a simplification of the SEIR process studied in Part I of these notes. We omit the state E and consider an SIR model (corresponding to the limiting case when $\nu\rightarrow +\infty$). Recall that the population size is $N$, that the infection rate is $\lambda$ and the removal rate $\gamma$. We assume that we observe the whole trajectory on a time window $[0,T]$ with $T>0$: $(S^N_t,I^N_t,R^N_t)_{t\in [0,T]}$. The successive times of events are $(T_i)_{1\leq i\leq K_N(T)}$, where $K_N(T)=\sum_{i\geq 0}\ind_{T_i\leq T}$ is the number of events. At each event, $J_i=0$ if we have an infection and $J_i=1$ if we have a recovery. Notice that we are here in the case where we have knowledge of all recovery and infection events, i.e.\  that we have \textit{complete epidemic data}. The case where some data are missing is treated in the next subsections.\\

Writing the likelihood of our data is important to calibrate the parameters of the model, $\theta=(\lambda,\gamma)\in \R_+^2$ in the case of the SIR model, but also because this is also useful for designing EM or MCMC procedures.

\begin{definition}We define the likelihood $\mathcal{L}^N_T(\theta)$ of the observations as the density, in $\D([0,T],[0,1]^3)$ of the process $(S^N_t,I^N_t,R^N_t)_{t\in [0,T]}$ with respect to the SIR process where intervals between events follow independent exponential distributions of parameter $2N$ and where each event is an infection with probability $1/2$ and a recovery with probability $1/2$. The likelihood is of course a function of $\theta\in \R_+^2$ and of the observations $(S^N_t,I^N_t,R^N_t)_{t\in [0,T]}$ which are omitted in the notation for the sake of notation.
\end{definition}

This definition has been proposed in \cite{arazozaclemencontranIV} for example. The dominating measure with respect to which the distribution of $(S^N_t,I^N_t,R^N_t)_{t\in [0,T]}$ is written is here the distribution of the process corresponding to the sequence $(J_i,T_i)$'s where the $J_i$'s are i.i.d.\ Bernoulli random variables with parameter $1/2$, and where the intervals $\Delta T_i=T_i-T_{i-1}$ are i.i.d.\ exponential random variables with expectation $1/(2N)$. With the notation above:

\begin{align}
 \mathcal{L}^N_T (\theta)&=  \mathcal{L}^N_T \big( (S^N_t,I^N_t,R^N_t)_{t\in [0,T]} ; \lambda,\gamma\big) \nonumber\\
 &=  \exp\big(NT - \int_0^T (\lambda S^N_s I^N_s - \gamma I^N_s)ds\big)\prod_{i=1}^{K_N(T)} (\lambda S_{T_i-}^N I_{T_i-}^N)^{1-J_i} (\gamma I_{T_i-}^N)^{J_i}.\label{likelihood}
\end{align}

%\textcolor{red}{Etienne, pourrais-tu ajouter dans ton chapitre 2 une remarque donnant l'equivalence entre l'ecriture page 36 avec $P_1$ et $P_2$ de l'exemple 2.2.1 et celle avec des mesures de Poisson $Q^1$ et $Q^2$.}

Taking the log, and using the formulation of the processes $(S_t, I_t, R_t)$ by means of Poisson point processes $Q^1$ and $Q^2$ as in Part I, Chapter 2 of these notes,
% Example 2.2.1 \cite{brittonpardouxIV}:
\begin{align*}
\log \mathcal{L}^N_T(\theta)&=  NT - \int_0^T (\lambda S^N_s I^N_s - \gamma I^N_s)ds \\
&\quad + \sum_{i=1}^{K_N(T)}\Big[(1-J_i) \log\big(\lambda S^N_{s_-} I^N_{s_-} \big)+J_i\log\big(\gamma I^N_{s_-} \big)\Big]\\
&=  NT - \int_0^T (\lambda S^N_s I^N_s - \gamma I^N_s)ds + \int_0^T \log\big(\lambda S^N_{s_-} I^N_{s_-} \big)\ind_{u\leq \lambda N S^N_{s_-} I^N_{s_-}}Q^1(ds,du)\\
 &\quad +\int_0^T \log\big(\gamma I^N_{s_-} \big)\ind_{u\leq \gamma N I^N_{s_-}}Q^2(ds,du).
\end{align*}

The above function is concave in $\lambda$ and $\gamma$, for a given observations $(S^N_t, I^N_t)_{t\in [0,T]}$, and maximizing it, we obtain:
\begin{proposition}\label{prop:emv}
The maximum likelihood estimator $\widehat{\theta}_N=(\widehat{\lambda}_N,\widehat{\gamma}_N)$ of $\theta$ (MLE) is then given by:
\begin{equation}\label{MLESIR}
\widehat{\lambda}_N=\frac{1}{N}\frac{\sum_{i=1}^{K_N(T)} (1-J_i)}{\int_0^T S^N_s I^N_s ds},\qquad \widehat{\gamma}_N=\frac{1}{N} \frac{\sum_{i=1}^{K_N(T)}J_i}{\int_0^T I^N_s ds}.
\end{equation}
\end{proposition}
These estimators have already been mentioned in \eqref{MLEPJP} and it had been noticed that the numerators of $\widehat{\lambda}_N$ and $\widehat{\gamma}_N$ are respectively the numbers of infections and recoveries on the period $[0,T]$.
Remark also that the estimators \eqref{MLESIR} are the same for the Cases (1) and (2) presented in Section \ref{chap3:intro}. In what follows, we concentrate on the Case (2).

%\begin{remark}Let us consider the case of $\widehat{\gamma}_N$, which involves less computation. The case of $\widehat{\lambda}_N$ is treated similarly. Using the Poisson point process $Q^2(ds,du)$ with intensity measure $ds\ du$, the Lebesgue measure on $\R_+^2$, we have:
%\begin{align*}
%\widehat{\gamma}_N= & \frac{1}{N} \frac{\int_0^T \int_{\R_+} \ind_{u\leq \gamma N I^N_{s_-}}Q^2(ds,du)}{\int_0^T I^N_s ds}\\
%= & \gamma +  \frac{\frac{1}{N}\int_0^T \int_{\R_+} \ind_{u\leq \gamma N I^N_{s_-}}Q^2(ds,du)-\int_0^T \gamma I^N_s ds}{\int_0^T I^N_s ds},
%\end{align*}
%where the numerator of the fraction in the right-hand side is a centered square-integrable martingale.
%Using the martingale (\ref{M1M2}) yields,
%$${\hat{\lambda}_t}=  \lambda+ \frac{M_1(t)}{\frac{1}{N}\int_0^t S(u)I(u)du};\; {\hat{\gamma}}_t= \gamma +
%\frac{M_2(t)}{\int_0^t I(u) du}.$$
%These estimators are unbiased with respective variance
%%E(\frac{M_1^2(t)}{(\frac{1}{N}\int_0^t S(u)I(u)du) ^2})$
%$$ E( \frac{<M_1>(t)}{(\frac{1}{N}\int_0^t S(u)I(u)du) ^2})= E(\frac{\lambda}{\frac{1}{N}\int_0^t %S(u)I(u)du})=E(\frac{\lambda^2}{N(t)-M_1(t)}).$$
%%$E(\frac{M_2^2(t)}{(\int_0^t I(u^-)du) ^2})=$
%$$ E\frac{<M_2>(t)}{(\int_0^tI(u)du)^2}= E(\frac{\gamma}{\int_0^tI(u)du})= E(\frac{\gamma^2}{R(t)-r_0}).$$
%Hence  estimators of these  two variances are
% $\hat{v_1}= \frac{\hat{\lambda}_t^2}{N(t)}$, $\hat{v_2}=\frac{\hat{\gamma}_t^2}{R(t)-r_0-M_2(t)}$.
%\end{remark}

Using the Law of Large Numbers  and the Central Limit Theorem stated in Part I, Section 2 of these notes
%(Sections \eqref{TP-EP_sec_LLN}, \eqref{TP-EP_sec_TCL}),
%er 2, Example 2.2.11 \cite{brittonpardouxIV}) and the Central Limit Theorem (see Chapter 2, Example 2.3.5\cite{brittonpardouxIV}),
we obtain that
\begin{proposition}
The estimator $\widehat{\theta}_N$ is convergent and asymptotically Gaussian when $N\rightarrow +\infty$:
$$\sqrt{N}\big(\widehat{\theta}_N-\theta\big)=\sqrt{N}\left(\begin{array}{c}\widehat{\lambda}_N-\lambda\\
\widehat{\gamma}_N-\gamma\end{array}\right)\Rightarrow \mathcal{N}\big(0_{\R^2},I^{-1}(\lambda,\gamma)\big),$$
where the Fisher information matrix is:
$$I(\lambda,\gamma)=\left(\begin{array}{cc}
V_{11}(t) & 0\\
0 & V_{22}(t)
\end{array}\right)$$
with $(s(t), i(t))_{t\in [0,T]}$ the solution of the limiting ODE that approximates  $(S^N_t, I^N_t)_{t\in [0,T]}$ when $N\rightarrow +\infty$
(see Example 2.2.10 in Part I ) and with
\begin{equation}\label{V}
V_{11} (t) = \frac{\int_0^T s(t) i(t) dt}{\lambda} = \frac{1-s(T)}{\lambda^2} ;\;
V_{22}(t) = \frac{\int_0^T i(t) dt}{\gamma}= \frac{1+\mu-s(T)-i(T)}{\gamma^2}.
\end{equation}
\end{proposition}

\begin{proof}Notice that the estimator $\widehat{\lambda}$ given in Proposition \ref{prop:emv} can be rewritten, with the notations of
Example 2.2.1 of Part I of these notes, as
%\cite[Chapter 2, Example 2.2.1]{brittonpardouxIV} as:
\[\widehat{\lambda}_N=\frac{1}{N}\frac{P_1\Big(\lambda N\int_0^T S^N_s I^N_s ds\Big)}{\int_0^T S^N_s I^N_s ds}.\]
Using the Law of Large Numbers given in Part I,  Section 2.2,
 %\ref{TP-EP_sec_LLN},
%Section Theorem 2.2.9,
the process $(S^N_t, I^N_t)_{t\in [0,T]}$ converges uniformly when $N\rightarrow +\infty$ to the unique solution of the ODE
\begin{align*}
s'(t)&=  - \lambda s(t) i(t),\\
i'(t)&=  \lambda s(t) i(t)-\gamma i(t).
\end{align*}
Moreover,
%Hence, using  Part I
%Proposition 2.2.8,
\begin{align*}
\lim_{N\rightarrow +\infty}\widehat{\lambda}_N = \lambda \frac{\int_0^T s(t)i(t) dt}{\int_0^T s(t)i(t) dt}=\lambda.
\end{align*}
Now,
\[\sqrt{N}\big(\widehat{\lambda}_N-\lambda\big)= \frac{1}{\int_0^T S^N_s I^N_s ds}\left[\frac{1}{\sqrt{N}}P_1\Big(\lambda N\int_0^T S^N_s I^N_s ds\Big)-\sqrt{N}\lambda \int_0^T S^N_s I^N_s ds\right].\]
From  Part I, Section 2.3,
%Lemma 2.3.4,
we have the following convergence in distribution
\[\frac{1}{\sqrt{N} }P_1\Big(\lambda N\int_0^T s(t) i(t) dt\Big) - \sqrt{N}  \lambda \int_0^T s(t)i(t) dt\Rightarrow B_1\Big(\lambda \int_0^T s(t)i(t) dt\Big)\]
where $B_1$ is a standard real Brownian motion. As in the proof of Proposition 2.3.1, the bracket in the right term is then shown to converge to the same limit $B_1(\lambda \int_0^T s(t)i(t) dt)$. Since the denominator of the right-hand side converges in probability to $\int_0^T s(t)i(t) dt$, we obtain the asymptotic normality of $\widehat{\lambda}_N$ with asymptotic variance
\[\frac{\lambda}{\int_0^T s(t)i(t)dt}.\]
Proceeding similarly for $\widehat{\gamma}_N$ and using the asymptotic independence between the two estimators provides the result. Notice that the Fisher information matrix can also be computed from the log-likelihood, and that all regularity assumptions of generic asymptotic normality results are satisfied (see e.g.\ Chapter 4 of \cite{linkovIV}).
\end{proof}

\begin{corollary}An estimator of  $R_0=\lambda/\gamma$ is $\widehat{R}_0^{(t)}=\frac{\hat {\lambda}_t}{\hat{\gamma}_t}.$
Applying the functional delta-theorem (e.g.\ \cite{vaa00IV}), it converges in distribution to
\begin{equation}\label{estimR0}
	\sqrt{n}(\widehat{R}_0^{(t)}- R_0)\rightarrow \mathcal{ N}(0, \sigma^2(t)) \quad  \mbox{ with  }\quad
	\sigma^2 (t)= \frac{V_{11}^{-1}(t)+ R_0^2 V_{22}^{-1}(t)}{\gamma^2}.
	\end{equation}
\end{corollary}

\begin{remark}[Maximum likelihood estimators in the Case (1)]
Let us denote by $(N_t)_{t\in \R_+}$ the counting processes associated to the infection process:
\[N_t= P_1\Big(\int_0^t \lambda N S^N_{s} I^N_{s} ds\Big),\]
and by $\tau_N$ the extinction time, when there is no infective individual left. Because the population is finite, $\tau_N<+\infty$ almost surely and $N(\tau_N)\leq N$. Let
\[A= \{\omega; N(\tau_N,\omega)\rightarrow \infty \;\mbox{as }  N\rightarrow \infty\}\]
be the event on which a major outbreak occurs. Ball \cite{ball83IV} proved that $\P(A)= 1-\min\{1,(\gamma/\lambda)^a\}$.
Moreover if $R_0=\lambda/\gamma >1$, then $\P(A)>0$ and as $n\rightarrow \infty$,
\[  \frac{N(\tau_N)}{N}\rightarrow \pi 1_{A} \mbox{ where $\pi$ is such that }\frac{\lambda}{\gamma}= -\frac{\log (1-\pi)}{\pi}.\]
Asymptotic results for the estimators are obtained on $A$ and $A^c$.
The maximum likelihood estimator satisfies that \[\hat{\lambda}_{N} \rightarrow \lambda 1_{A} + Z 1_{A^c}\]in distribution
where $Z$ is a positive explicit random variable such that $\E(Z) < 1/\lambda $ if $\lambda/\gamma>1$.
Note that in this case, $\hat{\lambda}_{N}$ is not a consistent estimator. We refer to \cite{ridaIV} for a detailed presentation of the results.
\end{remark}

These methods can be extended to other epidemic models. We will detail later for the SEIR and SIRS
epidemic models. The main  drawback of this approach is that the epidemic process is rarely observed in such details, which prevents this kind of statistical approach.
However, this study sums up the best statistical results that can be obtained when complete observations are available.  When  incomplete observations are available, the loss of information will be measured with respect to
this general reference.

\subsection{MCMC estimation}

The preceding subsection treated the case of complete observation. In practice, parameter estimation for SIR models is usually a difficult task because of missing observations, which is a recurrent issue in epidemiology.
O'Neill Roberts \cite{oneillrobertsIV} developed a Markov chain Monte Carlo method (MCMC) to make inferences about the missing data and the unknown parameters in a Bayesian framework. \\

We consider an SIR model as in Section \ref{sec:likelihood}. Instead of observing the sequence $(J_i, T_i)_{i\in \{1\dots K^N_T\}}$ (type -- infection or recovery -- and time of occurrence of the successive events, as described in the beginning of Section \ref{sec:likelihood}), we observe only the $T_i$'s such that $J_i=1$ (recovery events, that can also be detection events in some applications) and the total number of events $K^N_T$ is unknown. In this section, we adopt the following notation. Let us assume that there are $m$ infections at times $\mathbf{\sigma}=(\sigma_1<0,\dots \sigma_m)$ that are unobserved and $n$ removals at times $\mathbf{\tau}=(\tau_1=0,\dots \tau_n)$ which constitute our observations. For later purposes, we will denote by $\mathbf{\sigma}_{-1}=(\sigma_2,\dots \sigma_m)$ the vector of infection times starting from the second infection. We observe the total size of the population $N$, the number $n$ of removal times and the vector $\tau$ of these removal times. The parameter of interest is $(\lambda,\gamma,\sigma_{1})$ and the vector $\mathbf{\sigma}_{-1}$ is the vector of nuisance parameters. \\

The MCMC algorithm proposed by O'Neill and Roberts \cite{oneillrobertsIV} take place in a Bayesian framework. Given $\lambda$, $\gamma$ and the first infection time $\sigma_1$, the likelihood of $(\mathbf{\sigma}_{-1},\mathbf{\tau})=(\sigma_2\dots \sigma_m, \tau_1,\dots \tau_m)$ is obtained from adapting \eqref{likelihood}:
\begin{align}
 \mathcal{L}^N_T (\mathbf{\sigma}_{-1},\mathbf{\tau} | \lambda, \gamma, \sigma_1)= &  \exp\big(NT - \int_{\sigma_1}^T (\lambda S^N_s I^N_s - \gamma I^N_s)ds\big)\prod_{i=1}^{n} (\lambda S_{\sigma_i-}^N I_{\sigma_i-}^N) \prod_{i=1}^m (\gamma I_{\tau_i-}^N).\label{likelihood1}
\end{align}

\subsubsection{A priori distributions}We suppose that $\lambda$ and $\gamma$ have \textit{a priori} Gamma distribution with parameters $(\alpha_\lambda, \beta_\lambda)$ and $(\alpha_\gamma,\beta_\gamma)$ respectively, where we recall that the density of a Gamma distribution with parameter $(\alpha,\beta)$ is:
\[\frac{\beta^\alpha}{\Gamma(\alpha)}x^{\alpha-1}e^{-\beta x}\ \ind_{(0,+\infty)}(x)\]
where $\Gamma(x)$ is the gamma function such that for any positive integer $k$, $\Gamma(k)=(k-1)!$. Following \cite{oneillrobertsIV}, we also chose for the \textit{a priori} distribution of $\sigma_1$ the  `exponential' distribution with density (on $\R_-$) with $\rho>0$:
\[\rho e^{\rho \sigma_1} \ind_{(-\infty,0)}(\sigma_1).\]

\subsubsection{A posteriori distributions}
The purpose is now to generate a sample from the \textit{a posteriori} distribution
$\pi(\mathbf{\sigma}, \lambda,\beta  | \mathbf{\tau})$. For this, O'Neill and Roberts propose a Metropolis--Hastings algorithm. \\

Recall the principle of the Metropolis--Hastings algorithm used to obtain a sample $\x$ in a distribution with a density $\pi(x)$ that is proportional to some $f(x)$. Consider a transition kernel with a density $q(y | x)$ from which it is easy to simulate. Starting from a first point $x_0$, construct a sequence of points $(x_k)_{k\in N}$ with $f$ and $q$ as follows. Assume that $x_{k}$ has been constructed, then:
\begin{itemize}
\item draw $y$ from $q(y |x_{k})$.
\item With probability
\[\phi(x_k,y)=\min\Big(\frac{f(y)q(x_k |y)}{f(x_k)q(y|x_k )},1\Big)\]
define $x_{k+1}=y$.\\
With probability $1-\phi(x_k,y)$, define $x_{k+1}=x_k$.
\end{itemize}This defines a reversible Markov chain whose stationary distribution is $\pi$.\\

We apply the above idea to sample $\mathbf{\sigma}, \lambda,\beta$ from the \textit{a posteriori} distribution. To choose the transition kernels, notice first that with direct computation, we obtain:
\begin{align*}
\pi\big(\sigma_1 | \mathbf{\tau}, \mathbf{\sigma}_{-1},\lambda,\gamma\big) \sim  & (\rho +\lambda N+\gamma)e^{-(\theta +\lambda N+\gamma)(\sigma_2-y)}\ind_{y<\sigma_2}\\
\pi\big(\lambda | \mathbf{\tau}, \mathbf{\sigma},\gamma \big) \sim &  \Gamma\big(\alpha_\lambda+\int_{\sigma_1}^T S^N_s I^N_s ds, m-1+\beta_\lambda\big)\\
\pi\big(\gamma | \mathbf{\tau}, \mathbf{\sigma},\lambda\big) \sim &  \Gamma\big(\alpha_\gamma+ \int_{\sigma_1}^T I^N_s ds, n+\beta_\gamma\big).
\end{align*}
Hence, it is natural to choose the above distributions for the proposals of $\sigma_1$, $\lambda$ and $\beta$.
It remains to propose a transition kernel for $\mathbf{\sigma}_{-1}$. O'Neill and Roberts propose a Hasting algorithm with the three following moves:

%\item it remains to simulate $\mathbf{\sigma}$ conditionally to $\mathbf{\tau}, \sigma_1,\lambda,\gamma$.
%\end{itemize}

\begin{itemize}
\item Move an infection time chosen at random by sampling the candidate uniformly in $[0,T]$. If the infection time chosen at random was at time $s$ and the proposal time drawn uniformly in $[0,T]$ is $t$, the move is accepted with probability
\begin{align*}\phi(\mathbf{\sigma}, \mathbf{\sigma}\cup \{t\}\setminus \{s\})= & \frac{\mathcal{L}^N_T\big(\mathbf{\sigma}\cup \{t\}\setminus\{s\},\mathbf{\tau} | \lambda, \gamma, \sigma_{1}\big) \frac{1}{|\mathbf{\sigma}|-1} \frac{1}{T}}{\mathcal{L}^N_T\big(\mathbf{\sigma},\mathbf{\tau} | \lambda,\gamma,\sigma_{1}\big)\frac{1}{|\mathbf{\sigma}|-1}\frac{1}{T}}\wedge 1\\
= &  \frac{\mathcal{L}^N_T\big(\mathbf{\sigma}\cup \{t\}\setminus\{s\},\mathbf{\tau} | \lambda, \gamma, \sigma_{1}\big)}{\mathcal{L}^N_T\big(\mathbf{\sigma},\mathbf{\tau} | \lambda,\gamma,\sigma_{1}\big)}\vee 1.
\end{align*}
\item Remove an infection time chosen at random. If the chosen infection time was at time $s$, the acceptation probability is then:
$$\frac{\mathcal{L}^N_T\big(\mathbf{\sigma}\setminus \{s\}, \mathbf{\tau} | \lambda, \gamma, \sigma_1\big)\frac{1}{T-\sigma_1}}{\mathcal{L}^N_T\big(\mathbf{\sigma},\mathbf{\tau} | \lambda,\gamma,\sigma_1\big)\frac{1}{|\mathbf{\sigma}|-1}}\wedge 1=\frac{\mathcal{L}^N_T\big(\mathbf{\sigma}\setminus \{s\}, \mathbf{\tau} | \lambda, \gamma, \sigma_1\big)(|\mathbf{\sigma}|-1)}{\mathcal{L}^N_T\big(\mathbf{\sigma},\mathbf{\tau} | \lambda,\gamma,\sigma_1\big)(T-\sigma_1)}\wedge 1.$$
\item Add a new infection at a time $t$ drawn uniformly on $[0,T]$:
$$\frac{\mathcal{L}^N_T\big(\mathbf{\sigma}\cup \{t\}, \mathbf{\tau} | \lambda, \gamma, \sigma_1\big)\frac{1}{|\mathbf{\sigma}|}}{\mathcal{L}^N_T\big(\mathbf{\sigma},\mathbf{\tau} | \lambda,\gamma,\sigma_1\big)\frac{1}{(T-\sigma_1)}}\wedge 1=\frac{\mathcal{L}\big(\mathbf{\sigma}+\{t\}\big)(T-\sigma_1)}{\mathcal{L}\big(\mathbf{\sigma}\big) |\mathbf{\sigma}|}\wedge 1.$$
\end{itemize}

A numerical application is performed in \cite{oneillrobertsIV} for small epidemics. This algorithm is simulated and compared with other ones in Section \ref{sectionMCMC}.
%\pagebreak

\subsection{ EM algorithm for discretely observed Markov jump processes}

%Ref: Bladt and Sorensen JRSS B 2005\\
We consider now the situation where the Markov jump process is only observed at discrete time points. This has been considered by Bladt and Sorensen \cite{bladtsorensenIV}.
We study the  maximum likelihood estimation of the $Q$-matrix based on a discretely sampled Markov jump process.
The problem of identifiability and of existence and uniqueness of the MLE is related to the following problem in probability: can a given discrete time Markov chain be obtained as a discrete time sampling of a
continuous time Markov jump process?

\subsubsection{Likelihood function }
Let $X= (X(s), s \geq 0)$ be a Markov jump process with finite state space $E=\{1,\dots ,N\}$ and $Q$-matrix
${\mathbf Q}=(q_{kl})$.
If $X$ is continuously observed on the time interval $[0,T]$, the likelihood function is given by,
% using (\ref{Ntkl}), (\ref{Ntl})
\begin{equation}\label{likMPJ}
	L_T({\mathbf Q})= \prod_{k=1}^N \prod_{l\neq k} q_{kl}^{N_{kl}(T)}\;\exp(-q_{kl} R_k(T)), \; \mbox{where}
	\end{equation}
 the process $N_{kl}(t)$ is the number of transitions  from state $k$ to state $l$ in the time interval $[0,t]$  and $R_k(t)$ is the time spent in state $k$
before time $t$.
\begin{equation}\label{RkT}
R_{k}(t)= \int_0^t \delta_{\{X(s)= k\}} \;ds.
\end{equation}
For details see e.g. \cite{jacobsenIV} .\\
Therefore, if the process is continuously observed on $[0,T]$, the maximum likelihood estimator of its ${\mathbf Q}$- matrix is
easily obtained:
\begin{equation}\label{MLEQPJP}
\hat{{\mathbf Q}}_{kl}= \frac{N_{kl}(T)}{R_{k}(T)}.
\end{equation}

Assume now that the process is observed with a sampling interval $\Delta$ with $T=n\Delta$. Then, setting $X_i=X(t_i)$ is a discrete time Markov chain with transition matrix
$$P^{\Delta}({\mathbf Q}) \quad \mbox{where  } P^t({\mathbf Q})= \exp(t{\mathbf Q}),\quad t>0,$$
with $\exp(\cdot)$ denoting the matrix exponential function.

Hence the likelihood for the discrete observations  $(x_0, \dots, x_n)$ is
$$ L_{n,\Delta}({\mathbf Q})= \prod_{i=1}^n P^{\Delta}({\mathbf Q})_{x_{i-1}x_i},$$
 with the notation that the  $ij$ entry of a matrix $A$ is denoted $A_{ij}$.   Since it is a discrete time Markov chain, it satisfies,
\begin{eqnarray*}
	L_{n,\Delta}({\mathbf Q})&=& \prod_{k=1}^N \prod_{l=1}^N (P ^{\Delta}({\mathbf Q})_{kl}^{N^{kl}(n)},\\
N^{kl}(n)&=&\sum_{i=1}^{n} \delta_{\{X_{i-1}=k, X_{i}=l\}}.
\end{eqnarray*}
The random variables  $(N^{kl}(n))$ are the number of transitions from state $k$ to state $l$ before $n$.
% and  the time spent in state $l$ after $0$ in the discrete time chain $(X(t_1),\dots,X(t_n))$.
We have proved in Section \ref{MCcount}) that the associated MLE of the transition matrix  ${\hat {\mathbf P}}$ is explicit.
But  building an estimator of $Q$ from ${\hat {\mathbf P}}$ is not straightforward.

Indeed,
%for a Markov chain with finite state space $E$ and transition matrix ${\mathbf P}= (p_{kl})$ the maximum likelihood estimator ${\hat {\mathbf P}}$ is
%$$\hat{p}_{kl}= \frac{N^{kl}(n)}{N^k(n)} \;\; \mbox{with}\;\; N^{k} (n)=\sum_{i=1}^n \delta_{\{X_i= k\}}. $$
let $\mathcal{ P}_0= \{\exp\; {\mathbf Q} \;|\; {\mathbf Q} \in \mathcal{ Q}\}$ denote the set of transition matrices that correspond to discrete time observation of a continuous time Markov jump process. If
$ {\hat{\mathbf P}} \in \mathcal{ P}_0$, there exists a $\hat{{\mathbf Q}} \in \mathcal{ Q}$ such that $P^{\Delta} (\hat{{\mathbf Q}})= {\hat{\mathbf P}}$. This raises  two distinct  problems. First  the set $\mathcal{ P}_0$ is quite complex, and second the matrix exponential function is not an injection on its domain, so
${\hat{\mathbf Q}}$ may not be unique leading to identifiability questions for the statistical model.
 Additional assumptions are thus required in order to ensure the convergence of stochastic algorithms such as $EM$, $MCMC$. We refer to Bladt and Sorensen \cite{bladtsorensenIV} for details.

\subsubsection{The Expectation-Maximization (EM) algorithm}
This is a broadly used method for optimizing the likelihood function in cases where only partial information is available (see e.g.\ \cite{delyonlaviellemoulinesIV,dempsterlairdrubinIV,vaidaIV,jeffwu-emIV}).
A discretely observed Markov jump process is such an example where only data $Y_i= X(t_i); i=1,\dots,n $ are available. Let $X= \{X(t); 0 \leq t\leq T\}$ and $Y=\{Y_i; i=1 \dots,n\}$. The EM-algorithm aimed at estimating the $Q$-matrix $Q= (q_{ij},;i,j \in E )$  iterating the two steps: \\
\textbf{ E-step}: replace the unobserved parts by their conditional expected values given the data $Y=y $\\
\textbf{ M-step}:  perform maximum likelihood on the complete data.\\

The difficult part in the EM algorithm here is the  \textbf{ E-step}:\\ i.e. compute $\E_{Q_0}[\log L_T({\mathbf Q})| Y=y]$
where $Q_0$ is an arbitrary $Q$-matrix.\\
Indeed, consider the \textbf{ M-step}.  From equation \eqref{likMPJ}, we have
\begin{align*}
\E _{{\mathbf Q_0}}(\log L_T({\mathbf Q})| Y=y) =&\, \sum_{k=1}^N\sum_{l\neq k}\log(q_{kl})
\E _{{\mathbf Q_0}}(N_{kl}(T)|Y=y)\\
&- \sum_{k=1}^N\sum_{l\neq k}q_{kl}\E _{{\mathbf Q_0}}(N_k(T)| Y=y).
\end{align*}
This is the likelihood of a continuous time process with observed statistics\\
$\E _{{\mathbf Q_0}}(N_{kl}(T)|Y=y), \E _{{\mathbf Q_0}}(N_k(T)| Y=y)$.
It is maximized, as a function of $Q$, according to \eqref{MLEQPJP} by
\begin{equation} \label{EMMaxQ}
\hat{Q}_{kl}= \frac{\E _{{\mathbf Q_0}}(N_{kl}(T)|Y=y)}{ \E _{{\mathbf Q_0}}(N_k(T)| Y=y)}.
\end{equation}

 Therefore, to  perform the algorithm, we have to compute the two quantities $\E _{{\mathbf Q_0}}(N_{kl}(T)|Y=y)$ and
$ \E _{{\mathbf Q_0}}(N_k(T)| Y=y)$.\\
%More precisely, denote by $Q_0$ an initial intensity matrix, the algorithm finds $Q_1$ as follows:\\
%\textbf{ E-step}: Compute the function
%\begin{equation}\label{EstepPP}
%	f\;: {\mathbf Q }\rightarrow  \E _{{\mathbf Q_0}}(\log L_T({\mathbf Q})| Y=y).
%	\end{equation}
%\textbf{ M-step}: ${\mathbf Q_1}= \arg\max_{{\mathbf Q}} f({\mathbf Q})$.\\
%
%From (\ref{likMPJ}), we get that
%\begin{align*}
%\E _{{\mathbf Q_0}}(\log L_T({\mathbf Q})| Y=y) =&\, \sum_{k=1}^N\sum_{l\neq k}\log(q_{kl})
%\E _{{\mathbf Q_0}}(N_{kl}(T)|Y=y)\\
%&- \sum_{k=1}^N\sum_{l\neq k}q_{kl}\E _{{\mathbf Q_0}}(N_k(T)| Y=y).
%\end{align*}
%This is the likelihood of a continuous time process with observed statistics\\
%$\E _{{\mathbf Q_0}}(N_{kl}(T)|Y=y), \E _{{\mathbf Q_0}}(N_k(T)| Y=y)$. Using (\ref{MLEQ}), its maximum ${\mathbf Q_1}$ is
%\begin{equation}\label{EMPJP}
%	 q^1_{kl}(T)= \frac{\E _{{\mathbf Q_0}}(N^{kl}(T)| Y=y)}{\E _{{\mathbf Q_0}}(N^k(T)| Y=y)}.
%	\end{equation}
%Let us study more precisely these quantities.
For this,
let us consider a fixed intensity matrix ${\mathbf Q}$ and omit the index ${\mathbf Q}$.
Denote by $e_i$ the unit vector with $i^{th}$ coordinate equal to $1$, and for $U$ a vector or a matrix, let  $U^*$ the transpose of $U$.

Noting that $ N^k(T)= \sum_{p=1}^n	(N^k(t_p)-N^k(t_{p-1})) $, we get by the Markov property and the  time homogeneity of $X= X(t)$,
\begin{eqnarray*}
\E(N^k(t_p)-N^k(t_{p-1})/Y=y)&=& \E(N^k(t_p)-N^k(t_{p-1})|X(t_p)=y_p,X(t_{p-1})=y_{p-1})\\      &=&\E(N^k(t_p-t_{p-1})|X(t_p-t_{p-1})= y_p,X(0)= y_{p-1}).
\end{eqnarray*}
Similarly $N^{kl}(T)= \sum_{p=1}^n	(N^{kl}(t_p)-N^{kl}(t_{p-1})) $, and
$$\E(N^{kl}(t_p)-N^{kl}(t_{p-1})/Y=y) =\E(N^{kl}(t_p-t_{p-1})|X(t_p-t_{p-1})= y_p,X(0)= y_{p-1}).$$

Hence,
\begin{equation}\label{NkT}
	\E(N^k(T)|Y=y)= \sum_{p=1}^n E^k_{y_{p-1}y_p}(t_p-t_{p-1});\quad
	\E^{kl}(T)|Y=y)=\sum _{p=1}^n F^{kl}_{y_{p-1}y_p}(t_p-t_{p-1});
\end{equation}
where if $(i,j)$ and $ (k, l) \in E$, and $t>0$,
\begin{eqnarray*}
E_{ij}^k(t) &=& \E_{{\mathbf Q_0}}(N^k(t)| X(t)=j,X(0)=i),\\
F_{ij}^{kl}(t) &=& \E _{{\mathbf Q_0}}(N^{kl}(t)| X(t)=j,X(0)=i).
\end{eqnarray*}
%Let us consider a fixed intensity matrix ${\mathbf Q}$ and omit the index ${\mathbf Q}$.
%Denote by $e_i$ the unit vector with ith coordinate equal to $1$, and for $U$ a vector or a matrix, let  %$U'$ the transpose of $U$.
Fix $k \in E$ and define the matrix $M^k(t)$ by
\begin{equation}%\label{Mk}
	M^k_{ij}(t)= \E(N_k(t) 1_{X(t)=j}| X(0)=i).
	\end{equation}
	 Then, according to \cite{bladt2002IV},
	 %\cite{bladtsorensenIV}  %Bladt et al.\ \cite{bladt2002IV},
$$\frac{d}{dt}M_{ij}^k(t)= \sum_{l=1}^N M_{il}^k (t)q_{lj}+ \exp(t{\mathbf Q})_{ij}\delta_{jk};\quad M_{ij}^k(t_0)=0.$$
This equation has an explicit solution which reads as ${\mathbf M}^k(t)= (M^k_{ij}(t), i,j \in E)$,
\begin{equation}\label{Mk}
	{\mathbf M}^k(t)= \int_0^{t} \exp(s{\mathbf Q}) (e_k e^*_k)\exp((t-s)  {\mathbf Q})\;ds.
\end{equation}
%Hence, using that $\P(X(t)=j|X(0)=i)= e_i' \exp({\mathbf Q}t)e_j$ yields that
%$$E_{ij}^k(t)=\frac{M_{ij}^k(t)}{e_i'\exp({\mathbf Q}t)e_j} .$$
Fix now  $k, l \in E$ and define the matrix
${\mathbf f}^{kl}_{ij}(t)= \E(N^{kl}(t) 1_{X(t)=j}| X(0)=i)$. Similarly
\begin{equation}\label{fkl}
	{\mathbf f}^{kl}(t)= q_{kl}\int_0^t \exp(s{\mathbf Q})(e_k e^*_l)\exp((t-s){\mathbf Q})ds.
	\end{equation}
Hence, using that $\P(X(t)=j|X(0)=i)= e^*_i \exp({\mathbf Q}t)e_j$ yields that
\begin{equation}\label{EMexpli}
E_{ij}^k(t)=\frac{M_{ij}^k(t)}{e^*_i \exp(t{\mathbf Q})e_j}; \quad
F_{ij}^{kl}(t)= \frac{{\mathbf f}_{ij}^{kl}(t)}{{e^*_i \exp(t{\mathbf Q})e_j}}.
\end{equation}

So the EM-algorithm works along the successive iterations. Start from an initial
$Q$-matrix ${\mathbf Q}_0$. Let ${\mathbf Q}_m$ denote the $Q$-matrix of iteration $m$. Then
\begin{itemize}
	\item  For all $k,l \in E$, compute using (\ref{Mk}), (\ref{fkl}), (\ref{EMexpli}) the matrices
	$E_{y_i y_{i+1}}(t_{i+1}-t_i)$, and $F_{y_i y_{i+1}}^{kl}(t_{i+1}-t_i)$ associated to
	$Q= {\mathbf Q}_m$
\item Compute the two quantities $\E(N^k(T)|Y=y)$, $\E(N^{kl}(T)|Y=y)$ using (\ref{NkT})
\item  Define ${\mathbf Q}_{m+1}$ by \eqref{EMMaxQ}.
\end{itemize}

Let ${\mathbf Q}_0, {\mathbf Q}_1,\dots, {\mathbf Q}_p,\dots $ a sequence a $Q$- matrices obtained by the EM algorithm.
Then $L_{n,\Delta} ({\mathbf Q}_{p+1})\geq  L_{n,\Delta}({\mathbf Q}_{p})$ for $p=0,1,2, \dots$ (see e.g. \cite{dempsterlairdrubinIV}).
Under additional regularity conditions, one can prove (cf \cite{bladtsorensenIV}, Theorem 4) that, \\
 If ${\mathbf Q}_0$ satisfies that, for all $k,l \in E$,
 $({\mathbf Q}_0)_{kl} >0$, then the sequence $({\mathbf Q}_p)$ converge to a stationary point of the likelihood function $L_{n,\Delta}$ or
 $det\{exp({\mathbf Q}_p)\} \rightarrow 0$.

\section{ABC estimation}

 Markov Chain Monte Carlo (MCMC) methods that treat the missing data as extra parameters, have become increasingly popular for calibrating stochastic epidemiological models with missing data \cite{cau04IV,one02IV,oneillrobertsIV}. However, MCMC may be computationally prohibitive for high-dimensional missing observations \cite{cau08IV, chisstersinghfergusonIV} and fine tuning of the proposal distribution is required for efficient algorithms \cite{GilksRobertsIV}.
 The computation of the likelihood can sometimes be numerically infeasible because it involves integration over the unobserved events. In discrete time, or when the total population size is known and small as in \cite{oneillrobertsIV}, this is possible. But in \eqref{likelihood} for example, because we are in continuous time, the likelihood of removal times, when the infection times and $K^N_t$ are unknown, involves a summation over all possibilities which is impossible: the sum is over all the possible numbers of infections between each successive removal times, plus on the possible times of these infections. An alternative is given by Approximate Bayesian Computation (ABC), which was originally proposed for making inference in population genetics \cite{bea02IV}. This approach is not based on the likelihood function but relies on numerical simulations and comparisons between simulated and observed summary statistics. We detail here the ABC procedure and its application to epidemiology. For more information on ABC methods, the interested reader is referred to \cite{marinpudlorobertryderIV,sissonfanbeaumontbookIV}. In particular, there have been many refinements of the ABC method presented here, for instance using simulations to modify the sampling distributions (e.g.\ \cite{bea09IV,sissonfantanakaIV,tonietalIV}).\\

In \cite{blumtranIV}, the development of ABC estimation techniques for SIR models is motivated by the study of the Cuban HIV-AIDS database. In this case, the population is separated into the following compartments: 1) susceptible individuals who can be infected by HIV, 2) non-detected HIV positive infectious individuals who can propagate the disease, and 3) detected HIV positive individuals. When an individual is detected as HIV positive, we assume that the transmission of the disease ceases. So detection corresponds here to `recovery' events in the classical SIR model presented in Part I of this book. The Cuban database contains the dates of detection of the 8,662 individuals that have been found to be HIV positive in Cuba between 1986 and 2007 \cite{auvertIV}. The database contains additional covariates including the manner by which an individual has been found to be HIV positive. The individuals can be detected either by \textit{random screening} (individuals `spontaneously' take a detection test) or \textit{contact-tracing}. The total number of infectious individuals as well as the infection times are unknown. Blum and Tran \cite{blumtranIV} proposed an ABC estimation procedure when all detection times are known, which they then extend to noisy or binned detection times. They also propose an extension of ABC to path-valued summary statistics consisting of the cumulated number of detections through time.  They introduce a finite-dimensional vector of summary statistics and compare the statistical properties of point estimates and credibility intervals obtained with full and binned detection times. We present here these methods for a simple SIR model and compare numerically the posterior distributions obtained with ABC and MCMC. We refer the reader to \cite{blumtranIV} for more details and treatment of Cuban HIV data. Other use of ABC estimation techniques in public health can be found in \cite{drovandipettittIV,mckinleyrossdeardoncookIV} for example.

\subsection{Main principles of ABC}\label{section:presentationABC}

For simplicity, we deal here with densities and not general probability measures. Let $\mathbf{x}$ be the available data and $\pi(\theta)$ be the prior where $\theta$ is the parameter. Two approximations are at the core of ABC.\\

\noindent \textbf{Replacing observations with summary statistics} Instead of focusing on the posterior density $p(\theta\, |\, \mathbf{x})$, ABC aims at a possibly less informative \textit{target} density $p(\theta\, |\, S(\mathbf{x})=s_{obs})\propto {\textrm Pr}(s_ {obs} |\theta)\pi(\theta)$ where $S$ is a summary statistic that takes its values in a normed space, and $s_{obs}$ denotes the observed summary statistic. The summary statistic $S$ can be a $d$-dimensional vector or an infinite-dimensional variable such as a $L^1$ function. Of course, if $S$ is sufficient, then the two conditional densities are the same. The target distribution will also be coined as the \textit{partial posterior distribution}.\\

\noindent \textbf{Simulation-based approximations of the posterior} Once the summary statistics have been chosen, the second approximation arises when estimating the partial posterior density  $p(\theta\, |\, S(\mathbf{x})=s_{obs})$ and sampling from this distribution. This step involves nonparametric kernel estimation and possibly correction refinements.

\subsubsection{Sampling from the posterior}

The ABC method with smooth rejection generates random draws from the target distribution as follows (see e.g.\ \cite{bea02IV})
\begin{enumerate}
\item Generate $N$ random draws $(\theta_i,s_i)$, $i=1, \dots , N$. The parameter $\theta_i$ is generated from the prior distribution $\pi$ and the vector of summary statistics $s_i$ is calculated for the $i^{th}$ data set that is simulated from the generative model with parameter $\theta_i$.
\item Associate to the $i^{th}$ simulation the weight $W_i=K_{\delta}(s_i-s_{obs})$, where $\delta$ is a tolerance threshold and $K_\delta$ a (possibly multivariate) smoothing kernel.
\item The distribution $(\sum_{i=1}^N W_i \mathbf{\delta}_{\theta_i})/(\sum_{i=1}^N W_i)$, in which $\mathbf{\delta}_{\theta}$ denotes the Dirac mass at $\theta$, approximates the target distribution. \end{enumerate}

\subsubsection{Point estimation and credibility intervals}

Assume here that $\theta=(\theta_1,\dots \theta_d)$ is a $d$-dimensional vector. We denote by $\theta_i=(\theta_{1,i},\dots \theta_{d,i})$ the simulated vectors of parameters in the previous paragraph. Once a sample from the target distribution has been obtained, several estimators may be considered for point estimation of each one-dimensional component $\theta_j$, $j\in \{1,\dots d\}$. Using the weighted sample $(\theta_{j,i},W_i)$, $i=1,\dots ,N$, the \textit{mean} of the target distribution $p(\theta_j|s_{obs})$ is estimated by
    \begin{equation}
    \hat{\theta_j}=\frac{\sum_{i=1}^N \theta_{j,i} W_i}{\sum_{i=1}^N W_i}=\frac{\sum_{i=1}^N \theta_{j,i} K_\delta(s_i-s_{obs})}{\sum_{i=1}^N K_\delta(s_i-s_{obs})}, \quad j=1,2,3 \label{estimateurbayesien}
    \end{equation}
which is the well-known Nadaraya--Watson regression estimator of the conditional expectation $\E(\theta_j\, |\, s_{obs})$ (see e.g.\ \cite[Chapter 1]{tsybakov_livreIV}). We also compute the \textit{medians}, \textit{modes}, and $95\%$ credibility intervals (CI) of the marginal posterior  distribution (see Section 3 of the supplementary material).

\subsubsection{Summary statistics} We are here interested in estimating the parameter $\theta=(\lambda,\gamma)$ of a SIR model (see Part I of this book). Two different sets of summary statistics are considered.\\

First, we consider the (infinite-dimensional) statistics $(R_t,t\in [0,T])$ consisting of the cumulated number of recoveries at time $t$ since the beginning of the epidemic. Because the data consist of the recovery times this curve $(R_t, t\in [0,T])$ can simply be viewed as a particular coding of the whole dataset. It is thus a sufficient statistic implying that the partial posterior distribution $p(\theta\, |\, R^1,R^2)$ is equal to the posterior distribution $p(\theta\, |\, \mathbf{x})$.\\
The $L^1$-norm between the $i$th simulated path $R_i$ and the observed one $R_{obs}$ is
\begin{equation}
 \label{eqn:l1norm}
\|R_{obs}-R_i\|_1=\int_0^T  |R_{obs,s}- R_{i,s}|\, ds \quad , \; i=1, \dots , N.
\end{equation}
   The weights $W_i$ are then computed as $W_i=K_{\delta}(\|R_{obs}-R_i\|_1) $ where $\delta$ is a tolerance threshold found by accepting a given percentage $P_{\delta}$ of the simulations and where an Epanechnikov kernel is chosen for $K$.\\

Second, when there is noise or when the recovery times have been binned, the full observations $(R_t,t\in [0,T])$ are unavailable. Then, we replace these summary statistics by a vector of summary statistics such as the numbers of recoveries per year during the observation period. We consider a d-dimensional vector of summary statistics of three different types: 1) number $R_T$ of individuals detected by the end of the observation period, 2) for each year $j$, numbers of removed individuals $R_{j+1}-R_j$, 3) numbers of new infectious in the first years (assuming for instance that all of them have been detected since) $I_{j+1}-I_j$ for $j=0,\dots,J_0$, where $J_0$ is a small number of years where the information is supposed to be known, 4) mean time during which an individual is infected but has not been detected in the $J_0$ first years. This mean time corresponds to the mean sojourn time in the class I for the $J_0$ first years. Since these new summary statistics are not sufficient anymore, the new partial posterior distribution may be different from the posterior $p(\theta \,|\,\mathbf{x})$. \\
In order to compute the weights $W_i$, we consider the following spherical kernel $K_{\delta}(x)\propto K(\|\mathbf{H}^{-1}x\|/\delta)$. Here $K$ denotes the one-dimensional Epanechnikov kernel, $\|\cdot\|$ is the Euclidean norm of $\R^d$ and $\mathbf{H}^{-1}$ a matrix. Because the summary statistics may span different scales, $\mathbf{H}$ is taken equal to the diagonal matrix with the standard deviation of each one-dimensional summary statistic on the diagonal.

\subsection{Comparisons between ABC and MCMC methods for a standard SIR model}\label{sectionMCMC}

Following \cite{bea02IV} a performance indicator for ABC techniques consists in their ability to replicate likelihood-based results given by MCMC. Here the situation is particularly favourable for comparing the two methods since the partial and the full posterior are the same. In the following examples, we choose samples of small sizes ($n=3$ and $n=29$) so that the dimension of the missing data is reasonable and MCMC achieves fast convergence. For large sample sizes with high-dimensional missing data, MCMC convergence might indeed be a serious issue and more thorough updating scheme shall be implemented \cite{cau08IV, chisstersinghfergusonIV}.\\

We consider the standard SIR model with infection rate $\lambda$ and recovery rate $\gamma$. The data consist of the recovery times and we assume that the infection times are not observed. We implement the MCMC algorithm of \cite{oneillrobertsIV}. A total of $10,000$ steps are considered for MCMC with an initial burn-in of $5,000$ steps. For ABC, the summary statistic consists of the cumulative number of recoveries as a function of time. A total of $100,000$ simulations are performed for ABC.\\

The first example was previously considered by \cite{oneillrobertsIV}. They simulated recovery times by considering one initial infectious individual and by setting $S_0=9$, $\lambda=0.12$, and $\gamma=1$. We choose gamma distributions for the priors of $\lambda$ and $\gamma$ with a shape parameters of $0.1$ and rate parameters of $1$ and $0.1$. As displayed by Figure \ref{fig:3obs}, the posterior distributions obtained with ABC are extremely close to the ones obtained with MCMC provided that the tolerance rate is sufficiently small. We see that the tolerance rate changes importantly the posterior distribution obtained with ABC (see the posterior distributions for $\lambda$).\\
\begin{figure}[!ht]
\begin{center}
\includegraphics[height = 8cm,angle=270]{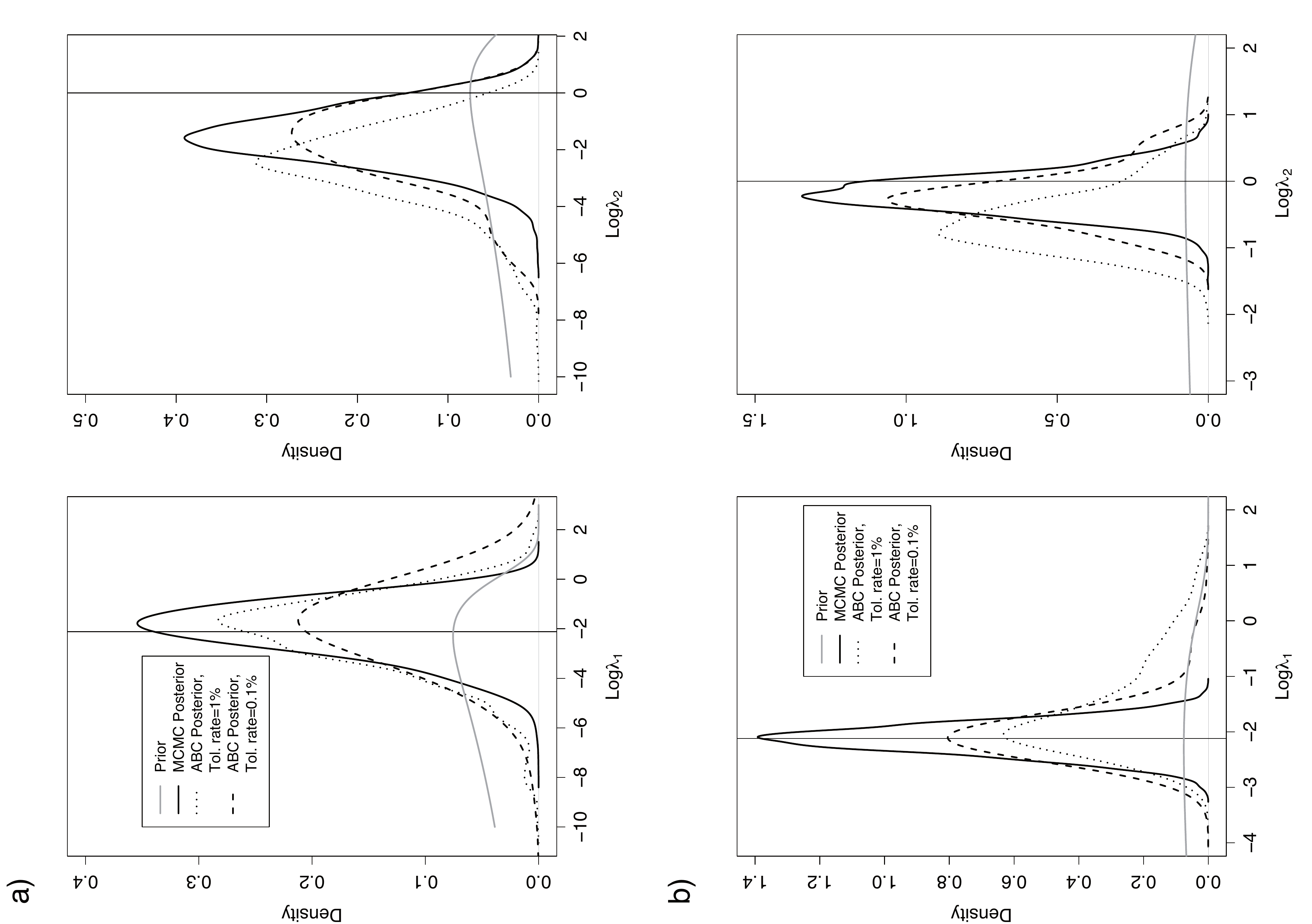}
\caption{{\small \textit{Comparison of the posterior densities obtained with MCMC and ABC. The vertical lines correspond to the values of the parameters used for generating the synthetic data. Left: the data consist of 3 recovery times that have been simulated by \cite{oneillrobertsIV}. Right: The data consist of 29 recovery times that we  simulated by setting $\lambda=0.12$, $\gamma=1$, $S_0=30$, $I_0=1$, and $T=5$ (see the supplementary material of \cite{blumtranIV} for the 29 recovery times).}}}
\label{fig:3obs}
\end{center}
\end{figure}

In a second example, we simulate a standard SIR trajectory with $\lambda=0.12$, $\gamma=1$, $S_0=30$ and $I_0=1$. The data now consist of 29 recovery times (and are given in the supplementary material of \cite{blumtranIV}). Once again, Figure \ref{fig:3obs} shows that the ABC and MCMC posteriors are close provided that the tolerance rate is small enough. ABC produces posterior distributions with larger tails compared to MCMC, even with the lowest tolerance rate of $0.1\%$. This can be explained by considering the extreme scenario in which the tolerance threshold $\delta$ goes to infinity: every simulation has a weight of 1 so that ABC targets the prior instead of the posterior. As the prior has typically larger tails than the posterior, ABC inflates the posterior tails.

\subsection{Comparison between ABC with full and binned recovery times}

\subsubsection{The curse of dimensionality and regression adjustments}\label{sec:correction}

In this case, the first set of summary statistics presented in Section \ref{section:presentationABC} can not be used any more and we have to use the second set of summary statistics, which constitute a vector of descriptive statistics as is much often encountered in the literature. In the case of a $d$-dimensional vector of summary statistics, the estimator of the conditional mean (\ref{estimateurbayesien}) is convergent if the tolerance rate satisfies $\lim_{N\rightarrow +\infty}\delta_N=0$, so that its bias converges to 0, and $\lim_{N\rightarrow +\infty}N\delta_N^d=+\infty$, so that its variance converges to 0 \cite{fan1IV}. As $d$ increases, a larger tolerance threshold shall be chosen to keep the variance small. As a consequence, the bias may increase with the number of summary statistics. This phenomenon known as the \textit{ curse of dimensionality} may be an issue for the ABC-rejection approach. The following paragraph presents regression-based adjustments that cope with the curse of dimensionality.\\

The adjustment principle is presented in a general setting within which the corrections of \cite{bea02IV} and \cite{blumfrancoisIV} can be derived. Correction adjustments aim at obtaining from a random couple $(\theta_i,s_i)$ a random variable distributed according to $p(\theta\,|\,s_{obs})$. The idea is to construct a coupling between the distributions $p(\theta|s_i)$ and $p(\theta | s_{obs})$, through which we can shrink the $\theta_i$'s to a sample of i.i.d.\ draws from $p(\theta | s_{obs})$. In the remaining of this subsection, we describe how to perform the corrections for each of the one-dimensional components separately. For $\theta\in \R$, correction adjustments are obtained by assuming a relationship $\theta=G(s,\varepsilon)=:G_s(\varepsilon)$ between the parameter and the summary statistics. Here $G$ is a (possibly complicated) function and $\varepsilon$ is a random variable with a distribution that does not depend on $s$. A possibility is to choose $G_s=F_s^{-1}$, the (generalized) inverse of the cumulative distribution function of $p(\theta|s)$. In this case, $\varepsilon=F_s(\theta)$ is a uniform random variable on $[0,1]$. The formula for adjustment is given by
\begin{equation}
\label{eqn:correc}
\theta_i^*=G_{s_{obs}}^{-1}(G_{s_i}(\theta_i)) \quad i=1, \dots , N.
\end{equation}For $G_s=F_s^{-1}$, the fact that the $\theta_i^*$'s are i.i.d.\ with density $p(\theta|s_{obs})$ arises from the standard inversion algorithm. Of course, the function $G$ shall be approximated in practice. As a consequence, the adjusted simulations $\theta_i^*$, $i=1,\dots ,N$, constitute an approximate sample of $p(\theta\,|\, s_{obs})$. The ABC algorithm with regression adjustment can be described as follows
\begin{enumerate}
\item Simulate, as in the rejection algorithm, a sample $(\theta_i,s_i)$, $i= 1,\dots, N$.
\item By making use of the sample of the $(\theta_i,s_i)$'s weighted by the $W_i$'s, approximate the function $G$ such that $\theta_i=G(s_i,\varepsilon_i)$ in the vicinity of $s_{obs}$.
\item Replace the $\theta_i$'s by the adjusted $\theta_i^*$'s. The resulting weighted sample $(\theta_i^{*},W_i)$, $i=1, \dots, N$, form a sample from the target distribution.
\end{enumerate}

\noindent \textbf{Local linear regression (LOCL)} The case where $G$ is approximated by a linear model $G(s,\varepsilon)=\alpha+ s^{t}\beta+\varepsilon$, was considered by \cite{bea02IV}. The parameters $\alpha$ and $\beta$ are inferred by minimizing the weighted squared error
\[\sum_{i=1}^N K_{\delta}(s_i-s_{obs}) (\theta_i-(\alpha+(s_i-s_{obs})^{T}\beta))^2.\]
Using (\ref{eqn:correc}), the correction of \cite{bea02IV} is derived as
\begin{equation}
\label{eqn:adj}
\theta_i^{*}=\theta_i-(s_i-s_{obs})^{T}\hat{\beta},\; i=1, \dots, N.
\end{equation}
Asymptotic consistency of the estimators of the partial posterior distribution with the correction (\ref{eqn:adj}) is obtained by \cite{blumIV}.\\

\noindent \textbf{Nonlinear conditional heteroscedastic regressions (NCH)}
To relax the assumptions of homoscedasticity and linearity inherent to local linear regression, Blum and Francois \cite{blumfrancoisIV} approximated $G$ by $G(s,\varepsilon)=m(s)+\sigma(s) \times \varepsilon$
where $m(s)$ denotes the conditional expectation, and $\sigma^2(s)$ the conditional variance. The estimators $\hat{m}$ and $\log\hat{\sigma}^2$ are found by adjusting two feed-forward neural networks using a regularized weighted squared error. For the NCH model, parameter adjustment is performed as follows
$$
\theta_i^{*}=\hat{m}(s_{obs})+(\theta_i-\hat{m}(s_{i}))\times \frac{\hat{\sigma}(s_{obs})}{\hat{\sigma}(s_i)},\; i=1, \dots , N.
$$
In practical applications of the NCH model, we train $L=10$ neural networks for each conditional regression (expectation and variance) and we average the results of the $L$ neural networks to provide the estimates $\hat{m}$ and $\log\hat{\sigma}^2$.\\

\noindent \textbf{Reparameterization} In both regression adjustment approaches, the regressions can be performed on transformations of the responses $\theta_i$ rather that on the responses themselves. Parameters whose prior distributions have finite supports are transformed via the logit function and non-negative parameters are transformed via the logarithm function. These transformations guarantee that the $\theta_i^{*}$'s lie in the support of the prior distribution and have the additional advantage of stabilizing the variance.\\

\noindent \textbf{Comparison between the first and second set of summary statistics} A simulation study is carried to compare the ABC methods based on the two different sets of summary statistics presented in Section \ref{section:presentationABC} has been carried in \cite{blumtranIV} using a slightly more elaborate SIR model with contact-tracing introduced in \cite{arazozaclemencontranIV}. Blum and Tran simulated $M=200$ synthetic data sets epidemic. % with $\mu_1=2\times 10^{-6}$, $\lambda_1=1.14\times 10^{-7}$, $\lambda_2=3.75\times 10^{-1}$, $\lambda_3=6.55\times 10^{-5}$, and $c=1$ \cite[]{arazozaclemencontran}.
%The initial conditions are set to $S_0=6\times 10^{6}$, the size of the Cuban population in the age-group 15-49, $I_0=232$  and $R_0=0$ \cite[]{auvert}. Here we simulate only 6 years of the epidemics.
%We study four variants of ABC for estimating $\lambda_1$, $\lambda_2$, and $\lambda_3$.
 When using the finite-dimensional vector of summary statistics, they perform the smooth rejection approach as well as the LOCL and NCH corrections with a total of 21 summary statistics. %: the 18 summary statistics corresponding to the yearly increases of $R^1$, $R^2$, and $I$; the final numbers of detected individuals $R^1_6$ and $R^2_6$; and the mean sojourn time in the class $I$.
 Each of the $M=200$ estimations of the partial posterior distributions are performed using a total of $N=5000$ simulations.
%\noindent\textbf{Prior distributions} The prior distributions for $\mu_1, \lambda_1, \lambda_2$ and $\lambda_3$ are chosen to be uniform on a log scale. The choice of a log scale reflects our uncertainty about the order of magnitude of the parameters. The prior distribution for $\log_{10}(\mu_1)$ is $\mathcal{U}(-6,-4)$ where $\mathcal{U}(a,b)$ denotes the uniform distribution on the interval $(a,b)$. The prior distribution is $\mathcal{U}(-9,-6)$ for $\log_{10}(\lambda_1)$, $\mathcal{U}(-4,3)$ for $\log_{10}(\lambda_2)$, and $\mathcal{U}(-8,2)$ for $\log_{10}(\lambda_3)$. The bounds of the uniform distributions are set to keep the simulations from being degenerate. The prior for $c$ is $\log(2)/\mathcal{U}(1/12,5)$ so that the half-life of $t\mapsto e^{-ct}$, which measures the individual contribution to the detection by contact-tracing, is uniformly distributed between $1/12$ and $5$ years.

%\includegraphics[width=0.5\linewidth, height=1 \textheight,trim=0cm 14cm 5cm 0cm]{Images/mcmc-abc.pdf}

\begin{figure}[!ht]
\begin{center}
\includegraphics[height = 11cm]{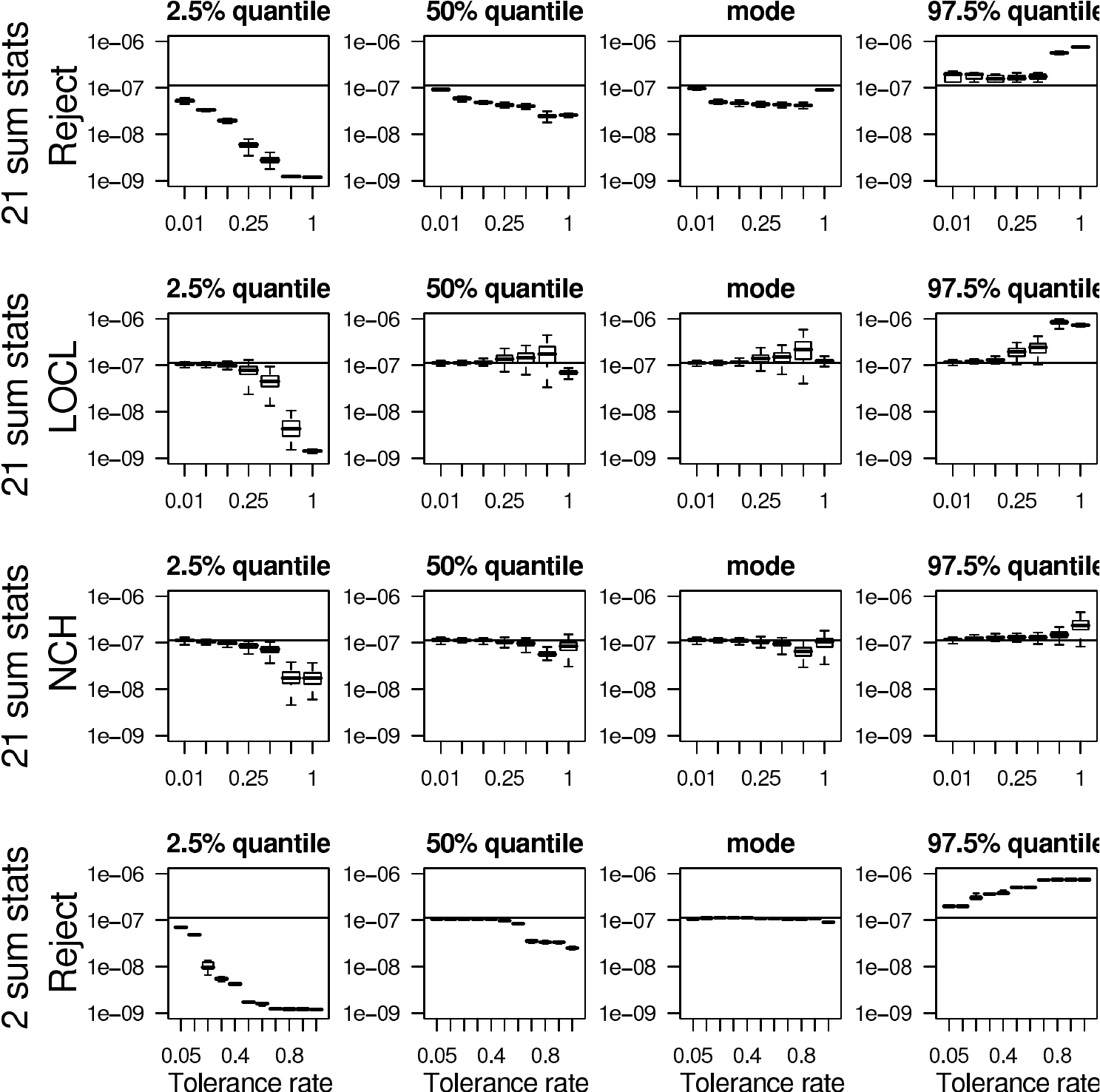}
\caption{{\small \textit{Boxplots of the $M=200$ estimated modes and quantiles ($2.5\%$, $50\%$, and $97.5\%$) of the partial posterior distributions of $\lambda$ in a model presented in Blum and Tran \cite{blumtranIV}. For each ABC method and each value of the tolerance rate, $200$ posterior distributions are computed for each of the 200 synthetic data sets. The horizontal lines correspond to the true value $\lambda=1.14 \times 10^{-7}$ used when simulating the 200 synthetic data sets. The different tolerance rates are 0.01, 0.05, 0.10, 0.25, 0.50, 0.50, 0.75, and 1 for all the ABC methods except the rejection scheme with the two summary statistics. For the latter method, the tolerance rates are $0.007$, $0.02$, $0.06$, $0.13$, $0.27$, $0.37$, $0.45$, $0.53$, $0.66$, $0.80$, $1$.}}}
\label{fig:boxplots}
\end{center}
\end{figure}

Figure \ref{fig:boxplots} displays the boxplots of the 200 estimated modes, medians, $2.5\%$ and $97.5\%$ quantiles of the posterior distribution for $\lambda$ as a function of the tolerance rate $P_\delta$.
First, the medians and modes are found to be equivalent except for the rejection method with 21 summary statistics for which the mode is less biased. For the lowest tolerance rates, the point estimates obtained with the four possible methods are close to the value $\lambda$ used in the simulations, with smaller CI for the LOCL and NCH variants. When increasing the tolerance rate, the bias of the point estimates obtained with the rejection method with 21 summary statistics slightly increases. By contrast, up to tolerance rates smaller than $50\%$, the biases of the point estimates obtained with the three other methods remain small. As can be expected, the widths of the CI obtained with the rejection methods increase with the tolerance rate while they remain considerably less variable for the methods with regression adjustment.\\

For further comparison of the different methods, we can compute the rescaled mean square errors (ReMSEs):
\begin{equation}
\mbox{ReMSE}(\lambda)=\frac{1}{M}\sum_{k=1}^{M} \frac{(\log(\hat{\lambda^k})-\log(\lambda))^2}{\mathrm{Range}(\mathrm{prior}(\lambda))^2},
\end{equation}
where $\hat{\lambda}^k$ is a point estimate obtained with the $k$th synthetic data set.

%The smallest values of the ReMSEs are usually reached for the lowest value of the tolerance rate (see Figure \ref{fig:RMSE}).
%For $\lambda_1$ and  $\lambda_2$,
%The ReMSEs of the point estimates obtained with the four different methods are comparable for the lowest tolerance rate. However, the smallest values of the ReMSEs are always found when performing the rejection method with the sufficient summary statistics $(R_t, t\in [0,T])$.\\

To compare the whole posterior distributions obtained with the four different methods, we can also compute the different CIs. The rescaled mean CI (RMCI) is defined as follows
\begin{equation}
\mbox{RMCI}=\frac{1}{M}\sum_{k=1}^{M} \frac{|IC^k|}{\mathrm{Range}(\mathrm{prior}(\lambda))},
\end{equation}
where $|IC^k|$ is the length of the $k$th estimated $95\%$ CI for the parameter $\lambda$.
As displayed by Figure \ref{fig:boxplots}, the CIs obtained with smooth rejection increase importantly with the tolerance rate whereas such an important increase is not observed with regression adjustment.

%In Figure \ref{fig:RMSE}, the CIs obtained with the NCH method are clearly the thinnest, those obtained with the rejection methods are the widest and those obtained with the LOCL method have intermediate width. %In the following, we perform the NCH correction when considering the finite dimensional vector of summary statistics. This choice is motivated by the small ReMSEs and RMCIs obtained with the NCH method (Figure \ref{fig:RMSE}).

%\begin{figure}[!ht]
%\begin{center}
%\includegraphics[height = 9cm]{Images/Figure4_5merged.eps}
%\end{center}
%\caption{{\small \textit{Rescaled mean squared error (ReMSE) of the posterior mode and rescaled mean credibility interval (RMCI).}}}
%\label{fig:RMSE}
%\end{figure}

\section{Sensitivity analysis}

Epidemiological models designed in order to test public health scenarios by simulations or disentangle various factors for a better understanding of the disease propagation are often over-parameterized. Input parameters are the rates describing the times that individuals stay in each compartment, for example. The sources that are used to calibrate the model can also be numerous: some parameters are for example obtained from epidemiological studies or clinical trials, but there can be uncertainty on their values due to various reasons. The restricted size of the sample in these studies brings uncertainty on the estimates, which are given with uncertainty intervals (classically, a 95\% confidence interval). Different studies can provide different estimates for the same parameters. The study populations can be subject to selection biases. In the case of clinical trials where the efficacy of a treatment is estimated, the estimates can be optimistic compared with what will be the effectiveness in real-life, due to the protocol of the trials. It is important to quantify how theses uncertainties on the input parameters can impact the results and the conclusion of an epidemiological modelling study. To check the robustness of some output with respect to the parameters, sensitivity analyses are often performed.\\

In a mathematical model where the output $y\in \R$ depends on a set of $p\in \N$ input parameters $x=(x_1,...x_p) \in \mathbb{R}^p$ through the relation $y=f(x)$, there are various ways to measure the influence of the input $x_\ell$, for $\ell\in \{1,\dots,p\}$, on $y$. In this article, we are interested in Sobol indices \cite{sobolIV}, which are based on an ANOVA decomposition (see \cite{saltellilivre2008IV,jacquesthesisIV,jacquesIV} for a review). These indices have been proposed to take into account the uncertainty on the input parameters that are here considered as a realisation of a set of independent random variables $X=(X_1,...X_p)$, with a known distribution and with possibly correlated components. Denoting by $Y = f(X)$ the random response, the first-order Sobol indices can be defined for $\ell\in \{1,\dots, p\}$ by
\begin{equation}
S_\ell=\frac{\Var\big(\E[Y\ |\ X_\ell]\big)}{\Var(Y)}.\label{def:Sobol1}
\end{equation}
This first-order index $S_\ell$ corresponds to the sensitivity of the model to $X_\ell$ alone. Higher order indices can also be defined using ANOVA decomposition: considering $(\ell,\ell') \in \{1,\dots, p\}$, we can define the second order sensitivity, corresponding to the sensitivity of the model to the interaction between $X_\ell$ and $X_{\ell'}$  index by
\begin{equation}
S_{\ell\ell'}=\frac{\Var\big(\E[Y\ |\ X_\ell,X_{\ell'}]\big)}{\Var(Y)} - S_{\ell} - S_{\ell'} \label{def:Sobol2}
\end{equation}
We can also define the total sensitivity indices by
\begin{equation}
S_{T_\ell} = \sum_{L \subset \{1,\dots, p\} \, |\, \ell \in L} S_L. \label{def:SobolTot}
\end{equation}
As the estimation of the Sobol indices can be computer time consuming, a usual practice consists in estimating the first-order and total indices, to assess 1) the sensitivity of the model to each parameter taken separately and 2) the possible interactions, which are quantified by the difference between the total order and the first-order index for each parameter. Several numerical procedures to estimate the Sobol indices have been proposed, in particular by Jansen \cite{jansenIV} (see also \cite{saltellilivre2000IV,saltellilivre2008IV}). These estimators, that we recall in the sequel, are based on Monte Carlo simulations of $(Y,X_1\dots X_p)$.\\

The literature focuses on deterministic relations between the input and output parameters. In a stochastic framework where the model response $Y$ is not unique for given input parameters, few works have been done, randomness being usually limited to input variables. Assume that:
\begin{equation}
Y=f(X,\varepsilon),\label{modelsto}
\end{equation}where $X=(X_1,\dots X_p)$ still denotes the random variables modelling the uncertainty of the input parameters and where $\varepsilon$ is a noise variable. When noise is added in the model, the classical estimators do not always work: $Y$ can be very sensitive to the addition of $\varepsilon$. Moreover, this variable is not always controllable by the user.\\

When the function $f$ is linear, we can refer to \cite{fortkleinlagnouxlaurentIV}. In the literature, meta-models are used: approximating the mean and the dispersion of the response by deterministic functions allows us to come back to the classical deterministic framework (e.g.\ Janon et al.\ \cite{janonnodetprieurIV}, Marrel et al.\ \cite{marrelioossdaveigaribatetIV}). We study here another point of view, which is based on the non-parametric statistical estimation of the term $\Var\big(\E[Y\ |\ X_\ell]\big)$ appearing in the numerator of \eqref{def:Sobol1}. Approaches based on the Nadaraya--Watson kernel estimator have been proposed by Da Veiga and Gamboa \cite{daveigagamboaIV} or Sol\'{\i}s \cite{solisthesisIV} while an approach based on warped wavelet decompositions is proposed by Castellan et al.\ \cite{castellancousientranIV}. An advantage of these non-parametric estimators is that their computation requires less simulations of the model. For Jansen estimators, the number of calls of $f$ required to compute the sensitivity indices is $n(p+1)$, where $n$ is the number of independent random vectors $(Y^i,X^i_1,\dots X^i_p)$ ($i\in \{1,\dots n\}$) that are sampled for the Monte Carlo procedure, making the estimation of the sensitivity indices time-consuming for sophisticated models with many parameters. In addition, for the non-parametric estimators, the convergence of the mean square error to zero may be faster than for Monte Carlo estimators, depending on the regularity of the model.

\subsection{A non-parametric estimator of the Sobol indices of order 1}\label{section:nonparam}

Denoting by $V_\ell=\E\big(\E^2(Y\ |\ X_\ell)\big)$ the expectation of the square conditional expectation of $Y$ knowing $X_\ell$, we have:
\begin{equation}
S_\ell = \frac{V_\ell -\E(Y)^2}{\Var(Y)},\label{def:S_ell}
\end{equation}which can be approximated by
\begin{equation}
\widehat{S}_\ell=\frac{\widehat{V}_\ell - \bar{Y}^2}{\widehat{\sigma}^2_Y}\label{def:generique}
\end{equation}where $$\bar{Y}=\frac{1}{n}\sum_{j=1}^n Y_j\mbox{  and  }\widehat{\sigma}^2_Y=\frac{1}{n} \sum_{j=1}^n (Y_j-\bar{Y})^2$$are the empirical mean and variance of $Y$. We consider here two approximations $\widehat{V}_\ell$ of $V_\ell$, based on Nadaraya--Watson and on warped wavelet estimators.\\
%At an advanced stage of this work, we learned that the Nadaraya--Watson-based estimator of Sobol indices of order 1 had also been proposed and studied in the PhD of Sol\'is \cite{solisthesis}. Using a result on estimation of covariances by Loubes et al.\ \cite{loubesmarteausolis}, they obtain an elbow effect. However their estimation is not adaptative. For the warped wavelet estimator, we propose a model selection procedure based on a work by Laurent and Massart \cite{laurentmassart} to make the estimator adaptative.

Assume that we have $n$ independent couples $(Y^i, X_1^i,\dots X_p^i)$ in $\R\times \R^p$, for $i\in \{1,\dots, n\}$, generated by \eqref{modelsto}. Let us start with the kernel-based estimator:

\begin{definition}\label{def:estimateur1}Let $K\, :\, \R\mapsto \R$ be a kernel such that $\int_{\R}K(u)du=1$.
Let $h>0$ be a window and let us denote $K_h(x)=K(x/h)/h$. An estimator of $S_\ell$ for $\ell\in \{1,\dots p\}$ is:
\begin{equation}
\widehat{S}_\ell^{(NW)}=\frac{\frac{1}{n}\sum_{i=1}^n \Big(\frac{\sum_{j =1}^n Y_j K_h(X_\ell^j-X_\ell^i)}{\sum_{j =1}^n K_h(X_\ell^j-X_\ell^i)} \Big)^2-\bar{Y}^2}{\widehat{\sigma}_Y^2}.
\end{equation}
\end{definition}

This estimator is based on the Nadaraya--Watson estimator of $\E(Y\, |\, X_\ell=x)$ given by (e.g.\ \cite{tsybakov_livreIV})
$$\frac{\sum_{j=1}^n Y_j K_h(X_\ell^j-x)}{\sum_{j=1}^n K_h(X_\ell^j-x)}.$$
Replacing this expression in \eqref{def:generique} provides $\widehat{S}_\ell^{(NW)}$. This estimator and the rates of convergence have been studied by Sol\'{\i}s \cite{solisthesisIV}. If we instead use a warped wavelet decomposition of $\E(Y\, |\, X_\ell=x)$ (see e.g.\ \cite{chagnyIV,kerkyacharianpicardIV}), this provides an estimator studied by Castellan et al.\ \cite{castellancousientranIV}. Let us present this second estimator.\\

Let us denote by $G_\ell$ the cumulative distribution function of $X_\ell$.
Let $(\psi_{jk})_{j\geq -1,k\in \Z}$ be a Hilbert wavelet basis of $L^2$, the space of real functions that are square integrable with respect to the Lebesgue measure on $\R$. In the sequel, we denote by $\langle f,g\rangle=\int_\R f(u)g(u)du$, for $f,g\in L^2$, the usual scalar product of $L^2$. The wavelet $\psi_{-10}$ is the father wavelet, and for $k\in \Z$, $\psi_{-1k}(x)=\psi_{-10}(x-k)$. The wavelet $\psi_{00}$ is the mother wavelet, and for $j\geq 0$, $k\in \Z$, $\psi_{jk}(x)=2^{j/2} \psi_{00}(2^j x -k)$.

\begin{definition}\label{def:estimateur2}Let us define for $j\geq -1$, $k\in \Z$,
\begin{equation}
\widehat{\beta}^\ell_{jk}=\frac{1}{n} \sum_{i=1}^n Y_i \psi_{jk}(G_\ell(X_\ell^i)).
\end{equation}Then, we define the (block thresholding) estimator of $S_\ell$ as
\begin{equation}
 \widehat{S}_\ell^{(WW)}=\frac{\widehat{V}_\ell-\bar{Y}^2}{\widehat{\sigma}_Y^2},
 \end{equation}
where $\widehat{V}_\ell$ is an estimator of the variance $V_\ell$ given by:
\begin{equation}\widehat{V}_\ell=\sum_{j= -1}^{J_n}  \Big[ \sum_{k\in \Z}\big(\widehat{\beta}^{\ell}_{jk}\big)^2 - w(j)\Big]\ind_{ \sum_{k\in \Z}\big(\widehat{\beta}^{\ell}_{jk}\big)^2 \geq w(j)} \label{def:hat_theta}
\end{equation}
with $w(j)=K \Big(\frac{2^j+\log 2}{n}\Big)$ and $J_n:=\big[\log_2\big(\sqrt{n}\big)\big]$ (where $[\cdot]$ denotes the integer part) and $K$ a positive constant.
\end{definition}

Let us present the idea explaining the estimator proposed in Definition \ref{def:estimateur2}. Let us introduce centered random variables $\eta_{\ell}$ such that
\begin{equation}
Y=f(X,\varepsilon)=\E(Y\, |\, X_\ell)+\eta_\ell.
\end{equation}
Let $g_{\ell}(x)=\E(Y\, |\, X_\ell=x)$ and $h_\ell(u)=g_\ell\circ G_\ell^{-1}(u)$. $h_\ell$ is a function from $[0,1]\mapsto \R$ that belong to $L^2$ since $Y\in L^2$.
Then
\begin{align}
h_\ell(u)=\sum_{j\geq -1} \sum_{k\in \Z} \beta_{jk}^\ell \psi_{jk}(u),\quad \mbox{ with }\\
\beta_{jk}^\ell=\int_0^1 h_\ell(u) \psi_{jk}(u)du=\int_\R g_\ell(x) \psi_{jk}(G_\ell(x)) G_\ell(dx).\end{align}
Notice that the sum in $k$ is finite because the function $h_\ell$ has compact support in $[0,1]$. It is then natural to estimate $h_\ell(u)$ by
\begin{equation}
\widehat{h}_\ell=\sum_{j\geq -1}\sum_{k\in \Z} \widehat{\beta}_{jk}^\ell \psi_{jk}(u),\label{def:hat_h}
\end{equation} and we then have:
%\pagebreak
\begin{align}
V_\ell&=\E\big(\E^2(Y\, |\, X_\ell)\big)\nonumber\\
&=  \int_{\R} G_\ell(dx)\Big(\sum_{j\geq -1} \sum_{k\in \Z} \beta_{jk}^\ell \psi_{jk}\big(G_\ell(x)\big)\Big)^2\nonumber\\
&=  \int_0^1 \Big(\sum_{j\geq -1} \sum_{k\in \Z} \beta_{jk}^\ell \psi_{jk}(u)\Big)^2 \ du\nonumber\\
&=  \sum_{j\geq -1}\sum_{k\in \Z} \big( \beta_{jk}^\ell\big)^2=\|h_\ell\|^2_2.
\end{align}
Adaptive estimation of $\|h_\ell\|^2_2$ has been studied in
\cite{laurentmassartIV}, which provides the block thresholding estimator $\widehat{V}_\ell$ in Definition \ref{def:estimateur2}. The idea is: 1) to sum the terms $\big( \beta_{jk}^\ell\big)^2$, for $j\geq 0$, by blocks $\{(j,k),\  k\in \Z\}$ for $j\in \{-1,\dots,J_n \}$ with a penalty $w(j)$ for each block to avoid choosing too large $j$'s, 2) to cut the blocks that do not sufficiently contribute to the sum, in order to obtain statistical adaptation.\\

Notice that $\widehat{V}_\ell$ can be seen as an estimator of $V_\ell$ resulting from a model selection on the choice of the blocks $\{(j,k),\ k\in \Z\}$, $j\in \{-1,\dots, J_n\}$ that are kept, with the penalty function $\mbox{pen}(\mathcal{J})=\sum_{j\in \mathcal{J}} w(j)$, for $\mathcal{J}\subset \{-1,\dots, J_n\}$. Indeed:
\begin{align}
\widehat{V}_\ell &= \sup_{\mathcal{J}\subset \{-1,0,\dots,J_n\}} \sum_{j\in \mathcal{J}}  \Big[ \sum_{k\in \N}\big(\widehat{\beta}^{\ell}_{jk}\big)^2 - w(j)\Big]\nonumber\\
&= \sup_{\mathcal{J}\subset \{-1,0,\dots,J_n\}}\sum_{j\in \mathcal{J}}  \sum_{k\in \N}\big(\widehat{\beta}^{\ell}_{jk}\big)^2 - \pen(\mathcal{J}).\label{lien_pen}
\end{align}

Note that the definition of the estimator and the penalization depend on a constant $K$ through the definition of $w(j)$. The value of this constant is chosen in order to obtain oracle inequalities. In practice, this constant is hard to compute, and can be chosen by a slope heuristic approach (see e.g.\ \cite{arlotmassartIV}).

\subsection{Statistical properties}

In this Section, we are interested in the rate of convergence to zero of the mean square error (MSE) $\E\big((S_\ell-\widehat{S}_\ell)^2\big)$. Let us consider the generic estimator $\widehat{S}_\ell$ defined in \eqref{def:generique}, where $\widehat{V}_\ell$ is an estimator of $V_\ell=\E(\E^2(Y\ |\ X_\ell))$ (not necessarily \eqref{def:hat_theta}).
We first start with a Lemma stating that the MSE can be obtained from the rate of convergence of $\widehat{V}_\ell$ to $V_\ell$.

\begin{lemma}\label{lemme1}Consider the generic estimator $\widehat{S}_\ell$ defined in \eqref{def:generique} and $\widehat{V}_\ell$ an estimator of $V_\ell$ (not necessarily \eqref{def:hat_theta}).
Then there is a constant $C$ such that:
\begin{equation}
\E\big((S_\ell-\widehat{S}_\ell)^2\big) \leq \frac{C}{n}+\frac{4}{\Var(Y)^2} \E\Big[\big(\widehat{V}_\ell - V_\ell\big)^2\Big].
\label{erreur:generique}
\end{equation}
\end{lemma}

\begin{proof}From \eqref{def:S_ell} and \eqref{def:generique},
\begin{align}
\E\big((S_\ell-\widehat{S}_\ell)^2\big)= & \E\Big[\Big(\frac{V_\ell-\E(Y)^2}{\Var(Y)}-\frac{\widehat{V}_\ell-\bar{Y}^2}{\widehat{\sigma}_Y^2}\Big)^2\Big]\nonumber\\
\leq & 2 \E\Big[\Big(\frac{\E(Y)^2}{\Var(Y)}-\frac{\bar{Y}^2}{\widehat{\sigma}_Y^2}\Big)^2\Big]+2\E\Big[\Big(\frac{V_\ell}{\Var(Y)}-\frac{\widehat{V}_\ell}{\widehat{\sigma}_Y^2}\Big)^2\Big].\label{etape1} \end{align}The first term in the right-hand side (r.h.s.)\ is in $C/n$.
For the second term in the right-hand side of \eqref{etape1}:
\begin{align}
\E\Big[\Big(\frac{V_\ell}{\Var(Y)}-\frac{\widehat{V}_\ell}{\widehat{\sigma}_Y^2}\Big)^2\Big] \leq & 2 \E\Big[\widehat{V}_\ell^2 \Big(\frac{1}{\Var(Y)}-\frac{1}{\widehat{\sigma}_Y^2}\Big)^2\Big]+\frac{2}{\Var(Y)^2} \E\Big[\big(\widehat{V}_\ell - V_\ell\big)^2\Big].\label{etape2PartIV}
\end{align}The first term in the r.h.s.\ is also in $C/n$, which concludes the proof.
\end{proof}

The preceding lemma implies that the rate of convergence of $\widehat{V}_\ell$ to $V_\ell$ is determinant for the rate of convergence of $\widehat{S}_\ell$. We recall the result of Sol\'{\i}s \cite{solisthesisIV}, where an elbow effect for the MSE is shown when the regularity of the density of $(X_\ell, Y)$ varies. The case of the warped wavelet estimator introduced by Castellan et al \cite{castellancousientranIV} is studied at the end of the section and the rate of convergence is stated in Corollary \ref{corol:vitesseconvergence}.

\subsubsection{MSE for the Nadaraya--Watson estimator}

Using the preceding Lemma, Loubes Marteau and Sol\'{\i}s prove an elbow effect for the estimator $\widehat{S}^{(NW)}_\ell$.
Let us introduce $\mathcal{H}(\alpha,L)$, for $\alpha,L>0$, the set of functions $\phi$ of class $[\alpha]$, whose derivative $\phi^{([\alpha])}$ is $\alpha-[\alpha]$ H\"{o}lder continuous with constant $L$.

\begin{proposition}[Loubes Marteau and Sol\'{\i}s \cite{loubesmarteausolisIV,solisthesisIV}]Assume that $\E(X_\ell^4)<+\infty$, that the joint density $\phi(x,y)$ of $(X_\ell,Y)$ belongs to $\mathcal{H}(\alpha,L)$, for $\alpha, L>0$ and that the marginal density of $X_\ell$, $\phi_\ell$ belongs to $\mathcal{H}(\alpha',L')$ for $\alpha'>\alpha$ and $L'>0$. Then:\\
If $\alpha\geq 2$, there exists a constant $C>0$ such that
$$\E\big((S_\ell-\widehat{S}_\ell)^2\big)\leq \frac{C}{n}.$$
If $\alpha<2$, there exists a constant $C>0$ such that
$$\E\big((S_\ell-\widehat{S}_\ell)^2\big)\leq C\big(\frac{\log^2 n}{n}\big)^{\frac{2\alpha}{\alpha+2}}.$$
\end{proposition}For smooth functions ($\alpha\geq 2$), Loubes et al.\ recover a parametric rate, while they still have a nonparametric one when $\alpha<2$. Their result is based on \eqref{erreur:generique} and a bound for $\E\Big[\big(\widehat{V}_\ell - V_\ell\big)^2\Big]$ given by \cite[Th.\ 1]{loubesmarteausolisIV}, whose proof is technical. Since their result is not adaptive, they require the knowledge of the window $h$ for numerical implementation. Our purpose is to provide a similar result for the warped wavelet adaptive estimator, with a shorter proof.

\subsubsection{MSE for the warped wavelet estimator}

Let us introduce first some additional notation. We define, for $\mathcal{J}\subset \{-1,\dots,J_n\}$, the projection $h_\mathcal{J,\ell}$ of $h$ on the subspace spanned by $\{\psi_{jk},\mbox{ with } j\in \mathcal{J},\, k\in \Z\}$ and its estimator $\widehat{h}_{\mathcal{J},\ell}$:
\begin{align}
& h_{\mathcal{J},\ell}(u)=\sum_{j\in \mathcal{J}}  \sum_{k\in \Z} \beta^\ell_{jk}\psi_{jk}(u)\\
& \widehat{h}_\mathcal{J,\ell}(u)=\sum_{j\in \mathcal{J}} \sum_{k\in \Z} \widehat{\beta}^\ell_{jk}\psi_{jk}(u).
\end{align}
We also introduce the estimator of $V_\ell$ for a fixed subset of resolutions $\mathcal{J}$:
\begin{align}
\widehat{V}_{\mathcal{J},\ell}= \| \widehat{h}_{\mathcal{J},\ell}\|^2_2=\sum_{j\in \mathcal{J}}   \sum_{k\in \Z}\big(\widehat{\beta}^{\ell}_{jk}\big)^2. \label{def:est_interm2}
\end{align}Note that $\widehat{V}_{\mathcal{J},\ell}$ is one possible estimator $\widehat{V}_\ell$ in Lemma \ref{lemme1}.\\

The estimators $\widehat{\beta}_{jk}$ and $\widehat{V}_{\mathcal{J},\ell}$ have natural expressions in term of the empirical process $\gamma_n(dx)$ defined as follows:
\begin{definition}\label{def:gamma_n}The empirical measure associated with our problem is:
\begin{equation}
\gamma_n(dx)=\frac{1}{n}\sum_{i=1}^n Y_i \delta_{G_\ell(X_\ell^i)}(dx)
\end{equation}where $\delta_a(dx)$ denotes the Dirac mass in $a$.\\
For a measurable function $f$, $
\gamma_n(f)=\frac{1}{n}\sum_{i=1}^n Y_i f\big(G_\ell(X_\ell^i)\big). $ We also define the centered integral of $f$ with respect to $\gamma_n(dx)$ as:
\begin{align}
\bar{\gamma}_n(f)= & \gamma_n(f)-\E\big(\gamma_n(f)\big)\label{def:gamma_bar}\\
= & \frac{1}{n}\sum_{i=1}^n \Big(Y_i f\big(G_\ell(X_\ell^i)\big)-\E\big[Y_if\big(G_\ell(X_\ell^i)\big)\big]\Big).
\end{align}
\end{definition}

Using the empirical measure $\gamma_n(dx)$, we have:
\begin{align*}
\widehat{\beta}^\ell_{jk}=\gamma_n\big(\psi_{jk}\big)=\beta^\ell_{jk}+\bar{\gamma}_n\big(\psi_{jk}\big).
\end{align*}
Let us introduce the correction term
\begin{align}
\zeta_n=& 2 \bar{\gamma}_n\big(h_\ell\big)\label{def:zetan}\\
= & 2 \Big[\frac{1}{n}\sum_{i=1}^n Y_i h_\ell\big(G_\ell(X_\ell^i)\big) - \E\Big(Y_1 h_\ell\big(G_\ell(X_\ell^1)\big)\Big)\Big]\nonumber\\
= & 2 \Big[\frac{1}{n}\sum_{i=1}^n  h^2_\ell\big(G_\ell(X_\ell^i)\big)- \|h_\ell\|_2^2 \Big]+\frac{2}{n} \sum_{i=1}^n \eta_\ell^i h_\ell\big(G_\ell(X_\ell^i)\big).\label{etape:zeta}
 %2\langle \widehat{h}_\ell-h_\ell,h_\ell\rangle=\sum_{j\geq -1}\sum_{k\in \Z} \beta_{jk}\big(\widehat{\beta}_{jk}-\beta_{jk}\big).
\end{align}

The rate of convergence of the estimator \eqref{def:hat_theta} is obtained in \cite{castellancousientranIV} based on the estimate presented in the next theorem. This result is derived using ideas due to Laurent and Massart \cite{laurentmassartIV} who considered estimation of quadratic functionals in a Gaussian setting. Because we are not necessarily in a Gaussian setting here, we rely on empirical processes and use sophisticated technology developed by Castellan \cite{castellanIV}.

\begin{theorem}[Castellan, Cousien, Tran \cite{castellancousientranIV}]\label{th1}
Let us assume that the random variables $Y$ are bounded by a constant $M$, and let us choose a father and a mother wavelets $\psi_{-10}$ and $\psi_{00}$ that are continuous with compact support (and thus bounded).
The estimator $\widehat{V}_\ell$ defined in \eqref{def:hat_theta} is almost surely finite, and:
\begin{align}
\E\Big[\big(\widehat{V}_\ell-V_\ell- \zeta_n \big)^2\Big]\leq  C\ \inf_{\mathcal{J}\subset \{-1,\dots,J_n\}} \Big(\|h_\ell -h_{\mathcal{J},\ell}\|^4_2 + \frac{\Card^2(\mathcal{J})}{n^2}\Big)+\frac{C' \log_2^2(n)}{n^{3/2}},\label{oracle}
\end{align}for constants $C$ and $C'>0$.
\end{theorem}

We deduce the following corollary from the estimate obtained above. Let us consider the Besov space $\mathcal{B}(\alpha,2,\infty)$ of functions $h=\sum_{j\geq -1}\sum_{k\in \Z} \beta_{jk} \psi_{jk}$ of $L^2$ such that
$$|h|_{\alpha,2,\infty}:= \sum_{j\geq 0} 2^{j\alpha}\sqrt{\sup_{0<v\leq 2^{-j}} \int_0^{1-v}|h(u+v)-h(u)|^2du} <+\infty.$$
 %We denote by $\mathcal{B}_{\alpha,2,\infty}(R)=\{h\in \mathcal{B}(\alpha,2,\infty),\ |h|_{\alpha,2,\infty}<R\}$ the ball of radius $R$ in this space.

For a $h\in \mathcal{B}(\alpha,2,\infty)$ and $h_\mathcal{J}$ its projection on
\[\Vect \{\psi_{jk},\ j\in \mathcal{J}=\{-1,\dots J_{\mbox{{\scriptsize max}}}\},\ k\in \Z\},\]
we have the following approximation result from \cite[Th.\ 9.4]{hardlekerkapicardtsybaIV}.
\begin{proposition}[H\"ardle, Kerkyacharian, Picard and Tsybakov]\label{prop:besov}
Assume that the wavelet function $\psi_{-10}$ has compact support and is of class $\mathcal{C}^N$ for an integer $N>0$. Then, if $h\in \mathcal{B}(\alpha,2,\infty)$ with $\alpha<N+1$,
\begin{align}
\sup_{\mathcal{J}\subset \N\cup \{-1\}} 2^{\alpha J_{\mbox{{\scriptsize max}}}} \|h-h_{\mathcal{J}}\|_2 = \sup_{\mathcal{J}\subset \N\cup \{-1\}} 2^{\alpha  J_{\mbox{{\scriptsize max}}}} \Big( \sum_{j\geq J_{\mbox{{\scriptsize max}}}} \sum_{k\in \Z} \beta_{jk}^2 \Big)^{1/2} <+\infty.
\end{align}
\end{proposition}
Notice that Theorem 9.4 of \cite{hardlekerkapicardtsybaIV} requires assumptions that are fulfilled when $\psi_{-10}$ has compact support and is smooth enough (see the comment after the Corol. 8.2 of \cite{hardlekerkapicardtsybaIV}).

\begin{corollary}\label{corol:vitesseconvergence}If $\psi_{-10}$ has compact support and is of class $\mathcal{C}^N$ for an integer $N>0$ and if $h_\ell$ belongs to a ball of radius $R>0$ of $\mathcal{B}(\alpha,2,\infty)$ for $0<\alpha<N+1$, then
\begin{align}
\sup_{h\in \mathcal{B}(\alpha,2,\infty)}\E\Big[\big(\widehat{V}_\ell-V_\ell\big)^2\Big]\leq & C \Big(n^{-\frac{8\alpha}{4\alpha+1}}+\frac{1}{n}\Big).
\end{align}
As a consequence, we obtain the following elbow effect:\\
If $\alpha\geq \frac{1}{4}$, there exists a constant $C>0$ such that
$$\E\big((S_\ell-\widehat{S}_\ell)^2\big)\leq \frac{C}{n}.$$
If $\alpha<\frac{1}{4}$, there exists a constant $C>0$ such that
$$\E\big((S_\ell-\widehat{S}_\ell)^2\big)\leq C n^{-\frac{8\alpha}{4\alpha+1}}.$$
\end{corollary}

\begin{proof}Using \eqref{oracle} and the fact that
\begin{equation}
\E\big(\zeta^2_n\big)= \frac{4}{n} \Var\Big(Y_1 h_\ell \big(G_\ell(X^1_\ell)\big)\Big)\leq \frac{2M^2 \|h_\ell\|_2^2}{n},
\end{equation}
we obtain:
\begin{equation}\label{etape13}
\E\Big[\big(\widehat{V}_\ell-V_\ell\big)^2\Big]\leq C \Big[\inf_{\mathcal{J}\subset \{-1,\dots, J_n\}} \Big( \|h_\ell-h_{\mathcal{J},\ell}\|_2^4 + \frac{\Card^2(\mathcal{J})}{n^2} \Big) + \frac{1+\|h_\ell\|_2^2}{n}\Big]. %\frac{D_\mathcal{J}}{n^2}+\big(\pen(\mathcal{J})- \frac{D_\mathcal{J}}{n}\big)^2+ \frac{\|h_\ell\|_2^2}{n} \Big].
\end{equation}
If $h_\ell\in  \mathcal{B}(\alpha,2,\infty)$, then from Proposition \ref{prop:besov}, we have for $\mathcal{J}=\{-1,\dots, J_{\mbox{{\scriptsize max}}}\}$ that $\|h_\ell-h_{\mathcal{J},\ell}\|_2^4 \leq 2^{-4 \alpha \ J_{\mbox{{\scriptsize max}}}}$. Thus, for subsets $\mathcal{J}$ of the form considered, the infimum is attained when choosing $J_{\mbox{{\scriptsize max}}}= \frac{2}{4\alpha+1} \log_2(n)$, which yield an upper bound in $n^{8\alpha/(4\alpha+1)}$. \\

For $h_\ell$ in a ball of radius $R$, $\|h_\ell\|_2^2\leq R^2$, and we can find an upper bound that does not depend on $h$. Because the last term in \eqref{etape13} is in $1/n$, the elbow effect is obtained by comparing the order of the first term in the r.h.s.\ ($n^{8\alpha/(4\alpha+1)}$) with $1/n$ when $\alpha$ varies. \hfill $\Box$

\end{proof}

Let us remark that in comparison with the result of Loubes et al.\ \cite{loubesmarteausolisIV}, the regularity assumption here is on the function $h_\ell$ rather than on the joint density $\phi(x,y)$ of $(X_\ell, Y)$. The adaptivity of the estimator is then welcomed since the function $h_\ell$ is \textit{a priori} unknown. Note that in applications, the joint density $\phi(x,y)$ also has to be estimated and hence has an unknown regularity. \\

When $\alpha<1/4$ and $\alpha\rightarrow 1/4$, the exponent $8\alpha/(4\alpha+1)\rightarrow 1$.
In the case when $\alpha> 1/4$, we can show from the estimate of Th. \ref{th1} that:
\begin{equation}
\lim_{n\rightarrow +\infty} n \E\Big[\big(\widehat{V}_\ell-V_\ell- \zeta_n \big)^2\Big]=0,
\end{equation}which yields that
$\sqrt{n}\big(\widehat{V}_\ell-V_\ell- \zeta_n \big)$ converges to 0 in $L^2$. Since $\sqrt{n}\zeta_n$ converges in distribution to $\mathcal{N}\Big(0,4 \Var\big(Y_1 h_\ell(G_\ell(X^1_\ell))\big)\Big)$ by the central limit theorem, we obtain that:
\begin{equation}
\lim_{n\rightarrow +\infty}\sqrt{n}\big(\widehat{V}_\ell-V_\ell\big)=\mathcal{N}\Big(0,4 \Var\big(Y_1 h_\ell(G_\ell(X^1_\ell))\big)\Big),
\end{equation}in distribution.

%The result of Corollary \ref{corol:vitesseconvergence} is stated for functions $h_\ell$ belonging to $\mathcal{B}(\alpha,2,\infty)$, but the generalization to other Besov space might be possible.

\subsubsection{Numerical illustration on an SIR model}

Let us consider an SIR model. The input parameters are the rates $\lambda$ and $\gamma$. The output parameter is the final size of the epidemic, i.e.\  at a time $T>0$  where $I^N_T=0$, $Y=R^N_T $. \\

Recall from Chapter \ref{StatMC} that the fractions $(S^N_t/N,I^N_t/N,R^N_t/N)_{t\in [0,T]}$ can be approximated by the unique solution $(s(t),i(t),r(t))_{t\in [0,T]}$ of a system of ODE (see Example \ref{ex:SEIRS} of Chapter \ref{chap_MarkovMod} in Part I of this volume). These limiting equations provide a natural deterministic approximating meta-model (recall \cite{marrelioossdaveigaribatetIV}) for which sensitivity indices can be computed. \\

For the numerical experiment, we consider a close population of 1200 individuals, starting with $S_0=1190$, $I_0=10$ and $R_0=0$. The parameters distributions are uniformly distributed with $\lambda/N \in [1/15000, 3/15000]$ and $\gamma \in [1/15,3/15]$. Here the randomness associated with the Poisson point measures is treated as the nuisance random factor in \eqref{modelsto}. \\
We compute the Jansen estimators of  $S_{\lambda}$ and $S_{\gamma}$ for the deterministic meta-model constituted by the Kermack--McKendrick ODEs of Chapter 2 in Part I of this volume, with $n=30,000$ simulations. For the estimators of $S_{\lambda}$ and $S_{\gamma}$ in the SDE, we compute the Jansen estimators
with $n=10,000$ (i.e.\  $n(p+1)=30,000$ calls to the function $f$), and the estimators based on Nadaraya--Watson and on wavelet regressions with $n=30,000$ simulations. \\

\begin{figure}[!ht]
\centering
\begin{tabular}{cc}
\includegraphics[scale=0.25]{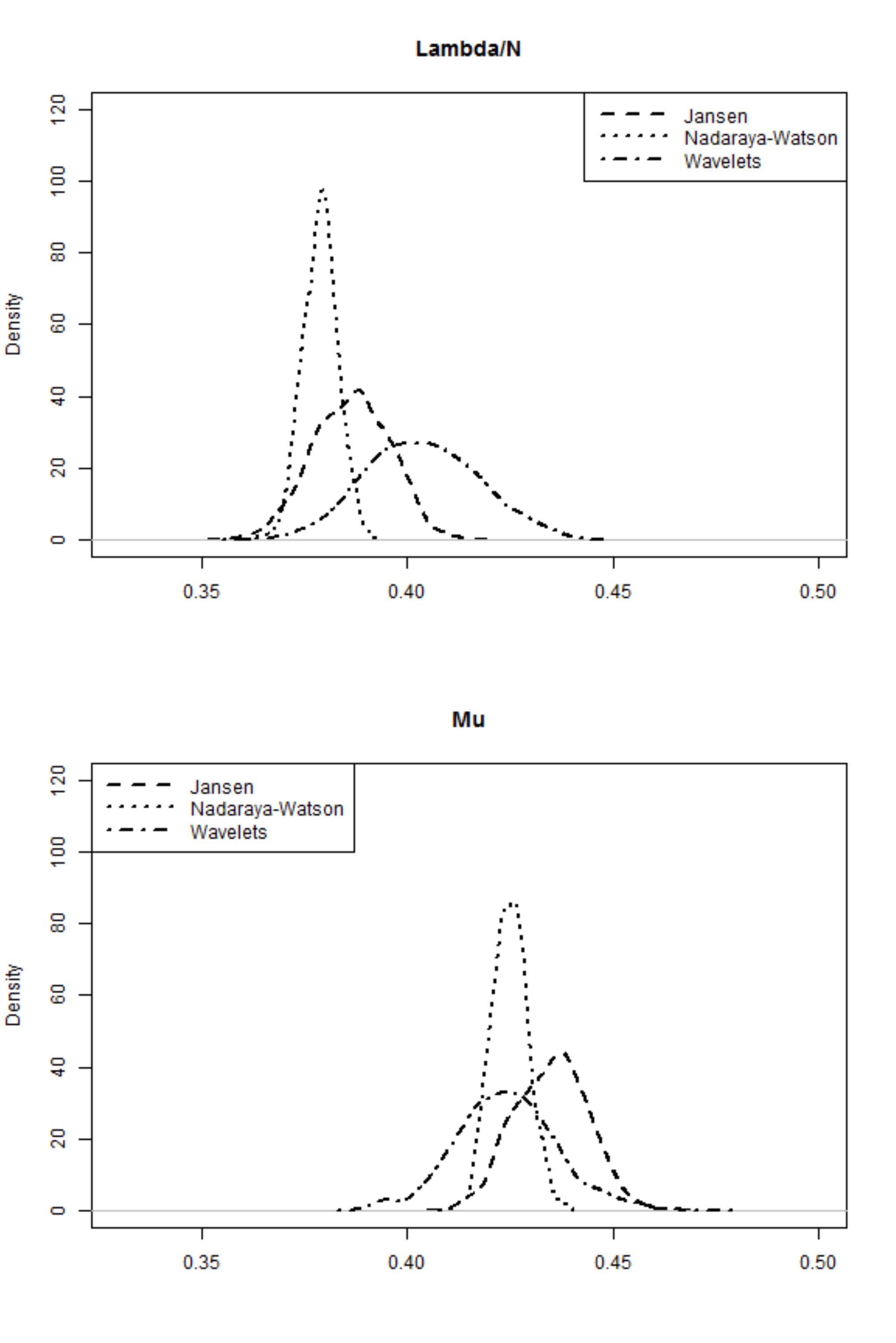} &
 \includegraphics[scale=0.25]{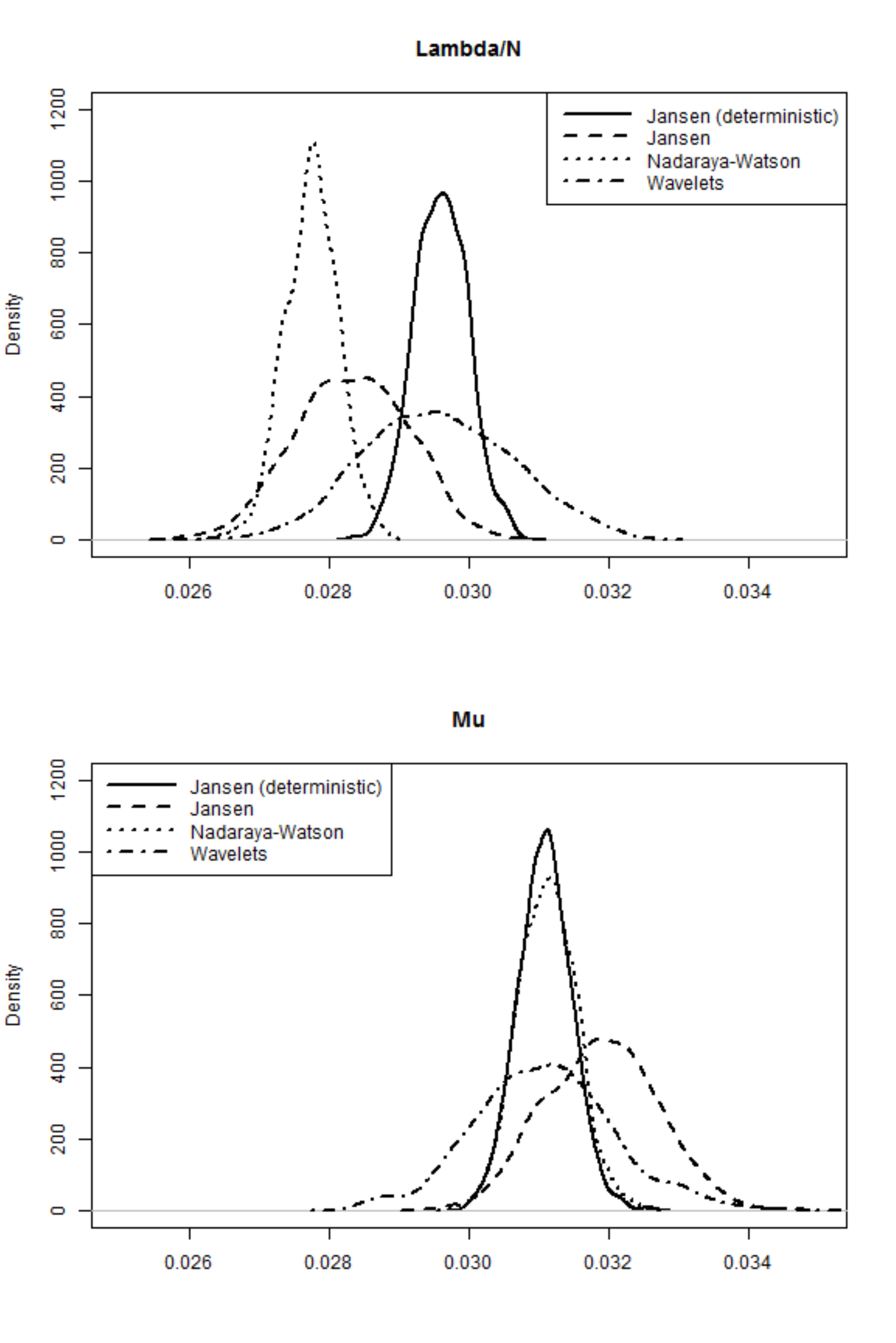}\\
(a) & (b)
\end{tabular}
\caption{{\small \textit{Estimations of the first-order Sobol indices, using Jansen estimators on the meta-model with $n=10,000$ and the non-parametric estimations based on Nadaraya--Watson and wavelet regressions. (a): the distributions of the estimators of $S_{\lambda}$ and $S_{\gamma}$ is approximated by Monte-carlo simulations. (b): the distributions of $\E(Y\ |\ \lambda)$ and $\E(Y\ |\ \gamma)$ are approximated by Monte Carlo simulations.}}}
\label{fig:comp}
\end{figure}

Let us comment on the results. First, the comparison of the different estimation methods is presented in Fig. \ref{fig:comp}. Since the variances in the meta-model and in the stochastic model differ, we start with comparing the distributions of $\E(Y\ |\ \lambda)$ and $\E(Y\ |\ \gamma)$ that are centered around the same value, independently of whether the meta-model or the stochastic model is used. These distributions are obtained from 1,000 Monte-carlo simulations. In Fig. \ref{fig:comp}(b), taking the meta-model as a benchmark, we see that the wavelet estimator performs well for both $\lambda$ and $\gamma$ while Nadaraya--Watson regression estimator performs well only for $\gamma$ and exhibit biases for $\lambda$. Jansen estimator on the stochastic model exhibit biases for both $\lambda$ and $\gamma$. \\

In a second time, we focus on the estimation of the Sobol indices for the stochastic model. The smoothed distributions of the estimators of $S_{\lambda}$ and $S_{\gamma}$, for 1,000 Monte Carlo replications, are presented in Fig. \ref{fig:comp} (a); the means and standard deviations of these distributions are given in Table \ref{table:resultSIR}.
Although there is no theoretical values for $S_\lambda$ and $S_\gamma$, we can see (Table \ref{table:resultSIR}) that the estimators of the Sobol indices with non-parametric regressions all give similar estimates in expectation for $\gamma$. For $\lambda$, the estimators are relatively different, with the Nadaraya--Watson showing the lower estimate. This is linked with the bias seen on Fig. \ref{fig:comp} (b) and discussed below. In term of variance, the Nadaraya--Watson estimator gives the tightest distribution, while the wavelet estimator gives the highest variance. \\

\begin{table}[!ht]
\centering
\begin{tabular}{|lccc|}
\hline
 & Jansen & Nadaraya--Watson & Wavelet \\
 \hline
 $\widehat{S}_\lambda$ & 0.39 & 0.38 & 0.40 \\
s.d. & (9.2e-3) & (4.3e-3) & (1.4e-2)  \\
 \hline
 $\widehat{S}_\gamma$ & 0.44 & 0.42 & 0.42  \\
s.d. & (9.0e-3) & (4.4e-3) & (1.2e-2) \\
\hline
\end{tabular}
\caption{{\small \textit{Estimators of the Sobol indices for $\lambda$ and $\gamma$ and their standard deviations using $n = $10,000 Monte Carlo replications of the stochastic SIR model.}}}
\label{table:resultSIR}
\end{table}

\begin{figure}[!ht]
\centering
\includegraphics[height=8cm]{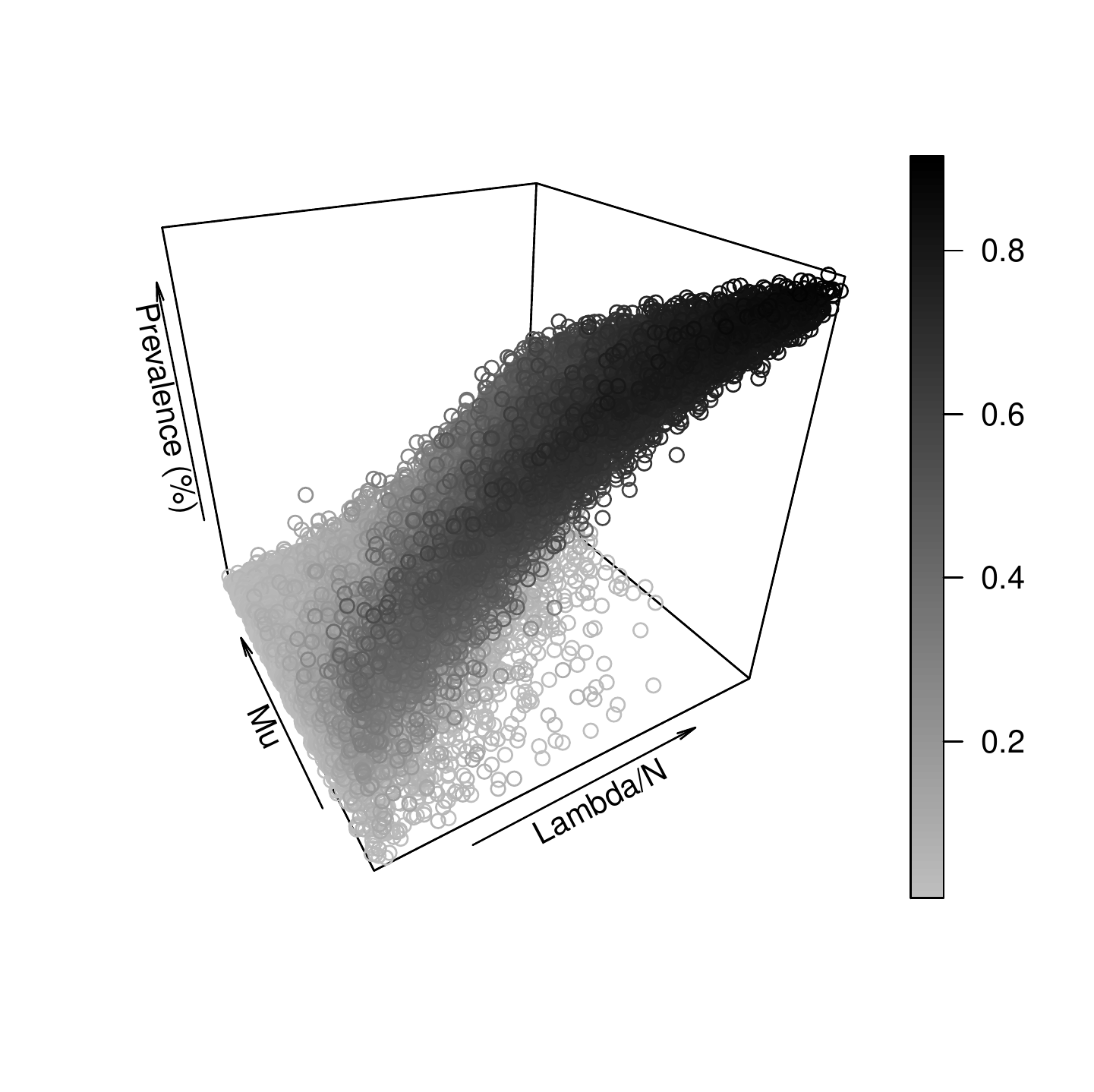}
\caption{{\small \textit{Prevalence ($Y$) simulated from the $n(p+1)=30,000$ simulations of $\lambda$ and $\gamma$, for the SIR model.}}
 \label{fig:exampleSIR30000}}
\end{figure}

The advantage of using the estimators with wavelets lies in their robustness to the inclusion of high frequencies and in the fact that
they can overcome some smoothing biases that the Nadaraya--Watson regressions exhibit (Fig. \ref{fig:comp} (b)). This can be understood when looking at Fig. \ref{fig:exampleSIR30000}: the simulations can give very noisy $Y$'s. For example, extinctions of the epidemics can be seen in very short time in simulations, due to the initial randomness of the trajectories. This produces distributions for $Y$'s that are not unimodal or with peaks at 0, which makes the estimation of $\E(Y\ |\ \lambda)$ or $\E(Y \ |\ \gamma)$ more difficult.
The variance of the estimator with wavelets is however the widest and in practice, finding the thresholding constants for the wavelet coefficients can be somewhat tricky when the number of input parameters is large.

\setcounter{chapter}{1}
\setcounter{section}{0}
\renewcommand{\thechapter}{\Alph{chapter}}
\chapter*{Appendix}%\label{Appendix-Part4}
\addcontentsline{toc}{chapter}{Appendix}

\section{Some classical results in statistical inference}\label{Classicstat}
In this section, we have gathered results on inference useful for this part of these notes.
\subsection{Heuristics on Maximum Likelihood Methods}\label{HeuMLE}
 As a guide for statistical inference for epidemic dynamics,  we first describe the heuristics for getting properties of Maximum likelihood Estimators, each family of statistical models having to be studied specifically (see \cite{cap05IV} for more details).

 Definitions and  properties are given for  general discrete time stochastic processes. Consider a sequence  $(X_1,\dots,X_n)$ of random variables with values in  $E$, and
let $P_{\theta}^n $ denote the distribution of $(X_1,\dots,X_n)$ on $(E^n,\mathcal{ E}^n)$.
Assume that the parameter set  $\Theta$ is included in $\R ^q$ and that  $\theta_0$  the true value of the parameter belongs to $Int(\Theta)$.

The properties on the MLE  relies on three basic results that hold  as $n\rightarrow \infty$ under $\P_{\theta_0}^n$:
\begin{enumerate}[aaa]
\item[(i)] a law of large numbers for the log-likelihood $\ell_n(\theta) $,
\item[(ii)] a central limit theorem for the score function  $\nabla_{\theta} \ell_n(\theta_0)$
\item[(iii)] a law of large numbers for the observed information $\nabla_{\theta} ^2 \ell_n(\theta_0)$ under $\P_{\theta_0}^n$.
\end{enumerate}

For a  regular statistical model with a standard rate of convergence $\sqrt{n}$,
\begin{enumerate}[aaa]
\item[\textbf{ (i)}] For all $ \theta \in \Theta$, $n^{-1 }\ell_n(\theta) \rightarrow  J(\theta_0,\theta) $ in  $P_{\theta_0}^n$-probability.
uniformly w.r.t.\ $\theta$,
$ \theta \rightarrow J(\theta_0,\theta) $  is a continuous  function with a global unique maximum at $\theta_0$.
%\vspace{0.3cm}
\item[\textbf{ (ii)}] $n^{-1/2} \nabla_{\theta}\ell_n(\theta_0) \rightarrow  \mathcal{ N}(0,\mathcal{ I}(\theta_0))$ in distribution under $P_{\theta_0}^n$,
%where $\mathcal{ I}(\theta))$ is the Fisher information matrix at $\theta$.\\
%\vspace{0.3cm}
\item[\textbf{ (iii)}]  $-\frac{1}{n} \nabla^2_{\theta} \ell_n(\theta_0) \rightarrow \mathcal{ I}(\theta_0)$  in $P_{\theta_0}^n$-probability.
\end{enumerate}
%$lim_{n \rightarrow \infty}sup_{|\theta-\theta_0|\leq \delta}
 %-\frac{1}{n}\nabla^2_{\theta} \ell_n(\theta)- \mathcal{ I}(\theta_\parallel  \rightarrow 0 \mbox{ as }
%\delta \rightarrow 0 $ $P_{\theta_0}-$ a.s.\\

Condition (i) ensures consistency of the MLE  $\hat{\theta} _n$.

Assuming that   $\mathcal{ I}(\theta_0)$ is  non-singular, a
Taylor expansion of the score function $\nabla_{\theta} \ell_n$ at point $\theta_0$ leads, using that
$\nabla_{\theta} \ell_n (\hat{\theta}_n)=0$,
\begin{equation}%\label{nablal}
	0= \nabla_{\theta} \ell_n (\hat{\theta}_n)= \nabla_{\theta} \ell_n (\theta_0)+ \Big(\int_0^1 \nabla^2_{\theta}\ell_n (\theta_0+t(\hat{\theta}_n-\theta_0))dt \Big)\;(\hat{\theta}_n-\theta_0).	
	\end{equation}
	From this expansion, we get, using  that $\mathcal{ I}(\theta_0)$ is non-singular,
	\begin{equation}%\label{normle}
	\sqrt{n}(\hat{\theta}_n- \theta_0)= \Big(-\frac{1}{n}\int_0^1 \nabla^2_{\theta}\ell_n(\theta_0+t(\hat{\theta}_n-\theta_0))dt \Big )^{-1} \Big(\frac{1}{\sqrt{n}}\nabla_{\theta} \ell_n(\theta_0)\Big).
	\end{equation}
Since  $\hat{\theta}_n \rightarrow \theta_0$ in $P_{\theta_0}^n$-probability  we get, using  (iii), that
\begin{enumerate}
\item[-] the	first factor of the r.h.s.\ of the equation above converges to $\mathcal{ I}(\theta_0)^{-1}$ $P_{\theta_0}$ a.s.
\item[-] the second factor converges in distribution under $P_{\theta_0}$ to $\mathcal{ N}(0,\mathcal{ I}(\theta_0))$.
\end{enumerate}
Finally, Slutsky's Lemma  yields that
	$\sqrt{n}(\hat{\theta}_n- \theta_0)\rightarrow_\mathcal{ L}\mathcal{ N}(0, \mathcal{ I}(\theta_0)^{-1}) $ under $P_{\theta_0}$.\\

\subsection{Miscellaneous results}\label{Miscstat}
We first state a theorem concerning the properties of the $\phi(\theta)$.
\begin{theorem}\label{phimle}
Let $(X_n)$ be a sequence of random variables with values in $\R^p$ and $a_n >0$ such that $a_n \rightarrow \infty$ as $n\rightarrow \infty$. Assume that $a_n(X_n -m )$  converges in distribution to a random variable $Z$.
Let $\phi: \R^p \rightarrow \R^q$ a continuously differentiable application. Then $a_n(\phi(X_n)-\phi (m))$ converges in distribution to the random variable $\nabla _x\phi (m) Z$, where $\nabla_x \phi$ is the Jacobian matrix of $\phi $: $\nabla_x \phi= (\frac{\partial \phi_k}{\partial x_l})_{1\leq k\leq q,1\leq l \leq p}$.
\end{theorem}
We refer to \cite{vaa00IV} for the proof.\\

For sake of clarity, we also give a recap on Exponential families of distributions (see e.g.\ \cite{bic15IV} or \cite{vaa00IV}). Indeed,
among parametric families of distributions, exponential families of distributions, widely used in statistics, provide here  a nice framework
to study the likelihood.

 Let $X$ be a random variable in $\R^k$ (or $\Z^k$) with distribution $P_{\theta}$ and density $p(\theta,x)$, with $\theta \in \Theta$, subset of $\R^q$.
\begin{definition}\label{expfam2}
The family $\{P_\theta, \theta \in \Theta\}$ is an exponential family if there exist  $q$ functions $(\eta_1,\dots,\eta_q)$ and $\phi$ defined on $\Theta$, $q$ real functions $T_1,\dots, T_q$  and a function $h(\cdot) $ defined on $\R^k$ such that\\
\begin{equation}\label{EF2}
p(\theta,x)= h(x) \exp\{\sum _{j=1}^q \eta_{j}(\theta)T_j(x)-\phi(\theta)\}\; ; x \in \R^k
\end{equation}
\end{definition}
%\pagebreak

Then $T(X)= (T_1(X),\dots,T_q(X))$ is a sufficient statistic in the i.i.d.\ case.
The random variable $X$ satisfies
\begin{equation}\label{momentexp2}
	m(\theta):= \E_{\theta}(X)= \nabla_{\theta} \phi (\theta);\quad \sigma^2(\theta):= \mathrm{Var}_{\theta}(X)= \nabla_{\theta}^2 \phi (\theta).
\end{equation}

\section{Inference for Markov chains}\label{ApStatMC}
 In order to present a good overview of the statistical problems, we  detail
the statistical inference  for Markov chains.
%Indeed, discrete time Markov chains models are interesting here because many questions that can arise for more complex models can be illustrated in this set-up. Moreover, continuous-time stochastic models are often observed in practice at discrete times, which might sum up to
%a Markov chain model. Therefore, this point of view allows to illustrate some classical statistical methods for stochastic models used in epidemics.
We have rather focus here on parametric inference since epidemic models always include in their dynamics  parameters that need to be estimated in order to derive predictions.

\subsection{Recap on Markov chains}\label{recapMC}

We first begin setting the notations used throughout this chapter and introducing the basic definitions.

Let  $(X_n, n \geq 0)$  a Markov chain on a probability space $(\Omega,{\mathbb F}, \P)$ with state space $(E,\mathcal{ E})$, transition kernel $Q$  and initial distribution
$\mu$ on $(E,\mathcal{ E})$. \\

The space of observations: $(E^{\N},\mathcal{ E}^{\otimes \N})$. Based on a  classical theorem of probability, there
	exists a unique probability measure on $(E^{\N},\mathcal{ E}^{\otimes \N})$, denoted
	$P_{\mu,Q}$ such that the coordinate process $(X_n,n\geq 0)$ is a Markov chain (with respect to its natural filtration) with initial distribution $\mu$ and transition  kernel $Q$.
Then, based on a classical  theorem in probability, there
	exists a unique probability measure on $(E^{\N},\mathcal{ E}^{\otimes \N})$, denoted
	$P_{\mu,Q}$ such that the coordinate process $(X_n,n\geq 0)$ is a Markov chain (with respect to its natural filtration) with initial distribution $\mu$ and transition kernel $Q$.

The probability $P_{\mu,Q}$ has the property:
\begin{enumerate}	
\item[-] if  $A_0, A_1, \dots, A_n$ are measurable sets in $E$, then
\[P_{\mu,Q}(X_i \in A_i; i=0,\dots,n )= \int_{A_{0}}\mu(dx_0)
\int_{A_1}Q(x_0,dx_1)\dots  \int_{A_n}Q(x_{n-1},dx_n).\]
\end{enumerate}

Let $\Theta$ denote some subset of ``probability measures $\times$ transition kernels on $(E,\mathcal{ E})$''.
 %\begin{definition}\label{CmsPmuQ}	
The canonical statistical model is
$(E^{\N}, \mathcal{ E}^{\N}, (P_{\mu,Q}, (\mu,Q) \in  {\Theta}))$.
Let us denote by $P_{\mu,Q}^n$ the distribution of $(X_0,\dots,X_n)$ on $ E^{n+1}$.
The successive observations of $(X_i)$ allow to estimate $\mu,Q$.
%One expects that longer observations lead to better estimators of $(\mu,Q)$.
%This is quantified by the Fisher information.\\
%Note that for $\omega =(x_0,\dots,x_n)$, $X_i(\omega)= x_i$ and $(X_0(\omega),\dots,X_n(\omega))= \omega)$.

Let $\alpha$ be a $\sigma$-finite positive measure on $(E, \mathcal{ E})$ dominating all the distributions $\{\mu(dy), (Q(x,dy), x \in E)\}$ and
assume that
$\mu(dy)= \mu (y)\alpha(dy), Q(x,dy)=Q(x,y) \alpha(dy)$.
Then, the likelihood of the observations $(x_0,\dots,x_n)$
is the probability  density function of $(X_0,\dots,X_n)$,
$P^n_{\mu,Q}$, with respect to the measure  $\alpha_n= \otimes_{i=0}^n \alpha^i(dy)$ on $E^{n+1}$, with $ \alpha^i(\cdot)$ copies of $\alpha(\cdot)$.
$$ \frac{d P^n_{\mu,Q}}{d\alpha_n}(x_i,i=0,\dots,n)=  \mu(x_0) Q(x_0,x_1)\dots
Q(x_{n-1},x_n).$$
Then, the likelihood function  at time $n$  is \\
\begin{equation}\label{deflik}
L_n(\mu,Q)=    \frac{d P^n_{\mu,Q}}{d\alpha_n}(X_0,\dots,X_n)=\mu(X_0) Q(X_0,X_1) \dots Q(X_{n-1}, X_n).
\end{equation}
%The  sequence $(L_n(\theta),n \in N)$ is
%also called Likelihood process.
The associated  Loglikelihood is
\begin{equation}\label {loglik} \ell_n(\mu,Q)= \log L_n(\mu,Q).
\end{equation}

\subsubsection{Maximum likelihood method for Markov chains}\label{MLEMC}
Let us consider the  case of positive recurrent Markov chains.
We follow the sketch detailed above to study the properties of MLE estimators.

Assume that the parameter set $\Theta$ is
is a  compact subset of $\R^q$.
%For $g(\theta,x)$ a $C^2$ function,  let   $\nabla_{\theta} g(\theta,x)= (\frac{\partial}{\partial \theta_i}g(\theta,x))_{i=1,\dots q}$  and $\nabla^2_{\theta} g(\theta,x)=(\frac{\partial^2}{\partial \theta_i\partial\theta_j} g (\theta,x))_{i=1,\dots q}$
\begin{definition}\label{Qf}
A family $(Q_{\theta}(x,dy),\theta \in \Theta)$ of transition probability kernels
on $(E,\mathcal{ E}) \rightarrow [0,1]$ is dominated by the transition kernel $Q(x,dy)$ if\\
$\forall x \in E, Q_{\theta} (x,dy)= f_{\theta}(x,y) Q(x, dy)$, with  $f_{\theta}: (E\times E,\mathcal{ E}\times \mathcal{ E}) \rightarrow \R ^+$
measurable.
\end{definition}

Assume that the initial distribution $\mu$ is known and let $\P_{\theta}$  (resp.\ ${\mathbb Q}$ denote
the distribution of the Markov chain $(X_n)$ with initial distribution $\mu$ and transition kernel  $Q_{\theta}$ (resp.\ $Q(x,dy)$.
Then the likelihood function and loglikelihood write
\begin{equation}\label{Ln}
 L_n(\theta)= \frac{d \P_{\theta}}{d {\mathbb Q}}(X_0,\dots,X_n)= \Pi_{i=1}^n f_{\theta}(X_{i-1},X_i) , \quad  \ell_n(\theta)= \sum_{i=1}^n \log f_{\theta}(X_{i-1},X_i).
 \end{equation}
The maximum likelihood estimator  is defined as: $ \hat{\theta}_n= \mbox{argsup}_ {\theta \in \Theta} \; L_n(\theta).$
%\begin{equation}\label
%L_n(\hat{\theta}_n)= sup\{L_n(\theta), \theta \in \Theta \}.
%\end{equation}

\subsubsection{Consistency}
Denote by $\theta_0$ the true value of the parameter.
In order to study the properties of t$ \hat{\theta}_n$  as $n\rightarrow \infty$,
we introduce some assumptions.
\begin{enumerate}[(H5):]
\item[\textbf{ (H0)}:] The family $(Q_{\theta}(x,dy),\theta \in \Theta)$ is dominated by the transition kernel $Q(x,dy)$.
\item[\textbf{ (H1)}:] The Markov chain $(X_n)$ with transition kernel $Q_{\theta_0}$ is irreducible, positive recurrent and aperiodic, with stationary measure $\lambda_{\theta_0}(dx)$ on $E$.
\item[\textbf{ (H2)}:] $\lambda_{\theta_0}(\{x, Q_{\theta}(x,\cdot) \neq Q_{\theta_0}(x,\cdot)\} ) >0$.
\item[\textbf{ (H3)}:] $\forall \theta , \log f_{\theta}(x,y)$ is integrable with respect to $\lambda_{\theta_0}(dx)Q_{\theta_0}(x,dy):= \lambda_{\theta_0}\otimes Q_{\theta_0}$.
\item[\textbf{ (H4)}:] $ \forall (x,y)\in E^2, \; \theta \rightarrow f_{\theta}(x,y)$ is continuous w.r.t.\ $\theta$.
\item[\textbf{ (H5)}:] There exists a function $h(x,y)$ integrable w.r.t.\ $\lambda_{\theta_0}\otimes Q_{\theta_0}$ and such that
$$\forall \theta  \in \Theta  , \vert \log f_{\theta }(x,y)\vert \leq h(x,y).$$
\end{enumerate}
Assumption (H0) ensures the existence of the likelihood, (H1) is
analogous for Markov chains to  repetitions in a $n$ sample of i.i.d.\ random variables,
(H2) corresponds to an identifiability assumption, which ensures that different parameter values lead to distinct distributions for the observations. Assumptions (H3)--(H5)
are regularity assumptions.

\begin{theorem}\label{consist}
Assume (H0)--(H5)  and that $\Theta $ is a compact subset of $\R^q$.
Then the MLE $\hat{\theta}_n$ is consistent: it converges in  $\P_{\theta_0}$-probability to $\theta_0$ as $n\rightarrow \infty$.
\end{theorem}

\begin{proof}
%The proof is divided in two steps.
%The first one  concerns a law of large numbers for $\ell_n(\theta)$. It relies on the property stated below.
%\begin{lemma}\label{ergodd}
Using that, under (H0),(H1), the sequence $(Y_n= (X_{n-1},X_n), n \geq 1)$ is a positive recurrent Markov chain on $(E\times E,\mathcal{ E}\times \mathcal{ E})$ with stationary distribution $ \lambda_{\theta_0}(dx) Q_{\theta_0}(x,dy)$, the ergodic theorem applies  to $(Y_n)$ and yields that, under (H3),
 \begin{equation}\label{limlog}
	\frac{1}{n}\sum_{i=1}^n \log f_{\theta}(X_{i-1},X_i) \rightarrow  J(\theta_0,\theta) :=
\int\int_{E\times E} \log f_{\theta}(x,y) \lambda_{\theta_0}(dx) Q_{\theta_0}(x,dy )\quad   P_{\theta_0} \mbox{-a.s.}
\end{equation}
%, which ensure that   "lim Argsup = Argsup lim" holds.
%\subsection{Study of Argsuplim}
%Let us study $J(\theta_0,\theta)$ defined in (\ref{limlog}).
Rewriting this equation yields that $J(\theta_0,\theta)$ defined in (\ref{limlog}),
$$J(\theta_0,\theta)= \int \int \log \frac{f_{\theta}(x,y)}{f_{\theta_0}(x,y)}\lambda_{\theta_0}(dx) Q_{\theta_0}(x,dy)+  A(\theta_0),$$
with $A(\theta_0)= \int \int \log f_{\theta_0}(x,y) \lambda_{\theta_0}(dx) Q_{\theta_0}(x,dy)$.
Under  (H0),
\[Q_{\theta}(x,dy)= f_{\theta}(x,dy) Q(x,dy),\]
so that
\begin{eqnarray*}
J(\theta_0,\theta) &=& \int \lambda_{\theta_0}(dx)\int\log\frac{Q_{\theta}(x,dy)}{Q_{\theta_0}(x,dy)}
Q_{\theta_0}(x,dy) +A(\theta_0)\\
&=& -\int K(Q_{\theta_0}(x,\cdot),Q_{\theta}(x,\cdot))\;\lambda_{\theta_0}(dx) +A(\theta_0),
\end{eqnarray*}
where $K(P,Q)$ denotes  the  Kullback--Leibler divergence between two probabilities.  Recall that it satisfies
\begin{enumerate}
\item[-] if $P << Q $, then $K(P,Q)= \E_P(\log\frac{dP}{dQ})= \int\log \frac{dP}{dQ}dP= E_Q(\phi(\frac{dP}{dQ}))$ with $\phi(x)=x \log(x)+1-x$.
\item[-] $K(P,Q)= +\infty $ otherwise.
\end{enumerate}
A well-known property is that
 $K(P,Q) \geq 0$ and $K(P,Q)=0$ if and only if  $P=Q$ a.s.
Assumption (H2) ensures that $\theta \rightarrow J(\theta_0,\theta)$ possesses a global unique maximum at $\theta=\theta_0$.

The MLE $\hat{\theta}_n$ satisfies that  $\hat{\theta}_n=\mathrm{Argsup}_{\theta}(\frac{1}{n}\ell_n(\theta))$.
The maximum of the right-hand side of (\ref{limlog}) is $\theta_0$.
Hence to get consistency, we have to prove that ``lim Argsup $\frac{1}{n}\ell_n(\theta)$'' is equal to  ``Argsup lim $\frac{1}{n}\ell_n(\theta)$'', which is $\theta_0$.
%This is ensured by (H4)--(H5).
%Indeed,
%of Theorem \ref{consist}.
Note that, for all $\theta \in \Theta$, $\ell_n(\hat{\theta _n}) \geq \ell_n(\theta)$ and
%$ \theta	 \rightarrow J(\theta_0,\theta)$ has a unique global maximum at $\theta_0$,
 $J(\theta_0,\theta_0 ) \geq
J(\theta_0, \hat{\theta _n})$. Combining these two inequalities  we get,
\begin{align*}
	0 \leq  J(\theta_0,\theta_0) -J(\theta_0,\hat{\theta _n})
	\leq &\,  J(\theta_0,\theta_0) -\frac{1}{n}\ell_n(\theta_0) + \frac{1}{n}\ell_n(\theta_0)-
	\frac{1}{n}\ell_n(\hat{\theta _n})\\
 &+\frac{1}{n}\ell_n(\hat{\theta _n})- J(\theta_0,\hat{\theta _n})\\
	\leq &\, 2 \sup_{\theta \in \Theta} |J(\theta_0, \theta)- \frac{1}{n}\ell_n(\theta)|.
	\end{align*}
Therefore, by taking $\Theta$ a compact subset of $\R^q$, we get that $J(\theta_0,\hat{\theta _n}) \rightarrow J(\theta_0,\theta_0)$
$\P_{\theta_0}$- a.s. as $n\rightarrow \infty$. Assumptions (H4),(H5)  ensure that $ J(\theta_0,\cdot)$ is continuous with a unique global maximum at $\theta_0$ so that the MLE converges to $\theta_0$ in  $\P_{\theta_0}$-probability.
\end{proof}

\subsubsection{Limit distribution}\label{limitdist}
This section is based on general results presented in \cite{hal80IV}.
 For $V$ a vector or a matrix, let  $V^*$ denote its transposition.
Define  the $q\times q$ matrix
\begin{equation}\label{FisherCh3}
\mathcal{ I}(\theta_0)= \int\int \frac{\nabla_{\theta}f_{\theta_0}(x,y)\;\nabla^*_{\theta}f_{\theta_0}(x,y) }{f_{\theta_0}(x,y)^{2}} \lambda_{\theta_0}(dx)Q_{\theta_0} (x,dy).
\end{equation}
 Let us introduce the additional assumptions.
\begin{enumerate}[(H8)]
\item[\textbf{ (H6)}] $ \theta \rightarrow \ell_n(\theta)$ is $C^2(\Theta)$  $\P_{\theta_0}$ -a.s.
\item[\textbf{ (H7)}] $\mathcal{ I}(\theta_0) $  defined in (\ref{FisherCh3}) is non-singular.
\item[\textbf{ (H8)}]   $\int \phi_{\theta_0}(r,x,y)\lambda_{\theta_0}(dx)Q_{\theta_0}(x,dy) \rightarrow 0 $ as $r\rightarrow 0$ where
\[\phi_{\theta_0}(r,x,y)=\sup \{\parallel \nabla^2_{\theta}\log f_{\theta} (x,y)- \nabla^2_{\theta}\log f_{\theta_0}(x,y)\parallel\cdot \parallel \theta-\theta_0\parallel \leq r \}.\]
\end{enumerate}
%Then $\int \phi_{\theta_0}(r,x,y)\lambda_{\theta_0}(dx)Q_{\theta_0}(x,dy) \rightarrow 0 $ as $r\rightarrow 0$
%\vspace{0.3cm}
We can state the result on the asymptotic normality of the MLE
\begin{theorem}\label{AsymptdistrMC}
Assume  (H0)--(H8). Then the MLE $\hat{\theta}_n$ is asymptotically Gaussian:  under  $P_{\theta_0}$,
 $$\sqrt{n}(\hat{\theta}_n-\theta_0) \rightarrow_{\mathcal{ L}}  \mathcal{ N}_{q}(0, \mathcal{ I}(\theta_0)^{-1}).$$
\end{theorem}
\begin{proof}
Under (H6), the score function is well defined and reads as
\begin{equation}\label{ScoreMC}
\nabla_{\theta}\ell_n (\theta)=\sum_{i=1}^n \nabla_{\theta} \log f_{\theta}(X_{i-1},X_i)= \sum_{i=1}^n  v_i(\theta).
\end{equation}
The score function satisfies
%\vspace{0.3cm}
\begin{proposition} \label{Martingalescore}
	Under assumptions (H0)--(H5), $\nabla_{\theta}\ell_n (\theta_0)$ is a   $q$-dimensional  $\P_{\theta_0}$-martingale
 w.r.t.\ $(\mathcal{ F}_n)_{n\geq 0}$, which  is centered  and square integrable.
\end{proposition}
%\vspace{0.3cm}
Proof:  By (\ref{ScoreMC}), we have,   $\nabla_{\theta}\ell_n (\theta_0) = \nabla_{\theta}\ell_{n-1} (\theta_0) +  v_n(\theta_0)$.
 We get, using that, under (H5), $\int \nabla _{\theta}f =\nabla _{\theta}(\int f)$ holds true,
 %\pagebreak
\begin{eqnarray*}
\E_{\theta_0} (v_i(\theta_0)|\mathcal{ F}_{i-1})&=& \E_{{\mathbb Q}} (\;\nabla_{\theta} \log f_{\theta_0}(X_{i-1},X_i) f_{\theta_0}(X_{i-1},X_i)|\mathcal{ F}_{i-1})\\
&=& \E_{{\mathbb Q}} (\nabla_{\theta}  f_{\theta_0}(X_{i-1},X_i)|\mathcal{ F}_{i-1})\\
&=& \nabla_{\theta} \; \E_{{\mathbb Q}}(f_{\theta_0}(X_{i-1},X_i)|\mathcal{ F}_{i-1})= \nabla_{\theta}1=0.
\end{eqnarray*}
Noting that $E_{\theta_0}(\nabla_{\theta}\ell_1(\theta_0))=\nabla_{\theta} (E_{\theta_0} 1)=0$,  $\nabla_{\theta}\ell_n (\theta_0)$ is a centered martingale.
%(\frac{1}{f_{\theta_0}(X_{i-1},)X_i)} \;\nabla_{\theta} f_{\theta_0}(X_{i-1},X_i) \frac{dP_{\theta_0}}{dQ}(X_{i-1},X_i) |\mathcal{ F}_{i-1}).$$.
%Using that $\frac{dP_{\theta_0}}{dQ}(X_{i-1},X_i)= f_{\theta_0}(X_{i-1},X_i)$  yields:\\
% \begin{eqnarray*}
% \E_{\theta_0}(v_n(\theta_0)|\mathcal{ F}_{i-1})&=&  \E_{Q}(\nabla_{\theta} f_{\theta_0}(X_{i-1},X_i)|\mathcal{ F}_{i-1})= \nabla_{\theta} (\E_{Q} (f_{\theta_0}(X_{i-1},X_i)|\mathcal{ F}_{i-1}))\\ &=&\nabla_{\theta} (\E_{\theta_0} (f_{\theta_0}(X_{i-1},X_i)\frac{dQ}{dP_{\theta_0}}(X_{i-1},X_i) |\mathcal{ F}_{i-1}))=0.
% \end{eqnarray*}
% Note that , under (H5),   $\int \nabla _{\theta}f =\nabla _{\theta}\int f$ holds true.

Consider now the increasing process associated with this martingale.  We have\\
$ \langle\nabla_{\theta}\ell_n (\theta_0)\rangle =\sum_{i=1}^n E_{\theta_0}(v_i(\theta_0) \;v_i^*(\theta_0)|\mathcal{ F}_{i-1}).$ \\
An application of the ergodic theorem yields
$\frac{1}{n}\sum v_i(\theta_0) \; v_i^*(\theta_0) \rightarrow  \mathcal{ I}(\theta_0)$ $\P_{\theta_0}$ a.s.  \\
Therefore  for $j=1,\dots q$, $E_{\theta_0}\langle\nabla_{\theta}\ell_n (\theta_0)\rangle_{jj} \rightarrow \infty$ as $n\rightarrow \infty$ .
 Applying a central limit theorem, we get that
 $$ \frac{1}{\sqrt{n}}\nabla_{\theta}\ell_n (\theta_0) \rightarrow \mathcal{ N}_q(0, \mathcal{ I}(\theta_0)).$$
The matrix $\mathcal{ I}(\theta_0)$ is  the Fisher information matrix .

 A Taylor expansion of the score function $\nabla_{\theta} \ell_n$ at point $\theta_0$ leads, using that
$\nabla_{\theta} \ell_n (\hat{\theta}_n)=0$, to
\begin{equation}\label{nablalCh3}
	0= \frac{1}{\sqrt{n}}\nabla_{\theta} \ell_n (\hat{\theta}_n)= \frac{1}{ \sqrt{n}}  \nabla_{\theta} \ell_n (\theta_0)+ \frac{1}{n} \Big(\int_0^1 \nabla^2_{\theta}\ell_n(\theta_0+t(\hat{\theta}_n-\theta_0))dt \Big)\;\frac{\hat{\theta}_n-\theta_0 }{\sqrt{n}}.\end{equation}
Now, \eqref{limlog} yields, using \eqref{deflik},
\[\frac{1}{n} \nabla^2_{\theta}\ell_n(\theta_0) \rightarrow \int \lambda_{\theta_0}(dx) \int \nabla^2_{\theta}(\log f_{\theta_0}(x,y)) Q_{\theta_0} (x,dy)= - \mathcal{ I}(\theta_0),\]
Indeed, the last equality is obtained  using   Assumptions (H3)--(H6) and
\[\int \frac{\nabla^2_{\theta} f_{\theta_0}(x,y)}{ f_{\theta_0}(x,y)}Q_{\theta_0} (x,dy)= \nabla_{\theta}^2 (\int  f_{\theta_0}(x,y) Q(x,dy) )=0.\]
% we get that $H(\theta_0
%\begin{eqnarray*}
%\int \nabla^2_{\theta}(\log f_{\theta_0}(x,y)) Q_{\theta_0} (x,dy)&=&\int \frac{\nabla^2_{\theta} f_{\theta_0}(x,y)}{ f_{\theta_0}(x,y)}Q_{\theta_0} (x,dy)
%-\int \frac{\nabla_{\theta} f_{\theta_0}(x,y) ^t \nabla_{\theta} f_{\theta_0}(x,y)}{f_{\theta_0}^2(x,y)} Q_{\theta_0} (x,dy)\\
%&=& \int \nabla^2_{\theta} f_{\theta_0}(x,y) Q(x,dy) -\int \frac{\nabla_{\theta} f_{\theta_0}(x,y) ^t \nabla_{\theta} f_{\theta_0}(x,y)}{f_{\theta_0}^2(x,y)}Q_{\theta_0} (x,dy)
%\end{eqnarray*} ta_0}(x,y)) Q_{\theta_0} (x,dy)

Therefore, from  expansion (\ref{nablalCh3}), we get,% using  that $\hat{\theta}_n$ is consistent and $\mathcal{ I}(\theta_0)$ is non-singular,
	\begin{equation}\label{normle}
	\sqrt{n}(\hat{\theta}_n- \theta_0)= \Big(-\frac{1}{n}\int_0^1 \nabla^2_{\theta}\ell_n(\theta_0+t(\hat{\theta}_n-\theta_0))dt \Big )^{-1} \Big(\frac{1}{\sqrt{n}}\nabla_{\theta} \ell_n(\theta_0)\Big).
	\end{equation}
Since  $\hat{\theta}_n \rightarrow \theta_0$ in $P_{\theta_0}^n$-probability  we get  that
 the	first factor of the r.h.s.\ of (\ref{normle}) converges to $\mathcal{ I}(\theta_0)^{-1}$ under  $\P_{\theta_0}$ a.s., and that
the second factor converges in distribution under $P_{\theta_0}$ to $\mathcal{ N}(0,\mathcal{ I}(\theta_0))$.
Finally, Slutsky's Lemma  yields that
	$\sqrt{n}(\hat{\theta}_n- \theta_0)$ converges to  $\mathcal{ N}(0, \mathcal{ I}(\theta_0)^{-1} \mathcal{ I}(\theta_0) \mathcal{ I}(\theta_0)^{-1})=\mathcal{ N}(0, \mathcal{ I}(\theta_0)^{-1}) $ in distribution.
	\end{proof}
	% under $P_{\theta_0}^n$.\\
%\pagebreak

%%%%%%%%%%%%%%%%%%%%%%%%%%%%%%%%%%%%%%%%%%%%%%%%%%%%%%%%%%%%%%%%%%%%%%%%%%%%%%%%%%%%%%%%%%%%%%%%%%%%%%%%%%%%%%%%%%%%%
\subsection{Other approaches than the likelihood}
It often occurs in practice that  the likelihood is difficult to  compute.  One way to overcome this problem relies on stochastic algorithms.
However, another way round is to build other processes than the likelihood to derive estimators. These methods include for the i.i.d.\  case the $M$-estimators (\cite{vaa00IV})  and, for stochastic processes, Estimating equations, approximate likelihoods, pseudolikelihoods.
(\cite{kes12IV}), Generalized Moment Methods  (\cite{hans95IV}), Contrast functions (\cite{dac93IV}).
%For inference for stochastic processes, a unified theory is contained by the notion of Contrast processes (see e.g.\ Dacun
%We can  collect  all the methods built on moments,

\subsubsection{Minimum contrast approaches}
What if, instead of the likelihood, another process  (contrast process) $ U_n(\theta)$ is used as for instance \ the C.L.S.\ method (in essence think of $U_n \simeq -\ell_n$))\\

Let us assume that   $U_n(\theta)=  U_n(\theta,X_0,\dots,X_n) $ satisfies
\begin{enumerate}[(H4b)]
\item[\textbf{ (H1b)}]  For all $ \theta \in \Theta$, $U_n(\theta)$ is $\mathcal{ F}_n$-measurable and $\theta \rightarrow  U_n(\theta)$ is under $\P_{\theta_0}$ a.s.  continuous and  twice continuously differentiable on a  subset $V(\theta_0)$.\\
\item[\textbf{ (H2b)}] For all $ \theta$, $n^{-1 }U _n(\theta) \rightarrow  K(\theta_0,\theta) $ in $P_{\theta_0}$-probability
	uniformly over compacts subsets of $\Theta$, where $ \theta \rightarrow K(\theta_0,\theta) $  is continuous with a unique global  minimum at $\theta_0$.\\
\item[\textbf{ (H3b)}] $n^{-1/2} \nabla_{\theta} U_n(\theta_0) \rightarrow
 {\mathcal N}_{q} (0, I_U(\theta_0))$ in distribution under $\P_{\theta_0}$.\\
\item[\textbf{ (H4b)}]  There exists a   symmetric positive matrix $J_U (\theta_0)$ such that
	\[\lim_{n \rightarrow \infty}\sup_{|\theta-\theta_0|\leq \delta}
	\parallel \frac{1}{n}\nabla^2_{\theta} U_n(\theta)- {J_U}(\theta_0)\parallel  \rightarrow 0\mbox{ as }\delta \rightarrow 0\quad P_{\theta_0}\mbox{-a.s.}\]
\end{enumerate}
	
Define  the MCE  estimator $\tilde{\theta}_n$ associated with $  U_n(\theta)$ as any solution of
\begin{equation}\label{MCE}
U_n(\tilde{\theta}_n)= \inf_{\theta \in \Theta} U_n(\theta).
\end{equation}

Then, using similar proofs than in Section	\ref{MLEMC} yields that
\begin{theorem}\label{LimDMCE}
 Assume that  (H1b)--(H4b) hold. Then, the MCE defined in \eqref{MCE}
\begin{enumerate}
\item[(1)] $\tilde{\theta}_n \rightarrow \theta_0$ in $P_{\theta_0}-$ probability.
\item[(2)] $\sqrt{n}(\tilde{\theta}_n- \theta_0) \rightarrow _{\mathcal{ L}} \mathcal{ N}_q (0,J_U(\theta_0)^{-1} I_U(\theta_0) J_U ^{-1}(\theta_0))$ under $P_{\theta_0}$.
\end{enumerate}
\end{theorem}
Note that contrary to the MLE where   $J_U(\theta_0)=I_U(\theta_0)$,  the asymptotic covariance matrix of  $\tilde{\theta}_n $ is no longer  $I_U(\theta_0) ^{-1}$.
Analytic properties of matrices yield that $J_U(\theta_0)^{-1} I_U(\theta_0) J_U ^{-1}(\theta_0))$ is always greater (as a linear form) than $I_U(\theta_0)^{-1}$.

%\textbf{ (ii)}, \textbf{ (iii)} obtained building a martingale $(M_n)$ from the observations.\\

%%%%%%%%%%%%%%%%%%%%%%%%%%%%%%%%%%%%%%%%%%%%%%%%%%%%%%%%%%%%%%%%%%%%%%%%%%%%%%%%%%%%%%%%%%%%%%%%%%%%%%%%%%%%%%%%%%%%%%%%%%%
\subsubsection{Conditional Least Squares}\label{CLS}
A classical approach associated to this method is 	the Conditional Least Squares  method.\\
Let 	$(X_n)$ be an Markov chain on $\R^p$  with transition kernel $Q_{\theta}(x, dy)$ on  $\R^p$ and initial distribution $\mu$.
Assume that it is positive recurrent with stationary distribution $\lambda_{\theta}(dx)$.
%Assume that $\Theta$ is a compact set of $\R^q$ and let  $ \theta_0$ be the true value.

Define the two functions
\begin{align*}
g(\theta,x) &= \int y Q_{\theta}(x, dy)\mbox{ and}\\
V(\theta,x) &= \int  \;^t(y-g(\theta,x)) \; (y- g(\theta,x)) Q_{\theta}(x,dy).
\end{align*}
Clearly,  $E_{\theta}(X_i |X_{i-1})= g( \theta, X_{i-1})$ and $\mathrm{Var}_{\theta}(X_i |X_{i-1})= V(\theta,X_{i-1})$. We assume\\
 The CLS  method is  associated with the process
\begin{equation}\label{CLSMC}
U_n(\theta)= \frac{1}{2} \sum_{i=1}^n \;  (X_i -E_{\theta}(X_i | X_{i-1}))^*\;(X_i -E_{\theta}(X_i|X_{i-1})).
\end{equation}
%Define the  function $g(\theta,x) = \int y Q_{\theta}(x, dy)$ and assume that, for all $x $, $\theta \rightarrow g (\theta,x)$ is in $C^2(\R^q)$.
%Follow the Euristics of Section ?, let us study (i)--(iii)\\
 Applying the ergodic theorem to $((X_{i-1},X_i), i\geq 1)$  yields  that,
 under $\P_{\theta_0}$
 $$\frac{1}{n}U_n(\theta) \rightarrow  K (\theta_0,\theta)= \frac{1}{2}\int\int \; (y- g(\theta, x)) ^* (y- g(\theta,x)) \lambda_{\theta_0}(dx)Q_{\theta_0}(x,dy)  \mbox{ a.s.}$$
Rewriting this limit yields that
\[K(\theta_0,\theta)=
  \frac{1}{2}\int\int \; (g(\theta ,x)-g(\theta_0,x))^* (g(\theta ,x)-g(\theta_0,x))  \lambda_{\theta_0}(dx)Q_{\theta_0}(x,dy) +  A(\theta_0)\]
  with
\[A(\theta_0)=\frac{1}{2}\int\int \; (y- g(\theta, x))^*  (y- g(\theta,x)) \lambda_{\theta_0}(dx)Q_{\theta_0}(x,dy).\]

To study the MCE $\tilde{\theta}_n= \mathrm{Argmin} \{U_n(\theta) , \theta \in \Theta\}$, we  assume
\begin{enumerate}[(A4)]
\item[\textbf{ (A1)}] For all $x \in\R^p$,  $ g(\theta,x) $ and  $V(\theta,x)$ are finite and  $C^2$ with respect to $\theta$.
\item[\textbf{ (A2)}] $\theta\rightarrow K(\theta_0,\theta)$ continuous and  $\lambda_{\theta_0}(\{ x, g(\theta ,x)\neq g(\theta_0,x)\})>0$.
\item[\textbf{ (A3)}] The matrix  $J_U(\theta)= \int \;(\nabla_{\theta}g(\theta, x)\; \nabla^*_{\theta} g(\theta,x)\lambda_{\theta}(dx) $ is non-singular at $\theta_0$.
\item[\textbf{(A4)}] The function $\phi(\delta,x) = \sup_{||\theta-\theta_0||\leq \delta} ||\nabla ^2_{\theta} g(\theta,x) -\nabla^2 _{\theta} g(\theta_0,x) ||$  satisfies
\[\int \phi(\delta,x )\lambda_{\theta_0}(dx)\rightarrow 0\mbox{ as }\delta \rightarrow 0.\]
\end{enumerate}
 	 	
Assumption (A1) ensures that $U_n$ is well defined, (A2) that $\theta \rightarrow K(\theta_0,\theta)$ has a global unique minimum at $\theta_0$. Assumption (A3),(A4) ensure that (H3b), (H4b) hold.

Let us study  $ \nabla_{\theta} U_n(\theta)$.
We have that\\
 \begin{equation}\label{nablaU}
	0= \nabla_{\theta} U_n (\tilde{\theta}_n)= \nabla_{\theta} U_n (\theta_0)+ \Big(\int_0^1 \nabla^2_{\theta} U_n(\theta_0+t(\hat{\theta}_n-\theta_0))dt \Big)\;(\hat{\theta}_n-\theta_0).\end{equation}
 The first term of the r.h.s.\ of (\ref{nablaU}) reads as
\[\nabla_{\theta} U_n(\theta_0) = -\sum_{i=1}^n \;  (\nabla_{\theta}g(\theta_0,X_{i-1}))^*\; (X_i -g(\theta_0,X_{i-1})).\]
 Hence,  under (A1), $ \nabla_{\theta} U_n(\theta_0)$  is a centered $L^2$-martingale under $P_{\theta_0}$ with
\[\langle\nabla_{\theta} U_n(\theta_0)\rangle =\sum_{i=1}^n  E_{\theta_0}\Big( \;(\nabla_{\theta} g(\theta_0,X_{i-1}))^*\; V(\theta_0,X_{i-1}) \nabla_{\theta}g(\theta_0,X_{i-1})\Big).\]

Applying the ergodic theorem yields	
\[\frac{1}{n}\langle\nabla_{\theta} U_n(\theta_0)\rangle_n \rightarrow \int \;(\nabla_{\theta}g_{\theta_0}(x)^*)\; V(\theta_0,x) \; \nabla_{\theta} g_{\theta_0}(x)\lambda_{\theta_0}(dx) := I_U (\theta_0)\mbox{ a.s.}\]
Therefore, we can apply the central limit theorem for martingales (see Theorem \ref{TCLMartmult}) and obtain,
	$$ \frac{1}{\sqrt{n}}\nabla_{\theta} U_n(\theta_0) \rightarrow_{\mathcal{ L}}  \mathcal{ N}_q(0, I_U(\theta_0)) \mbox{ under } P_{\theta_0}.$$
For the second term, we get
$\nabla^2_{\theta}U_n(\theta_0) = \sum_{i=1}^n \; \nabla_{\theta} g(\theta_0,X_{i-1}) ^* \nabla_{\theta}g(\theta_0,X_{i-1})$ which satisfies\\
\[\frac{1}{n}\nabla^2_{\theta}U_n(\theta_0) \rightarrow   J_U(\theta_0) :=\int \; \nabla _{\theta}g(\theta_0,x)^*  \; \nabla _{\theta} g(\theta_0,x) \; \lambda_{\theta_0}(dx) \quad P_{\theta_0}\mbox{ a.s.}\]

 Therefore under (A3), (A4), $J_U(\theta_0)$ is invertible.  Therefore, $\tilde{\theta}_n$ is consistent and
 $\sqrt{n}(\tilde{\theta}_n-\theta_0)\rightarrow \mathcal{ N}(0, \Sigma(\theta_0))$ with $ \Sigma(\theta_0)= J_U^{-1}(\theta_0)I_U(\theta_0) J_U^{-1}(\theta_0)$.

\subsection{Hidden Markov Models}\label{sec:HMM}
A Hidden Markov Model is, roughly speaking, a Markov chain observed with noise.
This raises new problems for the statistical inference of parameters ruling the Markov chain model $(X_n)$.\\
 Consider  a Markov chain $(X_n, n \geq 0)) $ with state space $E$.
The term "hidden " corresponds to the situation where the Markov chain cannot be directly observable,
Instead of $(X_n)$ , the observations consists in another stochastic process $(Y_n)$ whose distribution is ruled by $(X_n)$.
The simplest case  is for instance  the case of measurements errors $Y_n= X_n+ \epsilon_n$, with $(\epsilon_i)$ i.i.d. random variables.
All the statistical inference for $(X_n)$ has to be done in terms of $(Y_n)$ only, since $(X_n)$ cannot be observed.\\

For epidemic data, this situation occurs when the exact status of individuals cannot be observed or when there is a systematic error in the reporting rate
of Infected individuals. \\

The precise definition of a Hidden Markov Model (HMM) is:
\begin{definition}\label{def:HMM}
A Hidden Markov Model (HMM) is a bivariate discrete time process $ ((X_n,Y_n), n\geq 0)$  with state space $\mathcal{ X} \times \mathcal{ Y}$ such that\\
(i) $(X_n)$  is a Markov chain  with state space $\mathcal{ X}$.\\
(ii) For all  $i\leq n$, the conditional distribution of $Y_i$ given $(X_0,\dots,X_n)$ only depends on $X_i$.
\end{definition}

\noindent
A  classical example of Hidden Markov models is obtained as follows:\\
 Let  $(\epsilon_n)$ is a sequence of i.i.d.  random variables  on $E$ and  $F(.,.): \mathcal{ X} \times E \rightarrow \mathcal{ Y}$  a given measurable function.
  Then, if $Y_n=F(X_n,\epsilon_n)$,
the bivariate sequence $(X_n,Y_n)$ is a Hidden Markov Model.\\

\noindent
It follows from this definition that $(X_n,Y_n)$ is a Markov chain on $\mathcal{ X}\times \mathcal{ Y}$, while the sequence
$(Y_n)$ is no longer Markov: \\
$\mathcal{ L}(Y_n|Y_0,\dots,Y_{n-1})$ effectively depends on all the past observations.

This is why the inference for parameters ruling $(X_n)$ is difficult and rely on specific  tools  (see e.g. \cite{cap05IV}, \cite{vanh08IV}).

\section{Results for statistics of diffusions processes} \label{Recapdiff}
Inference for diffusion processes observed on a finite time-interval presents some specific properties.  For sake of comprehensiveness, a short recap of classical results for diffusion processes inference is then given.We first present the general framework required for time-dependent diffusions and then detail these results.
(see \cite{kes12IV} for a presentation of available results).
%------------------------------------------------------------------------------------------------------

On a probability space $(\Omega, \mathcal{ F},(\mathcal{ F}_t,t\geq 0),\P)$, consider the stochastic differential equation
\begin{equation}\label{defxi}
d\xi_t= b(t,\xi_t) dt +\sigma(t,\xi_t) dB_t, \xi_0=\eta.
\end{equation}

We assume that  $(B_t)$ is a $p$-dimensional Brownian motion, that $b$  and $\sigma$
 satisfy regularity assumptions which ensure the existence and uniqueness of solutions of \eqref{defxi} and that
%$\sigma(t,x)$ is known,
$\eta $  is $\mathcal{ F}_0$-measurable and that
%$\theta_0 \in \Theta $ , with $\Theta$ a compact subset of $\R^k$.\\

We detail results on the inference  on  parameters in the drift and diffusion coefficient depending on various kinds of observations  of $(\xi_t,t \in [0,T])$.
For this, let us recall some basic definitions concerning  these processes.
The state space of $(\xi_t, t\leq T)$ is
%a continuous time process observed on a finite time interval $[0,T]$ is
$ C_T= \{x= (x(t)): [0,T] \rightarrow \R^p \mbox{ continuous}, \mathcal{ C}_T\} $, where $ \mathcal{ C}_T$  denote the Borel filtration associated with the uniform topology.
Denote by $X_t: C_T \rightarrow \R^p ,\;\; X_t(x)=x(t)$.  the coordinate functions defined for $0\leq t\leq T$.
The distribution of  $ \xi^T: (\xi_t,t\in [0,T])$ on $(C_T,\mathcal{ C}_T)$ is denoted by $P_{b,\sigma}^T$ .
%: it is the probability distribution image of $\P$ by the r.v. $\xi^T $:\\
% \textit{ i.e.}
%$:= $  probability distribution of $(xi_t, t\in {Ø,T})$ on $(C_T,\mathcal{ C}_T)$,
%If $A_i$ borelian sets in $\R^p$, and $A=\{x \in C_T,\; x(t_1) \in A_1,\dots, x(t_k)\in A_k)$, then
%$\P(\xi^T \in A)= P_{b,\sigma}^T(X_{t_1}\in A_1,\dots, X_{t_k}\in A_k\}$.
%The Wiener measure $W^T$ is the distribution of $(B_t, t\in [0,T])$ on $(C_T,\mathcal{ C}_T)$.
\subsection{Continuously observed diffusions on [0,T]}
The distributions $  P_{b,\sigma} P_{b',\sigma'} $ of two diffusion processes having distinct diffusion coefficients are singular.
 Therefore, we assume that $\sigma(\cdot)= \sigma'(\cdot)$. From a statistical point of view, this means that $\sigma(\cdot)$ can be identified from the continuous observation
 of $(\xi_t)$.
%
% on $\R$
%$ d\xi_t= b( \theta;\xi_t) dt +\sigma(\xi_t) dB_t, \xi_0= x_0$,
%where $\sigma (x), b(\theta,x)$  are known, $x_0$ known and  $\theta   \in \Theta$ is unknown.
%Let  $P_{\theta}^T$ the distribution on $(C_T,\mathcal{ C}_T)$ of $(\xi_t)$ and  let $P_0^T$  denote the distribution of $\xi_t=x_0+\int_0 ^t \sigma(\xi_s) ds$, provided assumptions on $b, \sigma$ ensuring existence, uniqueness of solutions of the SDE and some additional assumptions, then the Girsanov formula holds (see e.g.\ \cite{lip01}).
%\begin{theorem}
%For all $\theta$, the distributions $P_{\theta}^T$ and  $P_{0}^T$ are equivalent and
%\begin{equation} \label{Girsanov}
%\frac{dP^T_{\theta}}{dP^T_{0}}(X)=\exp [ \int_0 ^T\frac{b(\theta, X_t)}{\sigma^2(X_t)}dX_t -\frac{1}{2} \int_0 ^T\frac{b^2(\theta,X_t)}{\sigma^2( X_t)}dt].
%\end{equation}
%\end{theorem}
%%\end{block}
%In \eqref{Girsanov}, the stochastic integral  is obtained with respect to the canonical process $(X_t)$.
%Under $P_0^T$, $ \int_0^t \frac{dX_s}{\sigma^2(X_s)}ds$ is a  Brownian motion, and
%under $P_{\theta}^T$, $\int_0^t \frac{dX_s-b(\theta,X_s) ds}{\sigma^2(X_s)}$ is a Brownian motion.
%
%%{ \bf Comments and extensions}
%%
%%$\bullet$ Diffusions having distinct diffusion coefficients $\sigma(x)$, $\sigma'(x)$ $\Rightarrow$ $P_{\sigma}$ and $P_{\sigma'}$ are singular distributions on $(C_T,\mathcal{ C}_T)$.\\
%%$\bullet$ Diffusion having distinct starting point $x_0,x'_0$ have singular distributions.
%
Consider the parametric model associated to the diffusion $(\xi_t)$ in $\R^p$:
%$$ d\xi_t= b(\theta,t\xi_T) dt + \sigma(\t,\xi_t)  d
%This theorem also holds for time-dependent multidimensional diffusion processes :
\begin{equation}\label{Appen-SDEgen}
 d\xi_t= b( \theta,t,\xi_t) dt +\sigma(t,\xi_t) dB_t, \xi_0= x_0.
 \end{equation}
 Define the diffusion matrix $ \Sigma (t,x)= \sigma (t,x)  \sigma^* (t,x).$\\
%Assume that the drift term is $b(\theta,t,x)$ and the diffusion coeficient  $\sigma(t,x)$ and

Consider the estimation of a $q$-dimensional parameter  $\theta \in \Theta$, with $\Theta$ a subset of $\R^q$.
Then, under  conditions ensuring existence and uniqueness of solutions (see e.g.\ \cite{kar00IV}) and additional assumptions for the Girsanov formula (cf.\ \cite{hop14IV}, \cite{lip01IV}) on $ C([0,T],\R^p),\mathcal{ C}_T)$,
\begin{align}
&L_T(\theta) =	\frac{dP^T_{\theta}}{dP^T_{0}}(X)\label{Girsanovgen}\\
&= \exp\left[ \int_0 ^T \Sigma^{-1}(t,X_t) b(\theta;t, X_t) dX_t -\frac{1}{2} \int_0 ^T \;  b^*(\theta;t,X_t)\Sigma^{-1}(t, X_t) b(\theta,t,X_t)dt\right].\nonumber
\end{align}
% . as $T\rightarrow \infty$.\\
The statistical model is  $(C_T, \mathcal{ C}_T, (P_{\theta,\sigma}^T,\theta \in \Theta))$.
%The likelihood function associated to the observation  of $(\xi_t)$ is \\
%\begin{equation}\label{likelcont}
%L_T(\theta) = \frac{dP^T_{\theta}}{dP^T_{0}}(X)=\exp [ \int_0 ^T\frac{b(\theta,t, X_t)}{\sigma^2(X_t)}dX_t -\frac{1}{2} \int_0 ^T\frac{b^2(\theta,t,X_t)}{\sigma^2( X_t)}dt].
%\end{equation}
The loglikelihood is $\ell_T(\theta) = \log L_T(\theta)$.
The Maximum Likelihood Estimator is $\hat{\theta}_T$ s.t.
	 \begin{equation}\label{MLEc}
	\ell_T(\hat{\theta}_T) = \sup\{\ell_T(\theta), \theta \in \Theta \} .
	 \end{equation}
 There is no  general theory for the properties of the MLE as $T \rightarrow \infty$, except in the case of ergodic diffusions.
%This can be illustrated by the   following examples.
%Consider the diffusion on $\R$
%$$d\xi_t= \theta f(t) dt +\sigma(t) dB_t;\; \xi_0=0,\quad \mbox{where }  f (\cdot), \sigma(\cdot) \mbox { are known functions }, \sigma (\cdot)>0.$$
%The loglikelihood is
%$$\ell_T(\theta)=\theta \int_0^T \frac{f(s)}{\sigma^2(s)} dX_s-\frac{\theta^2}{2}\int_0^T \frac{f^2(s)}{\sigma^2(s)} ds \Rightarrow  \hat{\theta}_T = \frac{\int_0^T \frac{f(s)}{\sigma^2(s)} dX_s}{\int_0^T\frac{f^2(s)}{\sigma^2(s)} ds}.$$
%Hence, under $P_{\theta_0}$,
%$$\hat{\theta}_T = \theta_0+ \frac{\int_0^T \frac{f(s)}{\sigma(s)}  dB_s} {\int_0^T \frac{f^2(s)}{\sigma^2(s)} ds} \Rightarrow
%\hat{\theta}_T \sim \mathcal{ N}(\theta_0,I_T^{-1}) \quad \mbox{with }  I_T= \int_0^T\frac{f^2(s)}{\sigma^2(s)} ds.$$
%%\hat{\theta}_T = \theta_0 + \frac{\int_0^T \frac{f(s)}{\sigma(s)} dB_s}
%The asymptotic behaviour as $T\rightarrow \infty$ depends on $I_T$.\\
% In the case of
%$f(t)=1, \sigma(t)=\sqrt{1+t^2}$ $ I_T= Arctan T \rightarrow \pi/2 $ as $T \rightarrow \infty$. The  MLE is not consistent.
%In the case $f(t)=1, \sigma(t)=1$, $ I_T= T^{-1}$. The MLE is consistent and $\sqrt{T} (\hat{\theta}_T-\theta_0)\rightarrow \mathcal{ N}(0,1)$.
%

Consider the case of an autonomous diffusion $\xi_t$  satisfying the stochastic differential equation on $\R^p$:
$$d\xi_t= b(\theta,\xi_t) dt +\sigma(\xi_t) dB_t; \;  \xi_0 \simeq \eta.$$
Assume that, for $\theta \in \Theta \in \R^q$, $(\xi_t)$ positive recurrent diffusion  process with stationary distribution $\lambda (\theta,x) dx$ on $ \R^p$.
Then, under assumptions ensuring that the statistical model is regular (see
\cite{ibr81IV} for general results   and  \cite{kut04IV} for ergodic diffusions), then, as $T\rightarrow \infty$,
 the MLE $\hat{\theta}_T$ is  consistent and
 \begin{eqnarray*}
 \sqrt{T}(\hat{\theta}_T-\theta_0)&\overset{\mathcal{ L}} \rightarrow& \mathcal{ N}_k(0,I^{-1}(\theta_0)) \; \mbox{under }\P_{\theta_0} , \mbox { with }\\
I(\theta)&=&
 I(\theta)= \int_{\R^p} \; \nabla_{\theta }b^*(\theta,x) \Sigma^{-1}(x) \nabla_{\theta}b(\theta,x) \lambda(\theta,x) dx.
 \end{eqnarray*}

\subsection{Discrete observations with sampling $\Delta$ on a time interval [0,T]}
Consider the stochastic differential equation \eqref{Appen-SDEgen}, where parameters in the drift are $\alpha$ and in the diffusion coefficient $\beta$.
\begin{equation}\label{SDEgen}
d\xi_t= b( \alpha,t,\xi_t) dt +\sigma(\beta,t,\xi_t) dB_t, \xi_0= x_0.
 \end{equation}
 Let  $ T=n\Delta$ and assume that the observations are obtained at times  $(t_i^n=  i\Delta; i=0,\dots n$) .\\
 The  space of observations is $((\R^p)^n, (\mathcal{ B}(\R^p))^n$. Let $\P_{\alpha,\beta}^n$ denote the distribution of the  $n$-tuple. Contrary to continuous observations, the probabilities $ \P_{\alpha,\beta}^n,\P_{\alpha',\beta'}^n$ are absolutely continuous, leading to a  likelihood $L_n(\alpha,\beta)$ for  the $n$-tuple.  However, it depends on the transition probabilities    $ \P_{\theta}(X(t_{i+1} )\in A | X(t_i)= x)$ of the underlying Markov chain.
 %This  leads to an intractable likelihood in most cases.	
%Other approaches based on  Estimating equations or  contrast functions have been developed.
The main difficulty here lies in the intractable likelihood. This is a well known problem for discrete observations of diffusion processes.  Alternative approaches based on M-estimators or contrast processes (see \cite{vaa00IV} for i.i.d.\ observations, \cite{kes12IV} for SDE) have to be investigated.\\

%\textit{ Remark}  Because of the unknown transition densities of $\xi_t$, there is  no explicit likelihood. But  there is no attempt
%in the methods used below to complete the sample paths to study the estimation.

Several cases can be considered according to $T$ and $\Delta$  with $T=n\Delta$.\\

\noindent\underline{\textbf{(a) $T\rightarrow \infty$}}. Results are obtained for ergodic diffusions.\\

\noindent 1- \underline{ $\Delta$ fixed}:
 Both parameters in the drift coefficient $\alpha$ and in the diffusion coefficient $\beta$ can be consistently estimated and
(\cite{kes97IV}),
\begin{equation}
 \sqrt{n} \begin{pmatrix} \hat{\alpha}_n - \alpha_0\\
\hat{\beta}_n -\beta_0\end{pmatrix} \rightarrow  \mathcal{ N}(0, I_{\Delta} ^{-1}(\alpha_0,\beta_0).
\end{equation}\\

\noindent 2- \underline{$\Delta= \Delta_n \rightarrow 0 $  and $T= n\Delta_n \rightarrow \infty$  as $n\rightarrow \infty$}.
As $ n \rightarrow \infty$, there is a double asymptotics
 $\Delta_n \rightarrow 0$ and $T=n\Delta_n \rightarrow \infty $. \\
 Both parameters in the drift coefficient $\alpha$ and in the diffusion coefficient $\beta$ can be consistently estimated and the following holds  (see \cite{kes97IV} and \cite{kes00IV}
\begin{enumerate}
\item[-] Parameters in the drift coefficient $\alpha$ are estimated at rate $\sqrt{n\Delta_n}$.
\item[-] Parameters in the diffusion coefficient $\beta$  are estimated at rate $\sqrt{n}$.
\end{enumerate}
\begin{equation}
  \begin{pmatrix}\sqrt{n\Delta_n} (\hat{\alpha}_n - \alpha_0)\\
\sqrt{n}(\hat{\beta}_n -\beta_0)\end{pmatrix}\rightarrow  \mathcal{ N}(0, I ^{-1}(\alpha_0,\beta_0).)
\end{equation}\\

\noindent\underline{\textbf{(b)}  $T= n\Delta_n$ fixed and $\Delta= \Delta_n \rightarrow 0 $ as $n\rightarrow \infty$}.\\
It presents the following properties.
\begin{enumerate}
\item[-]  Except for specific models,  there is no consistent estimators for parameters in the drift.
\item[-]  Parameters in the diffusion coefficient can be consistently estimated and satisfy
   $$ \sqrt{n}(\hat{\beta}_n-\beta_0) \overset{\mathcal{ L}} \rightarrow  Z = \eta \; U, \mbox{ with  }  \eta, U  \mbox{ independent, }  U \sim \mathcal{ N}( 0,I). $$
   The random variable $Z$ is not normally distributed but  Gaussian but  has Mixed variance Gaussian law.
 It corresponds to a Local Asymptotic Mixed Normal statistical model (see \cite{vaa00IV}, \cite{hop14IV} for general references on LAMN; \cite{doh87IV}, \cite{gen93IV}
and \cite{gob01IV} for diffusion processes).
\end{enumerate}

\subsection{Inference for diffusions with small diffusion matrix on  $[0,T]$}

The asymptotic properties of estimators are now studied with respect to the asymptotic framework ``$\epsilon\rightarrow 0$''.
Consider the SDE
$$ d\xi_t= b(\alpha,\xi_t) dt + \epsilon \sigma(\xi_t) dB_t, \xi_0=x_0.$$
Contrary to the previous section, it is possible to estimate parameters in the drift  $\alpha$.\\
For continuous observations on $[0,T]$, Kutoyants (\cite{kut84IV}) has studied the estimation of $ \alpha$ using the likelihood and proved that the MLE is consistent and satisfies
\begin{align}
\epsilon^{-1} ({\hat \alpha}_{\epsilon} - \alpha_0)&\rightarrow  \mathcal{N}(0, I_b^{-1} (\alpha_0))\mbox{ with }\label{IFischkKut}\\
I_b(\alpha)&= \int_0^T (\nabla_{\alpha}b)^*(\alpha,z(\alpha,t))\Sigma ^{-1}( z(\alpha,t))\nabla_{\alpha}b(\alpha,z(\alpha,t))dt.\nonumber
\end{align}
The Fisher information of this statistical model is $I_b( \alpha)$.\\

\noindent
 The statistical inference based on discrete observations of the sample path with sampling interval $\Delta= \Delta_n\rightarrow 0 $ has
 first been studied for one-dimensional diffusions with $\sigma\equiv 1$ (\cite{gen90IV}), and \cite{sor03IV}, \cite{glo09IV} assuming  a parameter
$\beta$ in the diffusion coefficient $\sigma(\beta,x)$.
%\textit{High frequency sampling}: $\Delta=\Delta_n\rightarrow 0$ with $T=n\Delta_n$\\
%This implies that the number of observations $n \rightarrow \infty$
Under assumptions linking the two asymptotics $\epsilon$ and $n$, \cite{glo09IV}   proved the existence of  consistent and asymptotically Gaussian estimators
$(\tilde{\alpha}_{\epsilon,n},\tilde{\beta}_{\epsilon,n})$ of $(\alpha_0,\beta_0)$, which converge at different rates, parameters in the drift function being estimated at rate $\epsilon ^{-1}$ and parameters in the diffusion coefficient at rate $\sqrt{n}=\Delta_n^{-1/2}$.
\begin{equation}\label{CVR}	
\begin{pmatrix}
	    \epsilon^{-1}(\hat{\alpha}_{\epsilon,n}-\alpha_0)\\
	\sqrt{n}(\hat{\beta}_{\epsilon,n}-\beta_0)
	   \end{pmatrix}
	\underset{n\rightarrow \infty, \epsilon \rightarrow 0}
{\longrightarrow}
	\mathcal{N} \left(0,\begin{pmatrix} I^{-1}_b(\alpha_0,\beta_0)&0\\0&
	I^{-1}_\sigma(\alpha_0,\beta_0)\end{pmatrix}\right) .
\end{equation}
The matrix $I_b$ is the matrix \eqref{IFischkKut} and the matrix  $I_{\sigma}$ is
\begin{align}
	&I_{\sigma}(\alpha,\beta)_{ij}= \label{Isigma2}\\
&\left( \frac{1}{2T}\int_0^T \mathrm{Tr}(\nabla_{\beta_i}\Sigma(\beta,s,z(\alpha,s)) \Sigma^{-1}(\beta,s,z(\alpha,s))\nabla_{\beta_j}\Sigma(\beta,s,z(\alpha,s)) ds \right),\nonumber
\end{align}
%$I_{\sigma}= (I_{\sigma})_{ij}, 1\leq i,j\leq b)$ is
%\begin{align}
%	&I_{\sigma}(\alpha,\beta)_{ij}= \label{Isigma2}\\
%&\left( \frac{1}{2T}\int_0^T \mathrm{Tr}(\nabla_{\beta_i}\Sigma(\beta,s,z(\alpha,s)) \Sigma^{-1}(\beta,s,z(\alpha,s))\nabla_{\beta_j}\Sigma(\beta,s,z(\alpha,s)) ds \right),\nonumber
%\end{align}
 where $I_{b}(\alpha_0,\beta_0)$ and $I_{\sigma}(\alpha_0,\beta_0)$ are assumed invertible.
 % ($\mathrm{Tr}(A)$ denoting  the trace of a matrix $A$).\\
%\noindent
%\textit{Low frequency sampling}: fixed $\Delta$ \\
% Since $T=n\Delta$ is finite, this implies that the number  of observations $n$ is finite. Investigating the inference for a finite number of observations, $n$, is not classical. But this occurs in practice for epidemics, along with the population size asymptotics (\textit{i.e.} $\epsilon \rightarrow 0$).
%%If the sampling interval $\Delta$ is assumed fixed, .
%This is a framework that we investigated in guy14 and that will be summarized in section \ref{LFO}.

\section{Some limit theorems for martingales and triangular arrays}\label{LimitTheo}

\subsection{Central limit theorems for  martingales}

This  Central Limit Theorem for martingales in $\R$  is stated  in \cite{hal80IV}.\\

Let  $M_n= \sum_{i=1}^n X_i$ and  $\langle M\rangle_n= \sum_{i=1}^nE(X_{i}^2/\mathcal{ F}_{i-1})$.
Set  $s_n^2= EM_n^2=E\langle M\rangle_n$.
\begin{theorem}\label{TCLMart1}
Assume that the sequence $(M_n)$ of $L^2$ centered martingales satisfy that, as $n\rightarrow \infty$,  $s_n^2 \rightarrow \infty$ and
\begin{enumerate}[(H2):]
\item[\textbf{ (H1)}:] $\forall \epsilon>0, \frac{1}{s_n^2}\sum_{i=1}^n E(X_i^{2} {\mathbf 1}_{|X_i|\geq s_n\epsilon}|\mathcal{ F}_{i-1}) \rightarrow 0 $ in probability.
\item[\textbf{ (H2)}:] $\frac{1}{s_n^2}\langle M\rangle_n \rightarrow \eta^2 $ in probability ($\eta$ is an r.v.\ such that, if $\eta^2< \infty$, $E\eta^2=1$).
\end{enumerate}
Then $(\frac{M_n}{s_n}, \frac{\langle M\rangle_n}{s_n^2}) \rightarrow_\mathcal{ L} (\eta\; N, \eta^2)$
with $\eta, N$ independent r.v.s, $N \sim \mathcal{ N}(0,1)$.
\end{theorem}

\noindent
Note that  $Z= \eta N $ satisfies $E(\exp(iuZ))= E(\exp(-u^2\eta ^2/2))$.\\

\noindent
The Lindberg condition (H1) is often replaced by the stronger assumption:
\begin{enumerate}[(H1b):]
\item[\textbf{ (H1b)}:] $\exists \;\delta>0 ,\;\frac{1}{s_n^{2+\delta}} \sum_{i=1}^n E(|X_i|^{2+\delta}|\mathcal{ F}_{i-1}) \rightarrow 0 $ in probability.
\end{enumerate}

%\subsubsection{Multidimensional martingales in $\R^q$}
If the dimension of the parameter  is $q$, the  score function $\nabla_{\theta}\ell_n(\theta_0)$ is a $\P_{\theta_0}$-martingale in $\R^q$.
 So we need theorems for multidimensional martingales in $\R^q$.\\

\noindent
Let $(M_n)$ be a sequence of random variables in $\R^q$ with $  M_n^*=(M_n^1,\dots, M_n^q)$.
%$ M_n= \begin{pmatrix}M_n^1\\.\\.\\M_n^q\end{pmatrix}$ .\\
%and associated increments $X_n= M_n-M_{n-1}= \begin{pmatrix} X_n^1\\.\\.\\X_n^q\end{pmatrix}$.\\
%$ \in R^q$
Then $(M_n)$
%with $ M_n= \begin{pmatrix}M_n^1\\.\\.\\M_n^q\end{pmatrix} \in R^q$
is a $\mathcal{ F}_n$-martingale if $(M_n^p)$ is a $\mathcal{ F}_n$-martingale for $p=1,\dots q$. \\

\noindent
Assume that $(M_n)$ is a centered $L^2$-martingale in $\R^q$ and set  $X_i= M_i-M_{i-1}$ with $X_i^*= (X_i^1,\dots,X_i^q)$.\\
 Then  the increasing process $\langle M\rangle_n$ is the $q\times q$ random matrix defined by $  \langle M\rangle_0= 0$ and
  $\langle M\rangle_n-\langle M\rangle_{n-1}= E(X_n\;   X_n^* |\mathcal{ F}_{n-1})=
 \big( E(X_n^pX_n^l|\mathcal{ F}_{n-1})\big)_{1\leq p,l\leq q}$.\\
Hence, for $1\leq p$, $l\leq q$,  $\langle M\rangle_n^{pl}= \sum_{i=1}^n E(X_i^p \;X_i^l|\mathcal{ F}_{i-1})$.\\

% \begin{theorem}\label{CLTgen}
%Let  $ R_n$ denote the diagonal  positive matrix   $R_n= diag( s_n^p) $. Assume that there exists a positive random matrix $\Gamma$  such that, as $n\rightarrow \infty$ \\
% (H1) $R_n^{-1}  \langle M\rangle_n R_n^{-1} \rightarrow \Gamma $ in probability,\\
%(H2) For all $p \leq q$, $\forall \epsilon >0,  \frac{1}{(s_n^p)^2} \sum_{i=1}^n E(|X_i^p| ^2 \ind_{ |X_i^p|\geq s_n^p\epsilon}|\mathcal{ F}_{i-1}) \rightarrow 0 $
%in probability.\\
%Then $( R_n^{-1} M_n, R_n^{-1}\langle M\rangle_n R_n^{-1} \rightarrow_\mathcal{ L} (\Gamma ^{1/2}N_q\;,   \Gamma)$
%with $\Gamma, N_q $ independent r.v., $N_q \sim \mathcal{ N}_q(0, I_q)$..
%\end{theorem}
This theorem is derived from a convergence theorem  for triangular arrays stated in  \cite{jac12IV}.\\
For each $p$, assume that $ \E(\langle M_n^p\rangle)= (s_n^p) ^2 \rightarrow \infty $ and define
 $$\zeta_i^{n,p}=\frac{X_i^p}{s_n^p} \quad \mbox{and  } (\zeta_i^{n})^*= (\zeta_i^{n,1}, \dots \zeta_i^{n,q}).$$
\begin{theorem}\label{TCLMartmult}
%Let  $ R_n$ denote the diagonal  positive matrix   $R_n= diag( s_n^p) $.
Assume that there exists a  positive random matrix $\Gamma$  such that, as $n\rightarrow \infty$,
\begin{enumerate}[(H2):]
\item[\textbf{ (H1):}]  $\sum_{i=1}^n \E(\zeta_i^{n}(\zeta_i^{n})^*|\mathcal{ F}_{i-1})  \rightarrow  \Gamma$  in probability.
% (i.e.\  for $1\leq p,l\leq q$, $ \sum_{i=1}^n \E( \zeta_i^{n,p} \zeta_i^ {n,|} |\mathcal{ F}_{i-1} ) \rightarrow  \Gamma_{pl} $).\\
\item[\textbf{ (H2):}] There exists $\delta>0$, $\sum_{i=1}^n E(\norm{\zeta_i^n} ^{2+\delta}|\mathcal{ F}_{i-1} ) \rightarrow 0 $ in probability.
\end{enumerate}
%\frac{1}{(s_n^p)^2} \sum_{i=1}^n E(|X_i^p| ^2 \ind_{ |X_i^p|\geq s_n^p\epsilon}|\mathcal{ F}_{i-1}) \rightarrow 0 $
Then the following holds
\[\left(\sum_{i=1}^n \zeta_i^n, \sum_{i=1}^n \E(\zeta_i^n (\zeta^n_i)^* |\mathcal{ F} _{i-1}\right) \overset{\mathcal{ L} } \rightarrow
\left(\Gamma ^{1/2}N_q\;,   \Gamma \right)\]
with    $N_q \sim \mathcal{ N}_q(0, I)$ and  $\Gamma, N_q$   independent.
\end{theorem}

\noindent
 Here again, if $Z= \Gamma^{1/2} N_q  $, then, for $u \in \R^q $,
$E(\exp(iuZ))= E(\exp(-\frac{u^* \Gamma u}{2}))$.\\

\noindent

%\chapter{Appendix} \ref{sect:appenCh5}
\subsection{Limit theorems for triangular arrays}\label{sect:appenCh5}
When dealing with discrete observations with small sampling interval, classical limit theorems for martingales can no longer be used since the $\sigma$-algebras
$\mathcal{ G}_k^n= \sigma(Z(s),s\leq k/n) $ do not satisfy the nesting property.   We need general theorems for triangular arrays as  stated in \cite{jac12IV}.

\subsubsection{Recap on triangular arrays}
Let $(\Omega,\mathcal{ F}, (\mathcal{ F}_{t}, t\geq 0),\P)$ be a filtered probability space satisfying the usual conditions.
Assume that for each $n$, there is a strictly increasing sequence  $(T(n,k),k\geq 0)$ of finite $(\mathcal{ F}_t)$-stopping times with limit  $+\infty$ and $T(n,0)=0$.
The stopping rule is defined as
$$N_n(t)= \sup\{ k, T(n,k)\leq t \}  = \sum_{k\geq 1}1_{T(n,k)\leq t}.$$
 A $q$-dimensional triangular array is a double sequence $(\zeta_k^n), n,k \geq 1)$ of $q$-dimensional variables $\zeta_k^n=(\zeta_k^{n,j})_{1 \leq j \leq q}$.
 such that each $\zeta_k^n$ is $\mathcal{ F}_{T(n,k)} $-measurable.

 We consider the behavior of the sums
$$ S_t^n= \sum_{k=1}^{N_n(t)} \;\zeta_k^n.$$
  %Define $T(n,k)= inf\{t,N_n(t)\geq k \} $.

 %\begin{definition}
  The triangular array is asymptotically negligible (A.N.) if \\
  $$\sum_{k=1}^{N_n(t)} \zeta_k^n \overset{u.c.p.}\rightarrow 0 \quad  \mbox {i.e.\  }
 \sup_{s \leq t} |\sum _{k=1}^{N_n(s)} \zeta_k^n | \overset{\P} \rightarrow 0.$$
%\end{definition}

 In the sequel, we assume that the  $ T(n,k)$ are non-random and set  $ \mathcal{ G}_k^n= \mathcal{ F}_{T(n,k)}$.\\
 The example we have in mind  consists in the deterministic  times
 \begin{equation}\label{exNnt}
  T(n,k)= \inf \{t,  [nt]  \geq k \Delta \}   \Rightarrow   N_n(t)= \sup\{ k,\; \frac{k \Delta}{n} \leq t\}.
 \end{equation}
%$T(n,k)= [k \Delta/n] ; N_n(t)=  [\frac{nt}{\Delta}] .
Triangular arrays often occur as follows: $\zeta_k^n$ may be a function of the increment $Y_{T(n,k)}- Y_{T(n,k-1)}$ for some underlying adapted c\`adl\`ag  process $Y$. For discretely observed diffusion processs, we have $\zeta_k^n= X(k\Delta/n)- X((k-1)\Delta/n).$
 %Hence   $\mathcal{ G}^n_k= \mathcal{ F}_{\frac{k}{n}}$.
 We first state a lemma proved in \cite{gen93IV}.
 %Assume that, for $n\geq 1, k\geq 1$, $\zeta_k^n$ is $\mathcal{ G}_k
\begin{lemma} \label{VGCJJ}
Let $\zeta_k^n, U$ be random variables with $\zeta_k^n$  being $\mathcal{ G}_k^n$-measurable. Assume that
\begin{enumerate}
\item[(i)] $\sum_{k=1}^n  \E(\zeta_k^n |\mathcal{ G}_{k-1}^n) \rightarrow U$  in  $\P$-probability,
\item[(ii)] $\sum_{k=1}^n  \E[(\zeta_k^n)^2 |\mathcal{ G}_{k-1}^n)] \rightarrow 0 $  in  $\P$-probability,
\end{enumerate}
Then
 $$ \sum_{k=1}^n \zeta_k^n  \rightarrow U \quad \mbox {in  $\P$-probability}.$$
\end{lemma}
%\begin{proof}This lemma is stated in \cite{gen93}  (Lemma 9 ) . For sake of clarity, we detail this proof. Setting\\
%\begin{eqnarray*}
%\xi_k^n=\zeta_k^n - \E(\zeta_k^n |\mathcal{ G}_{k-1}^n), &\quad & B_n= \sum_{k=1}^n \xi_k^n,\\
%C_n= \sum_{i=1}^n \E\left((\xi_k^n)^2| \mathcal{ G}_{k-1}^n\right) ,&\quad & D_n = \sum_{i=1}^n \E\left((\zeta_k^,)^2| \mathcal{ G}_{k-1}^n\right).
%\end{eqnarray*}
% It is enough to prove that $B_n \overset{\P} \rightarrow  0.$  We have that $\E(\xi_k^n| |\mathcal{ G}_{k-1}^n)=0$. Therefore, for $k\leq n$
% $M_k= \sum_{i=1}^k  (\xi_i^n)^2- \sum_{i=1}^k \E((\xi_i^n)^2| \mathcal{ G}_{i-1}^n) $ is a martingale with respect to $\mathcal{ G}_k^n$. Hence the Lenglart inequality
% (see e.g \cite{jac03}, I-3.30)  yields that $\P(B_n^2>a) \leq \frac{b}{a}+\P(C_n>b)$ for all $a,b>0$. Since $C_n \leq D_n $, Assumption (ii)  yields that
% $P(C_n>b)\rightarrow 0$ for all $b>0$.  Hence, the result.
%\end{proof}

\begin{corollary}
 Let $\zeta_k^n, U$ be $d$-dimensional random variables with $\zeta_k^n$  being $\mathcal{ G}_k^n$-measurable. Assume
 \begin{enumerate}
\item[(i)] $\sum_{k=1}^n  \E(\zeta_k^n |\mathcal{ G}_{k-1}^n) \rightarrow U$  in  $\P$-probability,
\item[(ii)] $\sum_{k=1}^n  \E[\norm{\zeta_k^n }^2  |\mathcal{ G}_{k-1}^n)] \rightarrow 0 $  in  $\P$-probability,
\end{enumerate}
Then
 $$ \sum_{k=1}^n \zeta_k^n  \rightarrow U \quad \mbox {in  $\P$-probability}.$$
\end{corollary}
\subsubsection{Convergence in law of triangular arrays}
Let  $(\zeta_k^n) $ be a triangular array of $d$-dimensional random variables such that  $\zeta_k^n$ is $\mathcal{ G}_k^n$-measurable.
\begin{theorem} Assume that $(\zeta_k^n)$ satisfy for $N_n(t)$  defined in \eqref{exNnt}
\begin{enumerate}
\item[(i)] $ \sum_{k=1}^{N_n(t)}   \E(\zeta_k^n|\mathcal{ G}_{k-1}^n) \overset{u.c.p.} \rightarrow A_t $ with  $A$ an $\R^d$-valued deterministic function.
\item[(ii)] $ \sum_{k=1}^{N_n(t)} \E(\zeta_k^{n,i} \zeta_k^{n,j} |\mathcal{ G}_{k-1}^n)-  \E(\zeta_k^{n,i} |\mathcal{ G}_{k-1}^n) \E(\zeta_k^{n,j}|\mathcal{ G}_{k-1}^n)
\overset{\P} \rightarrow C_t^{ij}  $ for $1\leq i,j\leq d$ and for all $t\geq 0$,
where $C=(C^{ij})$ is a deterministic  continuous $\mathcal{ M}^{+} _{d\times d} $-valued function.
\item[(iii)] For some $p >2$,   $\sum_{k=1}^{N_n(t)}  \E(\norm{\zeta_k^n}^p|\mathcal{ G}_{k-1}^n) \overset{\P} \rightarrow 0$.
\end{enumerate}
Then, we have
\begin{equation} \label{CLTTT}
\sum_{k=1}^{N_n(t)}  \zeta_k^n  \overset{\mathcal{ L}} \rightarrow A+Y, \mbox{ w.r.t.\  the Skorokhod topology,}
\end{equation}
where $Y$ is a continuous centered Gaussian process on $\R^d$ with independent increments s.t.  $ \E(Y_t^i Y_t^j)= C_t^{ij}$.
\end{theorem}

Remark: If (ii) holds for a single time $t$, the convergence  $\sum_{k=1}^{N_n(t)}  \zeta_k^n  \overset{\mathcal{ L}} \rightarrow A_t+Y_t$ for this particular $t$
fails in general. There is an exception detailed below (Theorem VII-2-36 of \cite{jac03IV}).

\begin{theorem}\label{TCLTb}
Assume that for each $n$, the variables $(\zeta_k^n, k\geq 1)$ are independent
and let $l_n$ be integers, or $\infty$. Assume that, for all $ i,j=1,\dots,d$ and for some $p>2$,\\
\begin{eqnarray*}
& \sum&_{k=1}^{l_n} \E(\zeta_k^{n,i}) \overset{\P} \rightarrow  A_i ,\\
& \sum&_{k=1}^{l_n}\left( \E (\zeta_k^{n,i} \zeta_k^{n,j} ) -  \E(\zeta_k^{n,i}) \E(\zeta_k^{n,j}) \right)  \overset{\P}\rightarrow C^{ij},\\
& \sum&_{k=1}^{l_n} \E(\norm{\zeta_k^{n}} ^p) \overset{\P} \rightarrow 0,
\end{eqnarray*}
where $  C^{ij} $ and $A_i$ are deterministic numbers. Then the variables
$ \sum_{k=1}^{l_n} \zeta_k^{n}$ converge in distribution to a Gaussian vector with mean  $A=(A^i)$ and covariance matrix $C=(C^{ij})$.
\end{theorem}

%%%%%%%%%%%%%%%%%%%%%%%%%%%%%%%%%%%%%%%%%%%%%%%%%
%%%%%%%%%%%%%%%%%%%%%%%%%%%%%%%%%%%%%%%%
%%%%%%%%%%%%%%%%%%%%%%%%%%%%%%%%%%%%%%%%%%

\section{Inference for pure jump processes}\label{LikPJP}

In statistical applications, we study likelihood ratios formed by taking Radon--Nikodym derivatives of members of the family of probability measures
$(P_{\theta},\theta \in \Theta \subset \R^q)$ with respect to one fixed reference distribution.

\subsection{Girsanov type formula for counting processes}\label{Gtffcp}

Rather than giving the general expression of the Girsanov formula for semi-martin\-gales (see  \cite{jac03IV}), we  state it first for the case of a counting process on $\N$ and then for multivariate counting processes.\\

Let $X$ be a stochastic process such that the predictable compensator $\Lambda$ of $X$ satisfies  $\Lambda(t)=\int_0^t\lambda (s) ds)$. assume that,
 under $\P_{\theta}$, it  is a counting process
with intensity $\lambda^{\theta}(t)$ where $\lambda^{\theta}(t)>0 $ for all $t>0$.
Denote by $T_1,T_2,\dots$ the sequence of jump times of $X$ and let $N(t)$ denotes the number of jumps up to time $t$. Then
\begin{equation}\label{Girsanovpp}
\frac{d\P_{\theta}}{d\P_{\theta_0}}|\mathcal{ F}_t= \exp\{\sum_{i=1}^{N(t)}[\log(\lambda^{\theta}(T_i))-\log(\lambda^{\theta_0}(T_i))]-
\int_0^t [\lambda^{\theta}(s)-\lambda^{\theta_0}(s)] ds. \} 	
\end{equation}

Consider now multivariate counting processes $N(t)=(N_1(t),\dots,N_k(t))$.
We refer to Jacod's  formula  (see e.g.\ Andersen \cite[II.7]{andersenborgangillkeidingIV}) for a general expression of two probability measures $\P,{\tilde \P}$ on a filtered probability space under which ${\mathbf N}$ has compensators $\bm{\Lambda}$, $\tilde{\bm{\Lambda}}$ respectively. Usually, we will have continuous or absolutely continuous compensators with intensities $\lambda_l(t),{\tilde \lambda_l(t)}$.
Since no jumps can occur simultaneously, the sequence of jump times $T_i$  is well defined, together with the mark $ J_i \in \{1,\dots,k \}$ ($J_i=l$ if the jump $T_i$ occurs in $N_l$ ($\Delta N_l(T_i) =1$).
The process $ N_.(t)= \sum _{l=1}^k N_l(t)$ is a counting process with compensator $\Lambda_.(t)= \sum_{l=1}^k \Lambda_l(t)$.
Assume ${\tilde {\P}}$ is absolutely continuous with respect to ${\P}$ (written ${\tilde {\P}}<< {\P}$).
\begin{theorem}\label{theoA51}
	Assume that ${\tilde {\P}}<< {\P}$. Then
	\[{\tilde \Lambda}_l << \Lambda_l\mbox{  for all }l=1,\dots,k,\quad  P\mbox{- a.s.}\]
	\[\Delta\Lambda_.(t)=1\mbox{ for any time } t \mbox{ implies }\Delta {\tilde \Lambda}_.(t)=1,\quad P\mbox{-a.s.}\]
	\begin{align*}
		&\frac{d{\tilde \P}}{d\P}|{\mathcal{ F}_t} = \frac{d{\tilde P}}{dP}|{\mathcal{ F}_0}\;
	\frac{\prod_{l=1}^k \prod_{s\leq t} {\tilde \lambda}_l(t)^{\Delta N_l(t)} \exp(-\int_0^t {\tilde \lambda}_{.}(s)ds)}{\prod_{l=1}^k \prod_{s\leq t} {\lambda}_l(t)^{\Delta N_l(t)} \;\exp(-\int_0^t { \lambda}_{.}(s)ds)}\\
	&= \frac{d{\tilde \P}}{d\P}|{\mathcal{ F}_0}\;\exp\left\{\sum_{l=1}^k \sum_{i=1}^{N(t)}\; [\log{\tilde{\lambda}_l(T_i)}-\log{\lambda}_l(T_i)]\Delta N_l(T_i)
	-\sum_{l=1}^k \int_0^t [\tilde{\lambda}_l(s)-\lambda_l(s)] ds\right\}.
\end{align*}
\end{theorem}
Note that the products in the above formula are just $\prod_n{\tilde \lambda}_{J_n}(T_n)$,
$\prod_n{ \lambda}_{J_n}(T_n)$.
%%\pagebreak

\subsection{Likelihood for Markov pure jump processes}

Let us consider a pure jump process with  state space  $E= \{0,\dots,N\}$ and $Q$-matrix ${\mathbf Q}= (q_{ij})$ observed up to time $T$.
The likelihood is
% using notations (\ref{Ntkl}),% (\ref{Ntk}),
\begin{equation}\label{likMPJ2}
	L_T({\mathbf Q})= \prod_{i=0}^N\prod_{j\neq i} q_{ij}^{N_{ij}(T)}\;\exp(-q_{ij} N_i(T)),
\end{equation}
where  the process $ N_{ij}(t)$ counts the number of transitions from state $i$ to state $j$ on the time interval $[0,t]$	and $N_i(t)$ is the time spent in state $i$
before time $t$:
$$N_i(t)= \int _0^t  \delta_{\{X(s)=i\}} ds.$$
We refer  to  \cite{jacobsenIV} for a complete study of Marked point processes.\\
This yields that the maximum likelihood estimator of ${\mathbf Q}$ is
\begin{equation}\label{MLEQ2}
	 \hat{q}_{ij}(T)= \frac{N_{ij}(T)}{N_{i}(T)}, \quad \mbox{for  }j\neq i  \quad \mbox{and  } N_i(T)>0 .
	\end{equation}
If $N_T(i)= 0$, the process has not been in state $ i$: there is no information about $q_{ij}$ in the observations and the MLE of $q_{ij}$ does not exist.
As for Markov chains with countable state space, $\hat{q}_{ij}(T)$ is the empirical estimate of $q_{ij}$.

\subsection{Martingale properties of likelihood processes}
In statistical applications, we want to consider a whole family of probability measures $\P$, not necessarily mutually absolutely continuous and therefore cannot apply the above theorem to obtain
$\frac{d{\tilde \P}}{d\P}|_{\mathcal{ F}_t}$ for each ${\tilde \P},\P $ considered. However, for any two probability measures ${\tilde \P},\P$, the measure ${\mathbf Q}=\frac{1}{2}({\tilde \P}+\P)$
dominates both ${\tilde \P}$ and ${\P}$. We can therefore calculate $d\P/d{\mathbf Q}$ and
$d{\tilde \P}/d{\mathbf Q}$ and finally set,
\begin{eqnarray*}
\frac{d{\tilde \P}}{d\P} & = & \frac{d{\tilde \P}}{d {\mathbf Q}}/\frac{d\P}{d{\mathbf Q}} \mbox{ where }
\frac{d \P}{d{ \mathbf Q}}>0,\\
\frac{d{\tilde \P}}{d\P} & = &
 \infty \mbox{ where }\frac{d \P}{d {\mathbf Q}}=0.
\end{eqnarray*}

Suppose now  that we have a statistical model $(\P_{\theta}, \theta \in \Theta)$ for some subset $\Theta\in \R^q$. Suppose that all $\P_{\theta}$ are dominated by a fixed probability measure ${\mathbf Q}$.
For simplicity, we assume that all the $\P_{\theta}$'s  coincide on $\mathcal{ F}_0$ and consider only the absolute continuous case: \\
under $\P_{\theta}$, ${\mathbf N}= (N_1,\dots,N_k)$ has compensator $\bm{\Lambda}^{\theta}= (\int \lambda_l^{\theta}), l=1,\dots,k )$ for certain intensity process $\lambda^{\theta}$. We consider the likelihood function as depending on both $t\in \R^+$ and $\theta \in \Theta$. Dropping the denominator in
Theorem \ref{theoA51} (which does not depend on $\theta$), we have that the likelihood at time $t$ as a function of $\theta$ is proportional to
\begin{eqnarray*}
	L(\theta,t)&=&\exp(-\sum_{l=1}^k \int_0^t \lambda_l^{\theta}(s) ds)\prod_{T_n\leq t} \lambda^{\theta}_{J_n}(T_n),\\
	&=& \exp \{\sum _{l=1}^k \int_0^t [\log \lambda_l^{\theta}(s) dN_l(s)-\lambda_l^{\theta}(s) ds]\}.
\end{eqnarray*}

\begin{remark}This is another expression of the general Girsanov formula given in the appendix of Part 1 of these notes.
%\cite[Section 5.5]{brittonpardouxIV}.
\end{remark}

The likelihood process $L(\theta,t)$ is a $({\mathbf Q},(\mathcal{ F}_t))$-martingale.
Indeed, let $Y$ a $\mathcal{ F}_s$ measurable random variable. We have
$ E_{{\mathbf Q}} (Y L(\theta,t)) = E_{{\mathbf Q}} ( Y \frac{d\P_{\theta}}{dQ})= \E_{\theta}(Y) =
E_{{\mathbf Q}} (Y L(\theta,s))$ since $Y \in \mathcal{ F}_s$. \\
Hence $E_{{\mathbf Q}}(L(\theta,t)|\mathcal{ F}_s) = L(\theta,s)$.

Consider now the log-likelihood
\begin{equation}\label{loglikcounting}
\log L(\theta,t)= \sum _{l=1}^k \int_0^t (\log \lambda _l ^{\theta}(s) dN_l(s)-\lambda_l^{\theta}(s) ds).
\end{equation}
The score process is defined as $\nabla_{\theta}\log L(\theta,t)$.
%$= (\frac{\partial L}{\partial \theta_j}(\theta,t), j=1,\dots,q)$.
Assuming that differentiation  may be taken under the integral sign, we get
%we differentiate with respect to $\theta$.
\begin{align}
	\nabla_{\theta_j} \log L(\theta, t)&=\frac{\partial}{\partial \theta_j}\log L(\theta,t)\label{scorecounting}\\
&= \sum _{l=1}^k\;\int_0^t \nabla_{\theta_j} \log \lambda_l^{\theta}(s)\; ( dN_l(s)-\lambda_l^{\theta}(s) ds ),\; j=1,\dots,q.\nonumber
\end{align}
Hence the score process is a $(P_{\theta},(\mathcal{ F}_t))$- local martingale in $\R^q$. It is a centered $L^2$-martingale with
associated predictable $q\times q$ matrix variation process
\begin{equation}\label{varscore}
	\langle\nabla_{\theta}\log L(\theta;\cdot)\rangle_{r,j}= \sum_{l=1}^k \int_0^t \nabla_{\theta_r} \log \lambda_l^{\theta}(s) \nabla_{\theta_j}\log \lambda_l^{\theta}(s) \;\lambda_l^{\theta}(s)  ds.
\end{equation}

The ``observed information'' at $\theta$ is obtained by differentiating again with respect to $\theta$. If differentiation can be taken under the integral sign, we get
\begin{align}
	\nabla ^2_{\theta_r\theta_j} \log L (\theta;t)=&\, \sum _{l=1}^k \;\int_0^t \nabla ^2_{\theta_r \theta_j j}\lambda_l^{\theta}(s) (dN_l(s)-\lambda_l(s) ds)\nonumber\\
& - \int_0^t \nabla_{\theta_r}\log \lambda_l^{\theta}(s)\nabla_{\theta_j}\log \lambda_l^{\theta}(s) \lambda_l^{\theta}(s) ds.\label{Fisherinfcounting}
	\end{align}
%\begin{equation}\label{Fisherinfcounting}
%\nabla ^2	\frac{\partial ^2 \log L}{\partial \theta_r \partial \theta_j}(\theta;t)= \sum _{l=1}^k \;\int_0^t %\nabla ^2(\lambda_l^{\theta}(s)) (dN_l(s)-\lambda_l(s) ds)- \int_0^t(\frac{\partial}{\partial \theta_r}\log %\lambda_l^{\theta}(s) )\frac{\partial}{\partial \theta_j}\log \lambda_l^{\theta}(s) ) \lambda_l^{\theta}(s) ds.
%\end{equation}
%Noting that
%$$<\frac{\partial}{\partial \theta}\log L(\theta;.)>(t)= -\sum_h \int_0^t (\frac{\partial}{\partial \theta}\log \lambda_h^{\theta}(s) )^2 \lambda_h^{\theta}(s)  ds,$$
Using (\ref{varscore}) yields that the compensator of the process $-\nabla^2\log L(\theta;\cdot)$ is  $\langle\nabla_{\theta}\log L(\theta;\cdot)\rangle$.
This is a version of a well-known result: the variance matrix of the score coincides with the expected information
matrix.

\newpage

\renewcommand{\bibsection}{\chapter*{References for Part IV}}

\end{document}